\documentclass{article}

\usepackage[
left=1.25in, right=1.25in]{geometry}
\usepackage[all]{xy}
\usepackage{graphicx}
\usepackage{indentfirst}
\usepackage{bm}
\usepackage{amsthm}
\usepackage{mathrsfs}
\usepackage{latexsym}
\usepackage{amsmath}
\usepackage{amssymb}

\newcommand{\grs}[1]{\raisebox{-.16cm}{\includegraphics[height=.5cm]{FD#1.eps}}}
\newcommand{\gra}[1]{\raisebox{-.4cm}{\includegraphics[height=1cm]{FD#1.eps}}}
\newcommand{\graa}[1]{\raisebox{-.6cm}{\includegraphics[height=1.5cm]{FD#1.eps}}}
\newcommand{\grb}[1]{\raisebox{-.8cm}{\includegraphics[height=2cm]{FD#1.eps}}}
\newcommand{\grc}[1]{\raisebox{-1.3cm}{\includegraphics[height=3cm]{FD#1.eps}}}

\newcommand{\gre}[1]{\raisebox{-2.3cm}{\includegraphics[height=5cm]{FD#1.eps}}}

\begin{document}

\title{Exchange relation planar algebras of small rank}
\author{Zhengwei Liu\\ zhengwei.liu@vanderbilt.edu\\ \small{Department of Mathematics, Vanderbilt University, Nashville, TN 37240, USA}}
\date{\today}
\maketitle

\theoremstyle{plain}
\newtheorem{theorem}{Theorem~}[section]
\newtheorem*{main}{Main Theorem~}
\newtheorem{lemma}[theorem]{Lemma~}
\newtheorem{proposition}[theorem]{Proposition~}
\newtheorem{corollary}[theorem]{Corollary~}
\newtheorem{definition}{Definition~}[section]
\newtheorem{notation}{Notation~}[section]
\newtheorem{example}{Example~}[section]
\newtheorem*{remark}{Remark~}
\newtheorem*{question}{Question}
\newtheorem*{claim}{Claim}
\newtheorem*{ac}{Acknowledgement}
\newtheorem*{conjecture}{Conjecture~}
\renewcommand{\proofname}{\bf Proof}

\begin{abstract}
The main purpose of this paper is to classify exchange relation planar algebras with 4 dimensional 2-boxes.
Besides its skein theory, we emphasize the positivity of subfactor planar algebras based on the Schur product theorem.
We will discuss the lattice of projections of 2-boxes, specifically the rank of the projections. From this point, several results about biprojections are obtained.
The key break of the classification is to show the existence of a biprojection.
By this method, we also classify another two families of subfactor planar algebras, subfactor planar algebras generated by 2-boxes with 4 dimensional 2-boxes and at most 23 dimensional 3-boxes; subfactor planar algebras generated by 2-boxes, such that the quotient of 3-boxes by the basic construction ideal is abelian. They extend the classification of singly generated planar algebras obtained by Bisch, Jones and the author.
\end{abstract}

\section{Introduction}
In \cite{Jon83}, Jones classified the indices of subfactors of type II$_1$ as follows,
$$\{4\cos^2(\frac{\pi}{n}), n=3,4,\cdots \}\cup [4,\infty].$$

One approach to the classification of subfactors is to treat the index. Thus the simplest subfactors are those of index less than 4 and then those of index between 4 and 5. An early result is the classification of subfactors of index at most 4, see \cite{Ocn88,GHJ,Pop94,Izu91}. This approach has been extremely successful in the hands of Haagerup \cite{Haa94} and others \cite{AsaHaa,Bis98,Izu91,SunVij,BMPS}. Recently the classification has been extended upto index 5, see \cite{ind50,ind51,ind52,ind53,ind54}.

Below index 4 a deep theorem of Popa's \cite{Pop90} showed that the $standard$ $invariant$
is a complete invariant of subfactors of the hyperfinite factor of type $II{_1}$.
Subfactor planar algebras were introduced by Jones as a diagrammatic axiomatization of the standard invariant \cite{JonPA}.
Other axiomatizations are known as Ocneanu's paragroups \cite{Ocn88} and Popa's $\lambda$-lattices \cite{Pop95}.

From the planar algebra perspective it seems far more
natural to say that the simplest subfactors are those whose standard invariants are
generated by the fewest elements satisfying the simplest relations.
The simplest subfactor planar algebra is the one generated by the sequence of Jones projections, also well-known as the Temperley-Lieb algebra, denoted by $TL(\delta)$, and $TL$ for short, where $\delta$ is the square root of the index.
The next most complicated planar algebras after Temperley-Lieb should be those generated
by a single element. See \cite{Wen,MPSD2n,Pet10,BMPS} for examples.

For a planar algebra $\mathscr{S}=\{\mathscr{S}_{n,\pm}\}_{n\in \mathbb{N}_0}$, an element in $\mathscr{S}_{n,\pm}$ is called an $n$-box.
Planar algebras generated by 1-boxes were completely analyzed by Jones in \cite{JonPA}. Subfactor planar algebras generated by a non-trivial 2-box were considered by Bisch and Jones, and
classified by them for $\dim(\mathscr{S}_{3,\pm})\leq 12$ in \cite{BisJon97}; for $\dim(\mathscr{S}_{3,\pm})=13$ in \cite{BisJon02}; by Bisch, Jones and the author for $\dim(\mathscr{S}_{3,\pm})=14$ in \cite{BJL}.
They are given by the crossed product group subfactor planar algebra $\mathscr{S}^{\mathbb{Z}_3}$, the $free$ $product$ of two $TL$'s, well known as Fuss-Catalan \cite{BisJonFC}; the crossed product subgroup subfactor planar algebra $\mathscr{S}^{\mathbb{Z}_2\subset \mathbb{Z}_5\rtimes\mathbb{Z}_2}$; BMW \cite{BirWen,Mur87,Wen}, precisely one family from quantum $Sp(4,\mathbb{R})$ and one from quantum $O(3,\mathbb{R})$, respectively.
The classification for $\dim(\mathscr{S}_{3,\pm})=15$ is still unclear. In these cases, we always have $\dim(\mathscr{S}_{2,\pm})=3$, since $\dim(\mathscr{S}_{2,\pm})^2\leq \dim(\mathscr{S}_{3,\pm})$.

In this paper, we hope to classify subfactor planar algebras generated by 4 dimensional 2-boxes.
Observe that the free product of the index 2 subfactor planar algebra and a subfactor planar algebra generated by a 2-box with 15 dimensional 3-boxes has 4 dimensional 2-boxes and 24 dimensional 3-boxes. So we can only expect a classification for at most 23 dimensional 3-boxes.

\begin{theorem}
Suppose $\mathscr{S}$ is a subfactor planar algebra generated by 2-boxes, with $\dim(\mathscr{S}_{2,\pm})=4$ and $\dim(\mathscr{S}_{3,\pm})\leq 23$, then
then $\mathscr{S}$ is one of the follows

(1)$\mathscr{S}^{\mathbb{Z}_4}$ or $\mathscr{S}^{\mathbb{Z}_2\oplus \mathbb{Z}_2}$;

(2a)$\mathscr{A}*TL$ or $TL*\mathscr{A}$, where $\mathscr{A}$ is generated by a non-trivial 2-box with $\dim(\mathscr{A}_{3,\pm})\leq 13$;

(2b)$\mathscr{B}*\mathscr{S}^{\mathbb{Z}_2}$ or $\mathscr{S}^{\mathbb{Z}_2}*\mathscr{B}$, where $\mathscr{B}$ is generated by a non-trivial 2-box with $\dim(\mathscr{A}_{3,\pm})\leq 14$;

(3)$\mathscr{S}^{\mathbb{Z}_2}\otimes TL$.
\end{theorem}

Another approach to the classification of planar algebras is to consider the relations of the generators, instead of the boundary of dimensions.
Several kinds of relations of 2-boxes appeared naturally in planar algebras generated by a non-trivial 2-box with at most 15 dimensional 3-boxes.
If $\mathscr{S}$ is a subfactor planar algebra generated by a non-trivial 2-box with $\dim(\mathscr{S}_{3,\pm})\leq 12$, then $\mathscr{S}_{3,+}/\mathscr{I}_{3,+}$ is abelian, where $\mathscr{I}_{3,+}$ is the basic construction ideal of $\mathscr{S}_{3,+}$, i.e., the two sided ideal of $\mathscr{S}_{3,+}$ generated by the Jones projection.
Motivated by this condition, we have the following classification.

\begin{theorem}
Suppose $\mathscr{S}$ is a subfactor planar algebra generated by 2-boxes, and $\mathscr{S}_{3,+}/\mathscr{I}_{3,+}$ is abelian,
then $\mathscr{S}$ is either depth 2 or the free product $\mathscr{A}_1*\mathscr{A}_2*\cdots *\mathscr{A}_n$, such that
$\mathscr{A}_1$ is Temperley-Lieb or the dual of $\mathscr{S}^{G_1}$, for a group $G_1$;
$\mathscr{A}_n$ is Temperley-Lieb or $\mathscr{S}^{G_n}$, for a group $G_n$;
$\mathscr{A}_m$, for $1<m<n$, is Temperley-Lieb or $\mathscr{S}^{G_m}$, for an abelian group $G_m$.
Verse Visa.
\end{theorem}

If $\mathscr{S}$ is a subfactor planar algebra generated by a non-trivial 2-box with $\dim(\mathscr{S}_{3,\pm})\leq 13$, then $\mathscr{S}$ is an $exchange$ $relation$ $planar$ $algebra$ \cite{Lan02}.
Motivated by the exchange relation, we have the following classification.

\begin{theorem}
Suppose $\mathscr{S}$ is an exchange relation planar algebra with $\dim(\mathscr{S}_{2,\pm})=4$, then $\mathscr{S}$ is one of the follows

(1)$\mathscr{S}^{\mathbb{Z}_4}$ or $\mathscr{S}^{\mathbb{Z}_2\oplus \mathbb{Z}_2}$;

(2)$\mathscr{A}*TL$ or $TL*\mathscr{A}$, where $\mathscr{A}$ is generated by a non-trivial 2-box with $\dim(\mathscr{A}_{3,\pm})\leq 13$;

(3)$\mathscr{S}^{\mathbb{Z}_2}\otimes TL$;

(4)$\mathscr{S}^{\mathbb{Z}_2\subset \mathbb{Z}_7\rtimes\mathbb{Z}_2}$.
\end{theorem}

The three classification results rely on a new approach to the complexity of subfactors, the rank of 2-boxes.
We will show that the rank of the $coproduct$ of two 2-box minimal projections is bounded by the number of length 2 paths between the two corresponding points in the principal graph, see Lemma \ref{copro connect}.
We can show that a subfactor planar algebra is a free product by looking at its principal graph, see Theorem \ref{virtual normalizer}.

In section \ref{back}, we recall some facts and notations about planar algebras. Some new results are obtained, especially a general form of Wenzl's formula \cite{Wen87}.
In section \ref{tensor product}, a diagrammatic interpretation of the tensor product is discovered based on the construction of an inner braid. It is related to the $flatness$ of a planar algebra with respect to two $biprojections$.
If a subfactor planar algebra contains two commuting and co-commuting biprojections, then the planar subalgebra generated by the flat parts with respect to the two biprojections forms a tensor product, see Theorem \ref{tensorembed}.
In section \ref{biprojections} , we prove the Schur product theorem for subfactor planar algebras, see Theorem \ref{P*P>0}. Based on it, a new equivalent definition of biprojections in given, see Theorem \ref{P*P=P}.
Consequently the support of the $pure$ $depth$ 2 $parts$ of an irreducible subfactor planar algebra is a biprojection, see Theorem \ref{P bipro}. From a von Neumann algbera perspective, that tells the existence of an intermediate subfactor which is the crossed product of the smaller factor by a Kac algebra. By the new definition of biprojections, we can talk about the biprojection generated by a 2-box, see Definition \ref{biprogen}. By the Schur product theorem, we show that the norm of the Fourier transform of a positive 2-box is achieved on the Jones projection, see Lemma \ref{norm1a}. Then we prove that the Fourier transform of the biprojection generated by a positive 2-box is the spectrum projection of the Fourier transform of the 2-box at its maximal spectrum, see Theorem \ref{maxcoeff}. This result generalises a well known result in representation theory, see Remark \ref{maxcoeffremark}.
In section \ref{code}, we discuss the construction and the decomposition of exchange relation planar algebras under the free product and the tensor product. We obtain two general constructions of exchange relation planar algebras, see Proposition \ref{exfreeprod},\ref{extensorprod}, and one family of exchange relation planar algebras, see Theorem \ref{ex2p2};
In section \ref{cla expa}, we prove the three main classification results.
\begin{ac}
The author would like to thank Vaughan F. R. Jones for suggesting and discussing this project, and Emily Peters for helpful conversations. The author was supported by DOD-DARPA grant HR0011-12-1-0009.
\end{ac}

\section{Preliminaries}\label{back}
We refer the reader to \cite{Jon12} for the definition of a planar algebra.
\begin{definition}
A subfactor planar algebra $\mathscr{S}=(\mathscr{S}_{n,\pm}, n\geq0)$ will be a spherical planar *-algebra over $\mathbb{C}$ with $\dim(\mathscr{S}_{n,\pm})<\infty$ for all $n$, $\dim(\mathscr{S}_{0,\pm})=1$, such that the Markov trace induces a positive definite inner product of $\mathscr{S}_{n,\pm}$.
\end{definition}
The dual of a planar algebra is given by switching its shading.
\subsection{Notations}
In a planar tangle, we use a thick string with a number $k$ to indicate $k$ parallel strings.
The distinguished intervals of a planar tangle are marked by $\$$'s (corresponding to $*$'s in \cite{Jon12}).

In this paper the planar algebra $\mathscr{S}=(\mathscr{S}_{n,\pm}, n\geq0)$ is always the standard invariant of an irreducible subfactor, which is automatically spherical. Equivalently we assume that $\dim(\mathscr{S}_{1,\pm})=1$.
Since we only work with $\mathscr{S}_{n,+}$, we write $\mathscr{S}_{n}$ for $\mathscr{S}_{n,+}$, and the dollar sign $\$ $ of a planar tangle is always in an unshaded region.
An element in $\mathscr{S}_{n}$ is written as a rectangle with the dollar sign on the left, called an $n$-box. The dollar sign and the boundary are omitted, if there is no confusion. For example, we may use $\gra{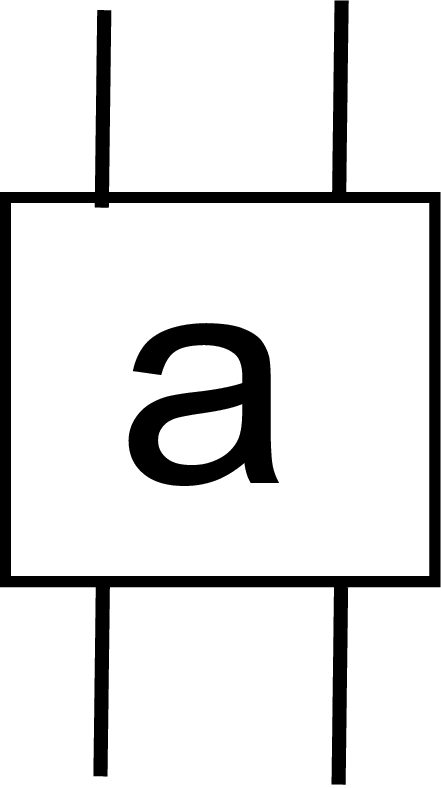}$ instead of $\gra{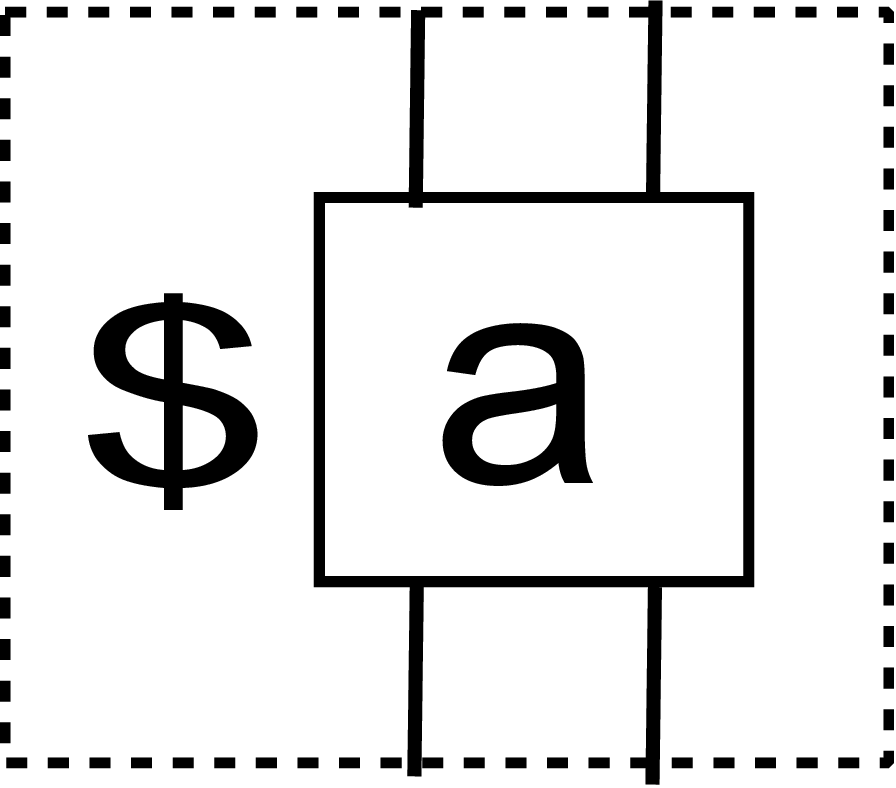}$, $\gra{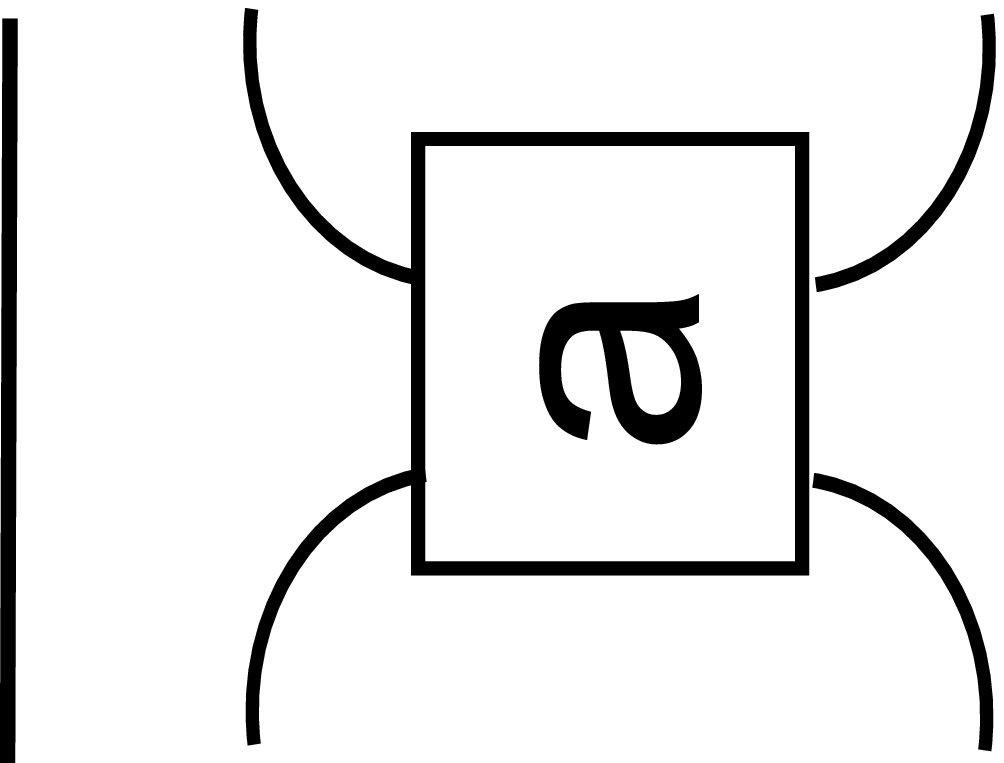}$ instead of $\gra{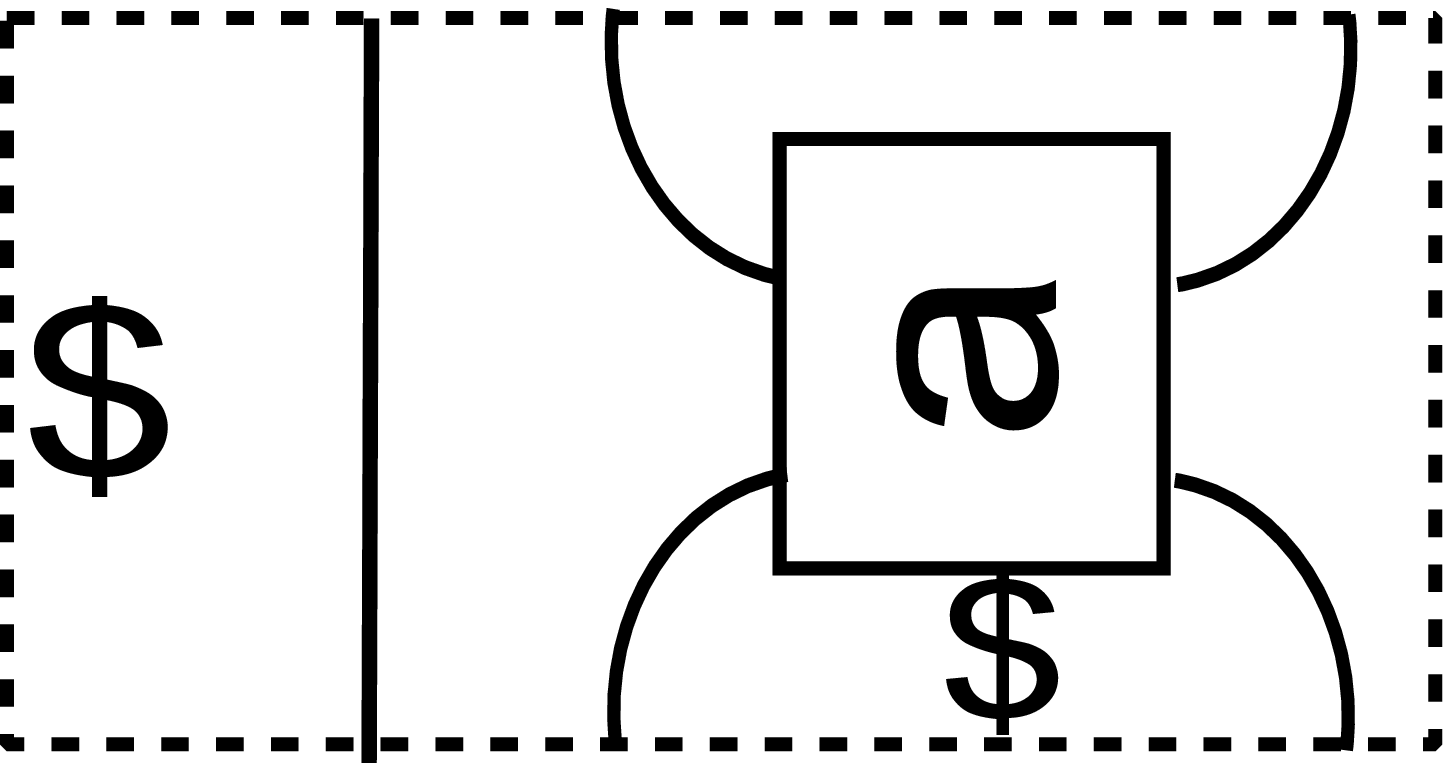}$.

The value of a closed circle is $\delta$;
$id$ means the identity of $\mathscr{S}_2$;
$e_n=\frac{1}{\delta}\gra{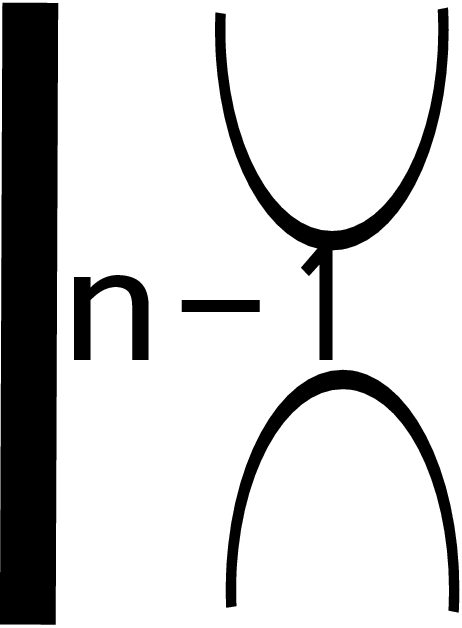}$ is the Jones Projection in $\mathscr{S}_{n+1}$;
$e=e_1$;
The (unnormalized) Markov trace on  $\mathscr{S}_n$ is denoted by $tr_n(x)=\gra{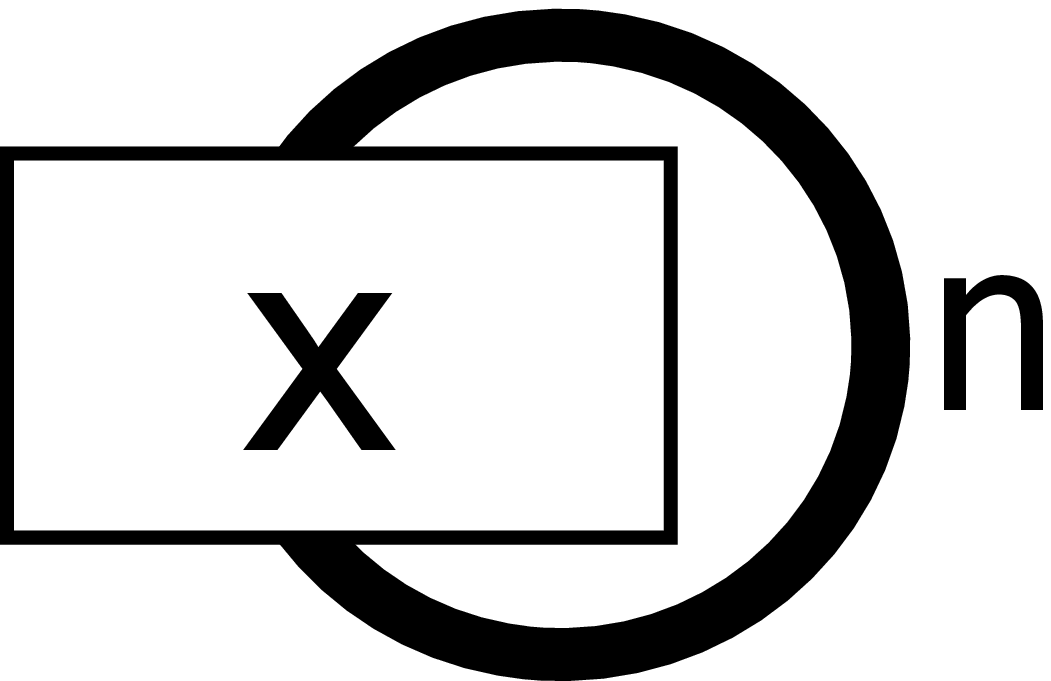}$, $\forall ~x\in\mathscr{S}_n$.
When $n=2$, we write $tr(x)$ for short.
For $a, b\in \mathscr{S}_n$, the product of $a$ and $b$ is defined as $ab=\graa{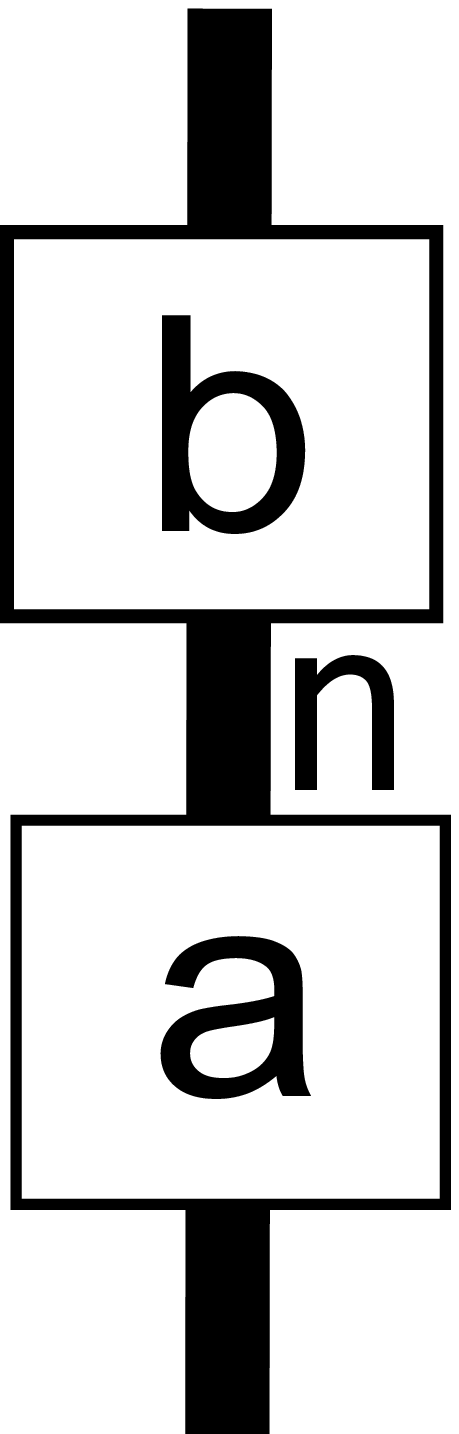}$.
If $a, b\in \mathscr{S}_2$,
then we define $a'=\gra{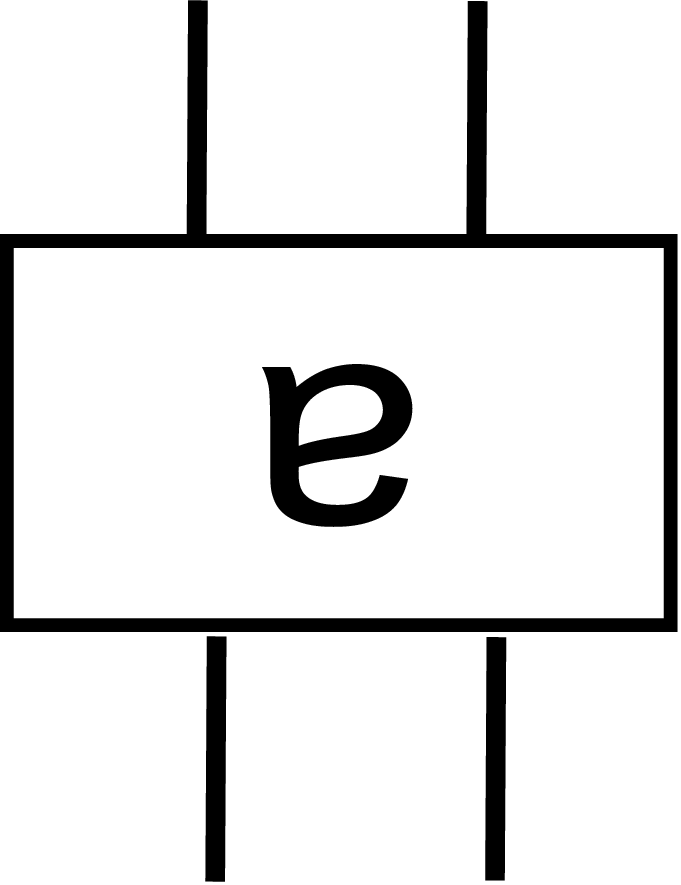}$ to be the contragredient of $a$,
$a*b=\gra{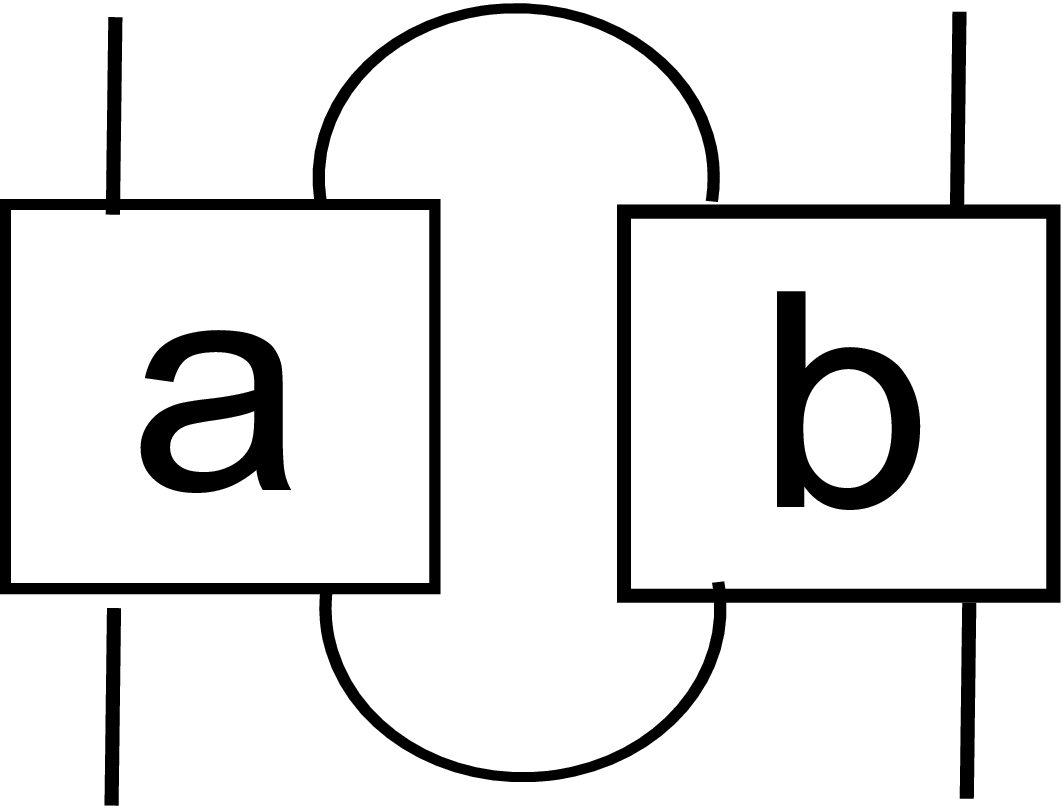}$ to be the ($1$-string) coproduct of $a$ and $b$;
Furthermore for $a_i\in\mathscr{S}_2,i=1,2,\cdots,k$,
We write $\prod*_{i=1}^k a_i$ for $a_1*a_2*\cdots *a_k$, and $a^{*k}$ for $\prod*_{i=1}^k a$.

Note that $\mathscr{S}_2$ is embedded in $\mathscr{S}_3$ by adding one string to the right.
Thus a 2-box $a$ can be viewed as an element in $\mathscr{S}_3$, still written as $a$.

The Fourier transform, i.e., the one click rotation, is an isomorphism from $\mathscr{S}_{2}$ to $\mathscr{S}_{2,-}$, and $\mathscr{S}_{2,-}$ is identified as a subspace of $\mathscr{S}_3$ by adding one string to the left.
Let us define $1\boxdot a$ to be the element $\gra{notation6}$ of $\mathscr{S}_3$, and $\mathscr{S}_{1,3}$ to be $\{1\boxdot z |z\in \mathscr{S}_2\}$. Then $\mathscr{S}_{1,3}$ is isomorphic to $\mathscr{S}_{2,-}$.
It is easy to check that $(1\boxdot a)(1\boxdot b)=1\boxdot (a*b)$, $e_2=\frac{1}{\delta}(1\boxdot id)$ and $1\boxdot a'$ is the adjoint of $1\boxdot a^*$, where $a^*$ is the adjoint of $a$.

\begin{definition}
For two self-adjoint operators $x$ and $y$, we say $x$ is weaker (resp. stronger) than $y$ if the support of $x$ (resp.$y$) is a subprojection of the support of $y$ (resp. $x$), written as $x\preceq y$ (resp. $y\succeq x$). If $x\preceq y$ and $y\preceq x$, then they have the same support, written as $x\sim y$.

\end{definition}
For a self-adjoint operator $x$ and a projection $p$, $x\preceq p$ is equivalent to $x=pxp$.

\begin{notation}
The support of a two sided ideal of a finite dimensional C*-algebra is the maximal projection in the ideal.
\end{notation}

\subsection{Principle graphs, Depth 2 Subfactors and Subgroup Subfactors}
We refer the reader to \cite{JS} for the definition of the (dual) principal graph of a subfactor.
It is also defined for a subfactor planar algebra, since it does not depend on the presumed subfactor \cite{Bis97}.

The principal graph and the dual principal graph are parts of the Ocneanu 4-partite principal graph \cite{Ocn88,JMN}.

Suppose $\mathscr{S}$ is the planar algebra of $\mathcal{N}\subset\mathcal{M}$.
Let us define $\mathscr{I}_{n+1}$ to be the two sided ideal of $\mathscr{S}_{n+1}$ generated by the Jones projection $e_n$, then $\mathscr{I}_{n+1}=\mathscr{S}_{n+1} e_n \mathscr{S}_{n+1}= \mathscr{S}_n e_n \mathscr{S}_n$. Let us define $\mathscr{S}_{n}/\mathscr{I}_{n}$ to be the orthogonal complement of $\mathscr{I}_n$ in $\mathscr{S}_n$,
there is a bijection between the equivalent classes of minimal projections of $\mathscr{S}_{n}/\mathscr{I}_{n}$ and points in the principal graph whose distance from the marked point is $n$.

\begin{definition}
In the principal graph, a point is said to be depth $n$, if its distance from the marked point is $n$.
Its multiplicity is the number of length $n$ paths from the marked point to it.
The depth of a principal graph is defined to be the maximal depth of its points.
\end{definition}

\begin{notation}
If the principal graph of a subfactor planar algebra is depth 2, equivalently $\mathscr{S}_3=\mathscr{I}_3$, then we call it a depth 2 subfactor planar algebra.
\end{notation}

There is a one to one correspondence between depth 2 subfactor planar algebras and finite dimensional Kac algebras, or finite dimensional C* Hopf algebras \cite{Sat97}\cite{KLS03}.
Precisely for any depth 2 subfactor planar algebra $\mathscr{S}$, $\mathscr{S}_2$ forms a Kac algebra.
On the other hand for any finite dimensional Kac algebra $K$, there is an
outer action of $K$ on the hyperfinite factor $\mathcal{R}$ of type $II_1$. Then $\mathcal{R}' \cap \mathcal{R}\rtimes K=\mathbb{C}$. Thus we obtain an irreducible subfactor planar algebra as the standard invariant of $\mathcal{R}\subset \mathcal{R}\rtimes K$, denoted by $\mathscr{S}^K$. Then $\mathscr{S}^K$ is depth 2, and $(\mathscr{S}^K)_{2}$ is isomorphic to the dual of $K$ as a Kac algebra.
Specially when $G$ is a finite group,
we obtain a subfactor planar algebra $\mathscr{S}^G$ of $\mathcal{R}\subset\mathcal{R}\rtimes G$.
If $H$ is a subgroup of $G$, then $\mathcal{R}\rtimes H$ is subfactor of $\mathcal{R}\rtimes G$. Thus we obtain a subfactor planar algebra of $\mathcal{R}\rtimes H\subset \mathcal{R}\rtimes G$.

\begin{definition}
Let us define $\mathscr{S}^{G}$ to be the planar algebra of the crossed product group subfactor $\mathcal{R}\subset \mathcal{R}\rtimes G$,
$\mathscr{S}^{H\subset G}$ to be the planar algebra of the crossed product subgroup subfactor $\mathcal{R}\rtimes H\subset \mathcal{R}\rtimes G$.
\end{definition}

The principal graph of a subgroup subfactor is described in \cite{JS}.
\subsection{Wenzl's formula}
Suppose $\mathscr{S}$ is a subfactor planar algebra,
$\mathscr{I}_{n+1}$ is the two sided ideal of $\mathscr{S}_{n+1}$ generated by the Jones projection $e_n$, and $\mathscr{S}_{n+1}/\mathscr{I}_{n+1}$ is its orthogonal complement in $\mathscr{S}_{n+1}$, then $\mathscr{S}_{n+1}=\mathscr{I}_{n+1} \oplus \mathscr{S}_{n+1}/\mathscr{I}_{n+1}$.
Let $s_{n+1}$ be the support of $\mathscr{S}_{n+1}/\mathscr{I}_{n+1}$.

If $\mathscr{S}$ is Temperley-Lieb (with $\delta^2\geq 4$), then $s_n$ is the $n_{th}$ Jones-Wenzl projection.
The following relation is called Wenzl's formula \cite{Wen87},
$$\grb{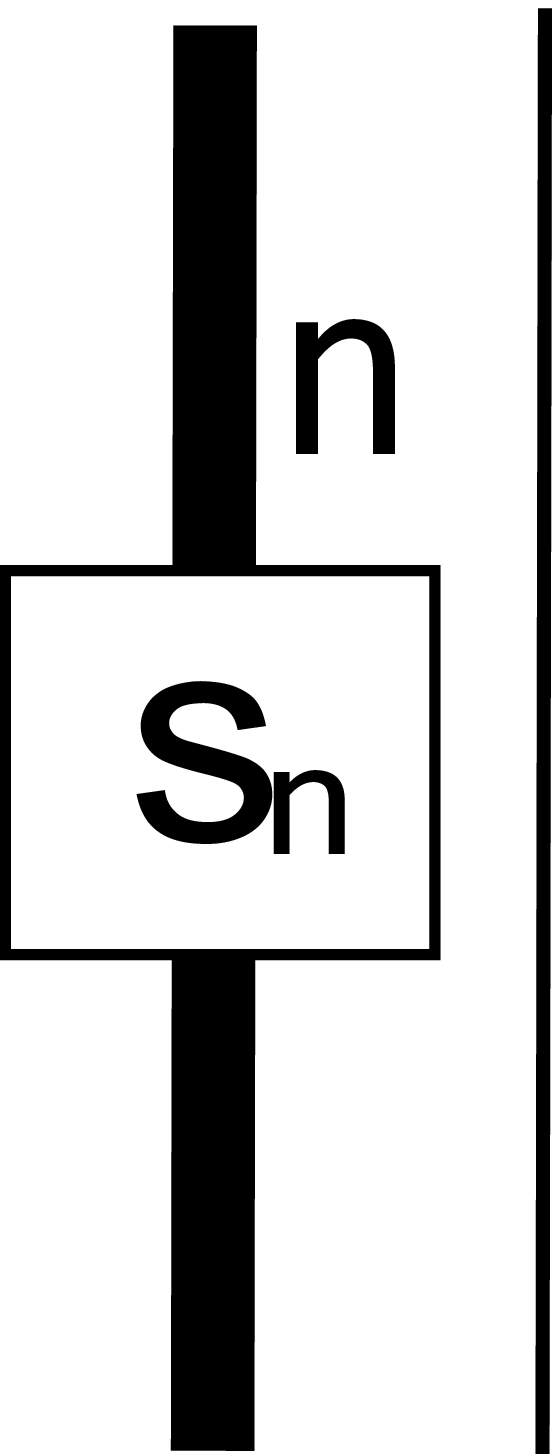}=\frac{tr_{n-1}(s_{n-1})}{tr_n(s_n)}\grb{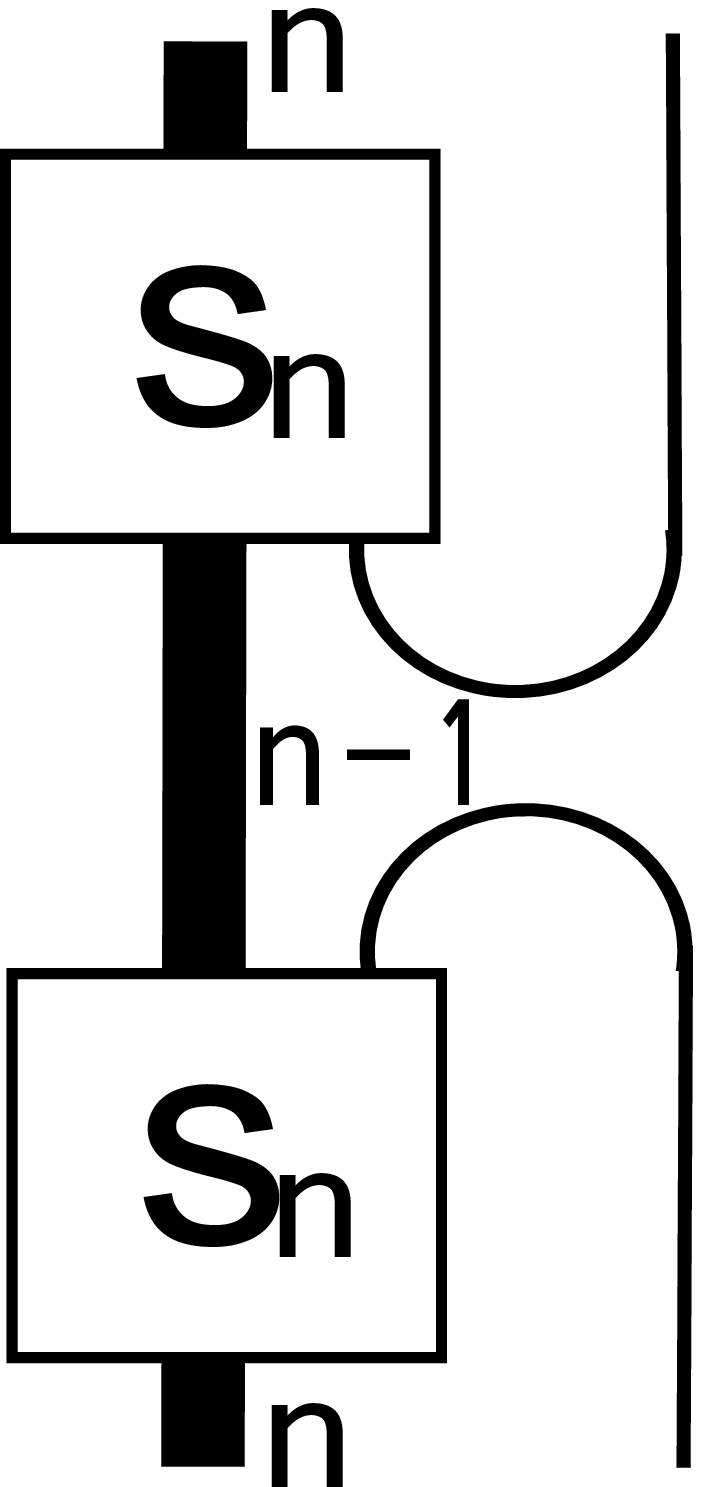}+\grb{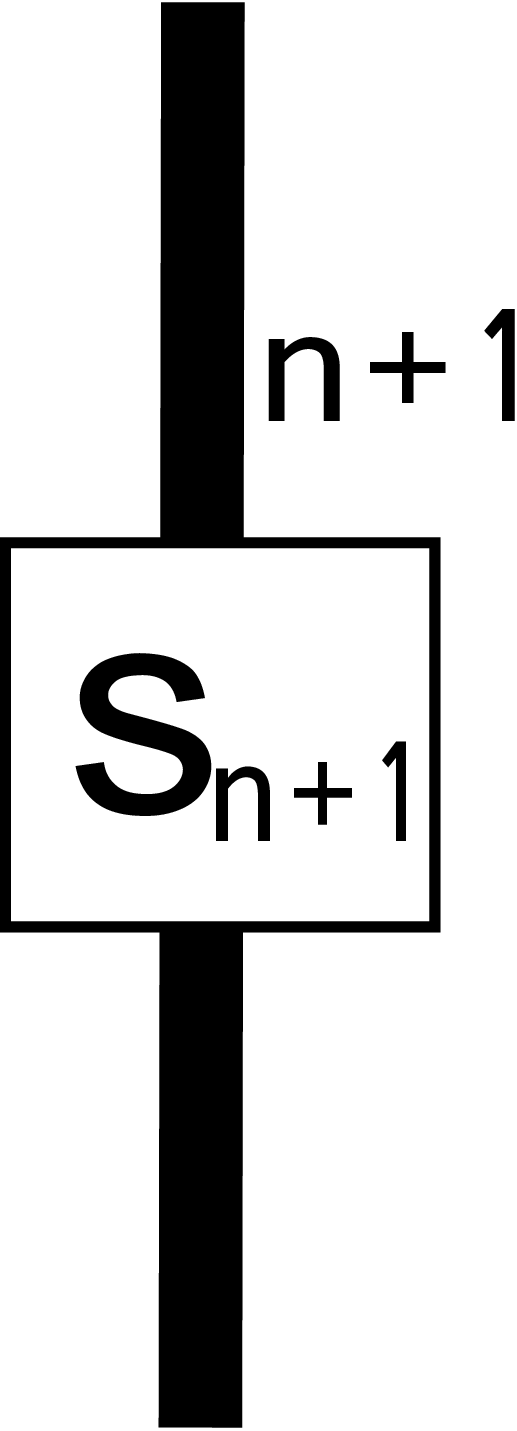}.$$
It tells how a minimal projection is decomposed after adding one string to the right.

In general, suppose $P$ is a minimal projection in $\mathscr{S}_n/\mathscr{I}_n$.
Note that $\mathscr{S}_{n+1}=\mathscr{I}_{n+1} \oplus \mathscr{S}_{n+1}/\mathscr{I}_{n+1}$. When $P$ is included in $\mathscr{S}_{n+1}$, it is decomposed as two projections $P=P_{old}+P_{new}$, such that
$P_{old}\in\mathscr{I}_{n+1}$ and $P_{new}\in\mathscr{S}_{n+1}/\mathscr{I}_{n+1}$. By the definition of $s_{n+1}$, we have $P_{new}=s_{n+1}P$. Now let us construct $P_{old}$.
Let $v$ be the depth $n$ point in the principal graph corresponding to P, $V$ be the central support of $P$.
Suppose $v_i$, $1\leq i \leq m$, are the depth $(n-1)$ points adjacent to $v$, the multiplicity of the edge between $v_i$ and $v$ is $m(i)$,
and $Q_i$ is a minimal projection in $\mathscr{S}_{n-1}$ corresponding to $v_i$.
For each $i$, take partial isometries $\{U_{ij}\}_{j=1}^{m(i)}$ in $\mathscr{S}_n$, such that $U_{ij}^*U_{ij}=P$, $\forall 1\leq j\leq m(i)$, and $\sum_{j=1}^{m(i)}U_{ij}U_{ij}^*=Q_iV$.
It is easy to check that $\frac{tr_{n-1}(Q_i)}{tr_n(P)}\grb{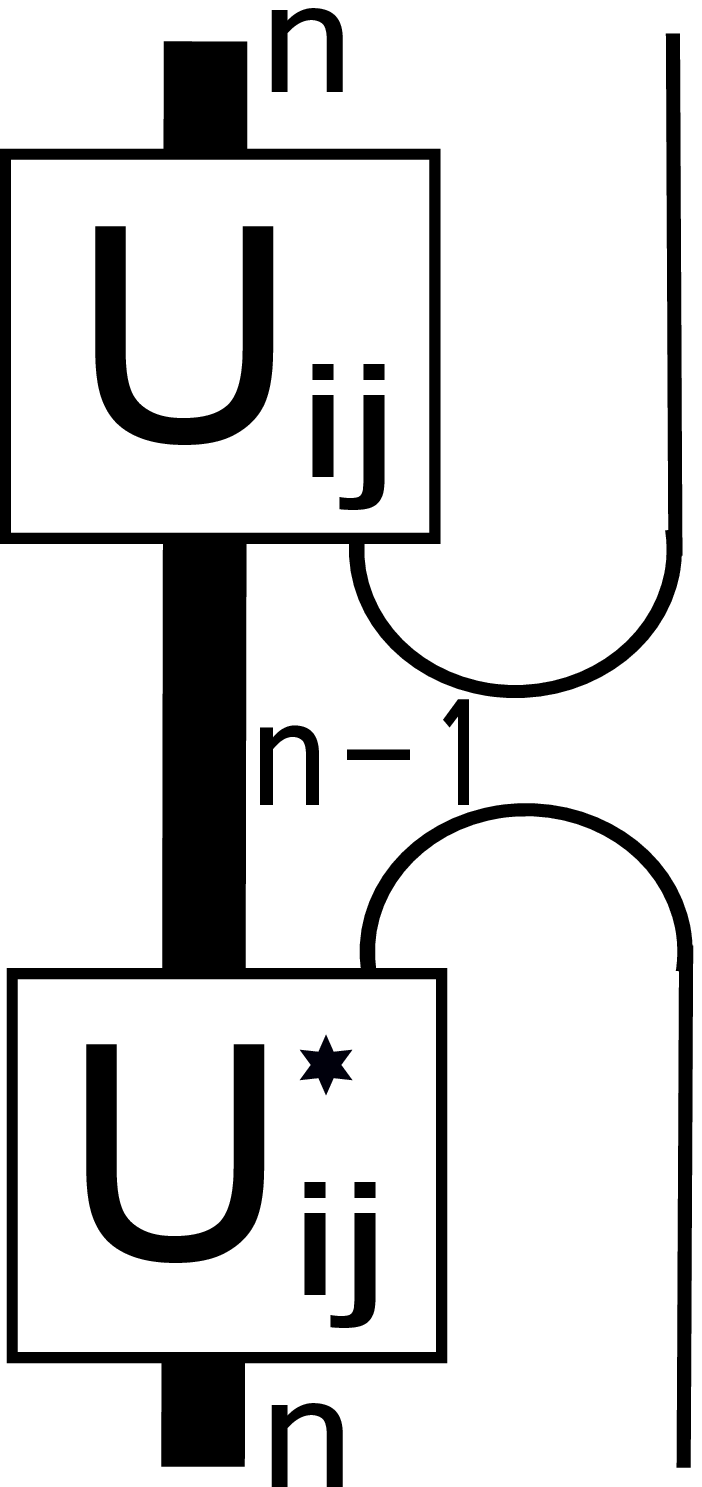}$ is a subprojection of $P$, and they are mutually orthogonal for all $i,j$.
By Frobenius reciprocity, their sum is $P_{old}$.
Then the general Wenzl's formula $P_{new}=P-P_{odd}$ is given as
$$s_{n+1}\grb{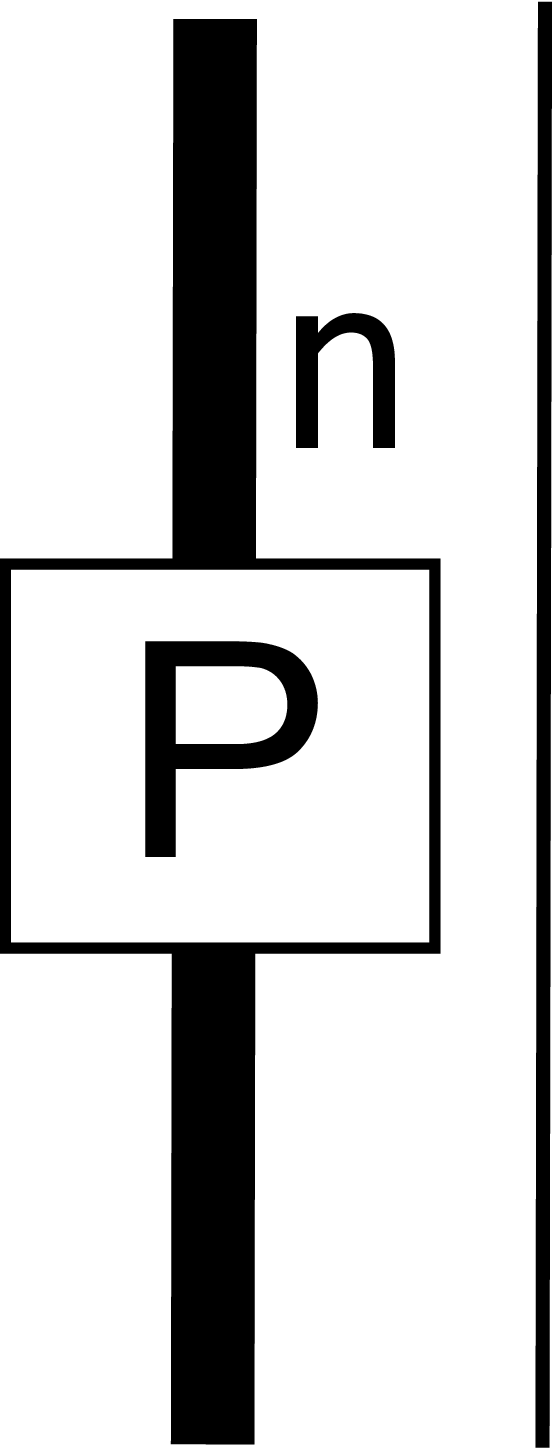}=\grb{wen4}-(\sum_{i=1}^m\sum_{j=1}^{n(i)} \frac{tr_{n-1}(Q_i)}{tr_n(P)}\grb{wen5}).$$

Now we give an alternative proof of the general Wenzl's formula without applying Frobenius reciprocity. This proof is very useful, since sometimes the planar algebra is constructed by generators and relations.
Based on this proof, we may derive the Bratteli diagram inductively without assuming it is a subfactor planar algebra.

\begin{proposition}
Take $P_+=\grb{wen4}-(\sum_{i=1}^m\sum_{j=1}^{n(i)} \frac{tr_{n-1}(Q_i)}{tr_n(P)}\grb{wen5})$. For any $x\in \mathscr{I}_{n+1}$, we have
$xP_+=P_+x=0$.
\end{proposition}

\begin{proof}
For $U_{ij}$, $1\leq i\leq m$, $1\leq j \leq n(i)$ ,$U_{kl}$, $1\leq k\leq m$, $1\leq l \leq n(k)$, we have
$$ \grb{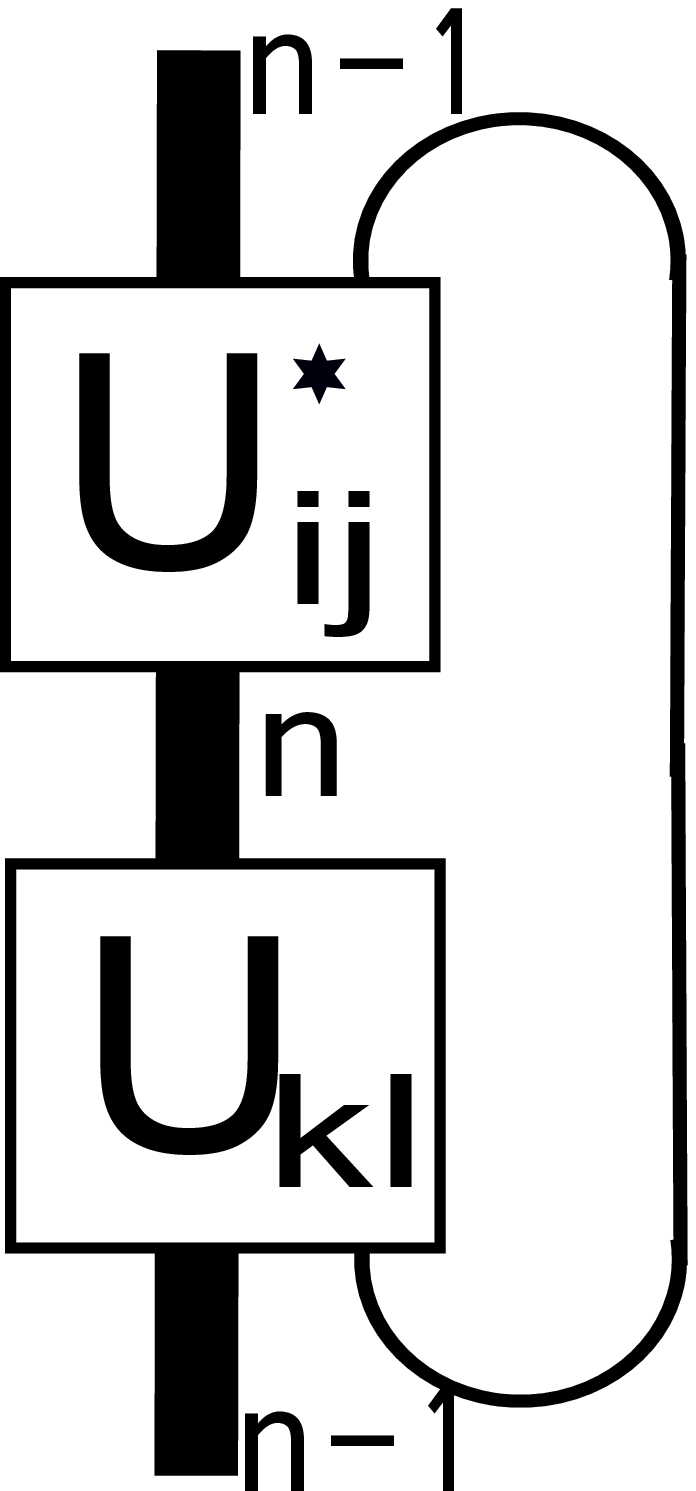} = Q_k \grb{wen6} Q_i = \delta_{i,k} Q_i \grb{wen6} Q_i=\delta_{i,k}\delta_{j,l}\frac{tr_n(P)}{tr_{n-1}(Q_i)}Q_i,$$
since $\{Q_i\}$ are mutually inequivalent minimal projections in $\mathscr{P}_{n-1}$, and the last equality follows from computing the trace.
So $e_{n}U_{i,j}p_+=e_{n}U_{i,j}-e_{n}U_{i,j}=0$.

For any $x\in \mathscr{I}_{n+1}$, we have $xP=\sum_{i=1}^m\sum_{j=1}^{n(i)} x_{i,j}e_{n}U_{i,j}$, for some $x_{i,j}\in\mathscr{P}_{n}$. So $xP_+=(xP)P_+=0$.
Similarly $P_+x=0$.
\end{proof}

\subsection{Tensor Products , Free Products and Biprojections}
Let us recall some facts about tensor products, free products and biprojections.

\begin{definition}
Suppose $\mathscr{A}$ and $\mathscr{B}$ are planar algebras, let us define $\mathscr{A}\otimes\mathscr{B}$ to be their tensor product \cite{JonPA}.
\end{definition}
The tensor product $\mathscr{A}\otimes\mathscr{B}=((\mathscr{A}\otimes\mathscr{B})_n,n\geq 0)$ is a planar algebra for which $(\mathscr{A}\otimes\mathscr{B})_n=\mathscr{A}_n\otimes\mathscr{B}_n$ and the action of an unlabeled tangle $T$ from $\otimes_{i=1}^k (\mathscr{A}\otimes\mathscr{B})_{n_i}$ to $(\mathscr{A}\otimes\mathscr{B})_m$ is defined as $T(\otimes_{i=1}^k (x_i\otimes y_i))=T(\otimes_{i=1}^k x_i)\otimes T(\otimes_{i=1}^k y_i)$, for any $x_i\otimes y_i\in \mathscr{A}_{n_i}\otimes\mathscr{B}_{n_i}, 1\leq i\leq m$.
If both $\mathscr{A}$ and $\mathscr{B}$ admit the adjoint operation $*$, then we define $(x\otimes y)^*$ to be $x^*\otimes y^*$.

\begin{proposition}
Suppose $\mathscr{A}$ and $\mathscr{B}$ are subfactor planar algebras. Then $\mathscr{A}\otimes\mathscr{B}$ is a subfactor planar algebras.
\end{proposition}

\begin{definition}
Suppose $\mathscr{A}$ and $\mathscr{B}$ are planar algebras. Let us define  $\mathscr{A}*\mathscr{B}$ to be the free product of $\mathscr{A}$ with $\mathscr{B}$ \cite{BisJonFC}\cite{BisJonfree}.
\end{definition}
Each element in the free product $\mathscr{A}*\mathscr{B}$ is a linear sum of A,B-colour diagrams which consist of non-intersecting A,B-colour strings, labels of $\mathscr{A}$ which only connect with A-coloured strings, and labels of $\mathscr{B}$ which only connect with B-coloured strings. The colour of its boundary points are ordered by $ABBA ABBA\cdots ABBA$. For an action of an unlabeled tangle $T$, we substitute each string of $T$ by a pair of parallel A,B-colour strings, then gluing the boundaries.

There is an equivalent definition of the free product. We say an element $x\otimes y\in (\mathscr{A}*\mathscr{B})_n$ is separated by a Temperley-Lieb $n$-tangle $T_n$, if $x$ can be written as a diagram in unshaded regions of $T_n$ and $y$ can be written a diagram in shaded regions of $T_n$. Then $\mathscr{A}*\mathscr{B}$ is the planar subalgebra of $\mathscr{A}\otimes\mathscr{B}$ consisting of all separated elements.
For example, the diagram $\gra{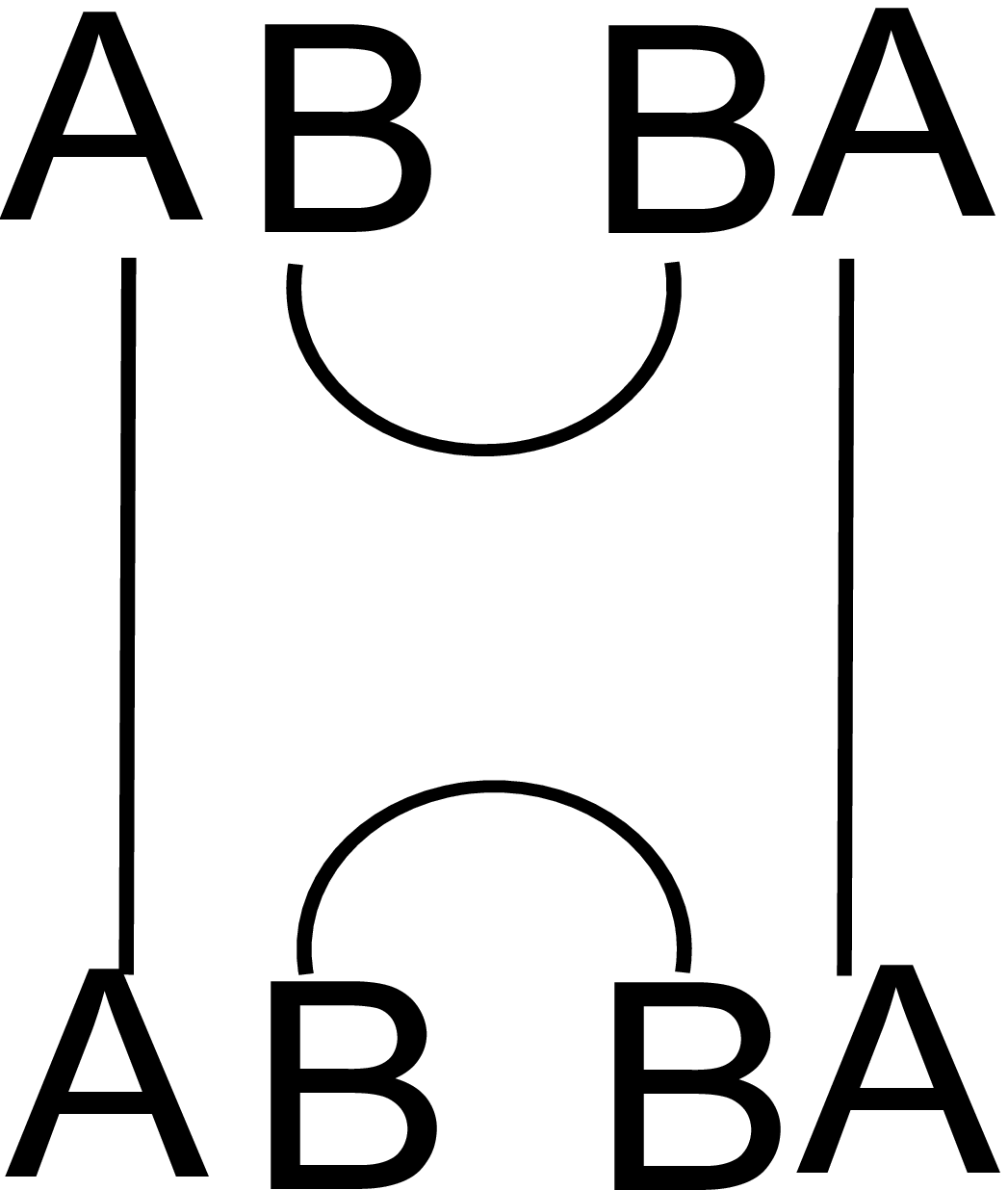}$ is separated by the tangle $\gra{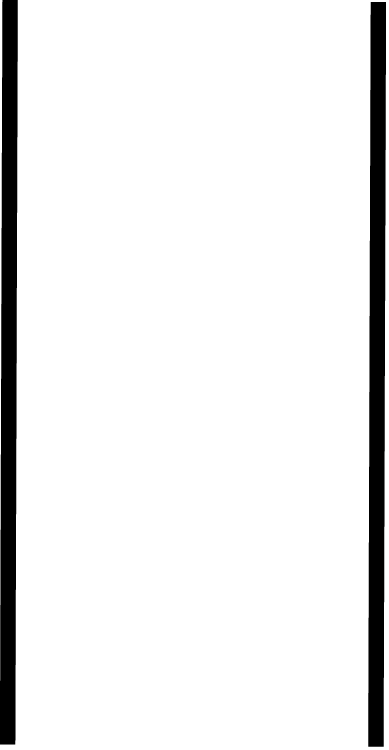}$, identified as $\delta_B id\otimes e$, where $\delta_B$ is the value of a circle of $\mathscr{B}$.
Consequently we have the following result.

\begin{proposition}
Suppose $\mathscr{A}$ and $\mathscr{B}$ are subfactor planar algebras. Then $\mathscr{A}*\mathscr{B}$ is a subfactor planar algebras.
\end{proposition}

\begin{notation}
The free product is associative but not commutative.
We say the free product of $\mathscr{A}$ with $\mathscr{B}$ for $\mathscr{A}*\mathscr{B}$;
We say the free product of $\mathscr{A}$ and $\mathscr{B}$ for either $\mathscr{A}*\mathscr{B}$ or $\mathscr{B}*\mathscr{A}$.
\end{notation}

\begin{proposition}
The dual of $\mathscr{A}*\mathscr{B}$ is the free product of the dual of $\mathscr{B}$ with the dual of $\mathscr{A}$.
\end{proposition}

A free product of Temperley-Lieb subfactor planar algebras is called $Fuss-Catalan$ \cite{BisJonFC}.

\begin{definition}
Suppose $Q$ is a projection in $\mathscr{S}_2$. If $1\boxdot Q$ is a multiple of a projection in $\mathscr{S}_{1,3}$, then we call $Q$ a biprojection.
In this case, $\frac{\delta}{tr(Q)}1\boxdot Q$ is the projection.
\end{definition}

There are two trivial biprojections, $e$ and $id$, in $\mathscr{S}_2$.

In a free product $\mathscr{A}*\mathscr{B}$, there is a non-trivial biprojection $\delta_B^{-1}\gra{bipro}$.

\begin{proposition}\label{ex of bipro}
Suppose $Q$ is a biprojection in a subfactor planar algebra. Then $Q$ satisfies
$$\gra{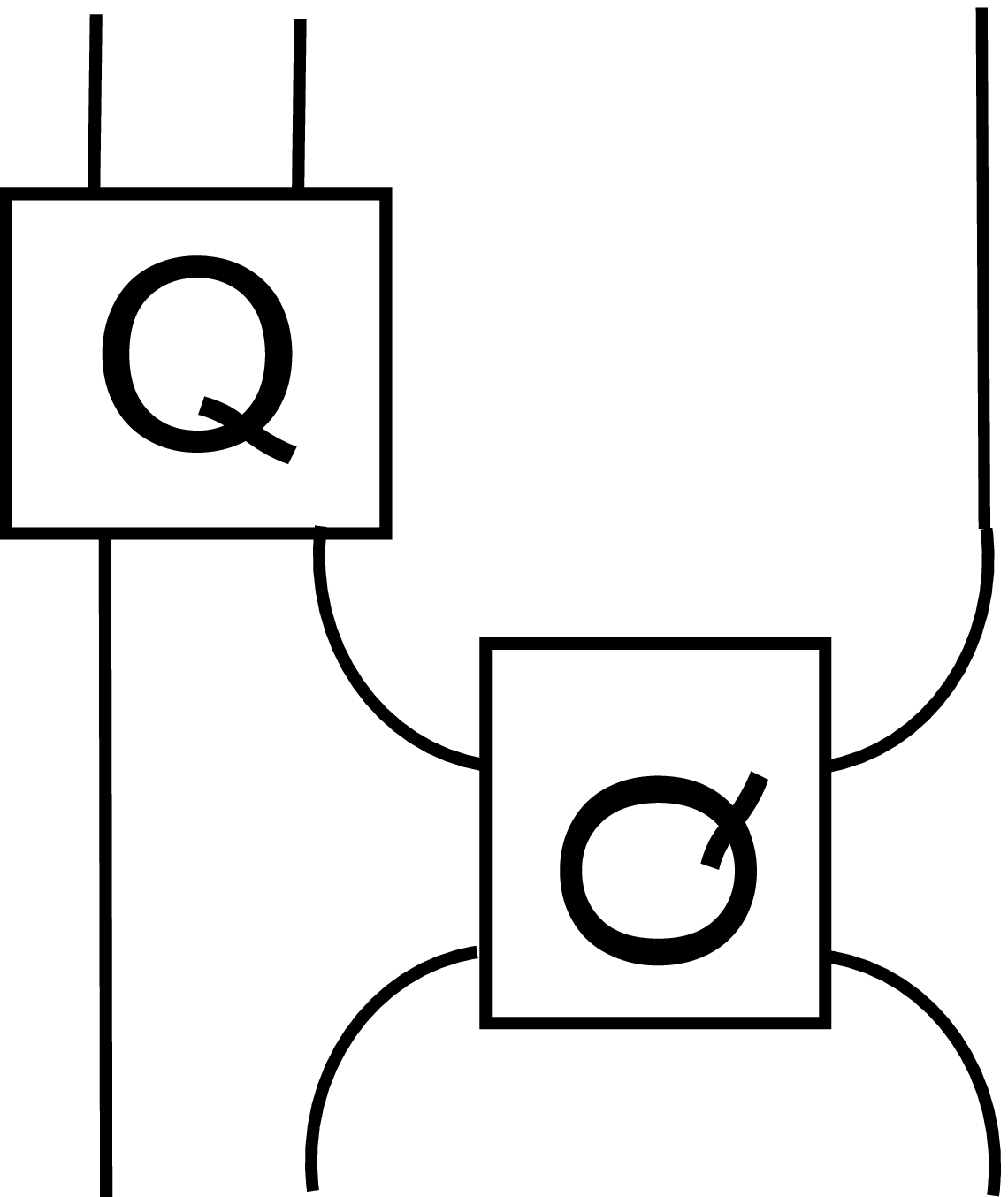}=\gra{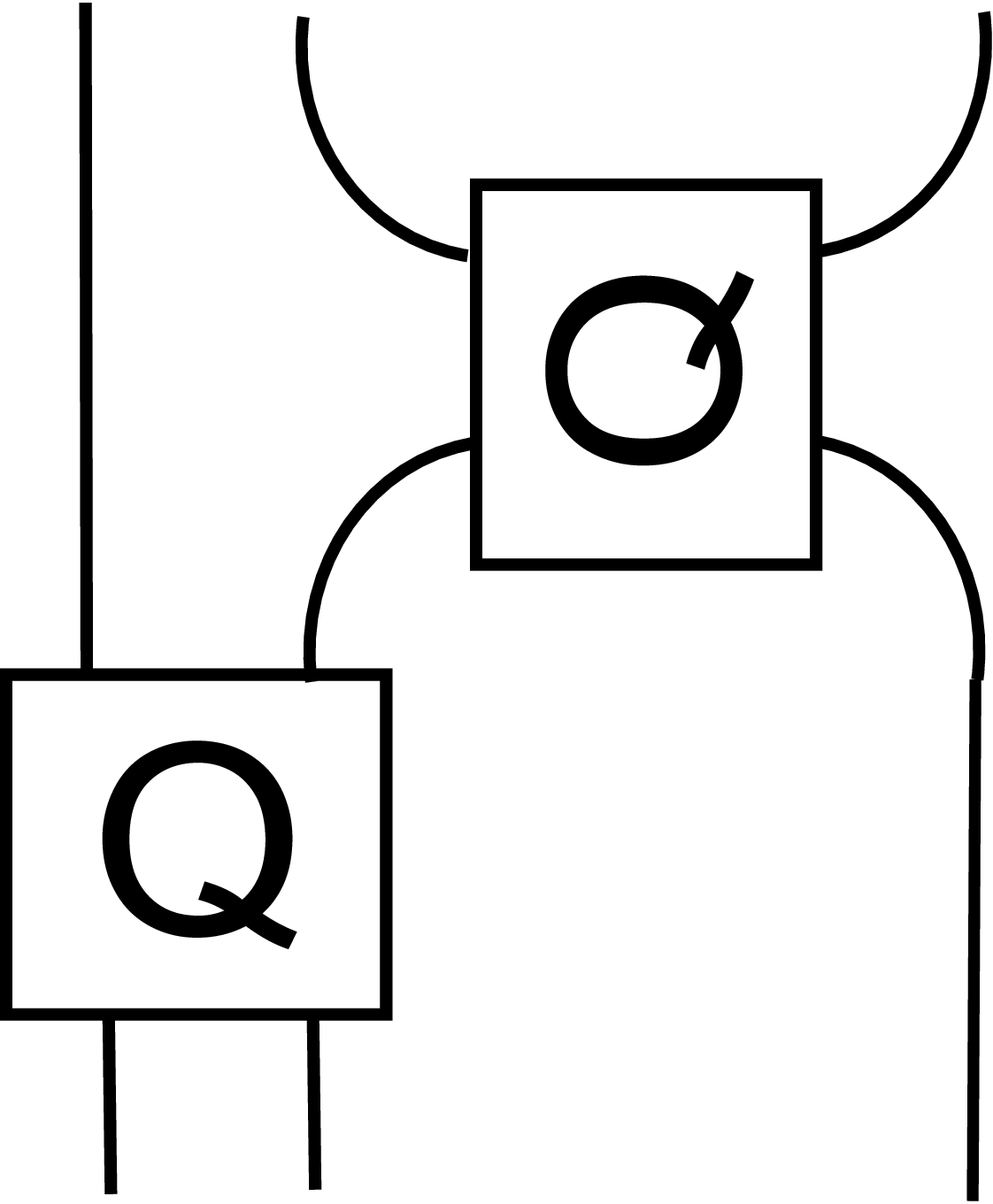},$$
called the exchange relation of the biprojection $Q$.
\end{proposition}
\begin{proof}
Let $x$ be $\gra{biproex1}-\gra{biproex2}$. Then it is easy to check that $tr_3(x^*x)=0$. By the positivity of the trace, we have $x=0$. That means $\gra{biproex1}=\gra{biproex2}$.
\end{proof}

A biprojection is discussed by Bisch while considering the projection onto an intermediate subfactor \cite{Bis94}.
Suppose $\mathscr{S}$ is the planar algebra of a subfactor $\mathcal{N}\subset \mathcal{M}$. Then each biprojection $Q$ in $\mathscr{S}_2$ corresponds to an intermediate subfactor $\mathcal{Q}$ of $\mathcal{N}\subset \mathcal{M}$, in the sense that $Q$ is the projection onto $L^2(\mathcal{Q})$ as a subspace of $L^2(\mathcal{M})$.

The planar algebra of $\mathcal{N}\subset \mathcal{Q}$ can be realised as a $Q$ cut down on shaded intervals of diagrams in $\mathscr{S}$, denoted by $\mathscr{S}_Q$,.
That means $(\mathscr{S}_Q)_n=\Psi_Q(\mathscr{S}_n)$, where $\Psi_Q$ is the annular action
$$\grb{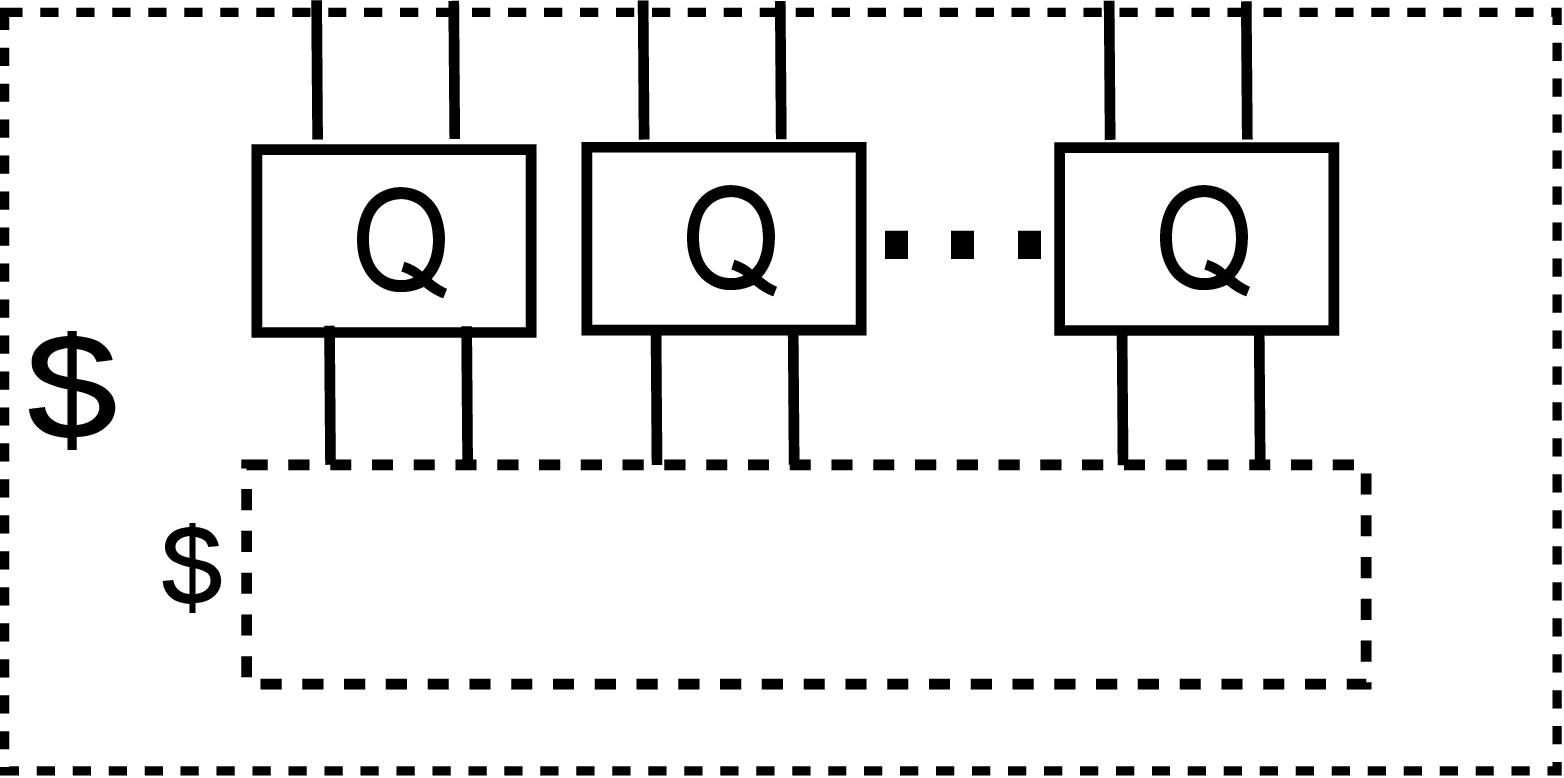};$$
and the action of an unlabeled tangle $T$ on $\mathscr{S}_Q$ is defined to be
$$(\frac{\delta}{\sqrt{tr(Q)}})^{fudge(T)}\Psi_Q\circ T,$$
where $fudge(T)=n-m$, $n$ is number of shaded intervals of outside boundary of $T$ and $m$ is the number of closed circles after adding a cap at each shaded interval of (outside and inside) boundary of $T$.
Considering the duality, the planar algebra of $\mathcal{Q}\subset\mathcal{M}$ is realised as a $\frac{\delta}{tr(Q)}Q$ cut down on unshaded intervals of diagrams in $\mathscr{S}$, denoted by $\mathscr{S}^Q$.

The meaning of the fudge factor is explained in the following proposition \cite{BisJonfree}.
\begin{proposition}\label{iso}
Let $\mathscr{A}$ and $\mathscr{B}$ be subfactor planar algebras with circle parameters $\delta_A$ and $\delta_B$ respectively. Form the free product $\mathscr{A}*\mathscr{B}$. Let $Q$ be the biprojection $\delta_B^{-1}\gra{bipro}$
and $v_n\in\mathscr{B}$ be $\gra{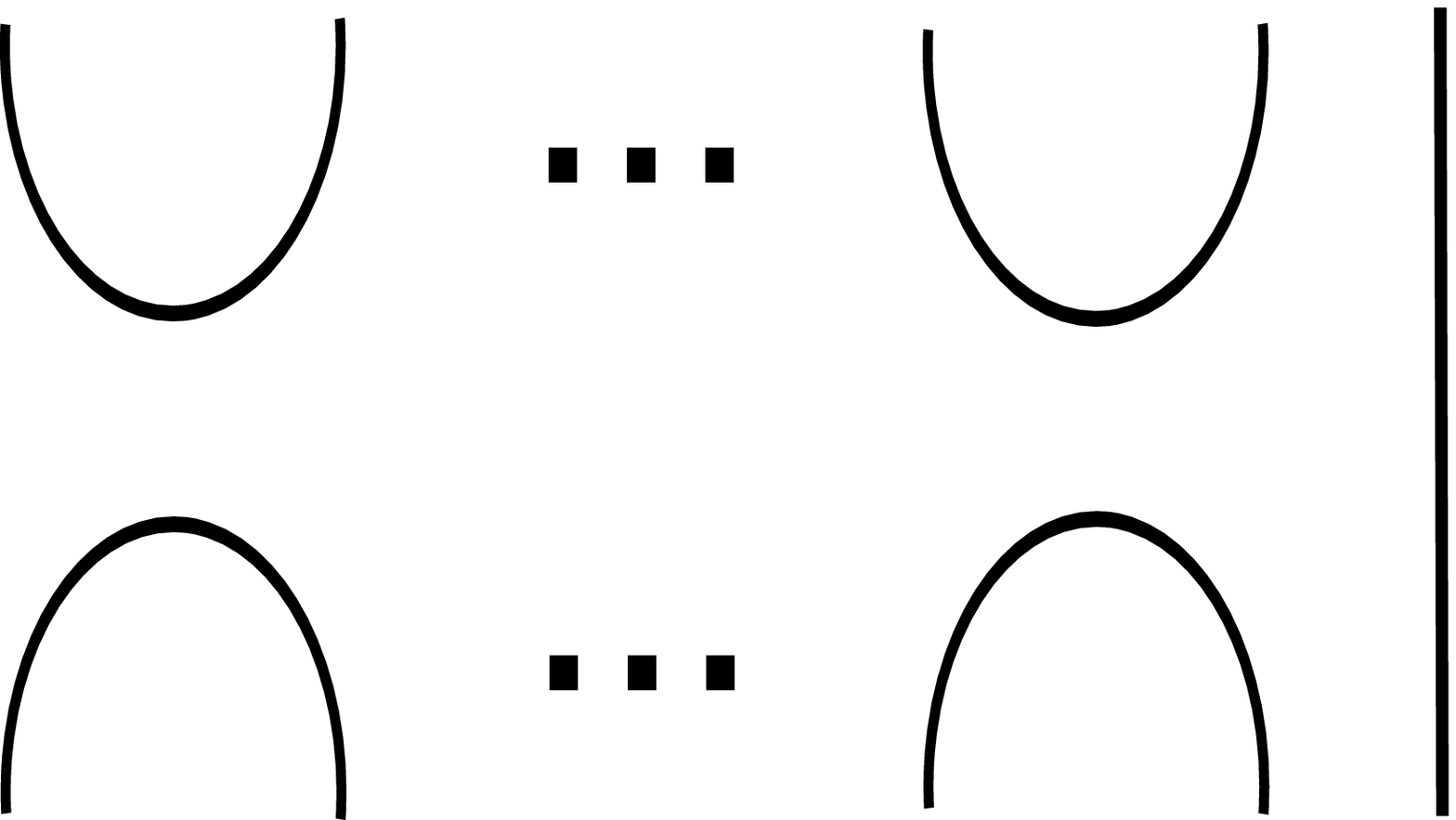}$ or $\gra{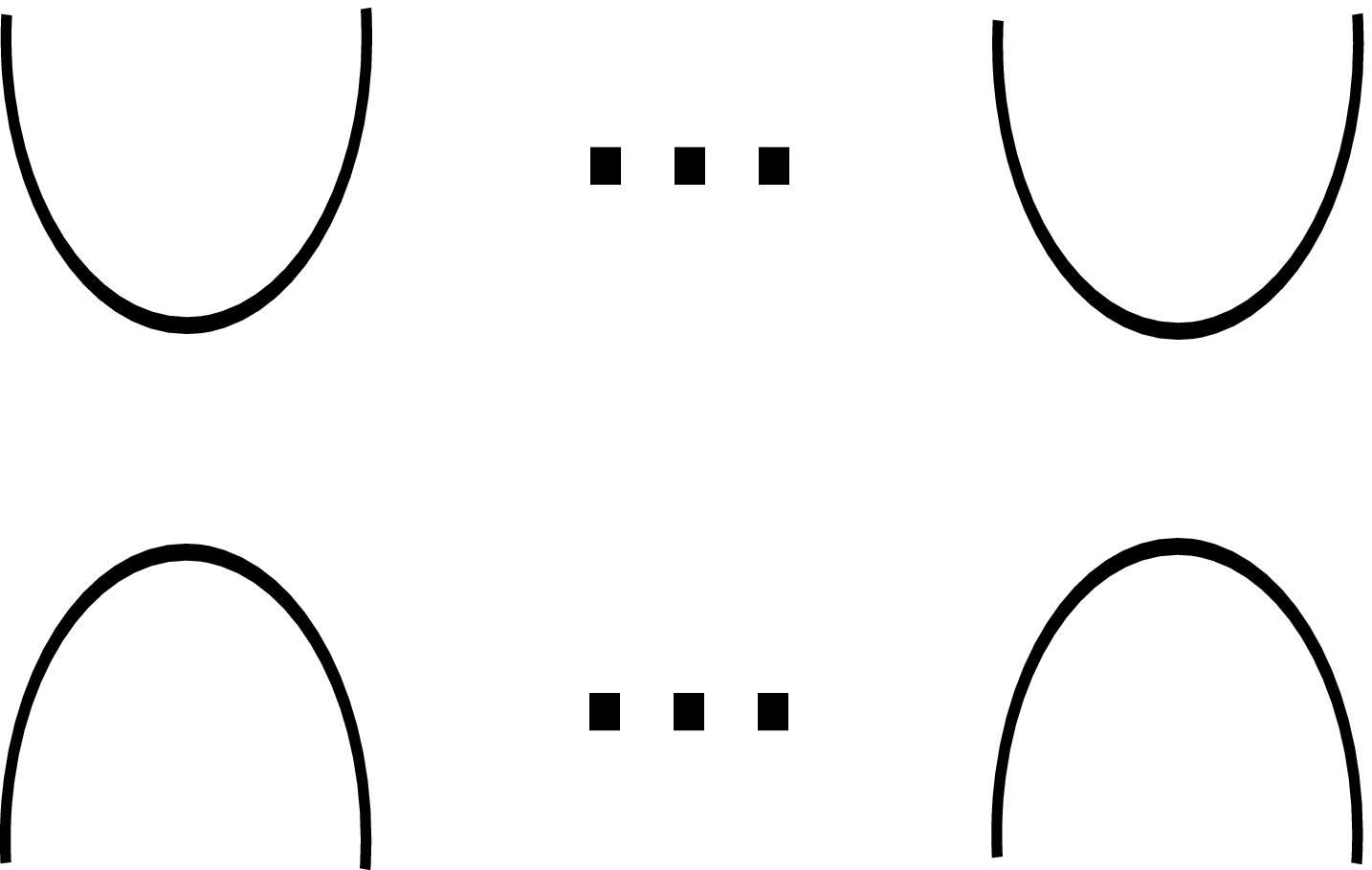}$, when $n$ is odd or even respectively. Then the map
$$\alpha_Q:\mathscr{A}\rightarrow (\mathscr{A}*\mathscr{B})_Q, ~\alpha_Q(x)=x\otimes v_n, ~\forall ~x\in \mathscr{A}_n,$$
 is a planar algebra isomorphism.
\end{proposition}

The following result was first known by Bisch and Jones \cite{BisJonfree}, see also \cite{BhaLan}.
\begin{theorem}
Let $\mathscr{S}$ be a subfactor planar algebra containing is a biprojection $Q$.
Then the planar subalgebra $\mathscr{S}_Q\vee\mathscr{S}^Q$ of $\mathscr{S}$ generated by the vector spaces $\mathscr{S}_Q$ and $\mathscr{S}^Q$ is naturally the free product  $\mathscr{S}_Q*\mathscr{S}^Q$.
\end{theorem}

\begin{corollary}\label{free222}
Suppose $\mathscr{A}$ and $\mathscr{B}$ are subfactor planar algebras generated by 2-boxes. Then $\mathscr{A}*\mathscr{B}$ is generated by 2-boxes.
\end{corollary}

\begin{theorem}\label{freedecom}
Suppose $\mathscr{S}$ is a subfactor planar algebra. If $Q$ is a biprojection in $\mathscr{S}_2$, and $\mathscr{S}$ is generated by $\{x \in \mathscr{S}_2| QxQ=x \quad or \quad Q*x*Q=(\frac{tr(Q)}{\delta})^2x\}$ as a planar algebra. Then $\mathscr{S}=\mathscr{S}_Q*\mathscr{S}^Q$, and both $\mathscr{S}_Q$ and $\mathscr{S}^Q$ are generated by 2-boxes. In this case, $\mathscr{S}$ is said to be separated by the biprojection $Q$ as a free product.
\end{theorem}

\begin{proof}
Suppose $\mathscr{A}$ is the planar subalgebra of $\mathscr{S}_Q$ generated by 2-boxes and
$\mathscr{B}$ is the planar subalgebra of $\mathscr{S}^Q$ generated by 2-boxes.
Then $\mathscr{A}*\mathscr{B}\subset \mathscr{S}_Q*\mathscr{S}^Q \subset \mathscr{S}$.
On the other hand,
if $x\in \mathscr{S}_2$ satisfies $QxQ=Q$, then  $x=\alpha_Q^{-1}(x)\otimes \delta_Be$;
if $y\in \mathscr{S}_2$ satisfies $Q*y*Q=(\frac{tr(Q)}{\delta})^2y$, then $y=id \otimes \beta_Q^{-1}(y)$.
If $\mathscr{S}$ is generated by these 2-boxes, then $\mathscr{S}\subset \mathscr{A}*\mathscr{B}$.
So $\mathscr{A}*\mathscr{B}=\mathscr{S}_Q*\mathscr{S}^Q=\mathscr{S}$.
Counting the dimensions by Volculescu's free product of dimension generating functions, we obtain $\mathscr{A}=\mathscr{S}_Q$ and $\mathscr{B}=\mathscr{S}^Q$.
\end{proof}

\begin{corollary}\label{freedecom222}
Suppose $\mathscr{A}$ and $\mathscr{B}$ are subfactor planar algebras.
If $\mathscr{A}*\mathscr{B}$ is generated by 2-boxes, then both $\mathscr{A}$ and $\mathscr{B}$ are generated by 2-boxes.
\end{corollary}

\begin{proof}
Suppose $Q$ is the central biprojection $id\otimes e$, then $(\mathscr{A}*\mathscr{B})_2=\{x \in \mathscr{S}_2| QxQ=Q \quad or \quad Q*x*Q=(\frac{tr(Q)}{\delta})^2x\}$.
The statement follows from Theorem \ref{freedecom}.
\end{proof}

Based on the Schur Product Theorem, see Theorem \ref{P*P>0}, we have the following result.

\begin{theorem}\label{inter free}
Suppose $\mathscr{A}$ and $\mathscr{B}$ are subfactor planar algebras. Then for any biprojection $Q$ in $(\mathscr{A}*\mathscr{B})_2$, either $Q\geq id\otimes e$ or $Q\leq id\otimes e$.
\end{theorem}
\begin{proof}
Note that $Q=a\otimes e+id\otimes b$ for some $a\in\mathscr{A}_2, b\in\mathscr{B}_2$. Furthermore we may assume $be=0$, otherwise $a,b$ are replaced by $\lambda id+a, b-\lambda e$, when $be=\lambda e$.
Then this decomposition is unique since $a\otimes e=Q (id\otimes e)$.
By assumption $Q$ is a projection, thus $a\otimes e$ and $id\otimes b$ are projections, then $a$ and $b$ are projections.
Moreover $Q=Q'$, thus both $a\otimes e$ and $id\otimes b$ are self-contragredient, then $a=a'$ and $b=b'$.
Furthermore $\frac{tr(Q)}{\delta}Q=Q*Q$, then
$$a\otimes e+id\otimes b$$
$$=a*a\otimes e*e +a*id\otimes e*b+id*a\otimes b*e+id*id\otimes b*b$$
$$=\frac{1}{\delta_1}a*a\otimes e+\frac{2tr(a)}{\delta_1\delta_2} id\otimes b+\delta_1 id\otimes b*b$$
$$=\frac{1}{\delta_1}a*a\otimes e+\delta_1 \frac{tr(b)}{\delta_2}id\otimes e+\frac{2tr(a)}{\delta_1\delta_2} id\otimes b+\delta_1 id\otimes (b*b-\frac{tr(b)}{\delta_2}e).$$
Both $b$ and $b*b-\frac{tr(b)}{\delta_2}e$ are orthogonal to $e$, so
$$\frac{tr(Q)}{\delta} a\otimes e=\frac{1}{\delta_1}a*a\otimes e+\delta_1 \frac{tr(b)}{\delta_2}id\otimes e.$$
If $tr(b)=0$, then $b=0$, because $b$ is a projection. Thus $P=a\otimes e \leq id\otimes e$.
Otherwise $tr(b)>0$. By Theorem \ref{P*P>0}, we have $a*a\otimes e>0$. So $id\otimes e\preceq a\otimes e$. While $a$ is a projection, so $a=id$. Then $P\geq id\otimes e$.
\end{proof}

\subsection{Skein theory}\label{sub ex}
Comparing to group theory,
a subfactor planar algebra could be constructed by generators and relations \cite{JonPA}.
While trying to construct a subfactor planar algebra $\mathscr{S}=\{\mathscr{S}_{n,\pm}\}_{n\in \mathbb{N}_0}$, we will encounter four problems:

(1) Is $\mathscr{S}$ finite dimensional, i.e., is $\mathscr{S}_{n,\pm}$ finite dimensional for each $n$?

(2) Is $\mathscr{S}$ evaluable, i.e., is $\mathscr{S}_{0,\pm}$ 1-dimensional?

(3) Is $\mathscr{S}$ the zero planar algebra?

(4) Is the Markov trace positive definite?

If $\mathscr{S}$ is the planar algebra of an irreducible subfactor, then $\dim(\mathscr{S}_{0,\pm})=\dim(\mathscr{S}_{1,\pm})=1$. We shall consider a planar algebra generated by a finite subset of 2-boxes, then each generator can be viewed as a crossing $\gra{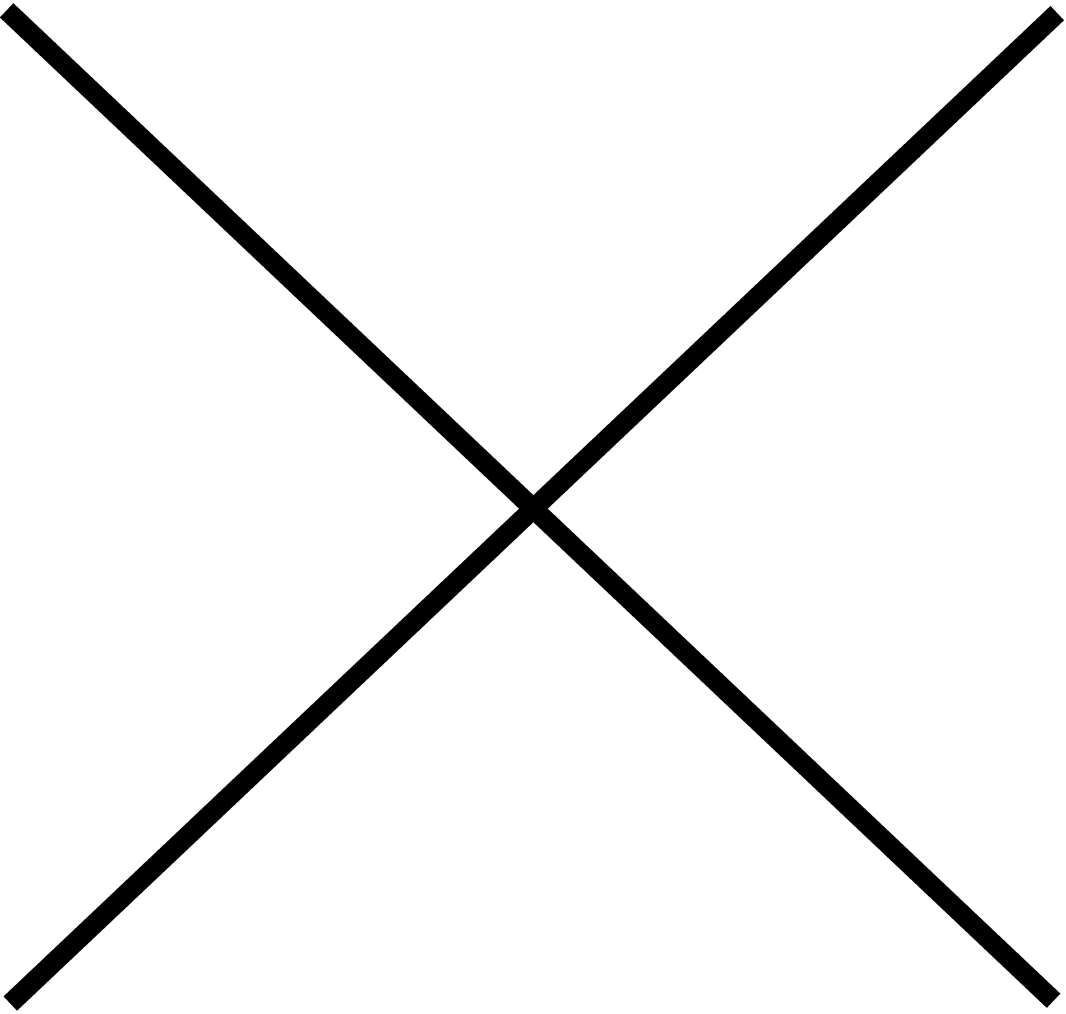}$ with a label at the intersection, and each element in $\mathscr{S}_{n,\pm}$ can be viewed as a linear combination of diagrams with $2n$ boundary points which consist of finitely many crossings and finitely many strings. For example, $\gra{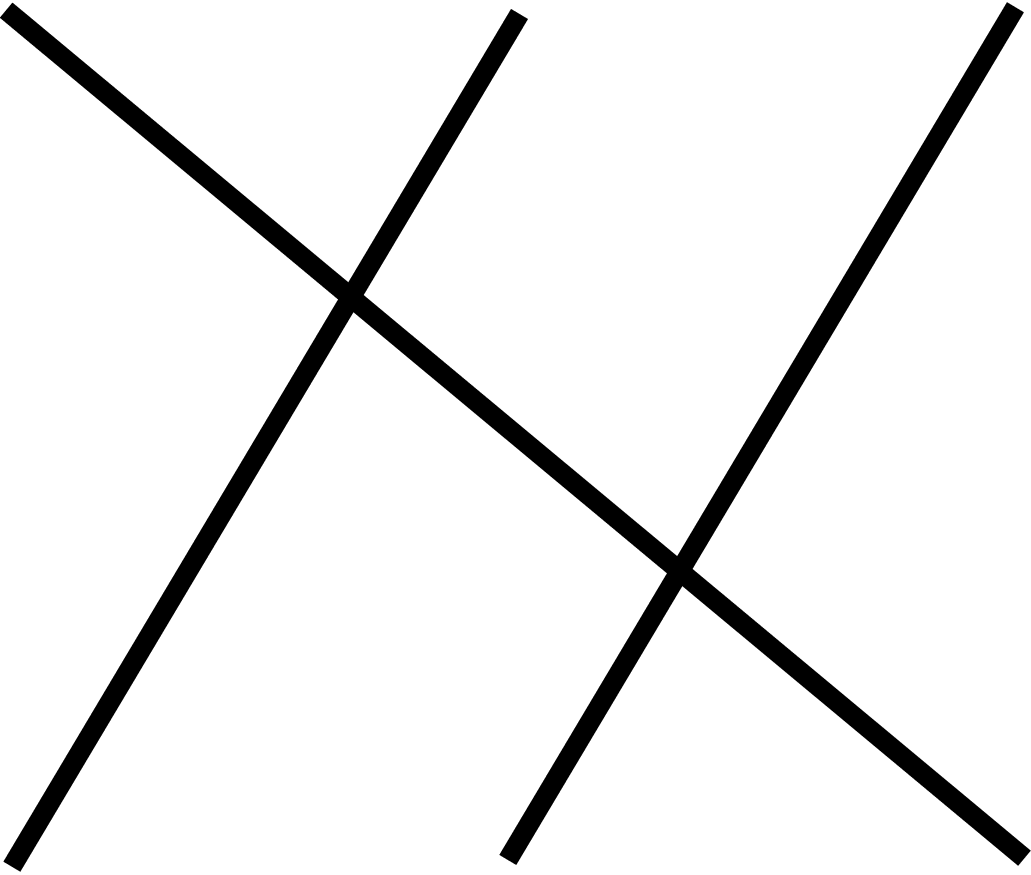}$, as an element in $\mathscr{S}_{3,+}$, is a diagram with $6$ boundary points and 2 crossings.
What kind of relations should be endowed?
One type of relation, termed an exchange relation, is discussed by Landau \cite{Lan02}. It is motivated by the exchange relation of a biprojection discovered by Bisch \cite{Bis94}, see Proposition \ref{ex of bipro}. The planar algebra $\mathscr{S}$ has an exchange relation means that the diagram $\gra{ex2}$ can be replaced by a finite sum of the diagrams $\gra{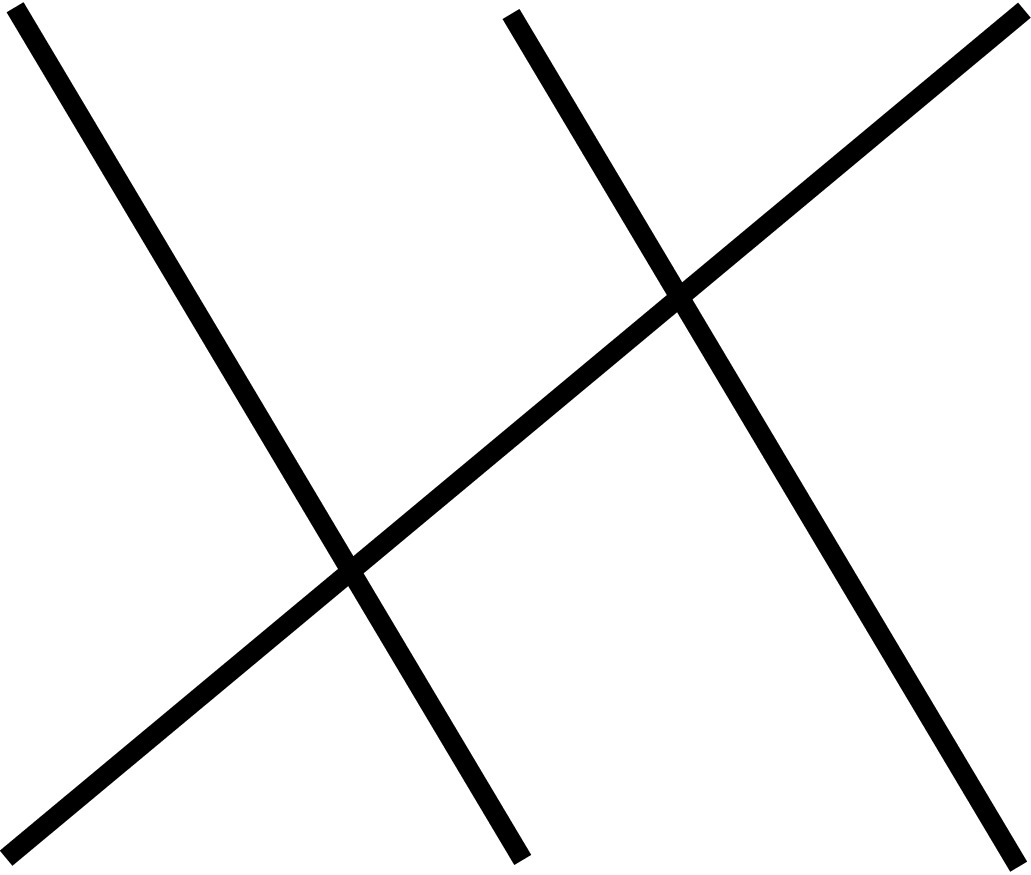}$ and $\gra{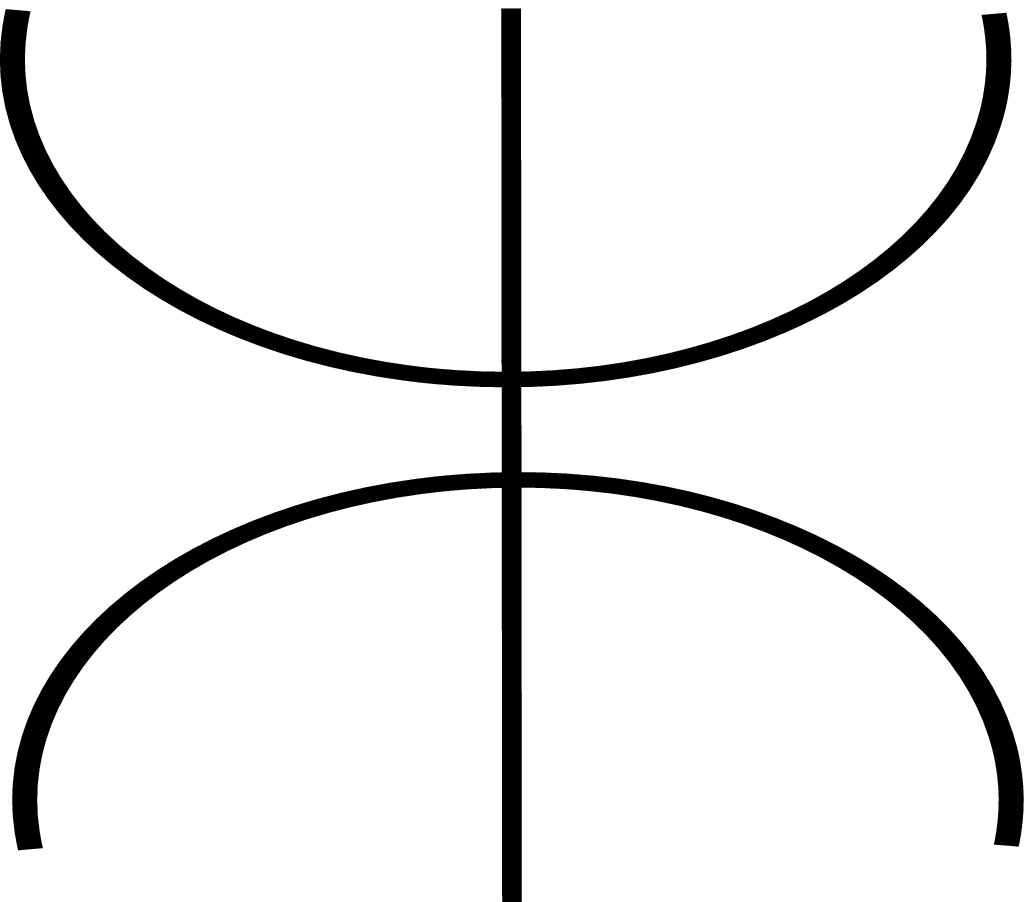}$, and the diagram $\gra{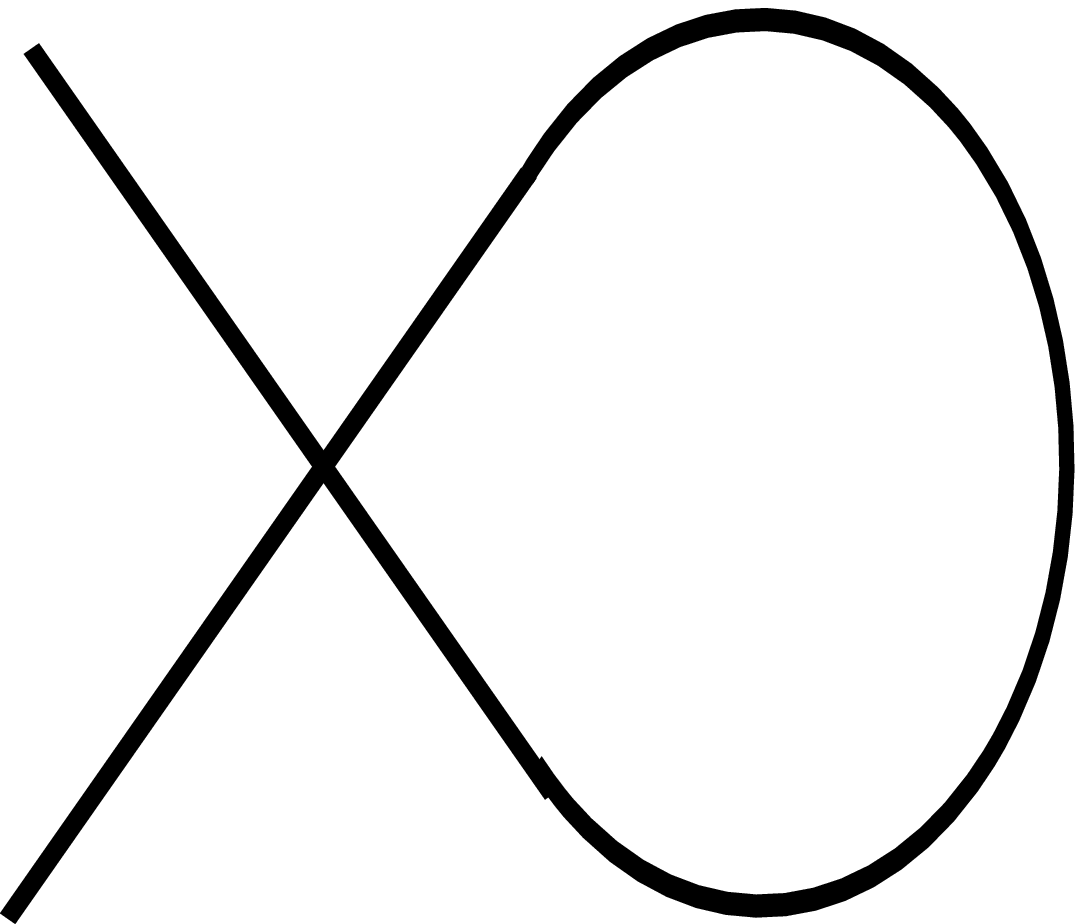}$ can be replaced by a multiple of a string $\gra{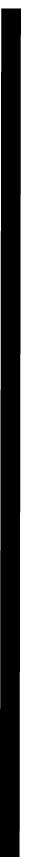}$. Note that a closed string contributes a scaler $\delta$.
By these three operations, a face of a diagram can be removed without increasing the number of crossings.
Given the number of boundary points, up to isotopy, there are only finitely many diagrams without faces and closed strings. Thus problem (1) is solved. Furthermore if such a diagram has no boundary points, then it has to be the empty diagram. Thus problem (2) is solved.
Given generators and an exchange relation, to solve problem (3) is equivalent to check a finite system of equations. But it is hard to solve theses equations directly. What's worse, it is much harder to solve problem (4).
In this paper, we focus on classifying exchange relation planar algebras. The ones appeared could be constructed by other methods. So we will not deal with problem (3) and (4) directly.

\begin{definition}
Suppose $\mathscr{S}$ is an irreducible subfactor planar algebra. If $\mathscr{S}$ is generated by $\mathscr{S}_2$ with the following relations: for any $a,b\in \mathscr{S}_2$,
$$(1\boxdot a)b=\Sigma_i c_i(1\boxdot d_i)+f_i(1\boxdot id)g_i,$$
for finitely many $c_i,d_i,f_i,g_i\in \mathscr{S}_2$,
then $\mathscr{S}$ is called an exchange relation planar algebra.
\end{definition}

It is easy to check that this definition is equivalent to Landau's definition in \cite{Lan02}.
By definition, Temperley-Lieb subfactor planar algebras and depth 2 subfactor planar algebras are exchange relation planar algebras.

\begin{proposition}\label{dimpn}
Suppose $\mathscr{S}$ is an exchange relation planar algebra.
Then
$$\dim(\mathscr{S}_{n+1})\leq\dim(\mathscr{S}_n)^2+(\dim(\mathscr{S}_2)-1)^n.$$
And $\mathscr{P_3}$ is generated by $\mathscr{S}_2$ and $\mathscr{S}_{1,3}$ as an algebra.
\end{proposition}

Specifically $\dim(\mathscr{S}_3)\leq\dim(\mathscr{S}_2)^2+(\dim(\mathscr{S}_2)-1)^2.$

\begin{proof}
We view 2-boxes as crossings, the labels at the intersection as points, the strings as edges. Then using exchange relations, we may replace $\grb{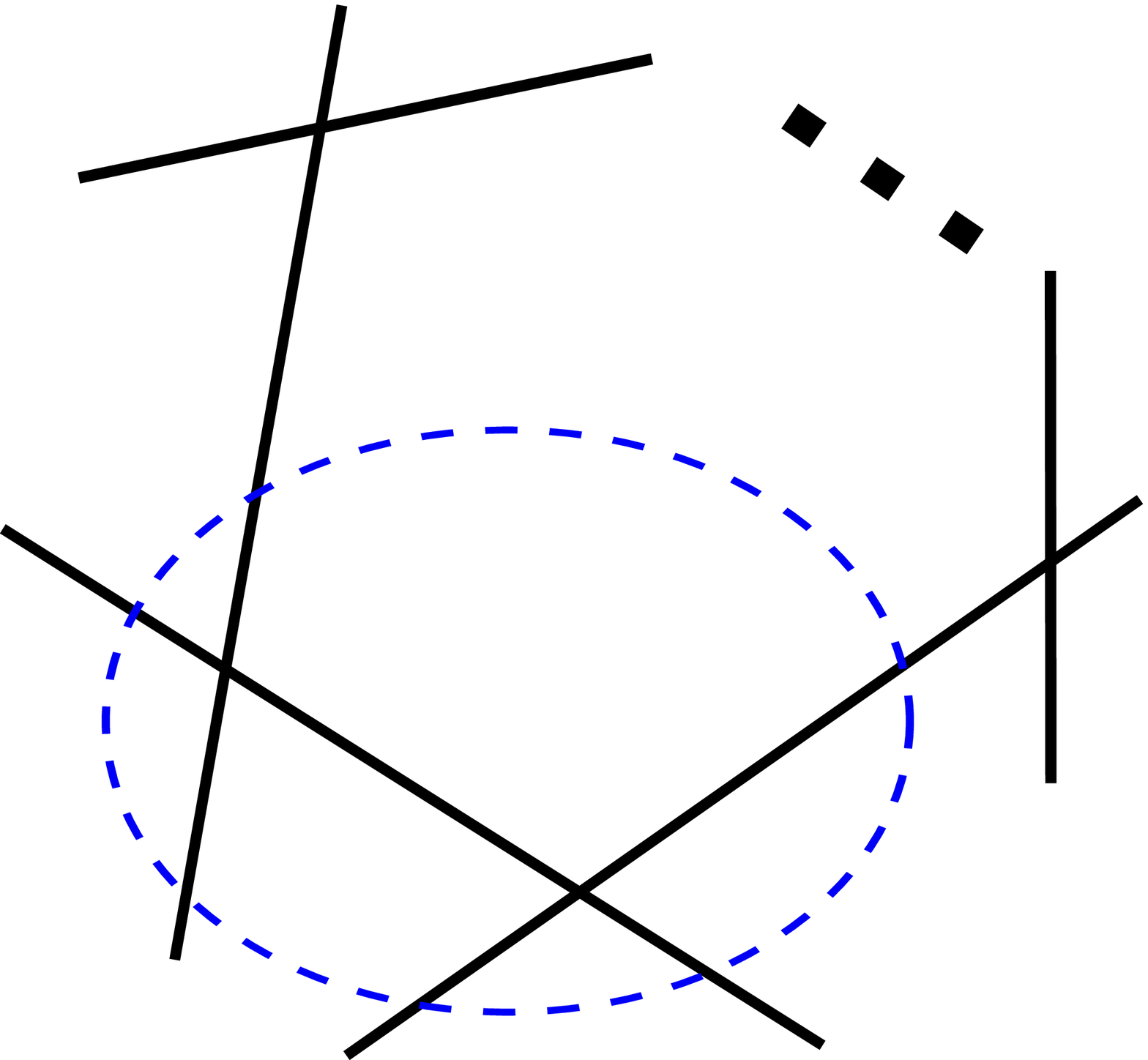}$ by $\grb{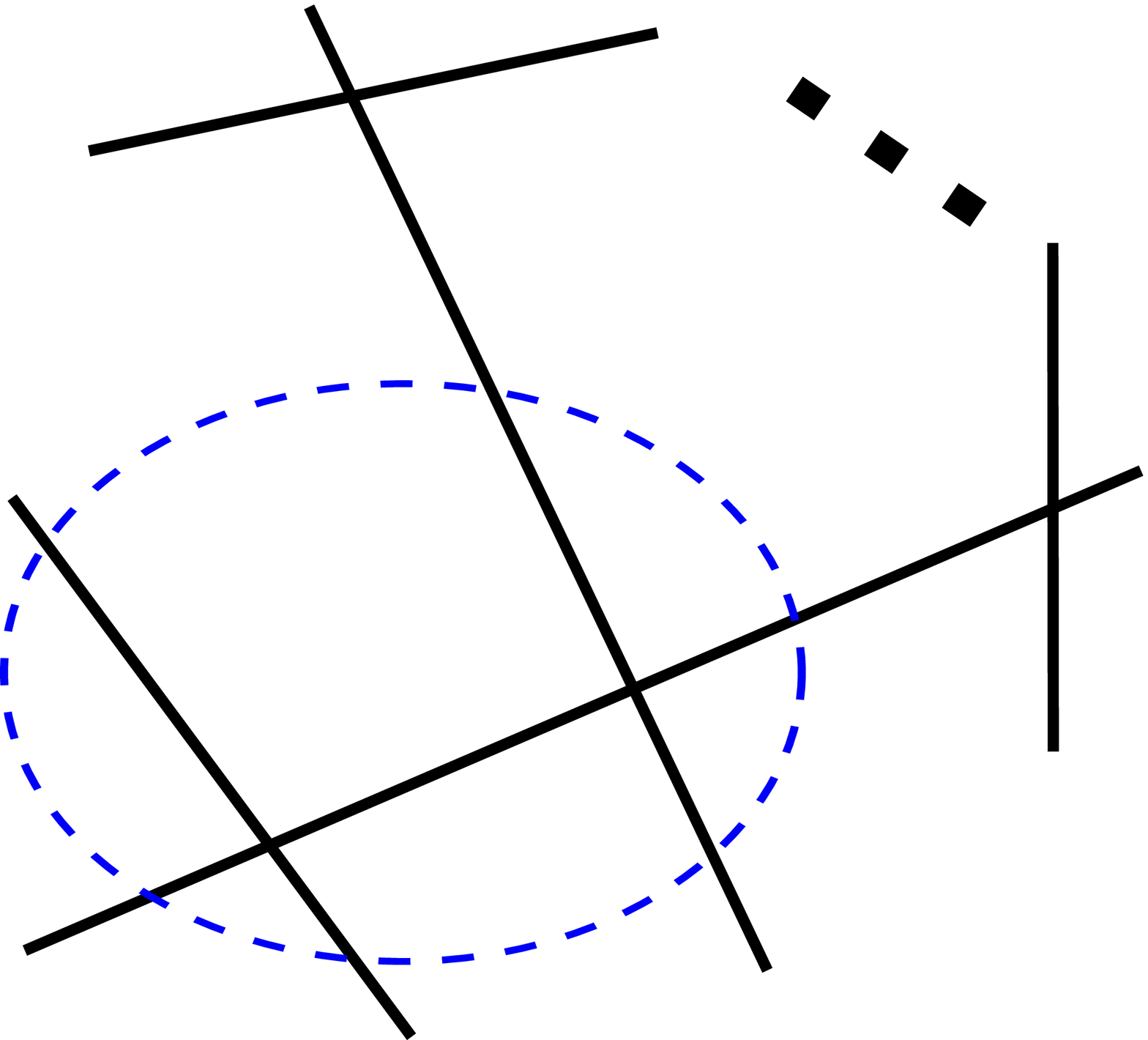}$ and $\grb{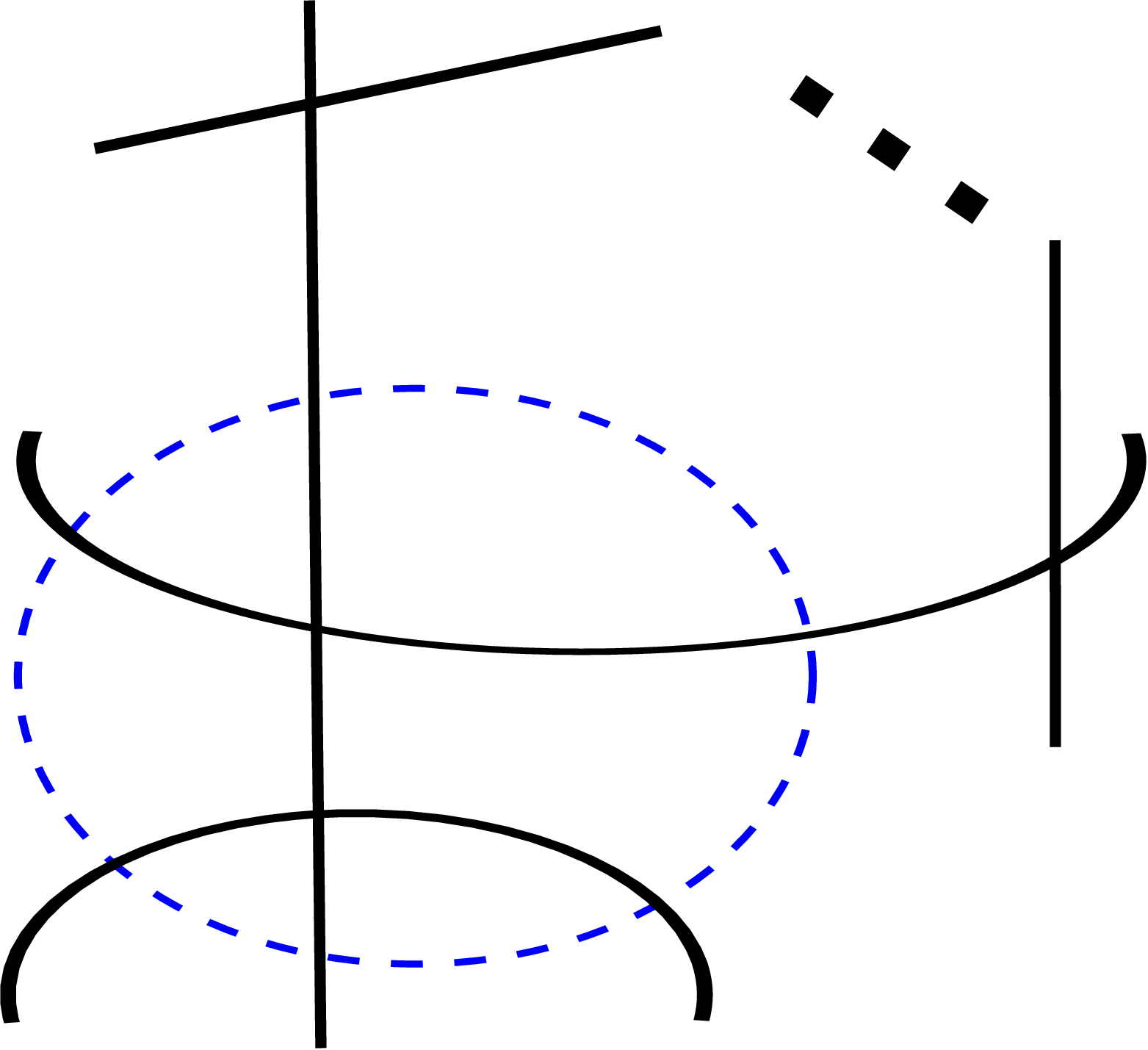}$. By this operation, the number of edges of one face will decrease without adding faces. Combining with the relation $\gra{ex5}=\gra{ex7}$, we only need to consider diagrams without faces.

If a diagram with $2n+2$ boundary points, $n+1$ on the top and $n+1$ on the bottom, has no faces, then either
it is in the ideal $\mathscr{I}_{n+1}$ generated by the Jones projection $e_n$, or it has $n+1$ through strings.
The dimension of the ideal of $\mathscr{I}_{n+1}$ is at most $\dim(\mathscr{S}_n)^2$.
In $\mathscr{S}_{n+1}/\mathscr{I}_{n+1}$, applying the exchange relation, we only need to consider one diagram $\gra{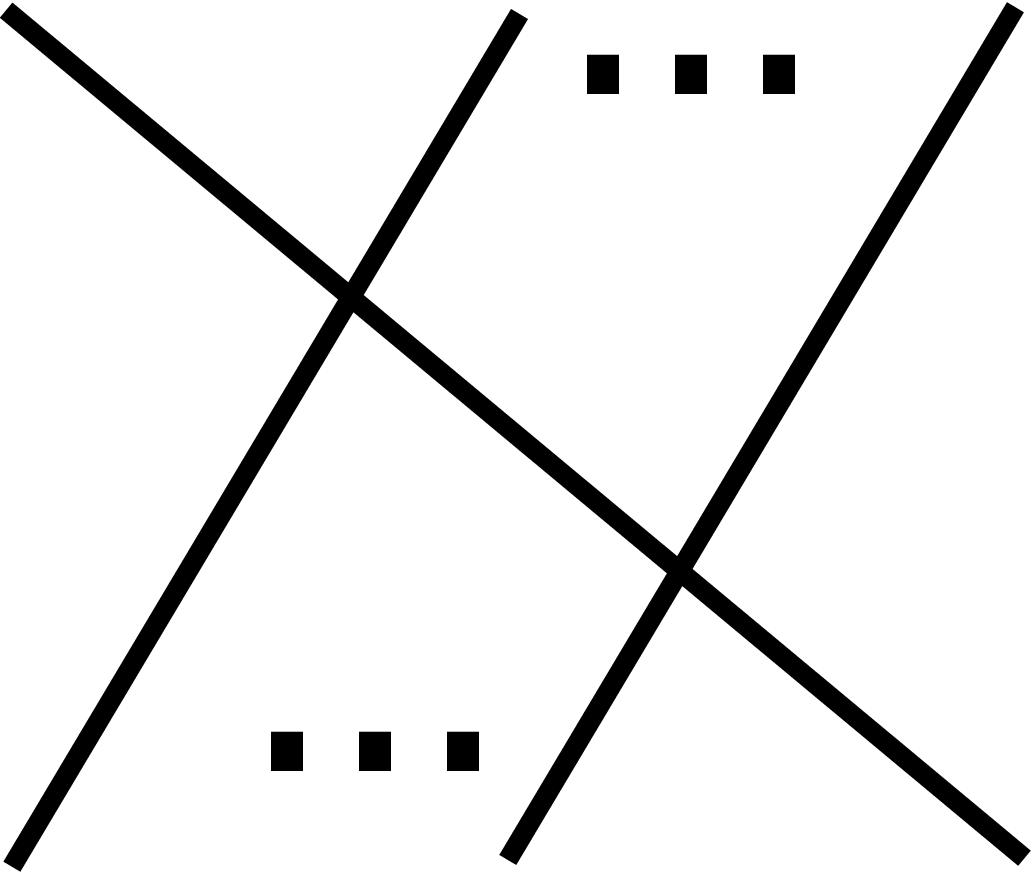}$.
If the label at an intersection is the Jones projection $e$, then this diagram is in the ideal $\mathscr{I}_{n+1}$.
Thus the dimension of $\mathscr{S}_{n+1}/\mathscr{I}_{n+1}$ is at most $(\dim(\mathscr{S}_2)-1)^n$.
Then $\dim(\mathscr{S}_{n+1})\leq\dim(\mathscr{S}_n)^2+(\dim(\mathscr{S}_2)-1)^n$.

If a diagram with $6$ boundary points has no faces and it has two crossings or more, then it has to be one of the following three diagram $\gra{ex2}$, $\gra{ex3}$, $\gra{ex4}$.
Thus $\mathscr{P_3}$ is generated by $\mathscr{S}_2$ and $\mathscr{S}_{1,3}$ as an algebra.
\end{proof}

If a closed diagram has no faces, then it is the empty diagram. So each closed diagram is evaluable based on the exchange relation.
Furthermore the exchange relation is determined by the algebraic structure of 2-boxes in the following sense.

\begin{definition}
The structure of 2-boxes of a subfactor planar algebra consists of the data of adjoints, contragredients, products and coproducts of 2-boxes.
\end{definition}

The following data is also derived from the structure of 2-boxes,
the identity $id$ is identified as the unique unit of 2-boxes under the product; the value of a closed circle $\delta$ is determined by the coproduct of two identities; $\delta e$ is identified as the unique unit of 2-boxes under the coproduct; the trace of a 2-box is determined by its coproduct with the identity $id$. If the planar algebra is irreducible, then capping a 2-box is also determined.

\begin{theorem}\label{p2ex}
Suppose $\mathscr{S}$ is an exchange relation planar algebra and $\mathscr{A}$ is a subfactor planar algebra generated by 2-boxes. If a linear map $\phi: \mathscr{S}_2\rightarrow\mathscr{A}_2$ is surjective and it preserves the structure of 2-boxes, i.e., adjoints, contragredients, products and coproducts, then $\phi$ extends to a planar algebra isomorphism from $\mathscr{S}$ to $\mathscr{A}$.
\end{theorem}

\begin{proof}
We extend $\phi$ to the universal planar algebra generated by 2-boxes of $\mathscr{S}$.
If $y=(1\boxdot a)b-(\Sigma_i c_i(1\boxdot d_i)+f_i(1\boxdot id)g_i)$ is a relation, then $tr(y^*y)=0$.
The computation of $tr(\phi(y)^*\phi(y))$ only depends on the structure of 2-boxes, which are preserved by $\phi$, so $tr(\phi(y)^*\phi(y))=tr(y^*y)=0$. Then $\phi(y)=0$ by the positivity of the trace. So $\phi$ induces a planar algebra homomorphism from the quotient $\mathscr{S}$ to $\mathscr{A}$. By assumption $\phi$ is surjective on 2-boxes, and $\mathscr{A}$ is generated by 2-boxes, so $\phi$ is a planar algebra isomorphism.
\end{proof}

\begin{remark}
A planar algebra homomorphism of subfactor planar algebras induces a homomorphism on the 0-box space $\mathbb{C}$, so it is either zero or injective.
\end{remark}

The following classification is given by Bisch and Jones \cite{BisJon97,BisJon02}.
\begin{theorem}\label{single}
Suppose $\mathscr{S}$ is a subfactor planar algebra generated by a non-trivial 2-box with $\dim(P_3)\leq13$, then $\mathscr{S}$ is one of the follows,
(1)$\mathscr{S}^{\mathbb{Z}_3}$;
(2)$TL*TL$;
(3)$\mathscr{S}^{\mathbb{Z}_2\subset \mathbb{Z}_5\rtimes \mathbb{Z}_2}$.
\end{theorem}

Suppose $\mathscr{S}$ is a subfactor planar algebra generated by $\mathscr{S}_2$ with $\dim(\mathscr{S}_2)=3$. Assume $id$, $e$ and $r$ forms a basis of $\mathscr{S}_2$.
In $\mathscr{S}_3$, there are $5$ Temperley-Lieb diagrams, $6$ diagrams with one $r$, the orbits or $r, 1\boxdot r$ under the rotation, and $3$ diagrams with two $r$, $r(1\boxdot id)r$, $r(1\boxdot r)$ and $(1\boxdot r)r$. If $\mathscr{S}_3\leq 13$, then those $14$ elements are linear dependent. Without loss of generality, we may assume one diagram with two $r$ is a linear sum of the other $13$ diagrams. Up to a 2-click rotation, it can be chosen as $(1\boxdot r)r$. That implies $\mathscr{S}$ is an exchange relation planar algebra.
Conversely if $\mathscr{S}$ is an exchange relation planar algebra with $\dim(\mathscr{S}_2)=3$, then $\dim(\mathscr{S}_3)\leq 2^2+3^2=13$.
Thus Theorem \ref{single} is equivalent to
\begin{theorem}\label{clas ex 1}
Suppose $\mathscr{S}$ is an exchange relation planar algebra with $\dim(P_2)=3$. Then $\mathscr{S}$ is one of the follows,
(1)$\mathscr{S}^{\mathbb{Z}_3}$;
(2)$TL*TL$;
(3)$\mathscr{S}^{\mathbb{Z}_2\subset \mathbb{Z}_5\rtimes \mathbb{Z}_2}$.
\end{theorem}

\section{Tensor Products}\label{tensor product}
The author would like to thank Dave Penneys for the discussion about tensor products.

In this section, sometimes we draw a diagram with all the boundary points on the top.
For example, $v_n$ is the Temperley-Lieb $n$-tangle $\grs{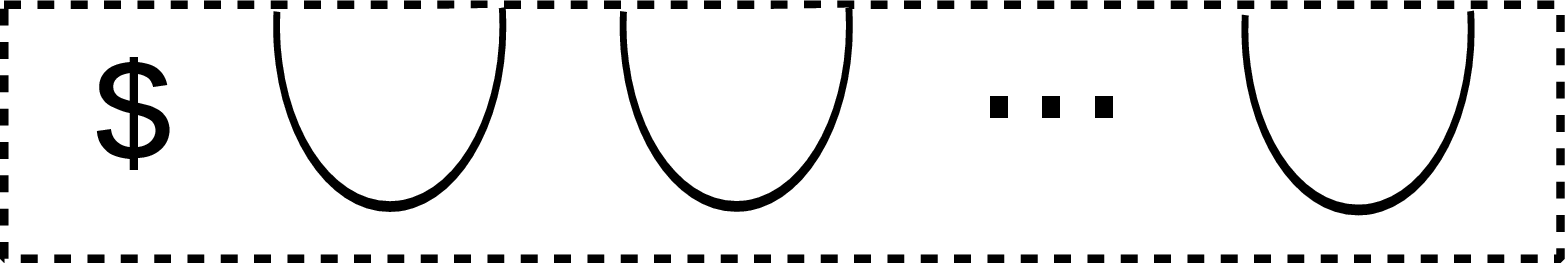}$.
Recall that $\mathscr{A}*\mathscr{B}\subset\mathscr{A}\otimes\mathscr{B}$, and it contains a biprojection $id\otimes e$.
By Proposition \ref{iso}, there is a planar algebra isomorphism
$\alpha_1:\mathscr{A}\rightarrow (\mathscr{A}*\mathscr{B})_{id\otimes e}, ~\alpha(a)=a\otimes v_n, ~\forall ~a\in\mathscr{A}.$
By the definition of the tensor product, we have $(\mathscr{A}\otimes\mathscr{B})_{id\otimes e}=\alpha_1(\mathscr{A})$.
So $(\mathscr{A}\otimes\mathscr{B})_{id\otimes e}=(\mathscr{A}*\mathscr{B})_{id\otimes e}$.
Then we have a planar algebra isomorphism
$$\alpha_1:\mathscr{A}\rightarrow (\mathscr{A}\otimes\mathscr{B})_{id\otimes e}, ~\alpha(a)=a\otimes v_n, ~\forall ~a\in\mathscr{A}.$$
Similarly we have a planar algebra isomorphism
$$\alpha_2:\mathscr{B}\rightarrow (\mathscr{A}\otimes\mathscr{B})_{e\otimes id}, ~\alpha(b)=v_n\otimes b, ~\forall ~b\in\mathscr{B}.$$

The tensor product of planar algebras is defined via simple tensors of vectors. We hope to interpret the simple tensor $a\otimes b$ as a diagram in terms of $\alpha_1(a)$ and $\alpha_2(b)$.

\begin{theorem}
$\mathscr{A}\otimes\mathscr{B}$ is generated by the two vector spaces $(\mathscr{A}\otimes\mathscr{B})_{id\otimes e}$ and $(\mathscr{A}\otimes\mathscr{B})_{e\otimes id}$ as a planar algebra.
\end{theorem}

\begin{proof}
For any $a\in\mathscr{A}_n$ and $b\in\mathscr{B}_n$, let us construct an element in $(\mathscr{A}\otimes\mathscr{B})_n$ as
$$\gre{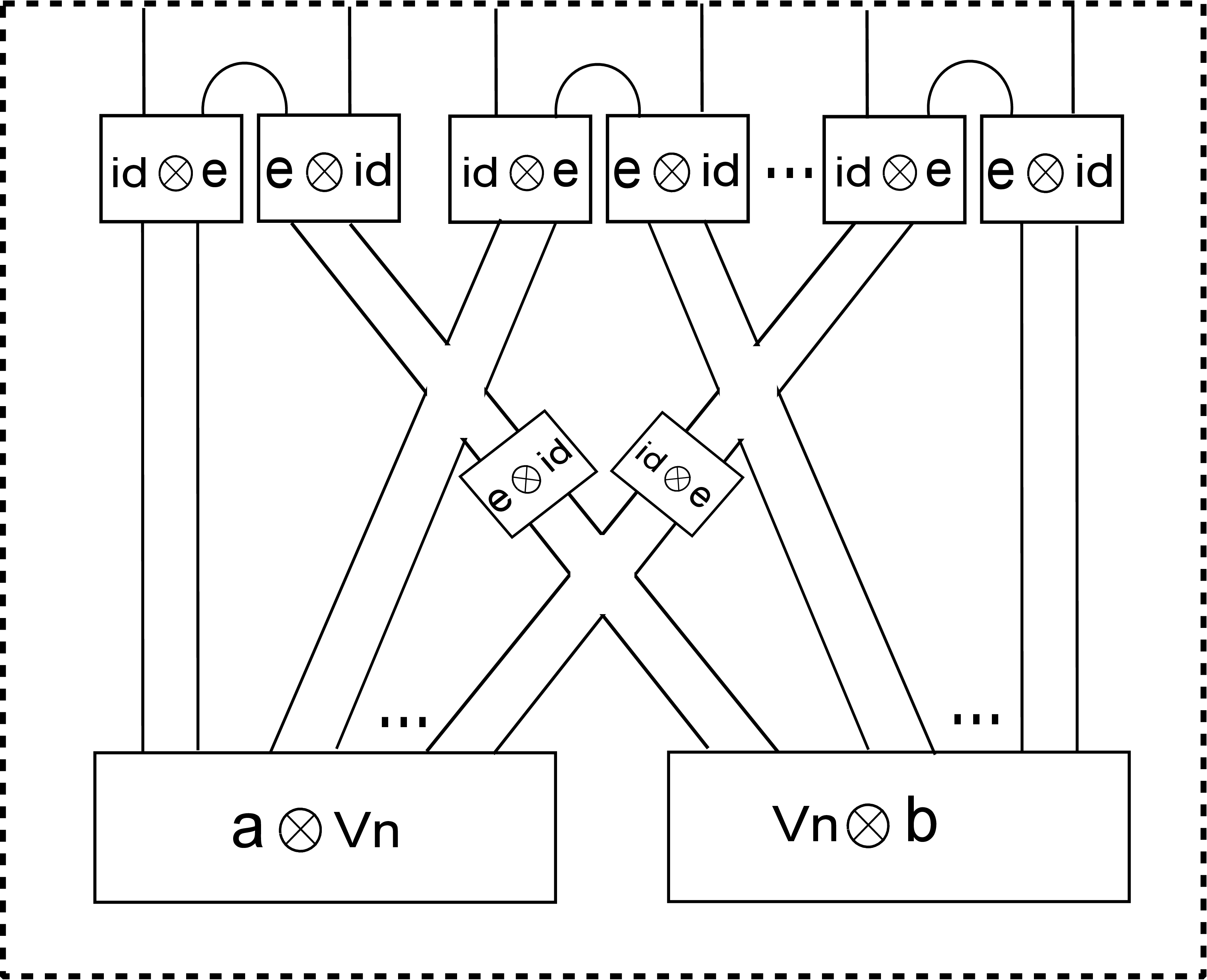}.$$
By the definition of the tensor product, it is $\lambda_n a\otimes\lambda_n b$ for some $\lambda_n>0$.
Note that $a\otimes v_n\in (\mathscr{A}\otimes\mathscr{B})_{id\otimes e}$, $v_n\otimes b \in(\mathscr{A}\otimes\mathscr{B})_{e\otimes id}$,
and $\mathscr{A}\otimes\mathscr{B}$ is generated by all $a\otimes b$'s.
Thus $\mathscr{A}\otimes\mathscr{B}$ is generated by $(\mathscr{A}\otimes\mathscr{B})_{id\otimes e}$ and $(\mathscr{A}\otimes\mathscr{B})_{e\otimes id}$ as a planar algebra.
\end{proof}

\begin{corollary}\label{tensor222}
If both $\mathscr{A}$ and $\mathscr{B}$ are generated by 2-boxes, then $\mathscr{A}\otimes\mathscr{B}$ is generated by 2-boxes.
\end{corollary}

Suppose $\mathscr{S}$ is a subfactor planar algebra, and $A, B$ are two biprojections, such that $AB=e$ and $A*B$ is a multiple of $id$.
Then $AB=(AB)^*=B^*A^*=BA$, $A*B=(A*B)'=B'*A'=B*A$.
By computing the trace of $A*B$, we have $A*B=\frac{tr(A)tr(B)}{\delta^3}id$. Then $tr((A*B)e)=\frac{tr(A)tr(B)}{\delta^3}$.
On the other hand, we view $tr((A*B)e)$ as a diagram, then $tr((A*B)e)=\frac{1}{\delta}tr(AB)=\frac{1}{\delta}$. Thus $tr(A)tr(B)=\delta^2$, and $A*B=\frac{1}{\delta}id$.
Using the exchange relation of biprojections, we have
\begin{equation}\label{ppqq}
\grb{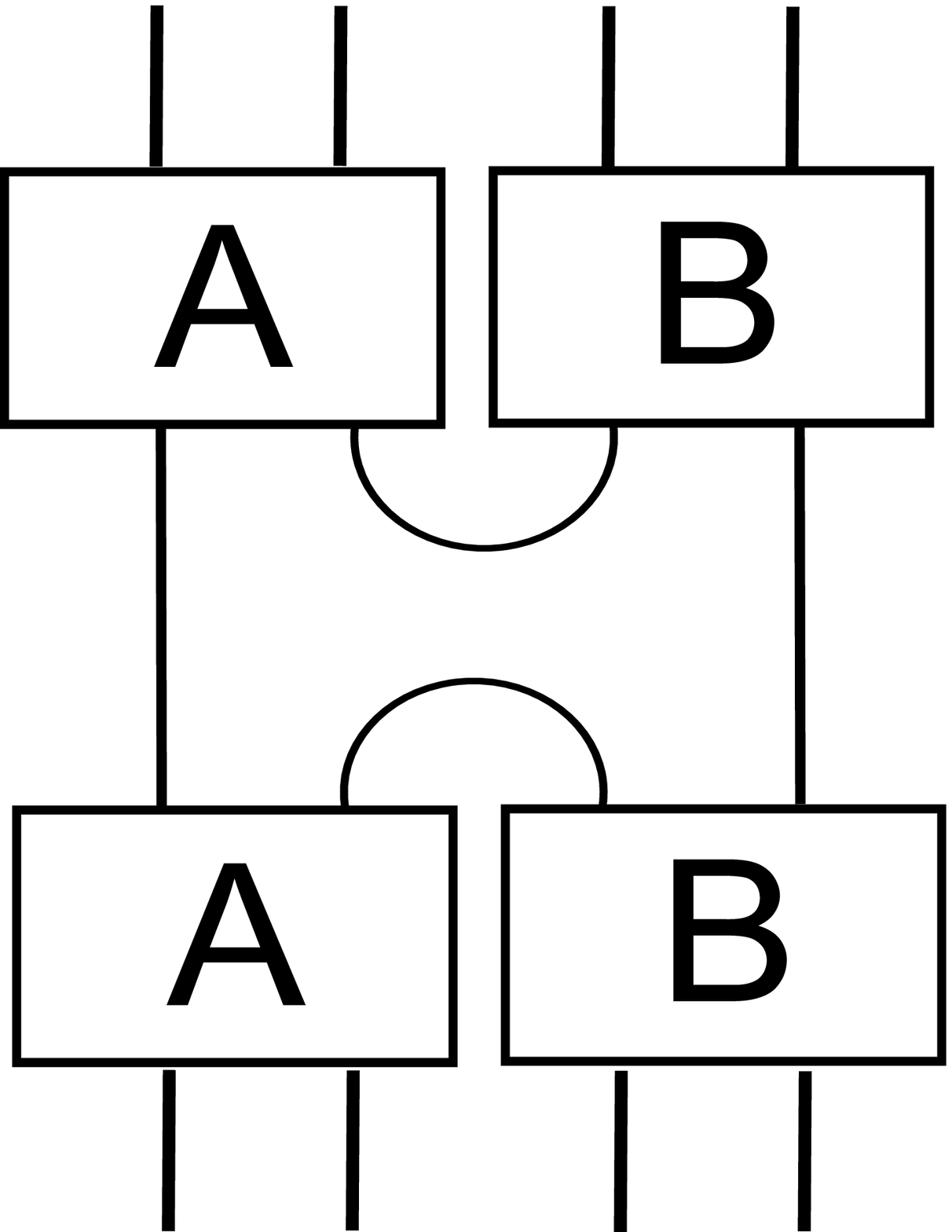}=\grb{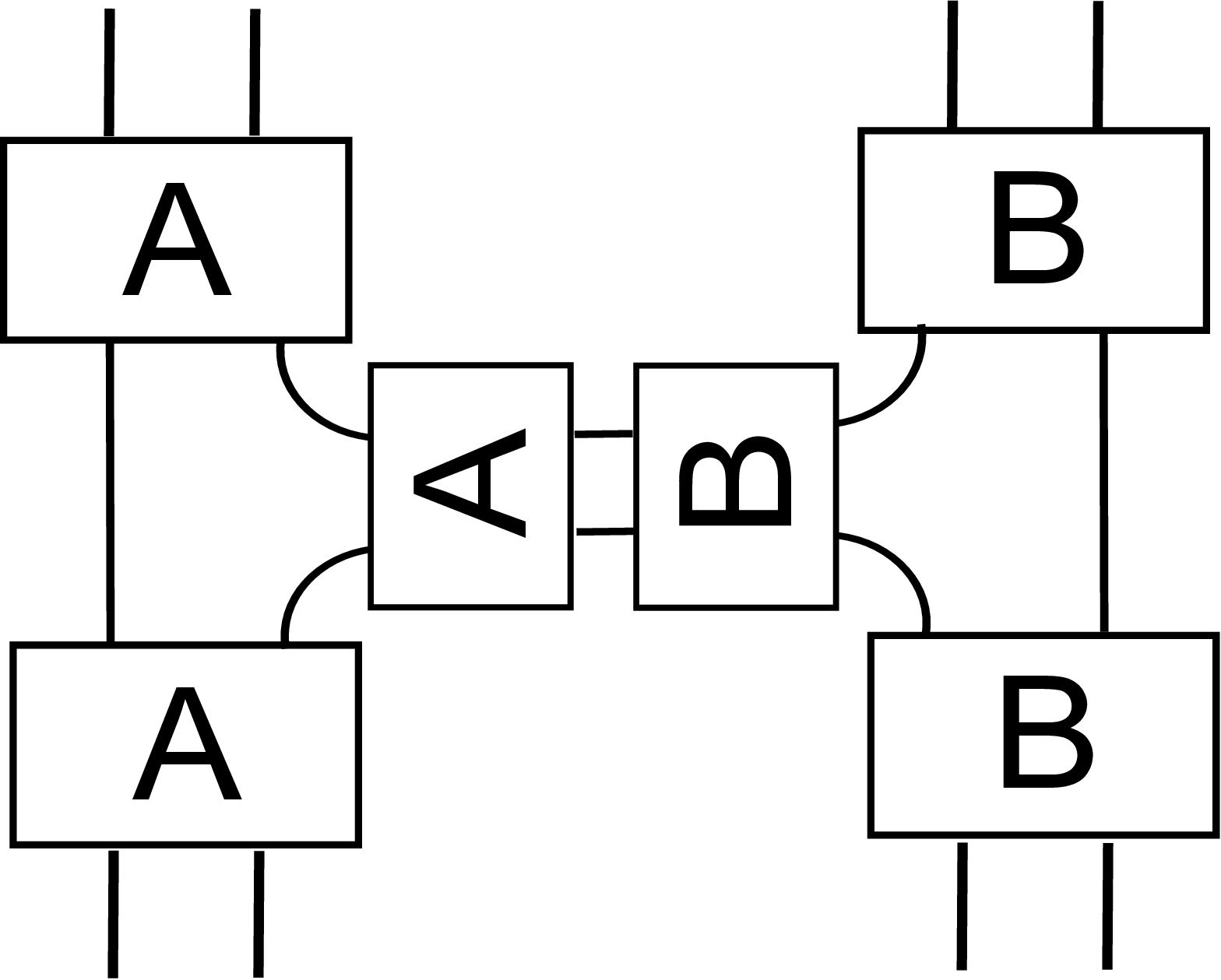}=\grb{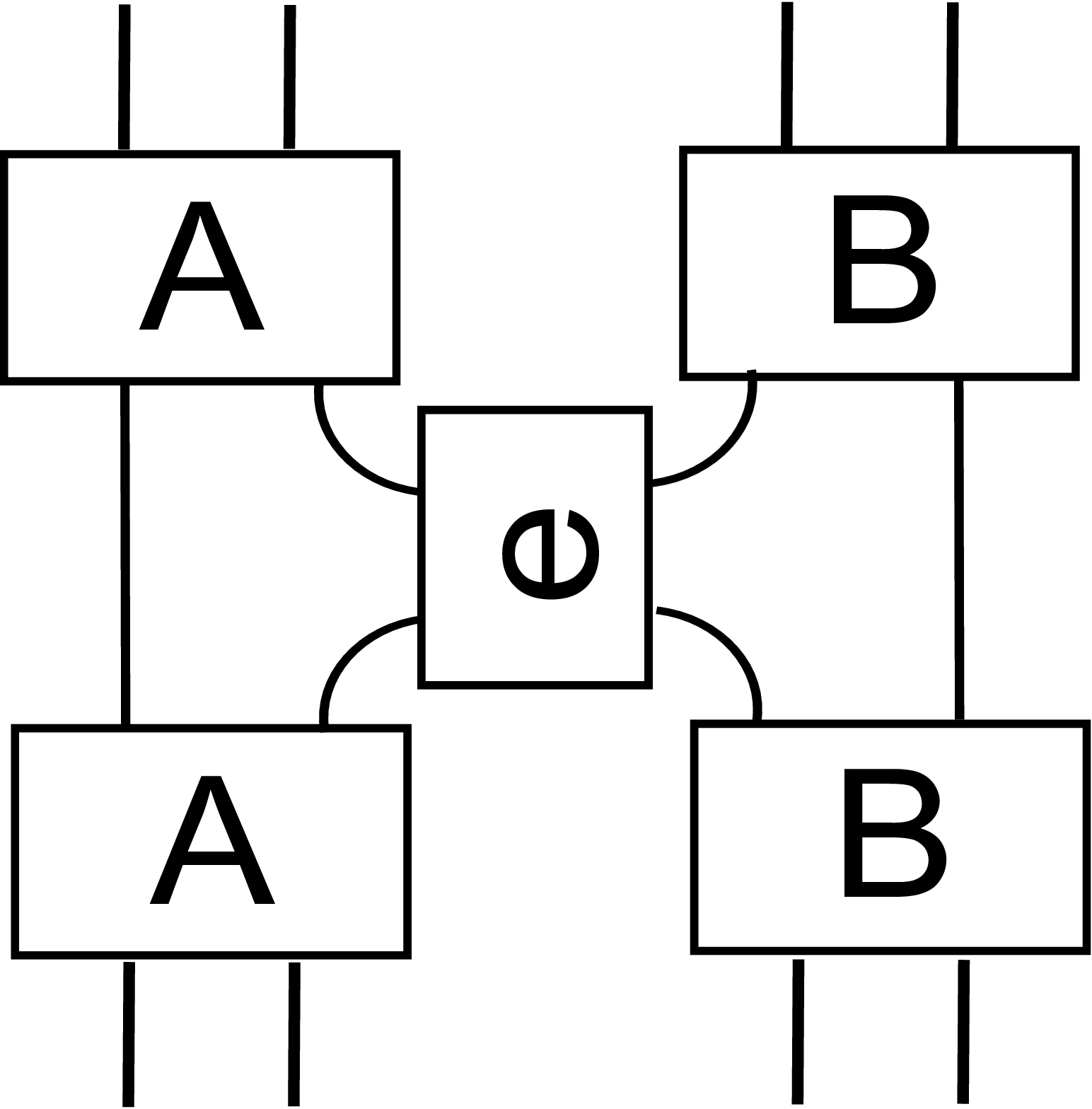}=\frac{1}{\delta}\gra{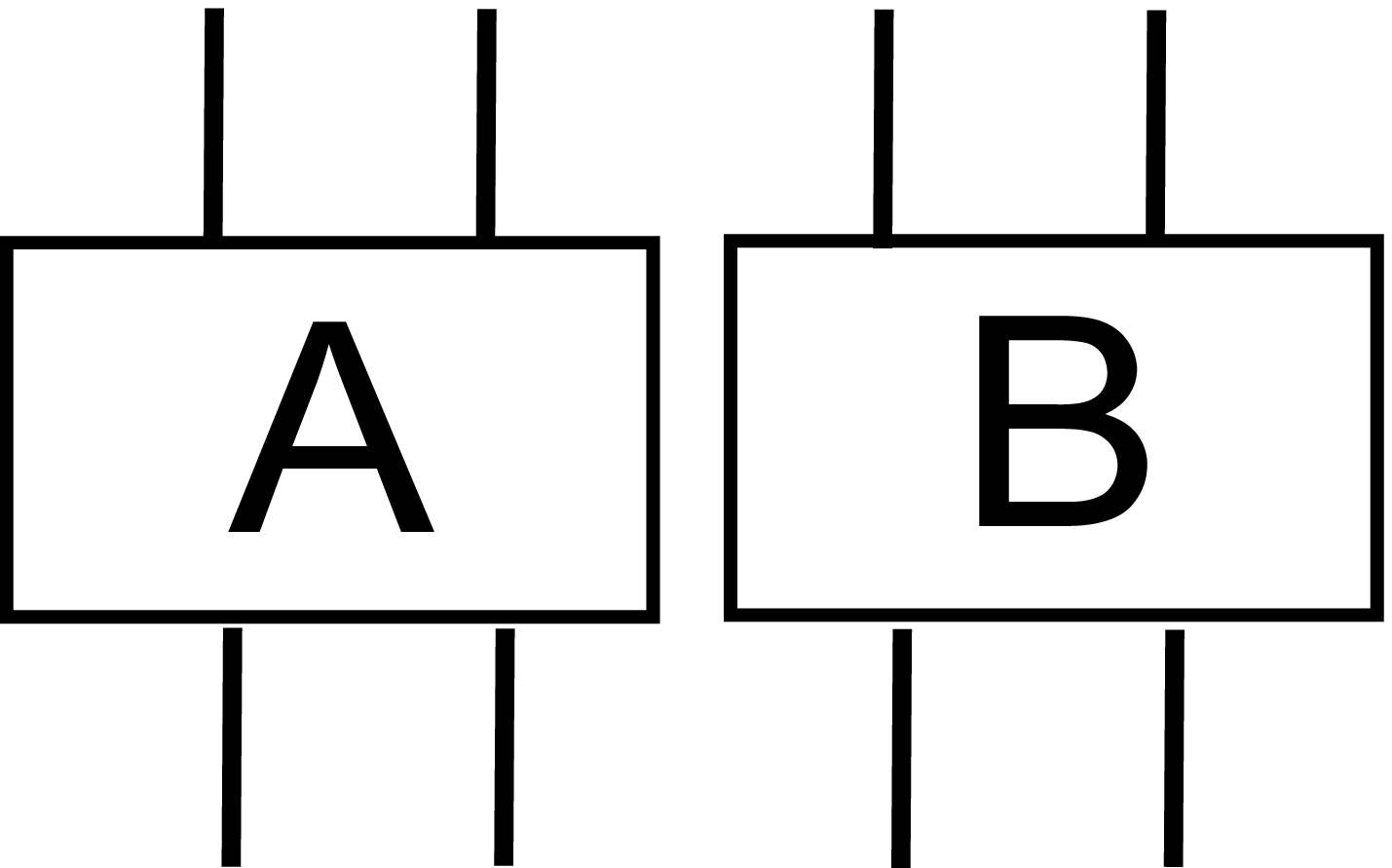}.
\end{equation}
Let $\delta_A,\delta_B$ be $\sqrt{tr(A)},\sqrt{tr(B)}$, then they are the value of a closed circle in $\mathscr{S}_{A},\mathscr{S}_B$ respectively, and $\delta=\delta_A\delta_B.$

To show $\mathscr{S}_A$ and $\mathscr{S}_B$ are ``independent",
we use two kinds of coloured strings, an A-colour string $- - - - -$ connecting with elements in $\mathscr{S}_A$ and a B-colour string $\cdot\cdot\cdot\cdot\cdot$ connecting with elements in $\mathscr{S}_B$. A crossing $\graa{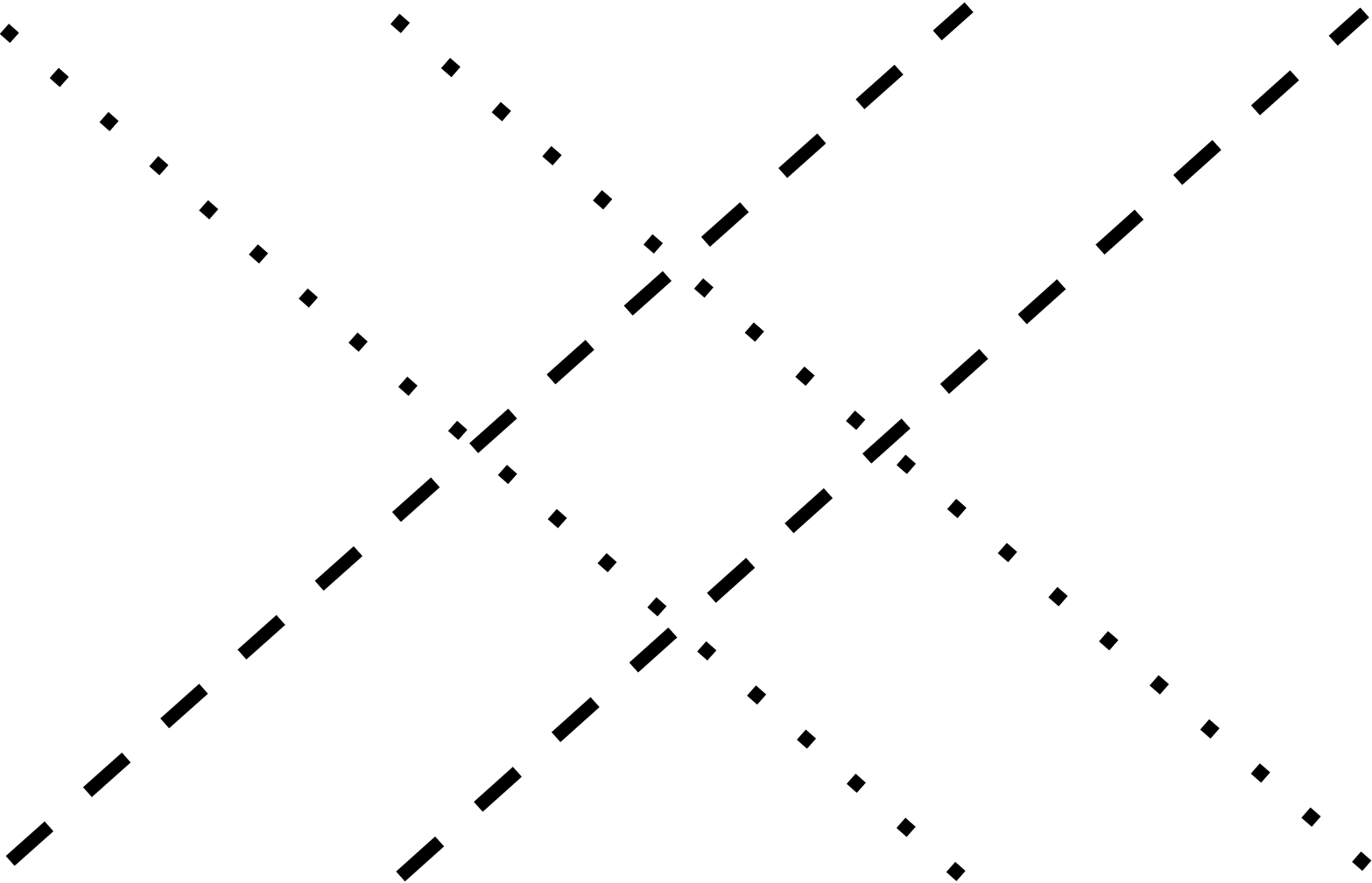}$ means $\delta\graa{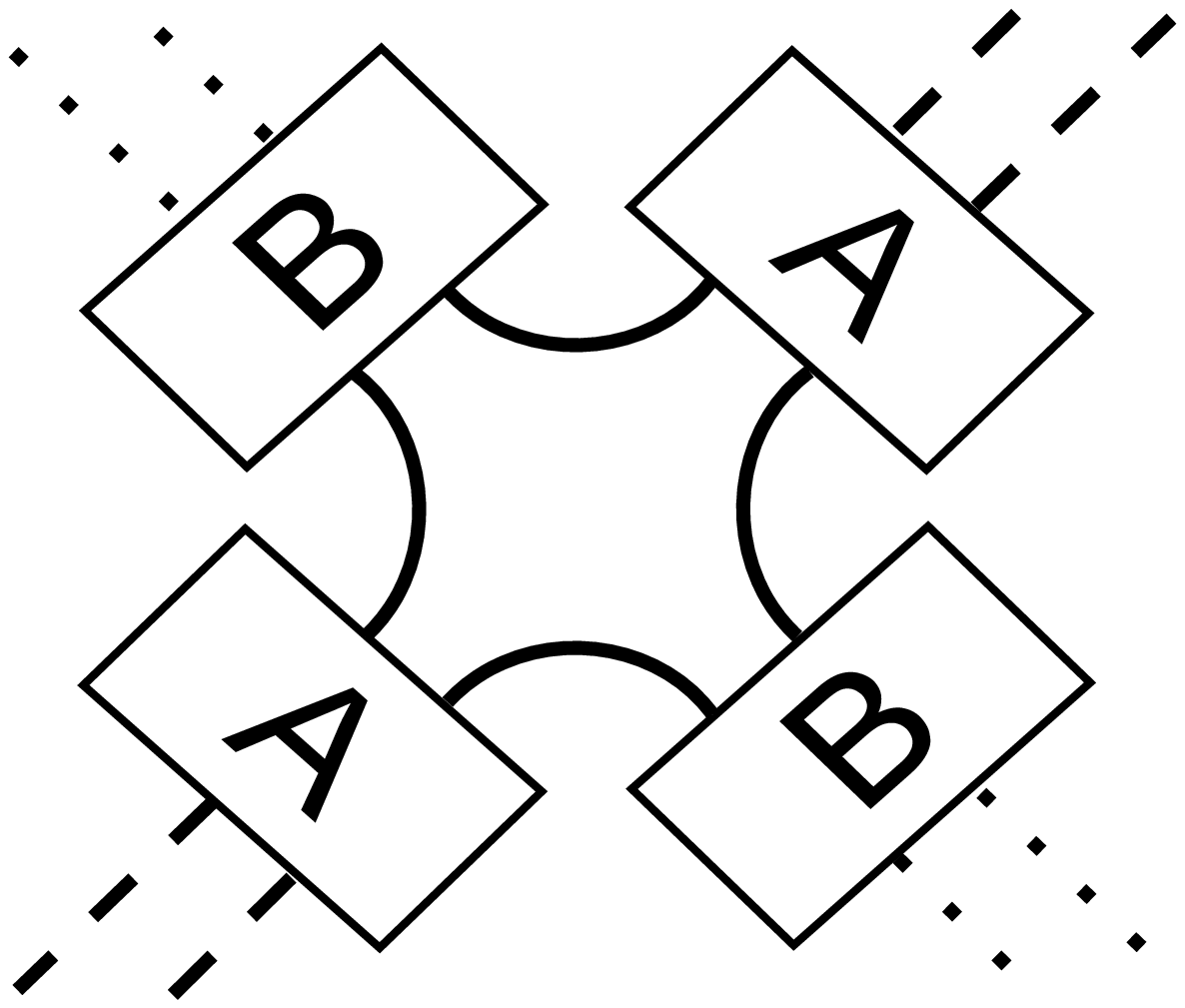}$. Because $A$ and $B$ are biprojections, it does not matter where we put the $\$$'s. If a non-closed A-colour string does not intersect with a B-colour string, then the A-colour string is just a common string. By our assumption, an A-colour string $- - - - -$ only connects with elements in $\mathscr{S}_A$, thus $\gra{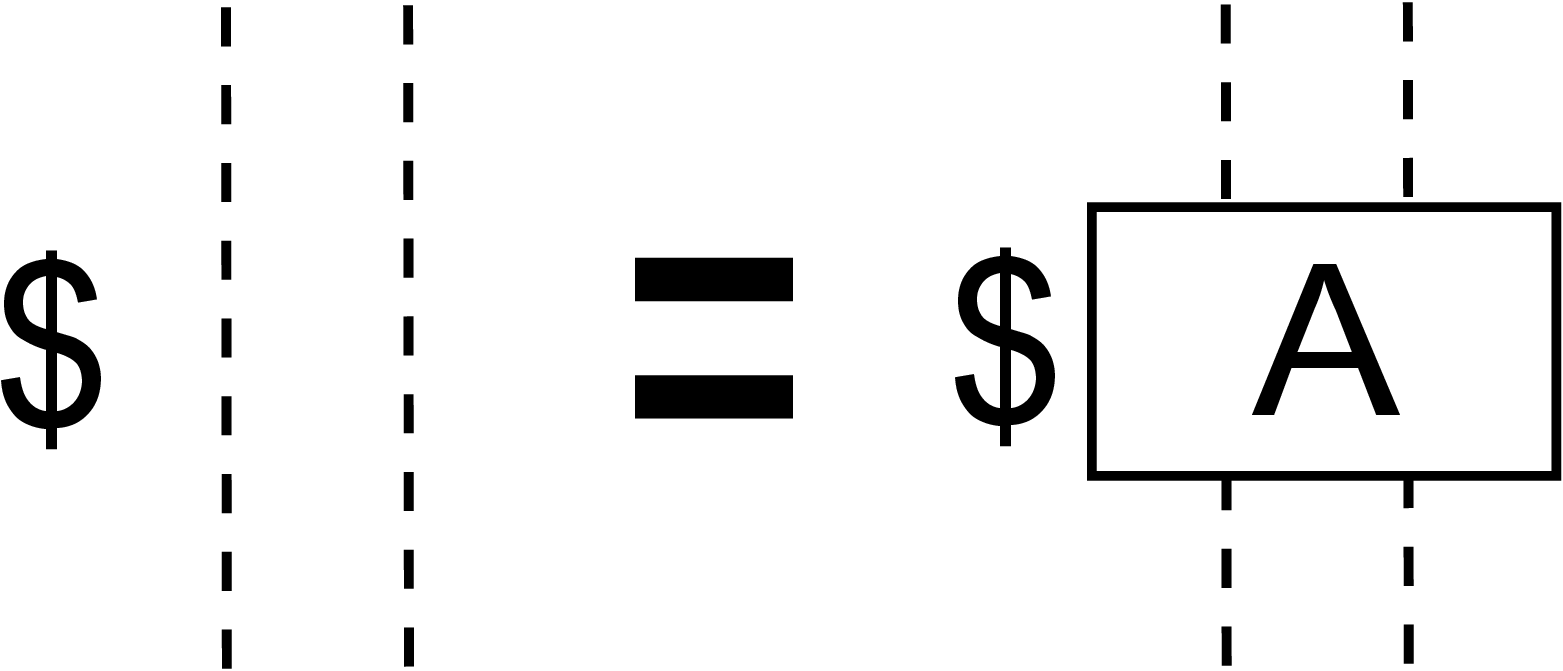}$. Moreover we can view $\gra{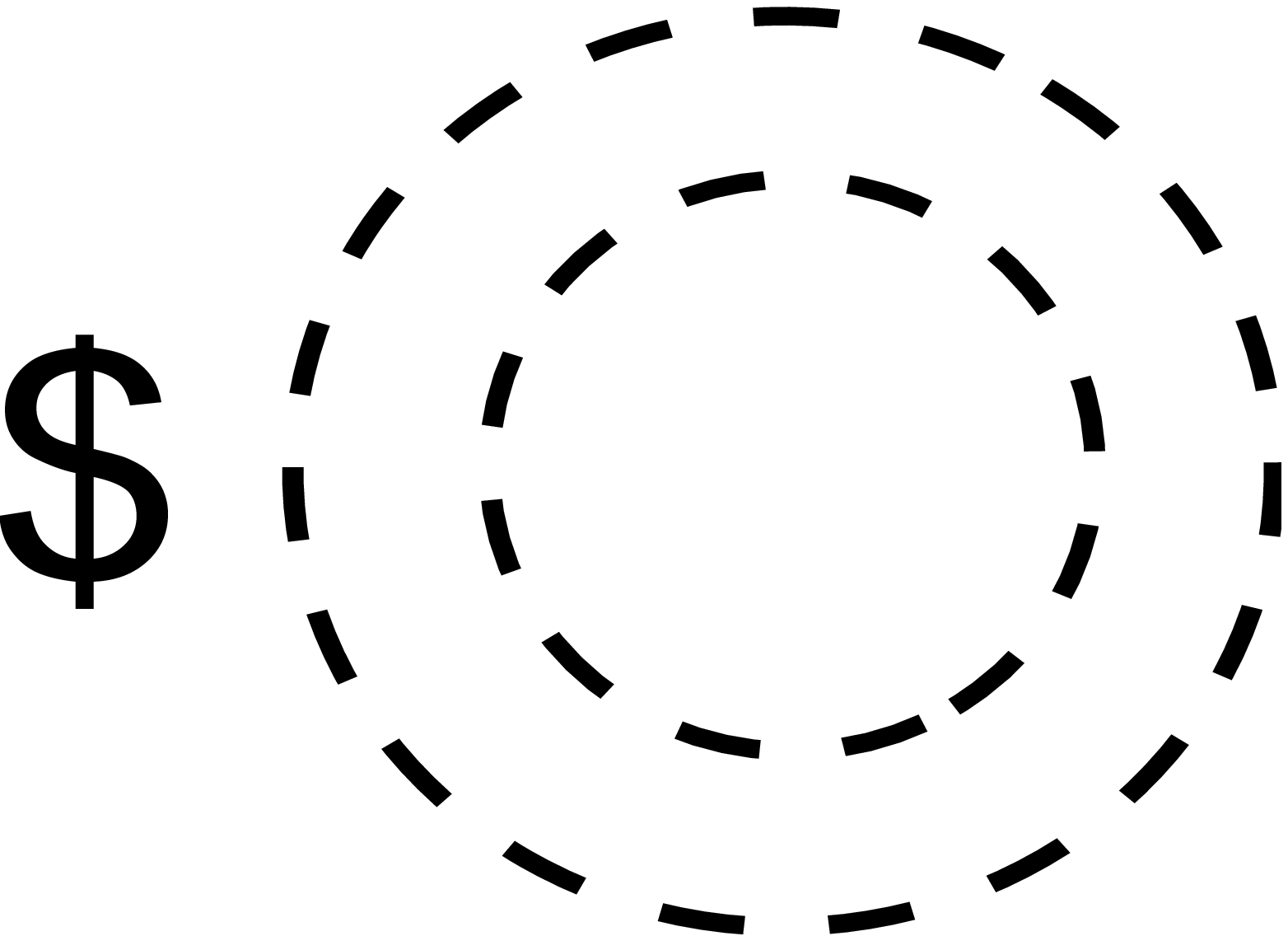}$ as $\gra{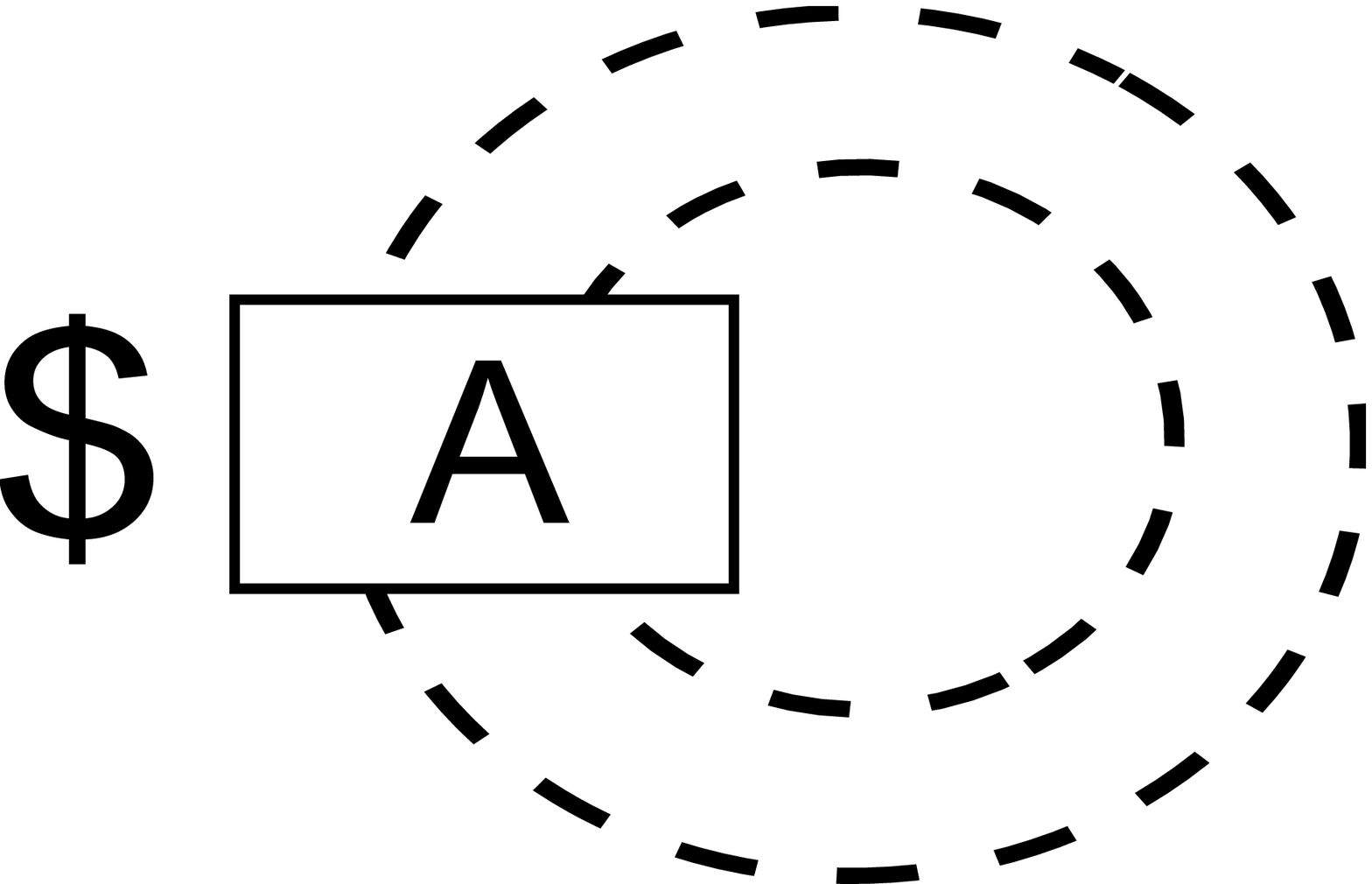}$, i.e., $\gra{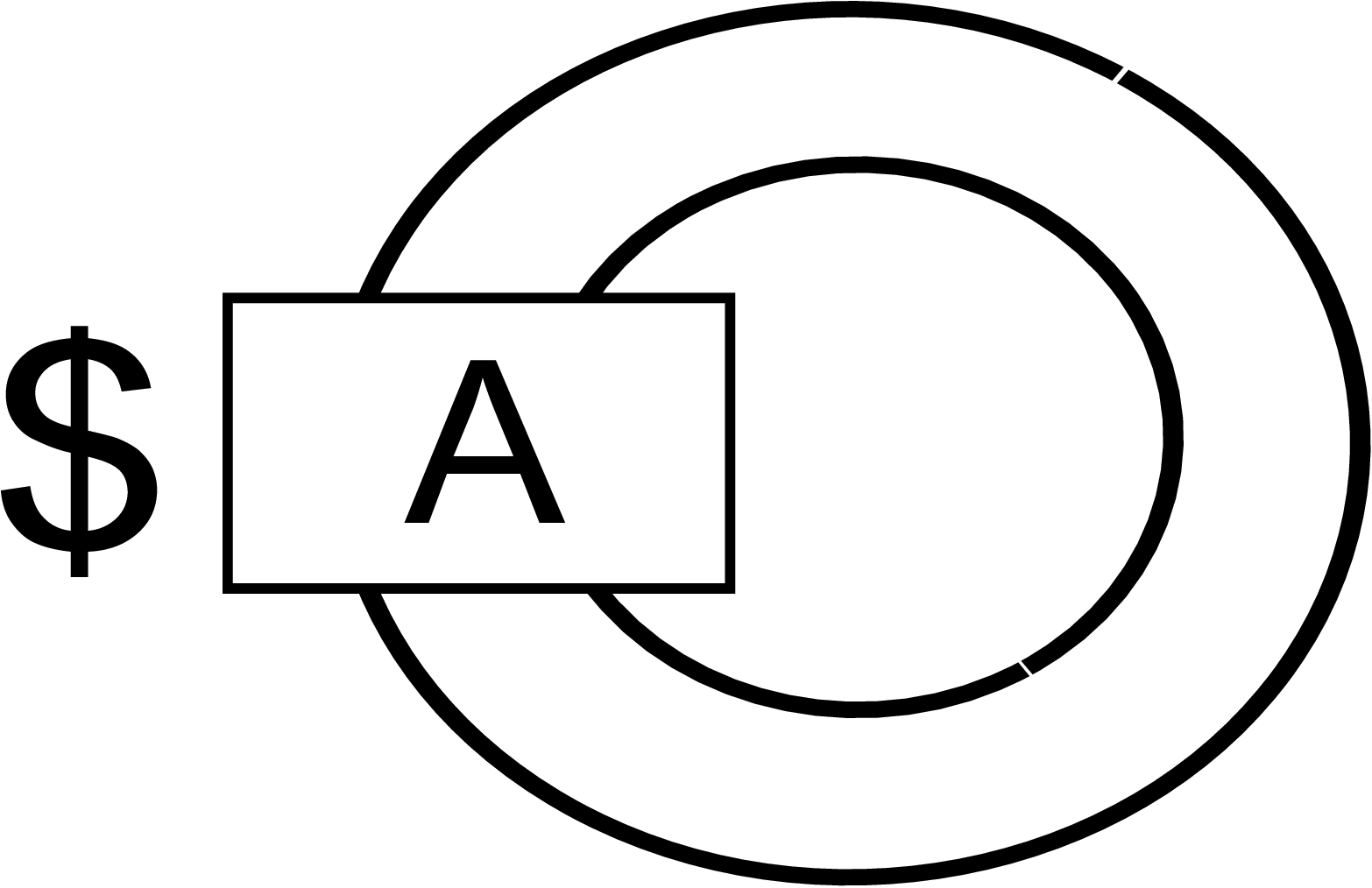}$. Thus $\gra{colourtracea1}=tr(A)$. Similarly for $B$.

\begin{proposition}\label{pqpq}
$\graa{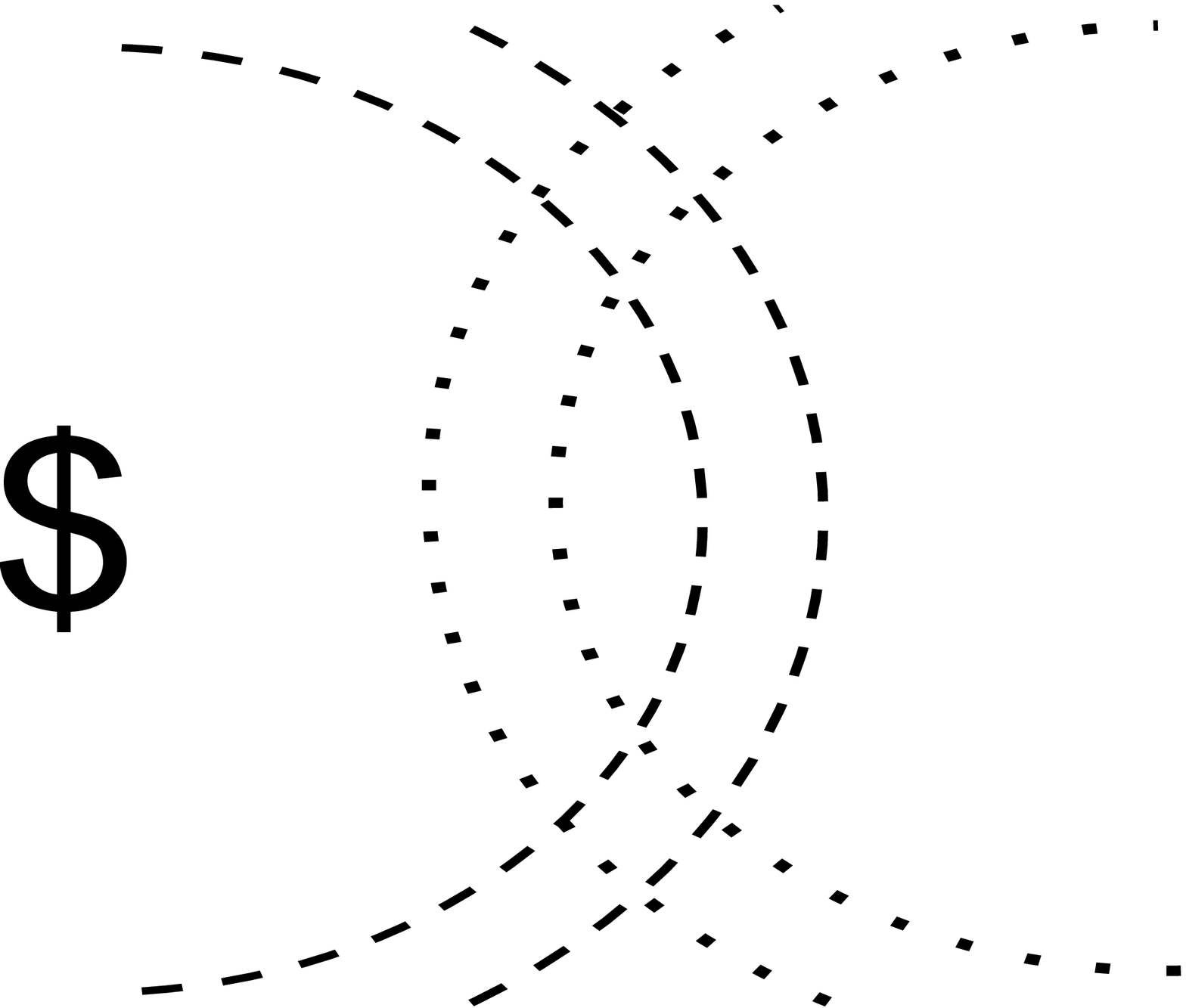}=\graa{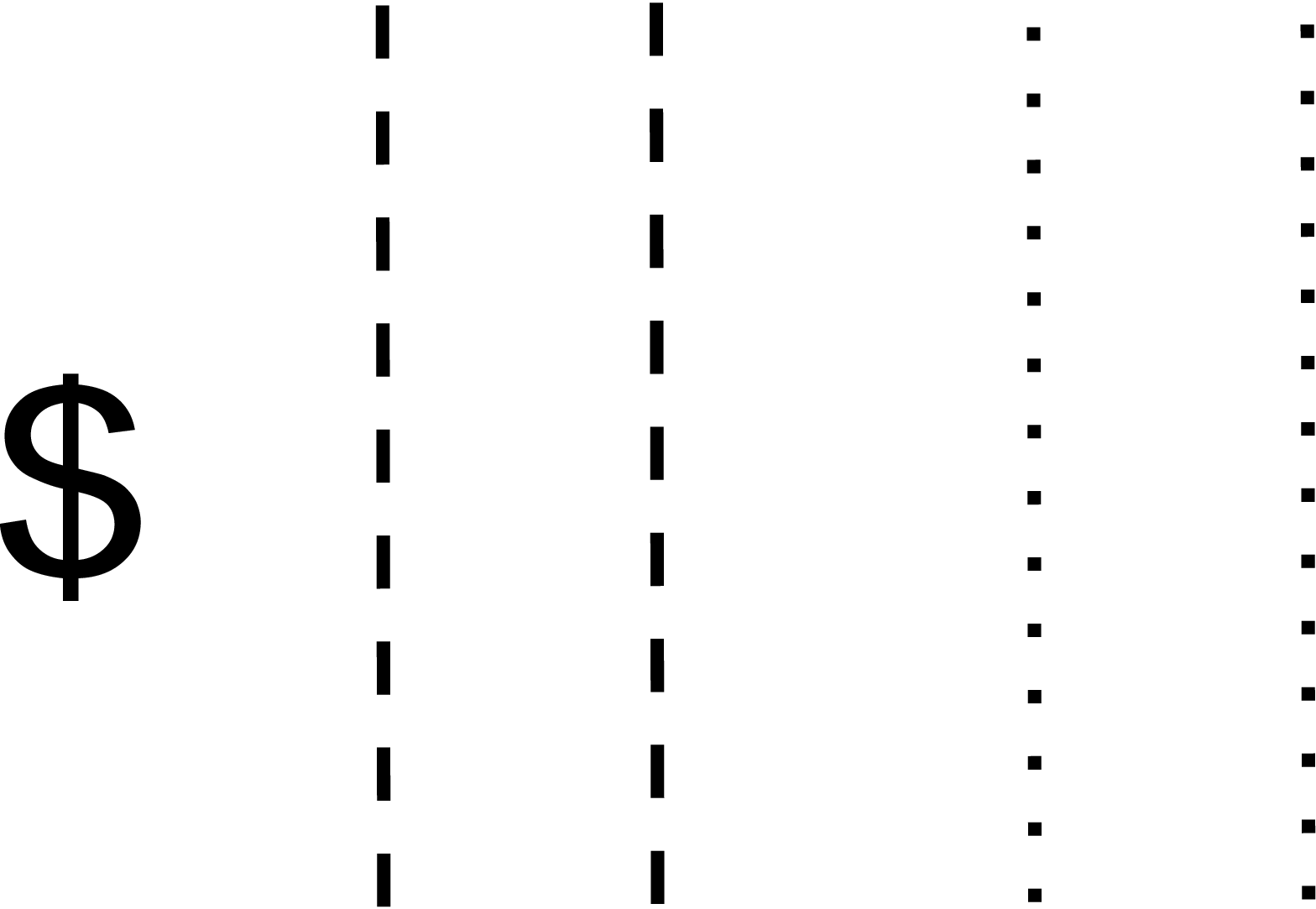}$.
\end{proposition}

\begin{proof}
$\graa{braidab1}=\delta^2\grb{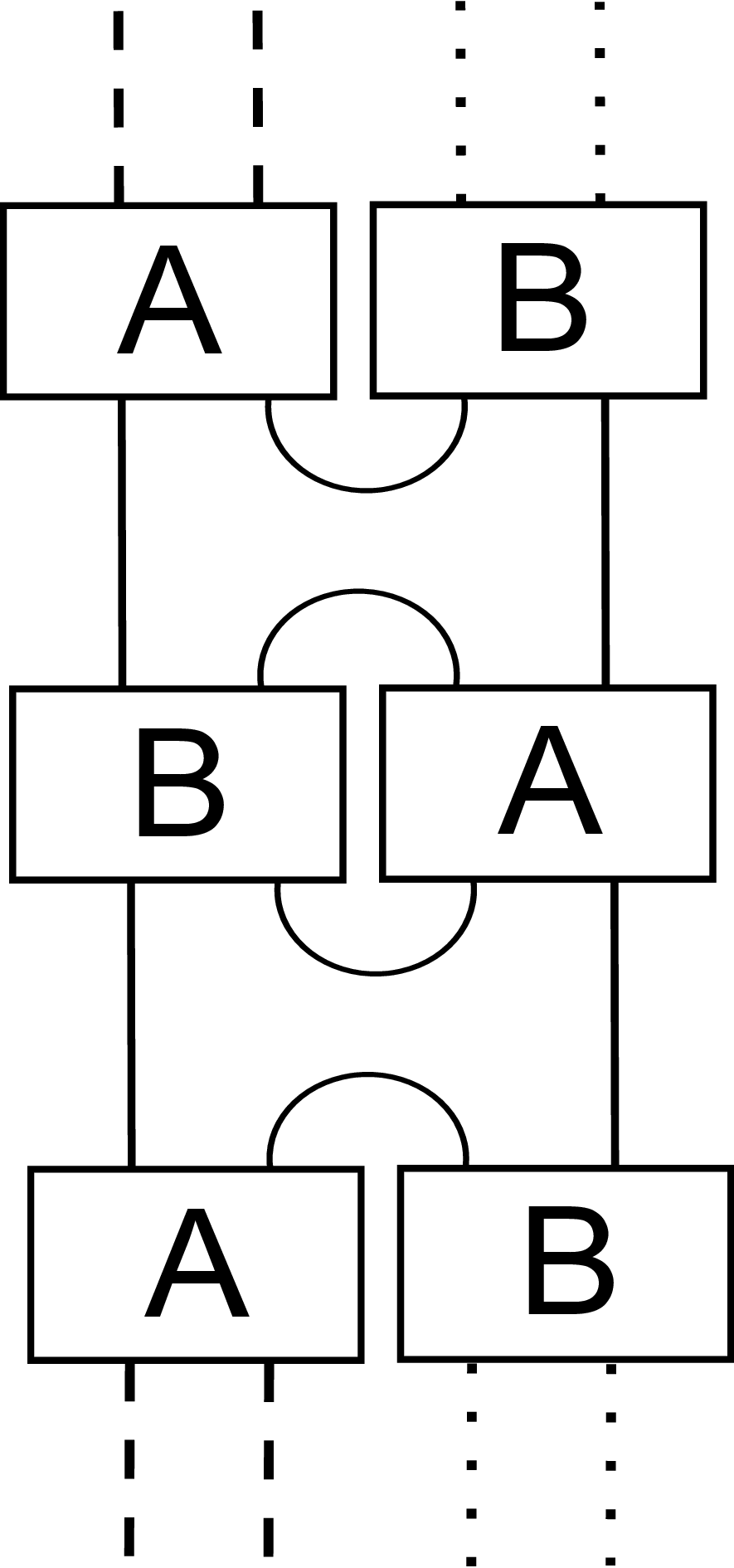}=\delta^2\grb{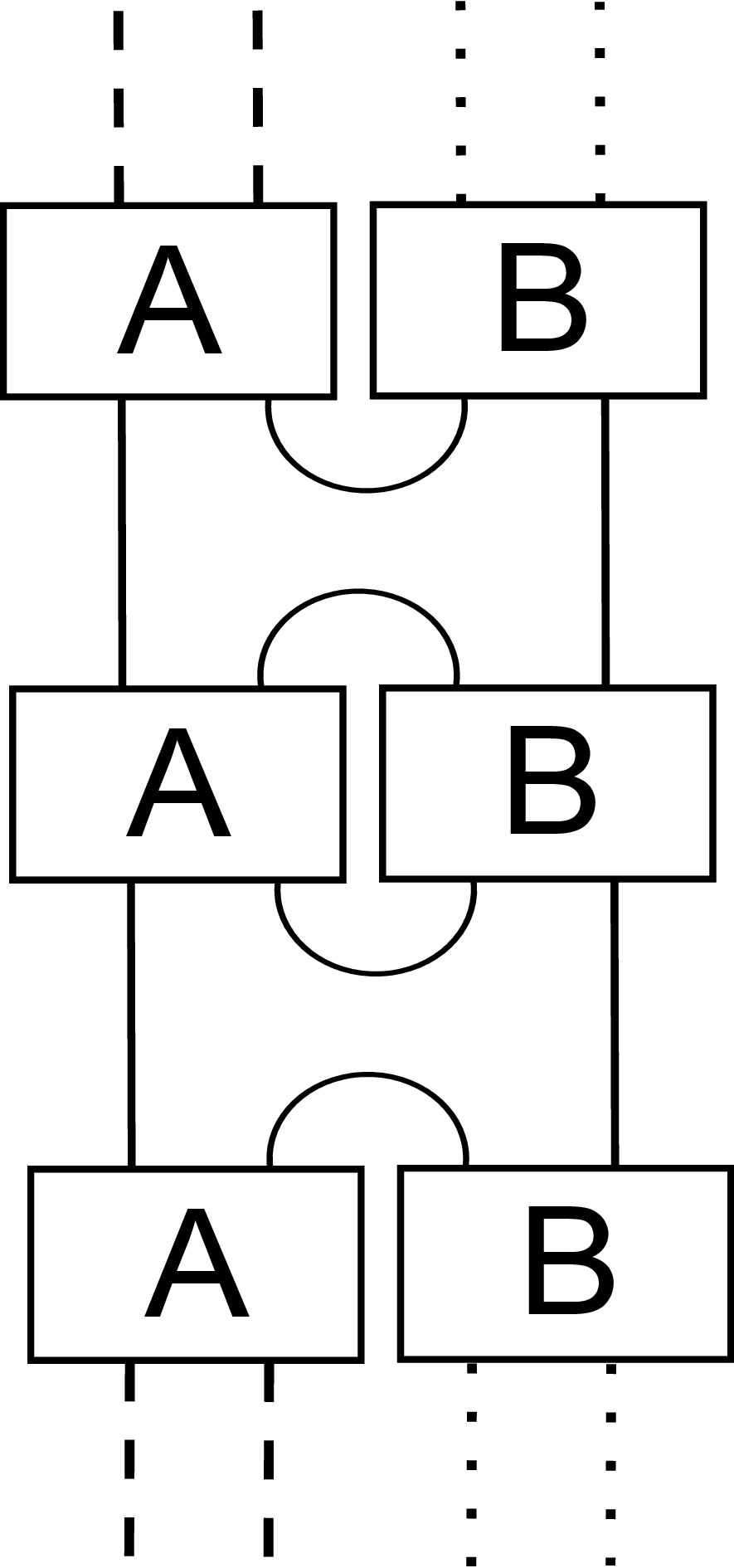}=\graa{braidab2}$.
\end{proof}

\begin{definition}
For $a\in\mathscr{S}_A$, we call $a$ B-flat, if
$$\graa{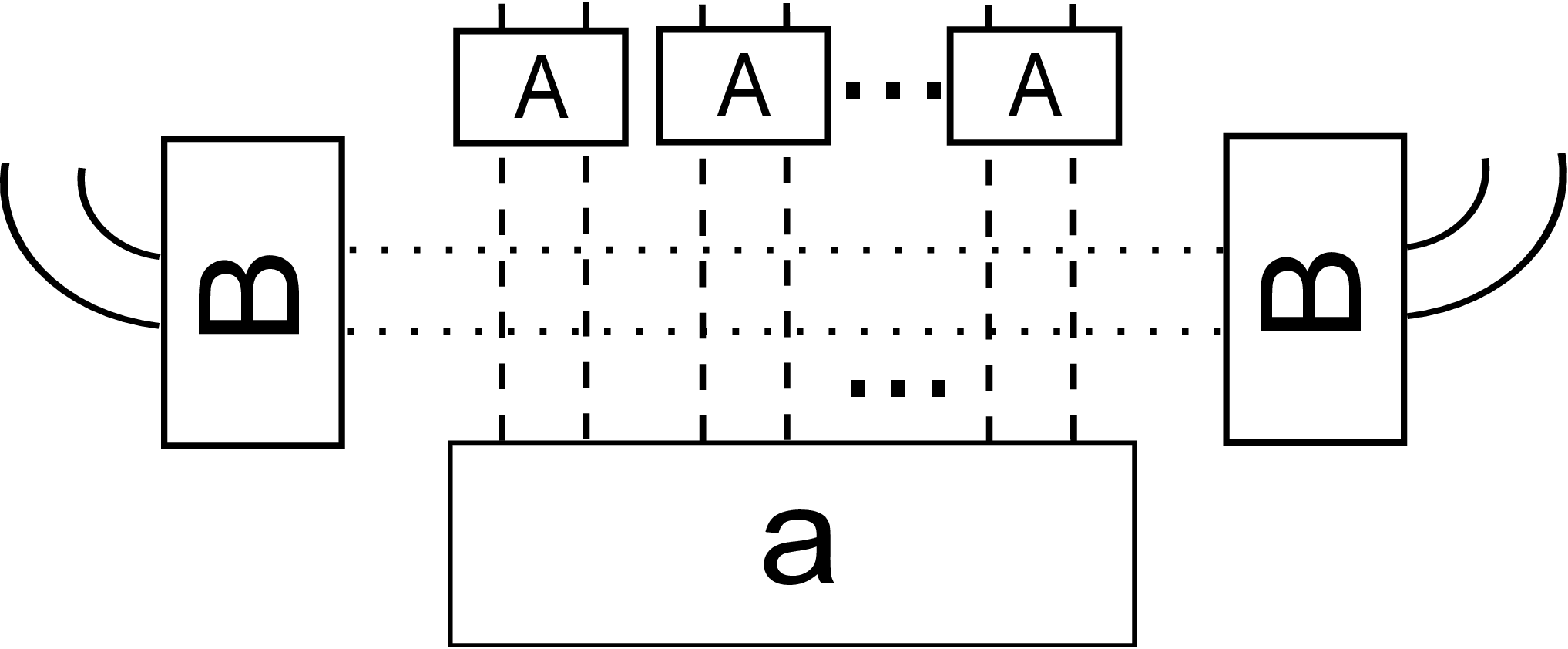}=\graa{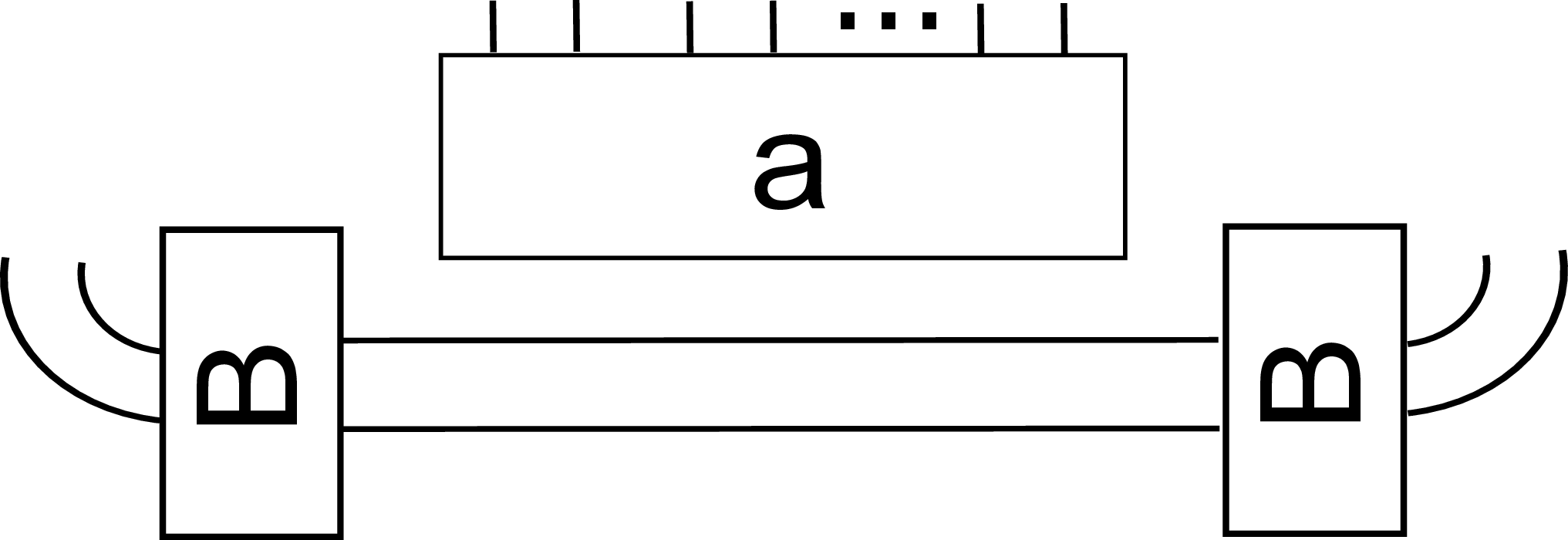}.$$
\end{definition}

By definition, the empty diagram $\emptyset$ in $(\mathscr{S}_A)_0$ is B-flat.

\begin{definition}
Let us define a map $\wedge: \mathscr{S}_m \otimes \mathscr{S}_n \rightarrow \mathscr{S}_{m+n}$,
$$x\wedge y=\gra{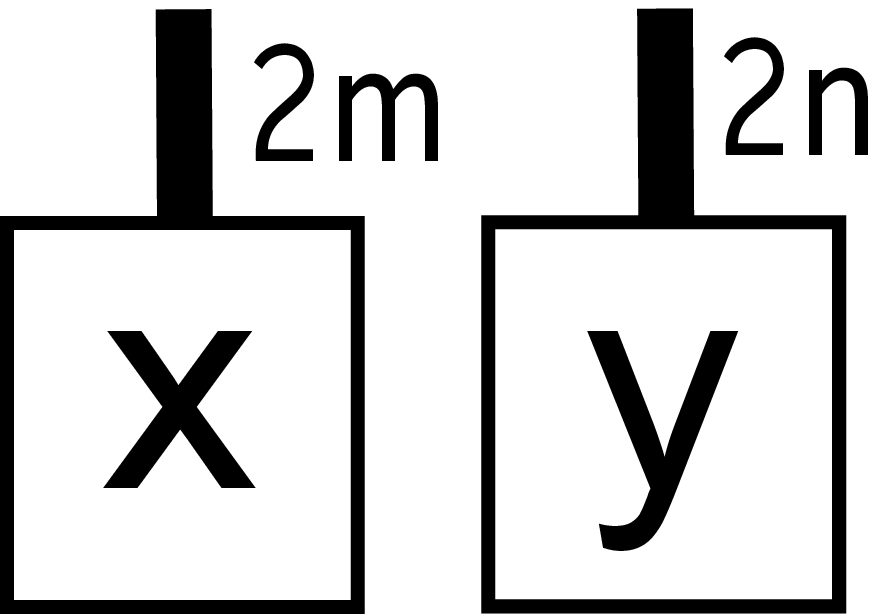}, ~\forall~ x\in \mathscr{S}_m, ~y\in\mathscr{S}_n, m,n\geq0.$$
\end{definition}

Then its restriction on $\mathscr{S}_A$ is a map $\wedge: (\mathscr{S}_A)_m \otimes (\mathscr{S}_A)_n \rightarrow (\mathscr{S}_A)_{m+n}$.

\begin{proposition}
Suppose $x, y\in \mathscr{S}_A$ are B-flat, then $x\wedge y$ is B-flat.
\end{proposition}
\begin{proof}
$$\graa{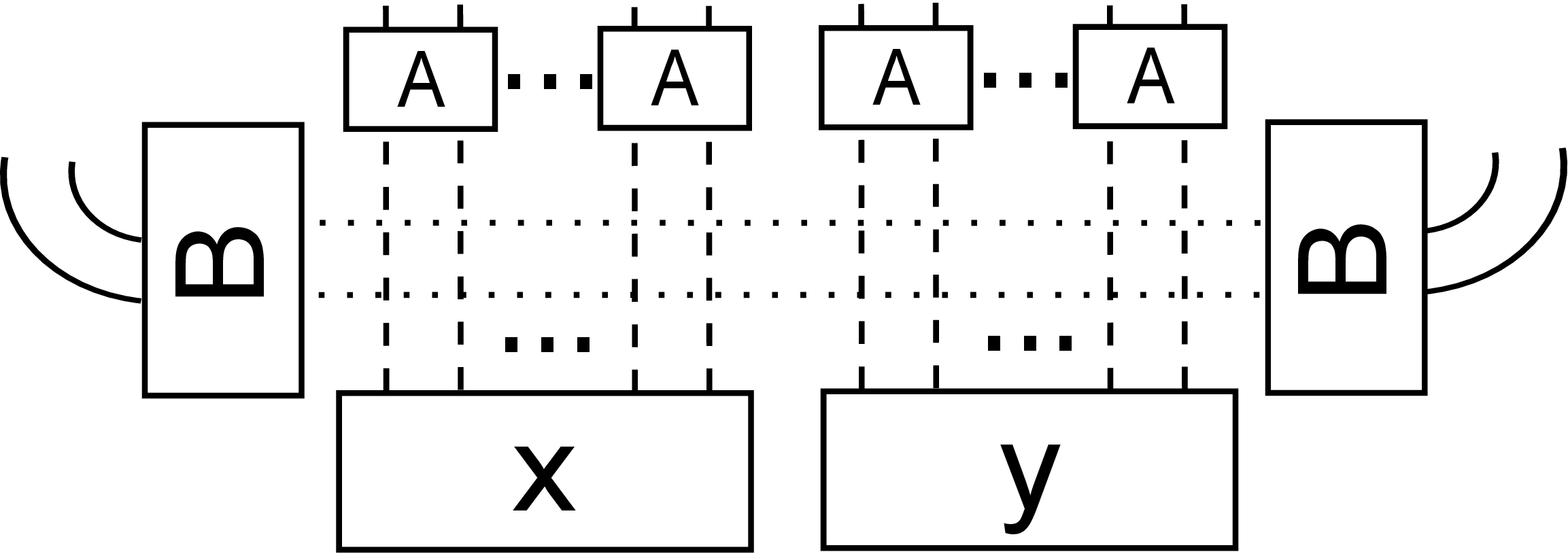}$$
$$=\graa{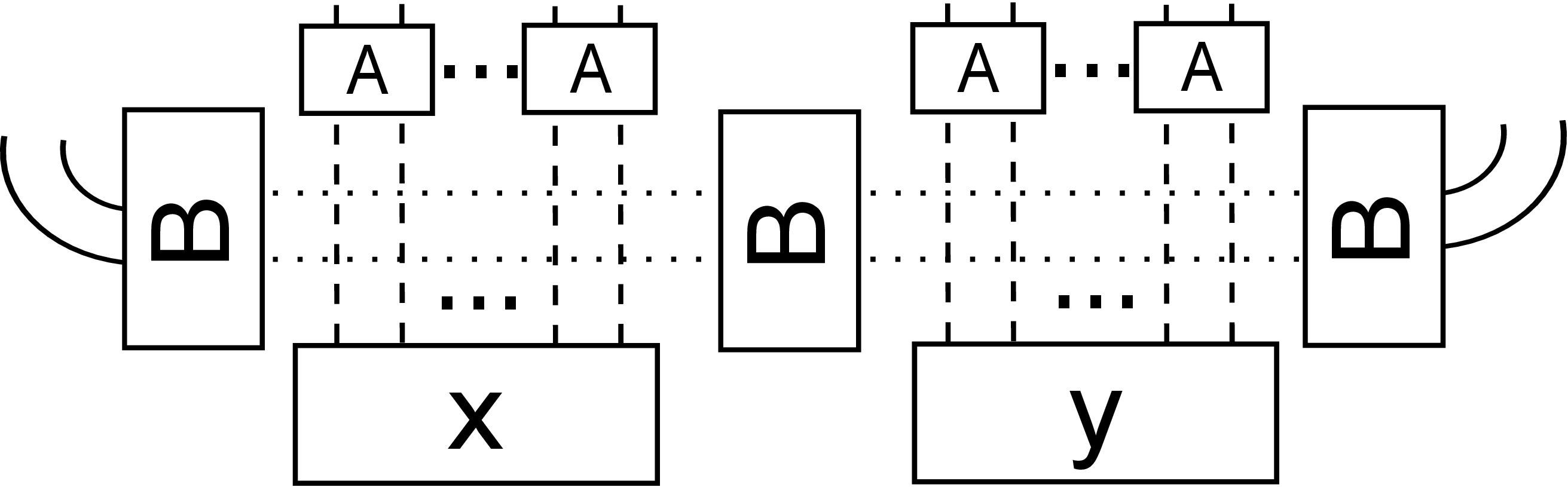}$$
$$=\graa{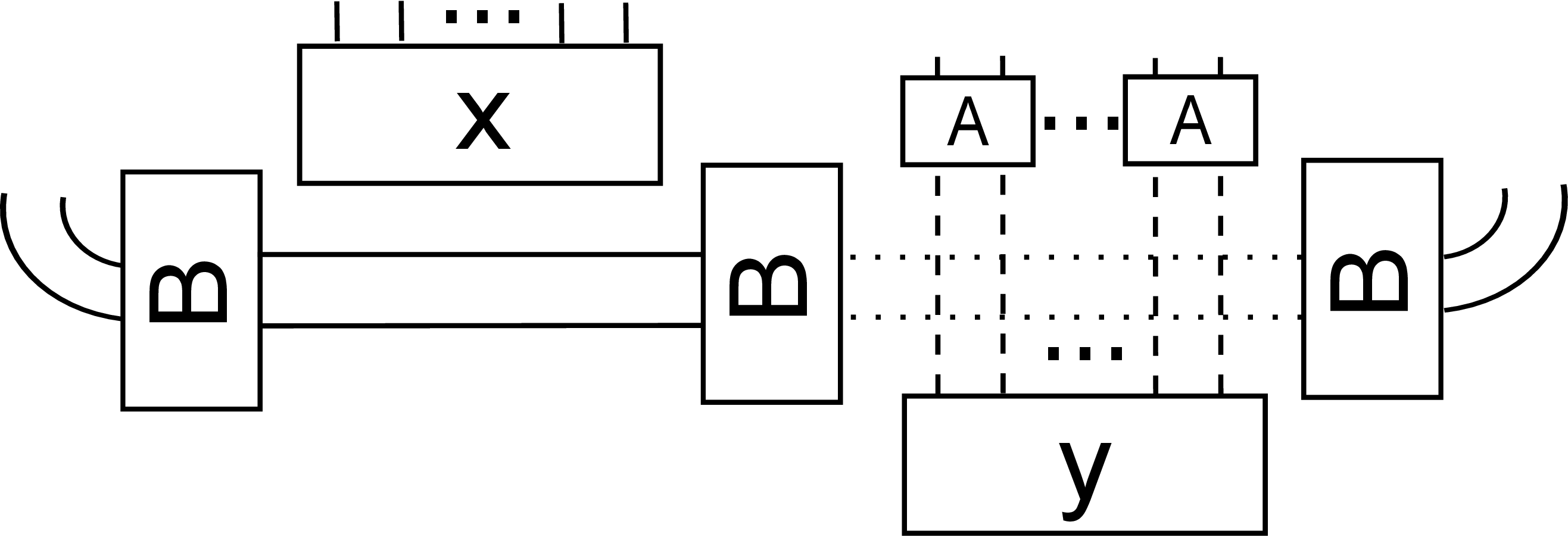}$$
$$=\gra{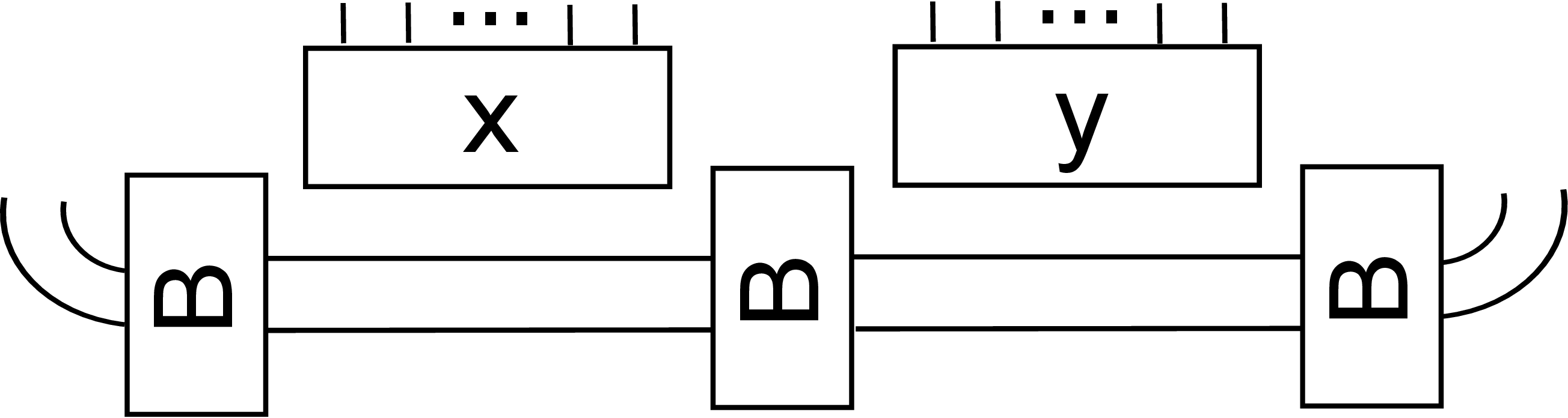}$$
$$=\gra{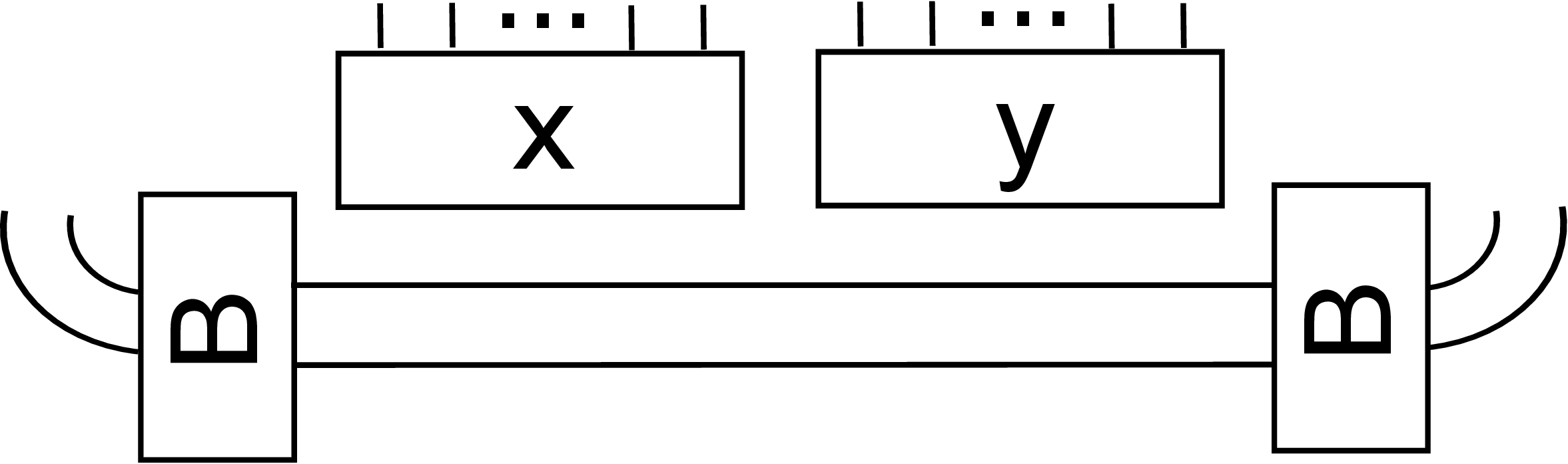}.$$
Thus $x\wedge y$ is B-flat.
\end{proof}

\begin{definition}
Let us define the annular action $\Theta_B$ of B-colour strings on $\mathscr{S}_A$ as
$$\Theta_B(a)=\grb{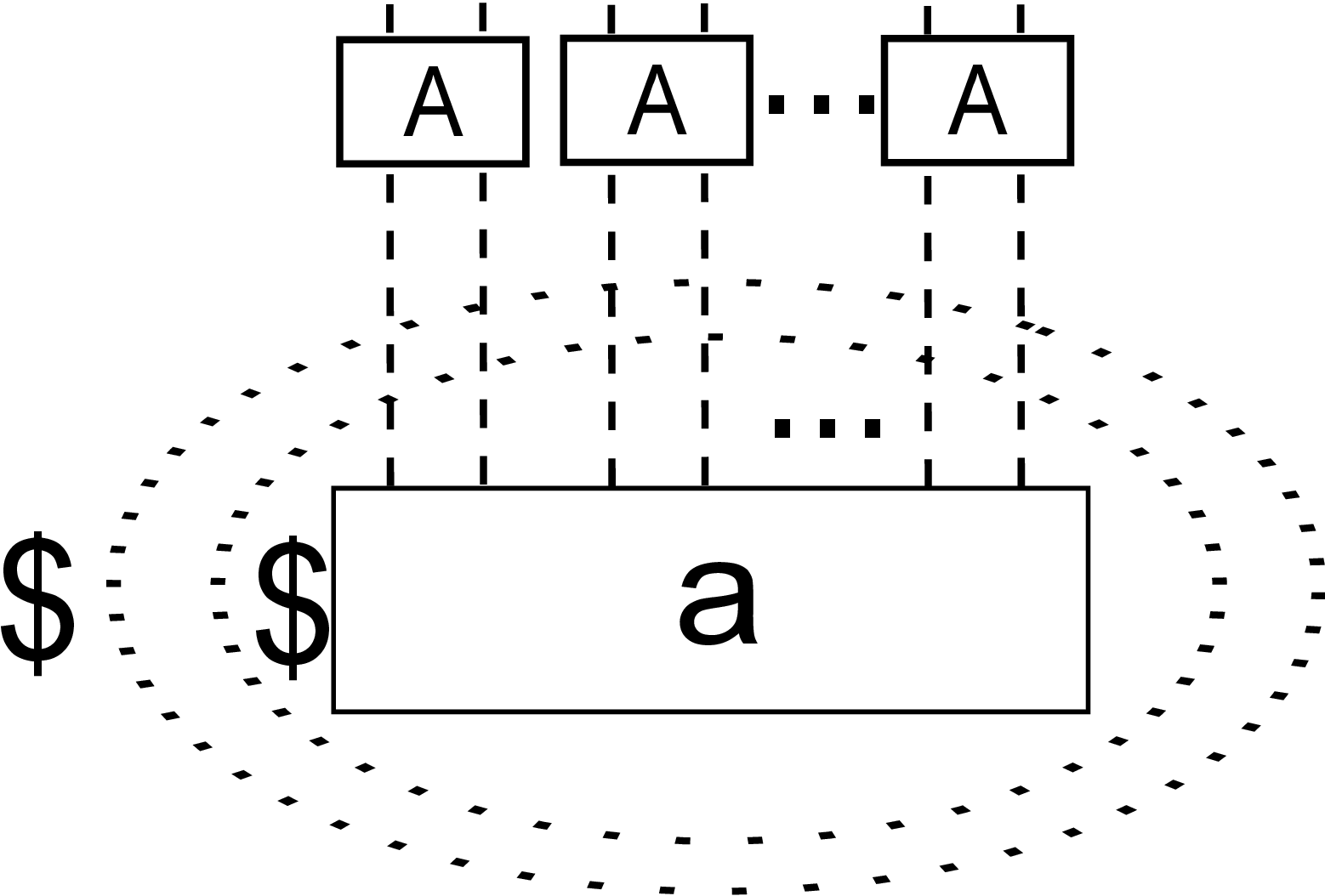}, ~\forall ~a\in\mathscr{S}_A.$$
\end{definition}

\begin{proposition}\label{annularflat}
For any $a\in\mathscr{S}_A$, $a$ is B-flat if and only if $\Theta_B(a)=tr(B)a$.
\end{proposition}

\begin{proof}
If $a$ is B-flat, then $\Theta_B(a)=\Theta_B(\emptyset)a=tr(B)a$.

On the other hand, we assume that $\Theta_B(a)=tr(B)a$, then $\Theta_B(a^*)=(\Theta_B(a))^*=tr(B)a^*$.
Let $x$ be
$$\graa{flatab1}-\graa{flatab2}.$$
Then
$tr(xx^*)=\Theta_B(\emptyset)tr(aa^*)+\Theta_B(tr(aa^*)\emptyset)-tr(\Theta_B(a)a^*)-tr(a\Theta_B(a^*))
=tr(B)tr(aa^*)+tr(B)tr(aa^*)-tr(B)tr(aa^*)-tr(B)tr(aa^*)=0.$
Thus $x=0$. That means $a$ is B-flat.
\end{proof}

\begin{corollary}
For any $a\in\mathscr{S}_A$, if $a$ is B-flat, then $a^*$ is B-flat.
\end{corollary}

\begin{theorem}
All B-flat elements in $\mathscr{S}_A$ denoted by $Flat_B(\mathscr{S}_A)$, forms a non-zero planar subalgebra of $\mathscr{S}_A$.
Consequently all Temperley-Lieb elements of $\mathscr{S}_A$ are B-flat.
\end{theorem}

\begin{proof}
We have known that the empty diagram $\emptyset\in(\mathscr{S}_A)_0$ is B-flat.
And $Flat_B(\mathscr{S}_A)$ is invariant under the adjoint action and the map $\wedge$.
Thus we only need to show that $Flat_B(\mathscr{S}_A)$ is invariant under the annular Temperley-Lieb action.
By Proposition \ref{annularflat}, it is sufficient to prove that the annular Temperley-Lieb action commute with $\Theta_B$.
Remember that, each Temperley-Lieb annular action is a composition of the following five operations,
(1) a two-string rotation;
(2) adding a cap on a shaded boundary;
(3) adding a cap on a unshaded boundary;
(4) adding a string in a shaded boundary;
(5) adding a string in a unshaded boundary. Thus we only need to prove $\Theta_B$ commutes with (1)-(5).

$\Theta_B$ commute with (1) follows from their definitions.

$\Theta_B$ commutes with (2) is equivalent to $\grb{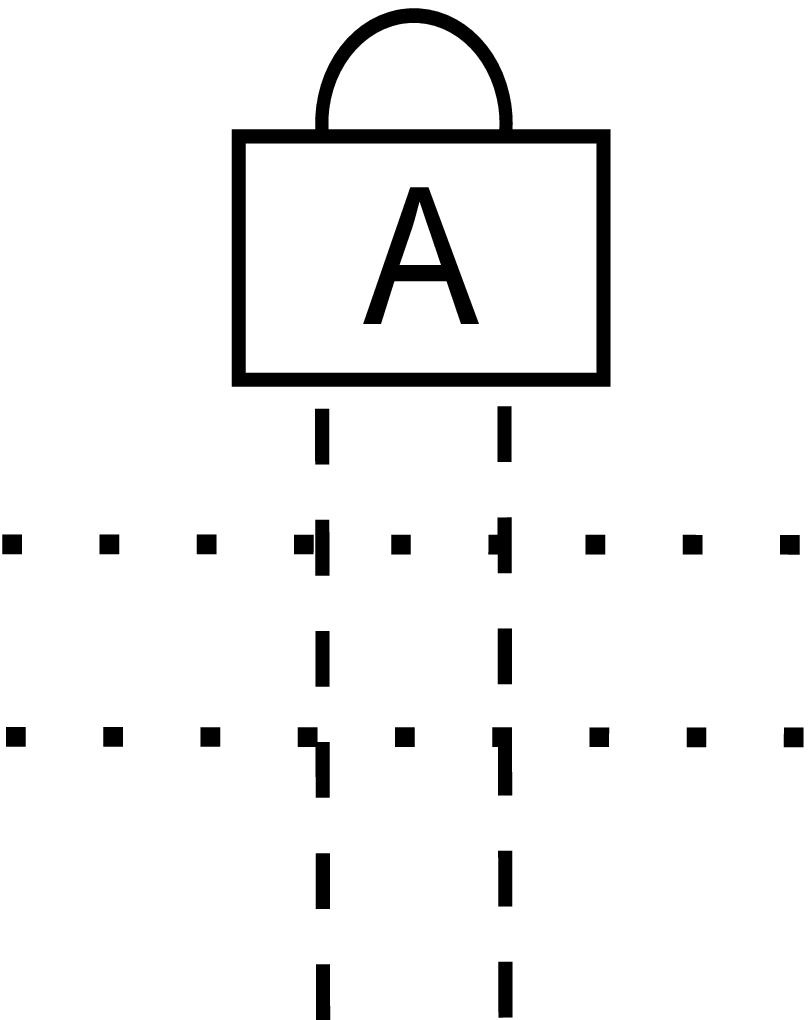}=\grb{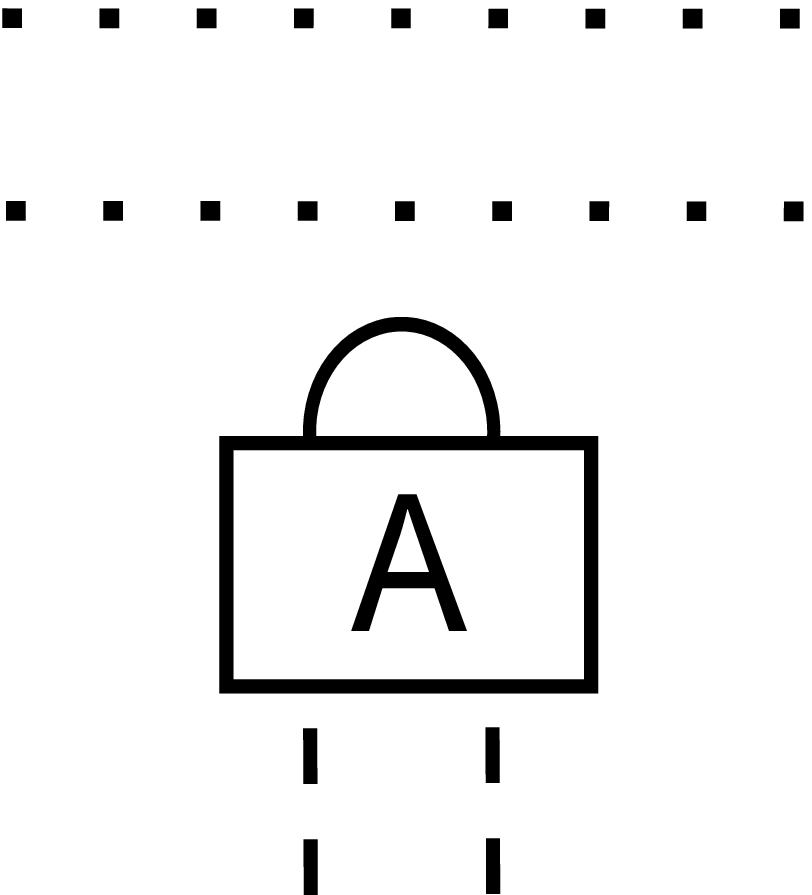}$. That is
$$\grb{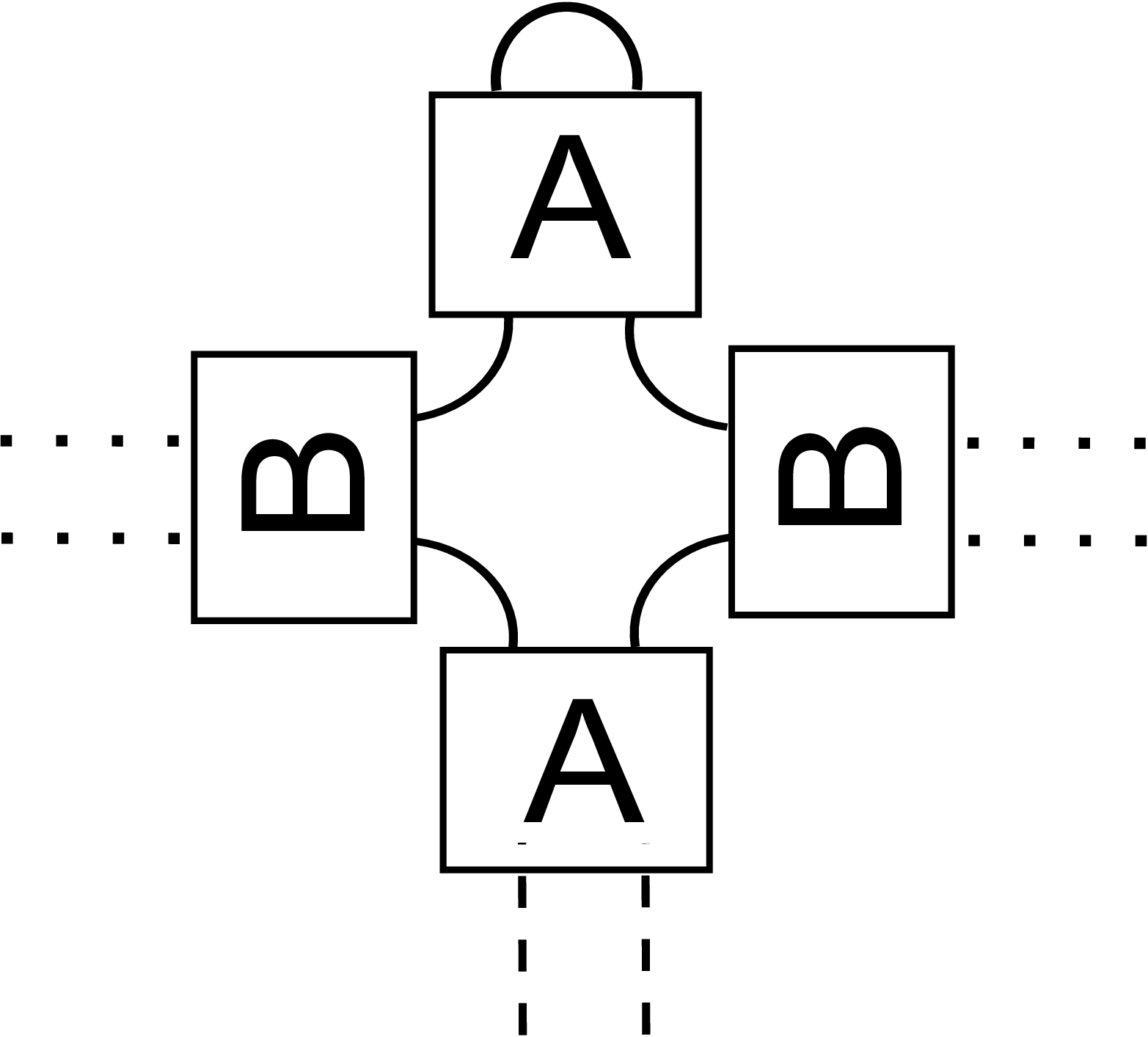}=\grb{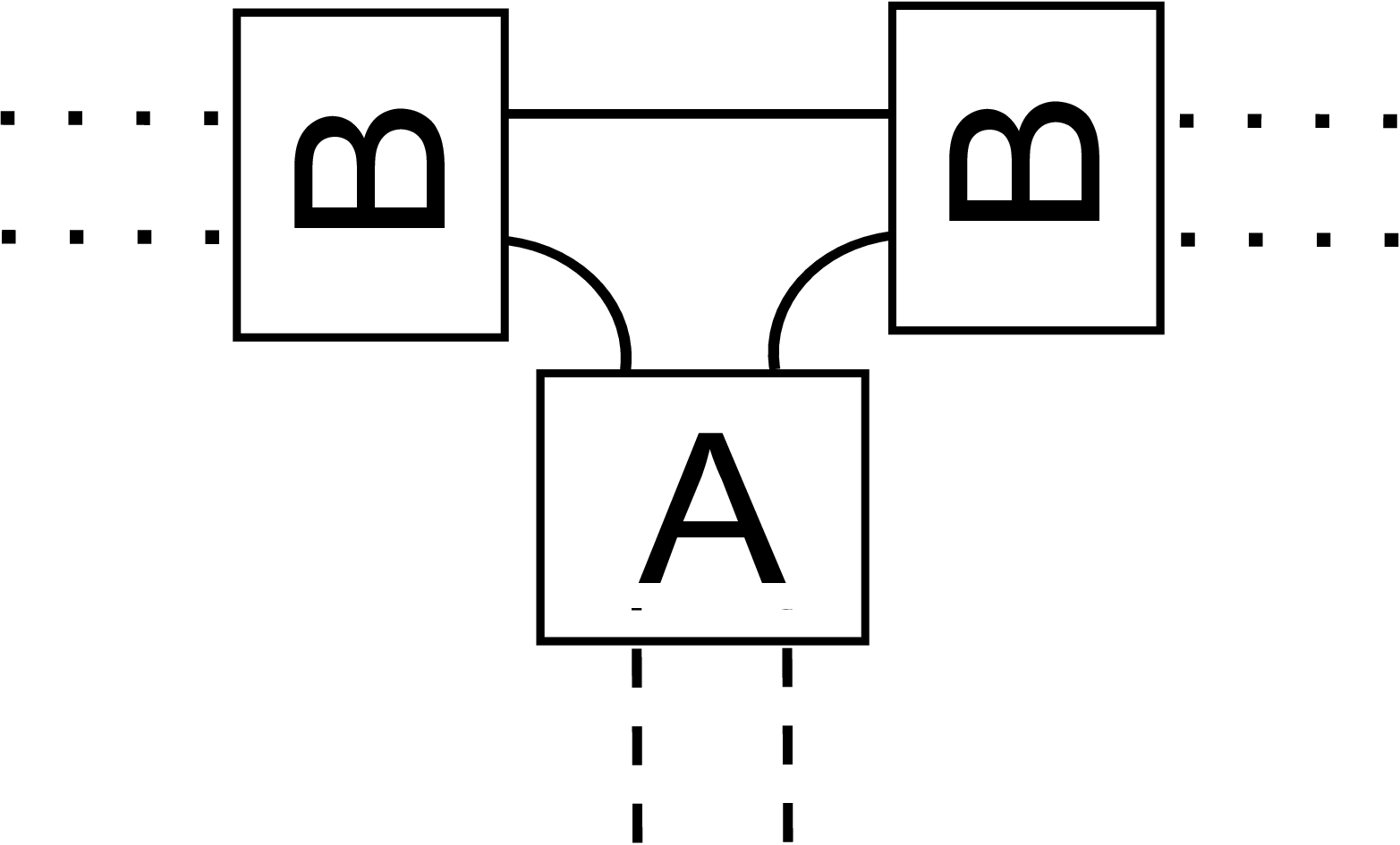}=\grb{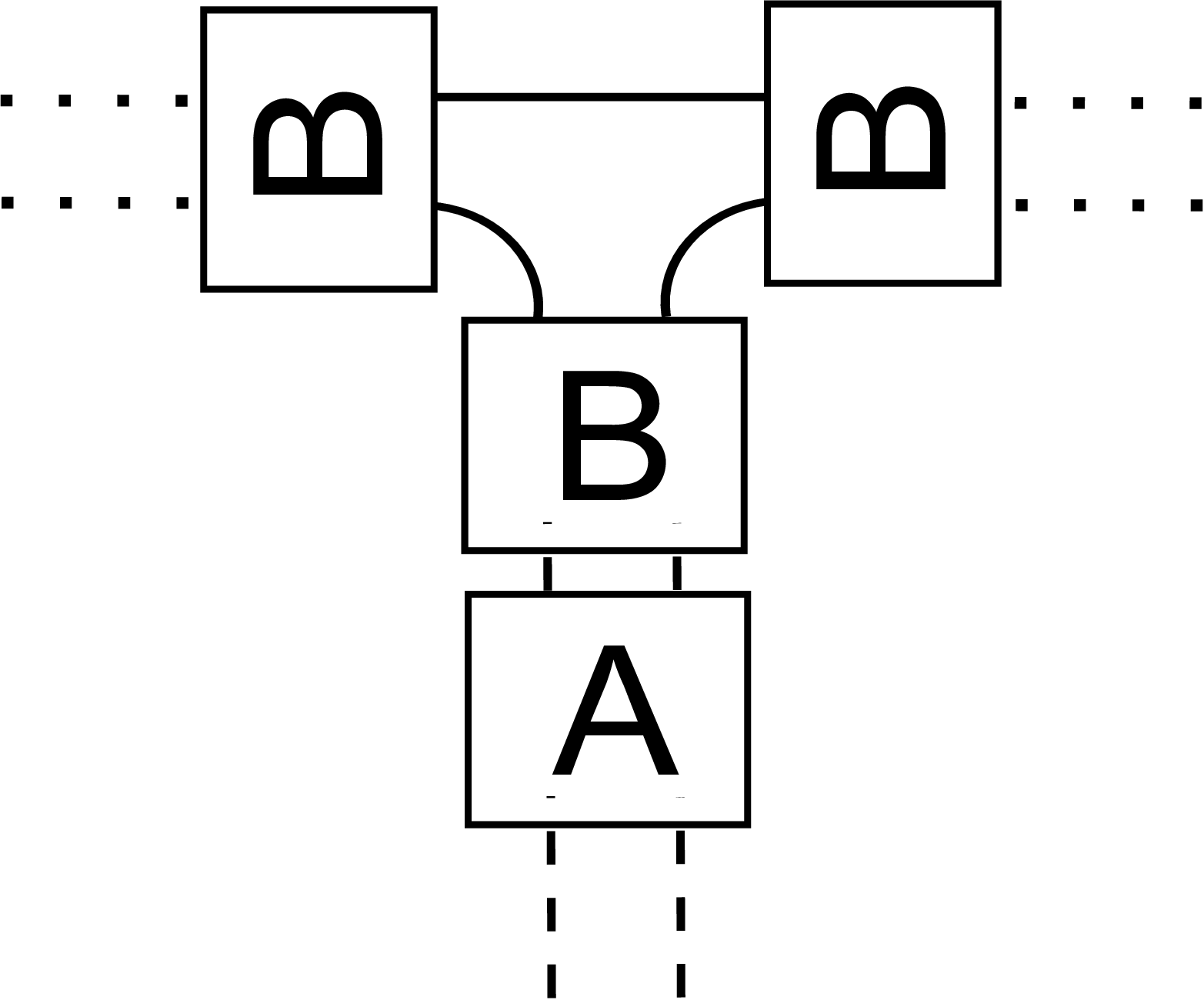}.$$

$\Theta_B$ commutes with (3) is equivalent to $\grb{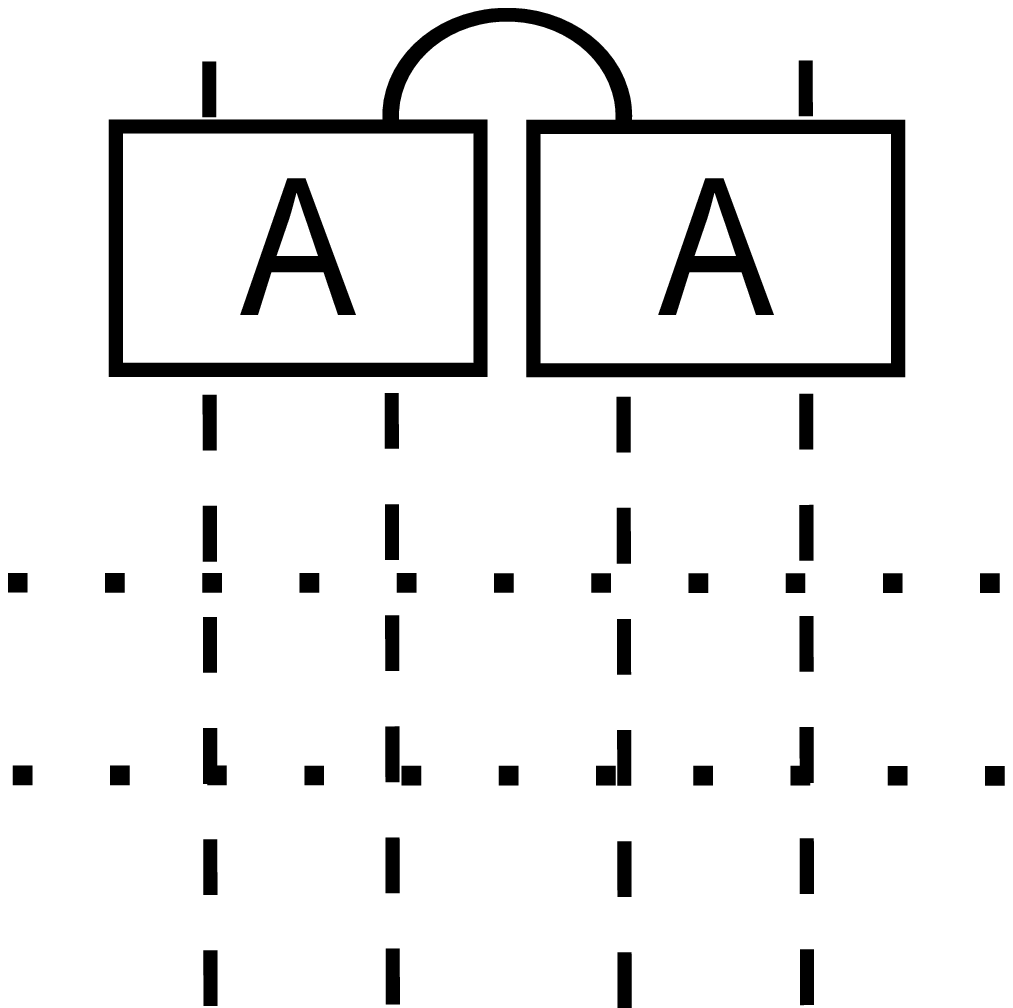}=\grb{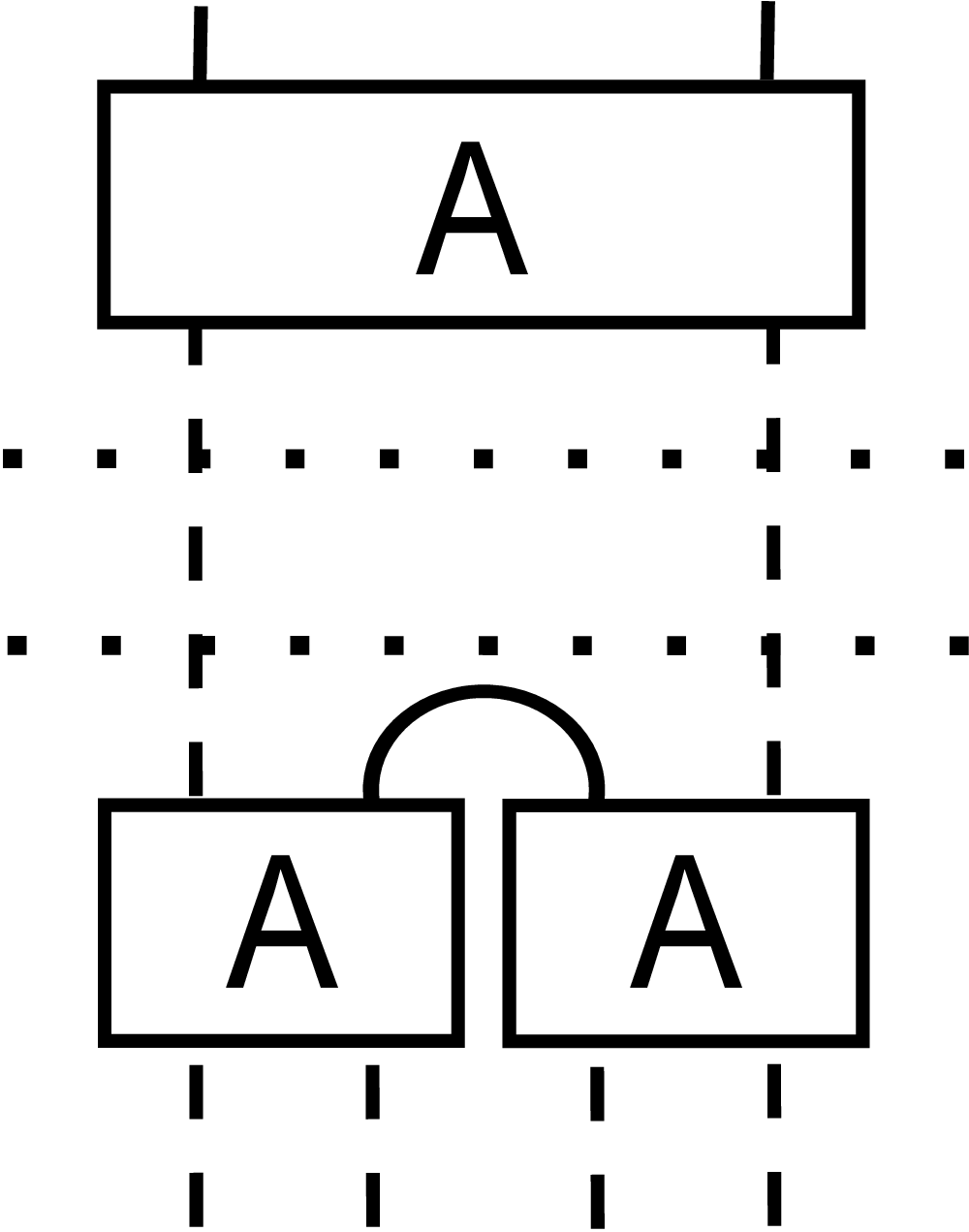}$. That is
$$\grb{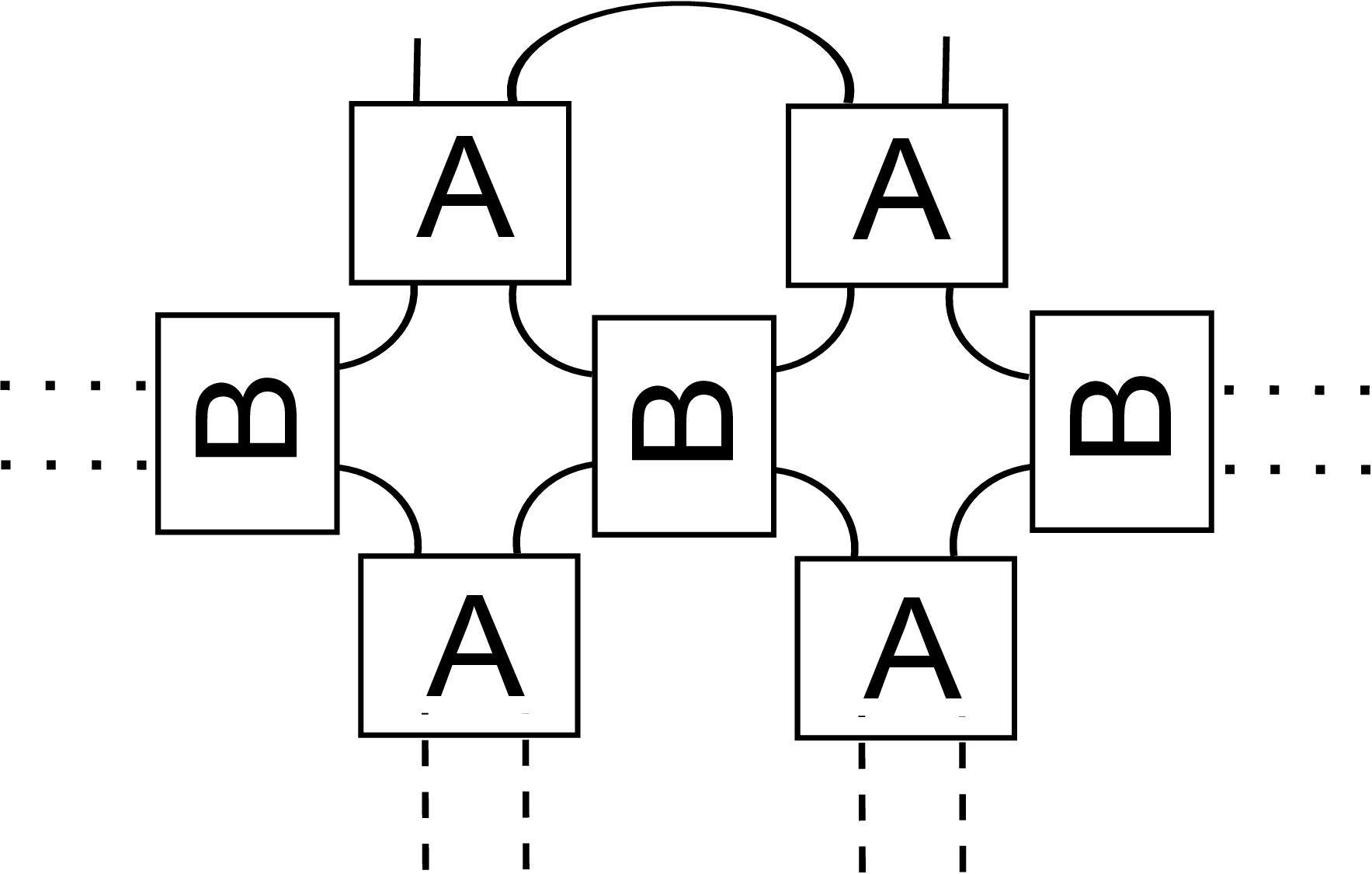}=\grb{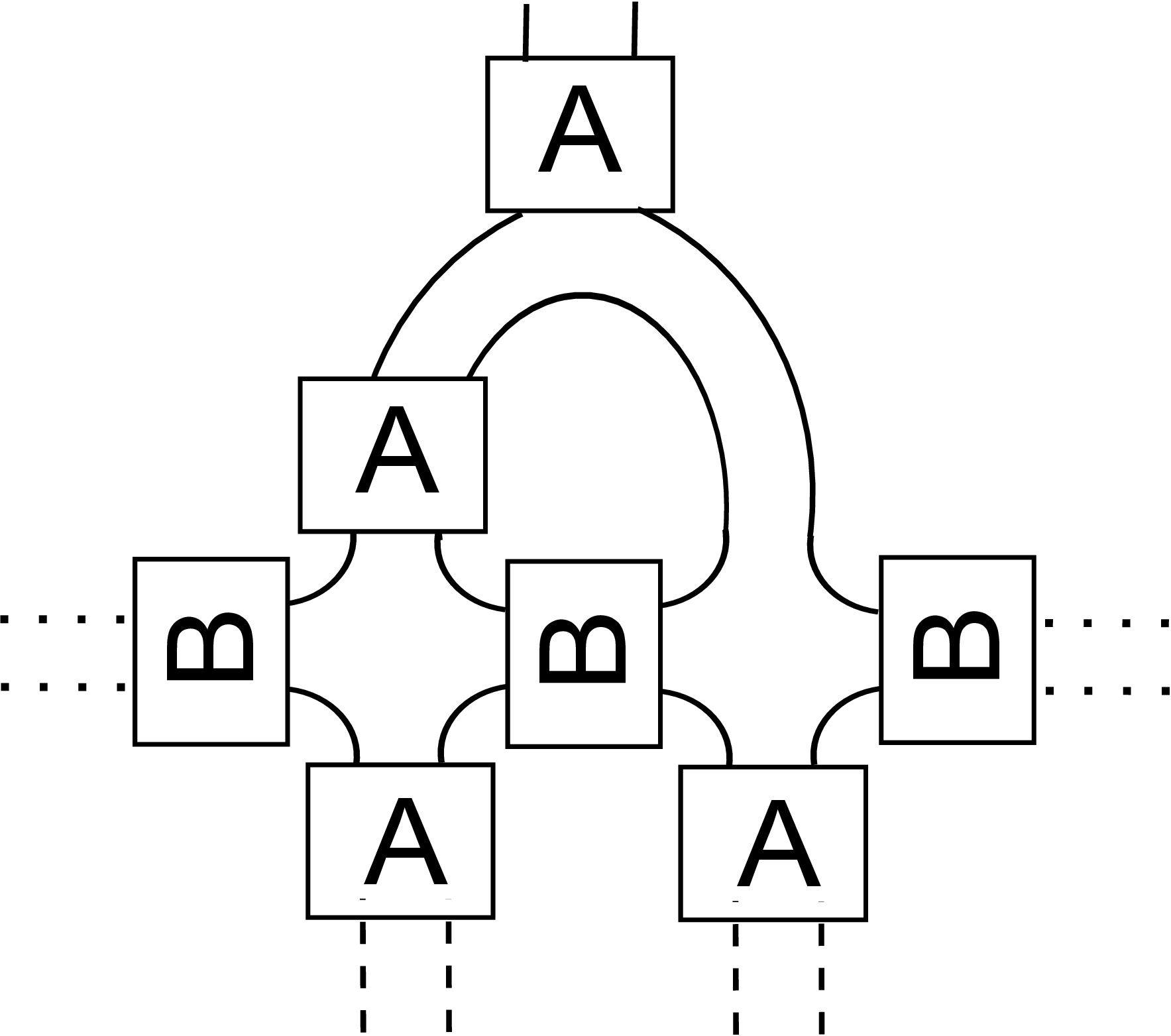}=\grb{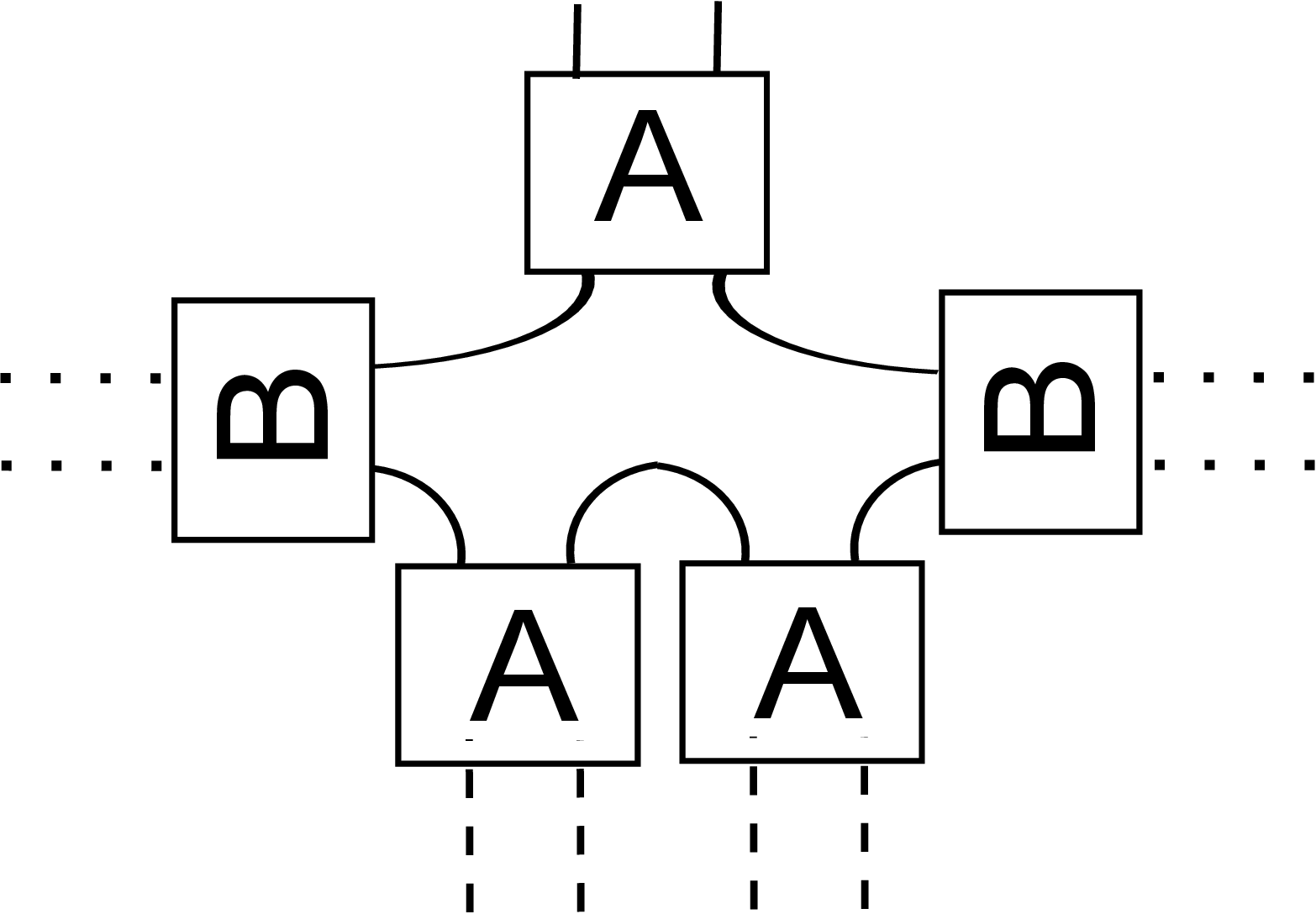}=\grb{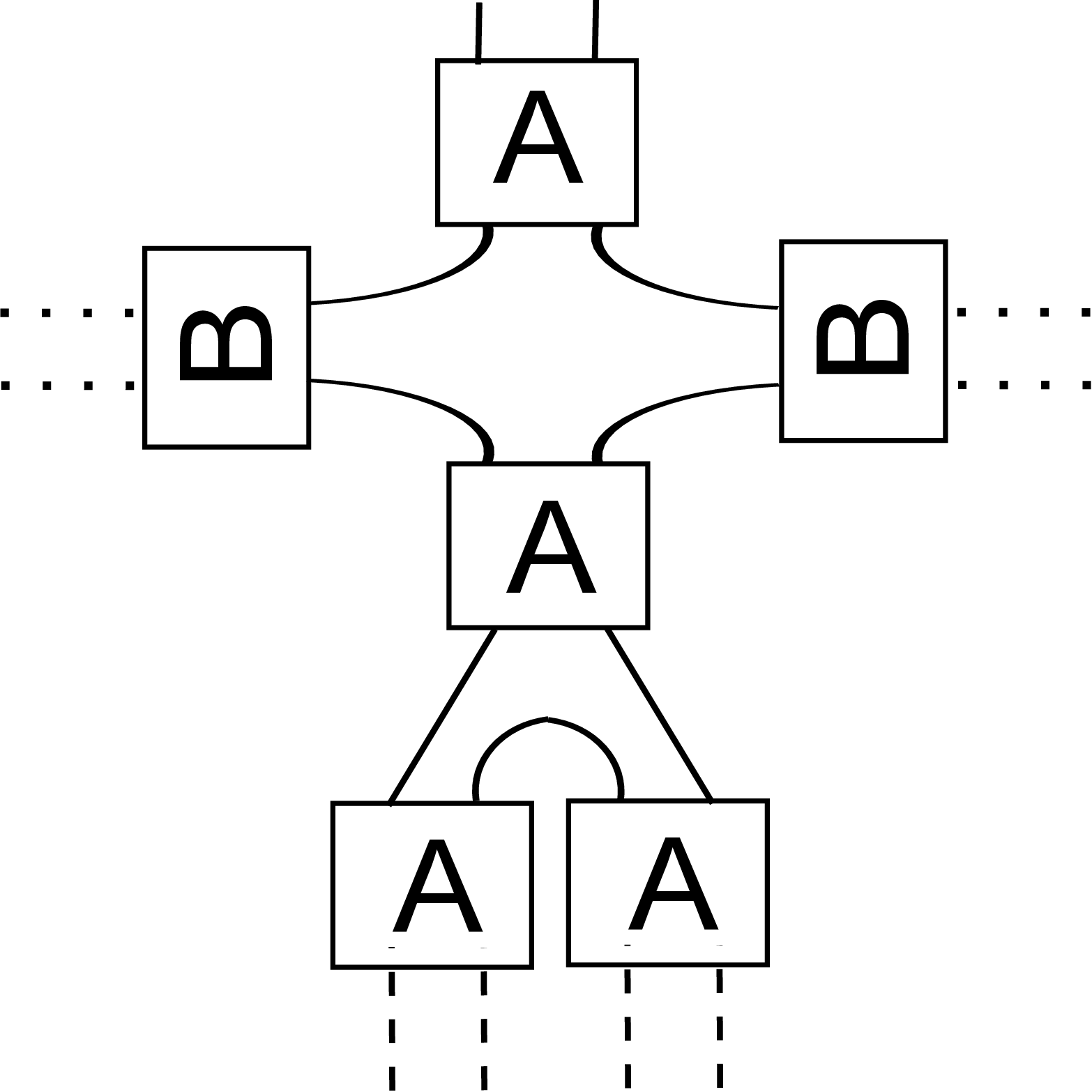}.$$

Observe that $(\mathscr{S}_A)_n$ with the Markov trace forms a sequence of Hilbert spaces, and
$\Theta_B$ is a self-adjoint operator on it, because the Markov trace is spherical.
Moreover (4) (resp. (5)) is a multiple of the adjoint of (3) (resp. (2)), thus $\Theta_B$ commutes with  (4) (resp. (5)).
\end{proof}

Symmetrically let us define $Flat_A(\mathscr{S}_B)$ to be the planar subalgebra of $\mathscr{B}$ consisting of all A-flat elements.
\begin{theorem}\label{tensorembed}
The planar subalgebra $Flat_B(\mathscr{S}_A)\vee Flat_A(\mathscr{S}_B)$ of $\mathscr{S}$ generated by the two vector spaces $Flat_B(\mathscr{S}_A)$ and $Flat_A(\mathscr{S}_B)$ is naturally isomorphic to the tensor product
$Flat_B(\mathscr{S}_A)\otimes Flat_A(\mathscr{S}_B)$ .
\end{theorem}
\begin{proof}
Let us define a map $\Phi : Flat_B(\mathscr{S}_A)\otimes Flat_A(\mathscr{S}_B) \rightarrow \mathscr{S}$, as a linear extension of
$$\Phi(a\otimes b)=\grc{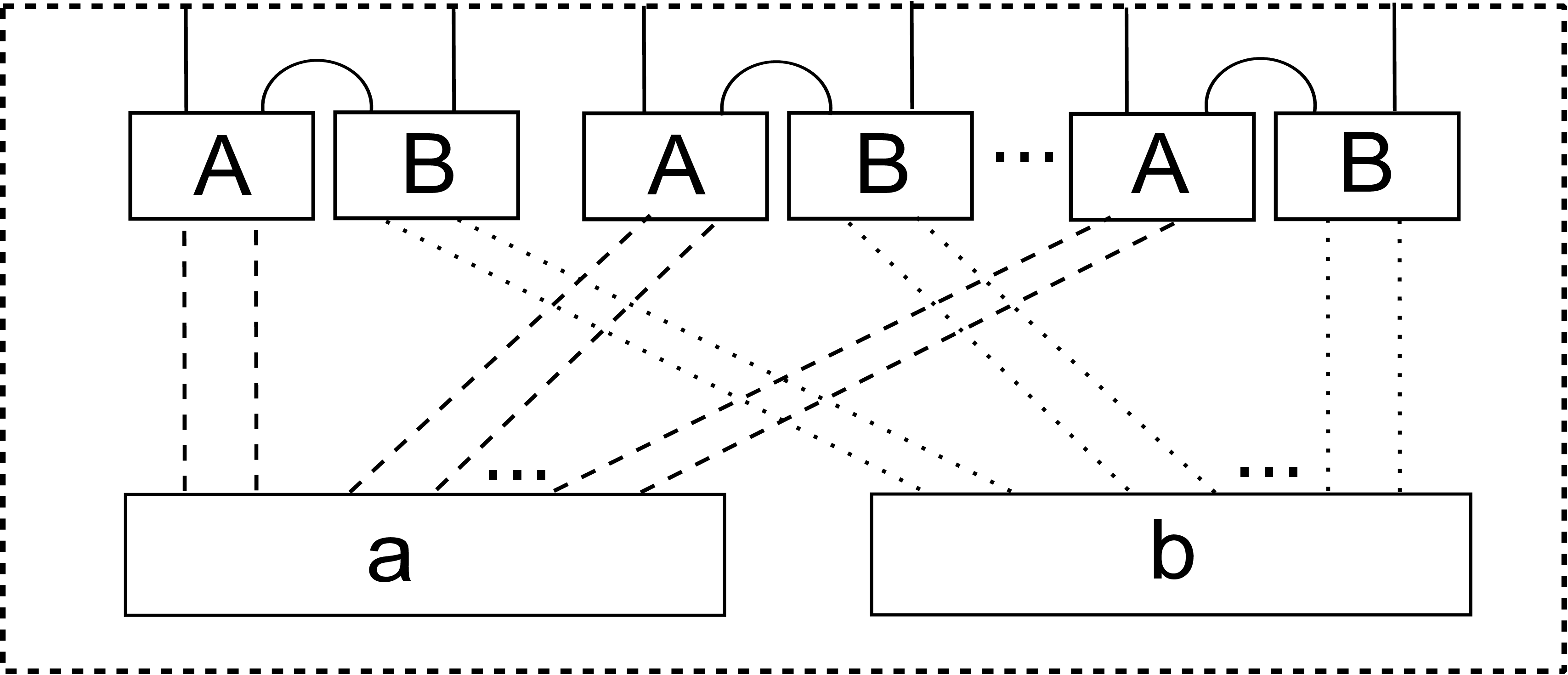},$$
for any $a\otimes b\in (Flat_B(\mathscr{S}_A)\otimes (Flat_A\mathscr{S}_B))_n, n>0;$
and $\Phi(\emptyset\otimes\emptyset)=\emptyset$, where $\emptyset$ is the unshaded empty diagram.
To show it is well defined, we prove that $\Phi$ preserves the inner product.
By Equation \ref{ppqq} and Proposition \ref{pqpq}, for any $a_1\otimes b_1,~a_2\otimes b_2\in (Flat_B(\mathscr{S}_A)\otimes (Flat_A\mathscr{S}_B))_n$, we have $tr_n(\phi(a_1\otimes b_1)^*\phi(a_2\otimes b_2))=\delta^{-n}tr_n(a_1^*a_2)tr_n(b_1^*b_2)$.
While computing the inner product in $\mathscr{S}_A$ and $\mathscr{S}_B$, the fudge factor is involved. The inner product of $a_1$ and $a_2$ is $\delta_A^{-n}tr(a_1^*a_2)$; the inner product of $b_1$ and $b_2$ is $\delta_B^{-n}tr(b_1^*b_2)$. So the inner product of $a_1\otimes b_1$ and $a_2\otimes b_2$ is $\delta_A^{-n}tr(a_1^*a_2)\delta_B^{-n}tr(b_1^*b_2)=\delta^{-n}tr_n(a_1^*a_2)tr_n(b_1^*b_2)=tr_n(\phi(a_1\otimes b_1)^*\phi(a_2\otimes b_2))$.

To prove $\Phi$ is a planar algebra isomorphism, we need to check that $\phi$ commutes with the following seven operations,
(1) the 2-click rotation;
(2) adding a cap on a shaded boundary;
(3) adding a cap on a unshaded boundary;
(4) adding a string in a shaded boundary;
(5) adding a string in a unshaded boundary;
(6) the adjoint operation;
(7) the $\wedge$ operation.

(1) follows from the flatness of $a$ and $b$, and the fudge factor is $0$ under the 2-click rotation.

(7) follows from the flatness of $a$ and $b$, and the fudge factor is $0$ under the $\wedge$ operation.

(6) follows from $\graa{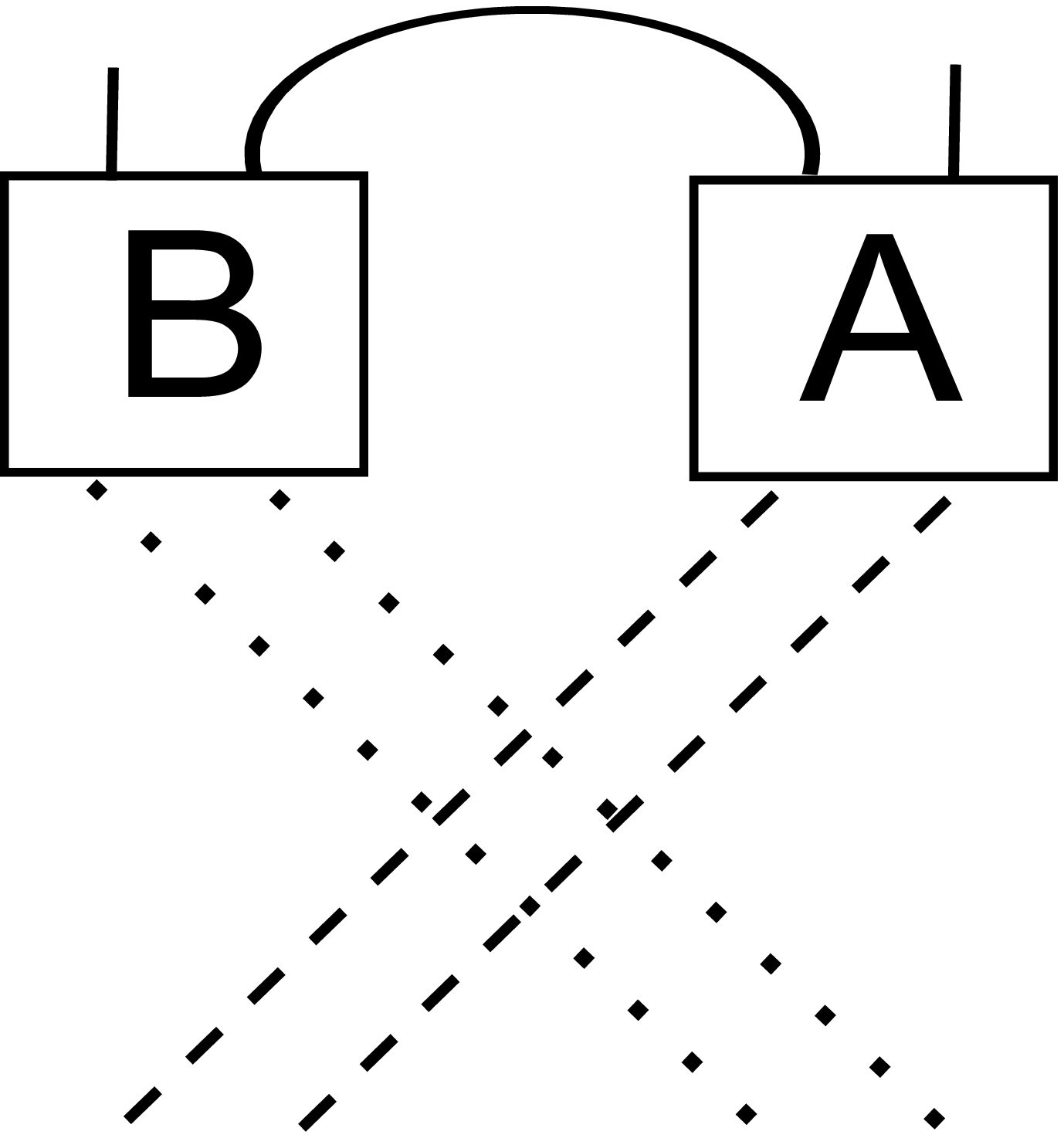}=\delta\grb{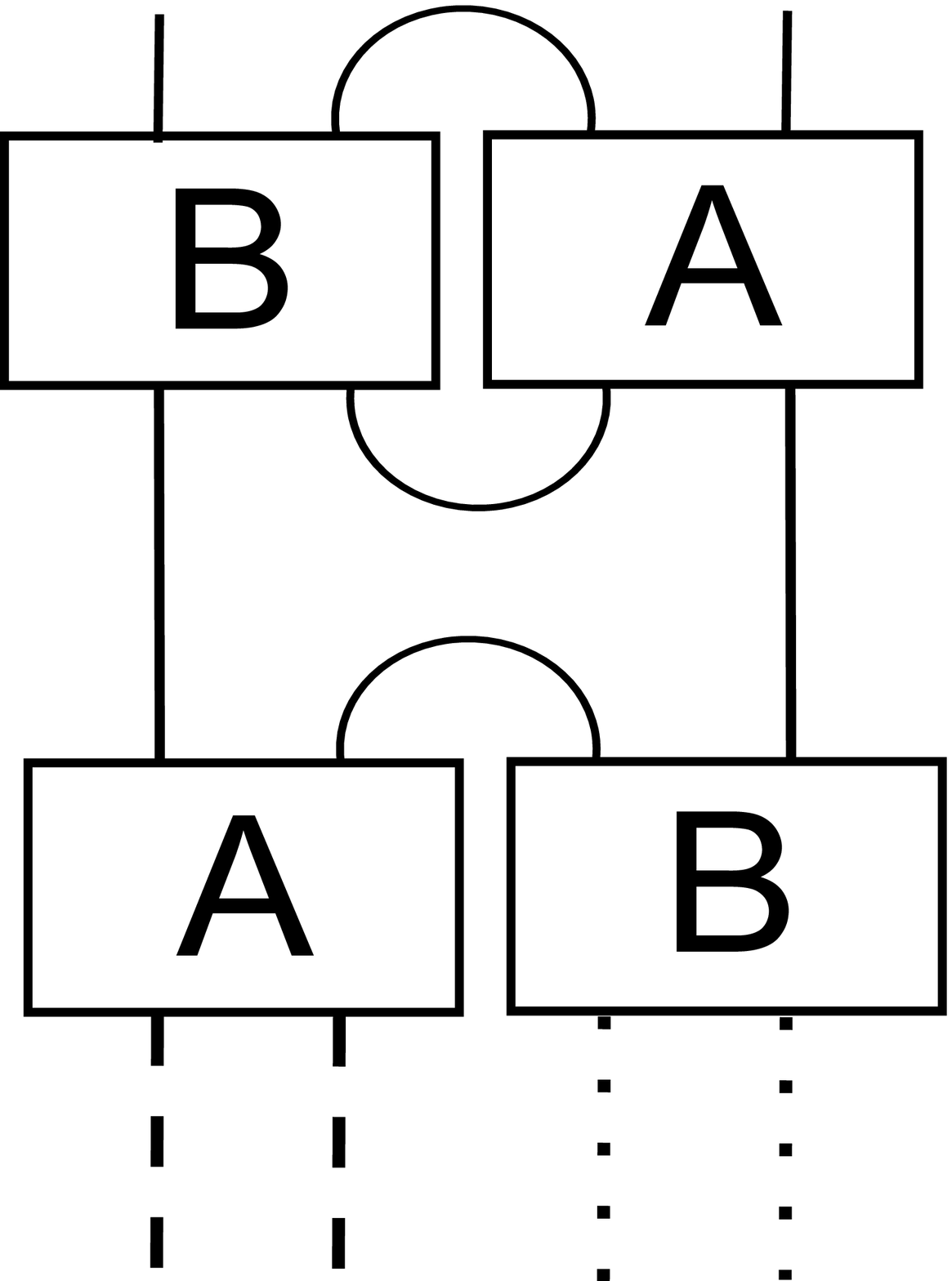}=\gra{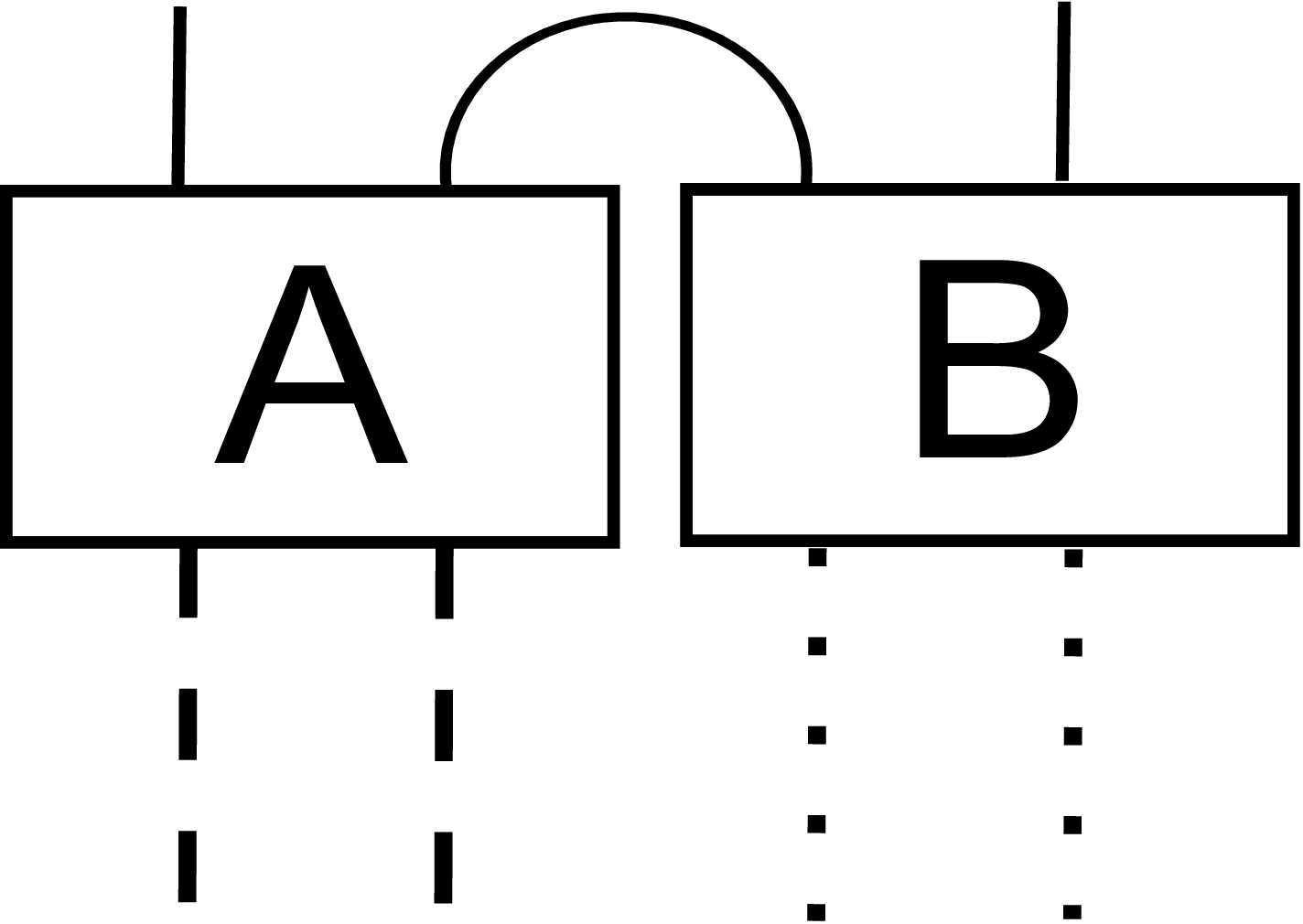}$,
then the flatness of $a$ and $b$.

(2) follows from $\gra{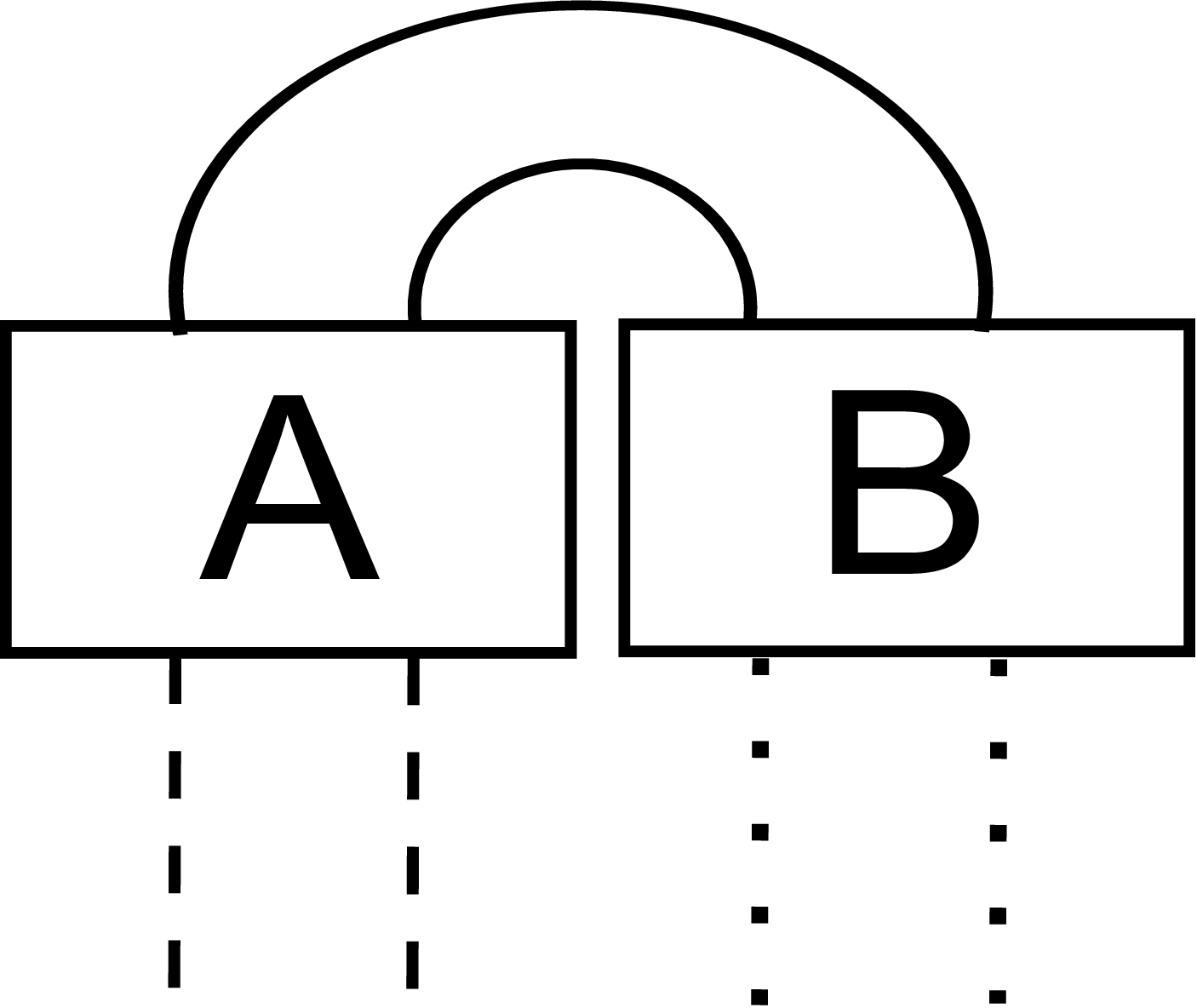}=\frac{1}{\delta}\gra{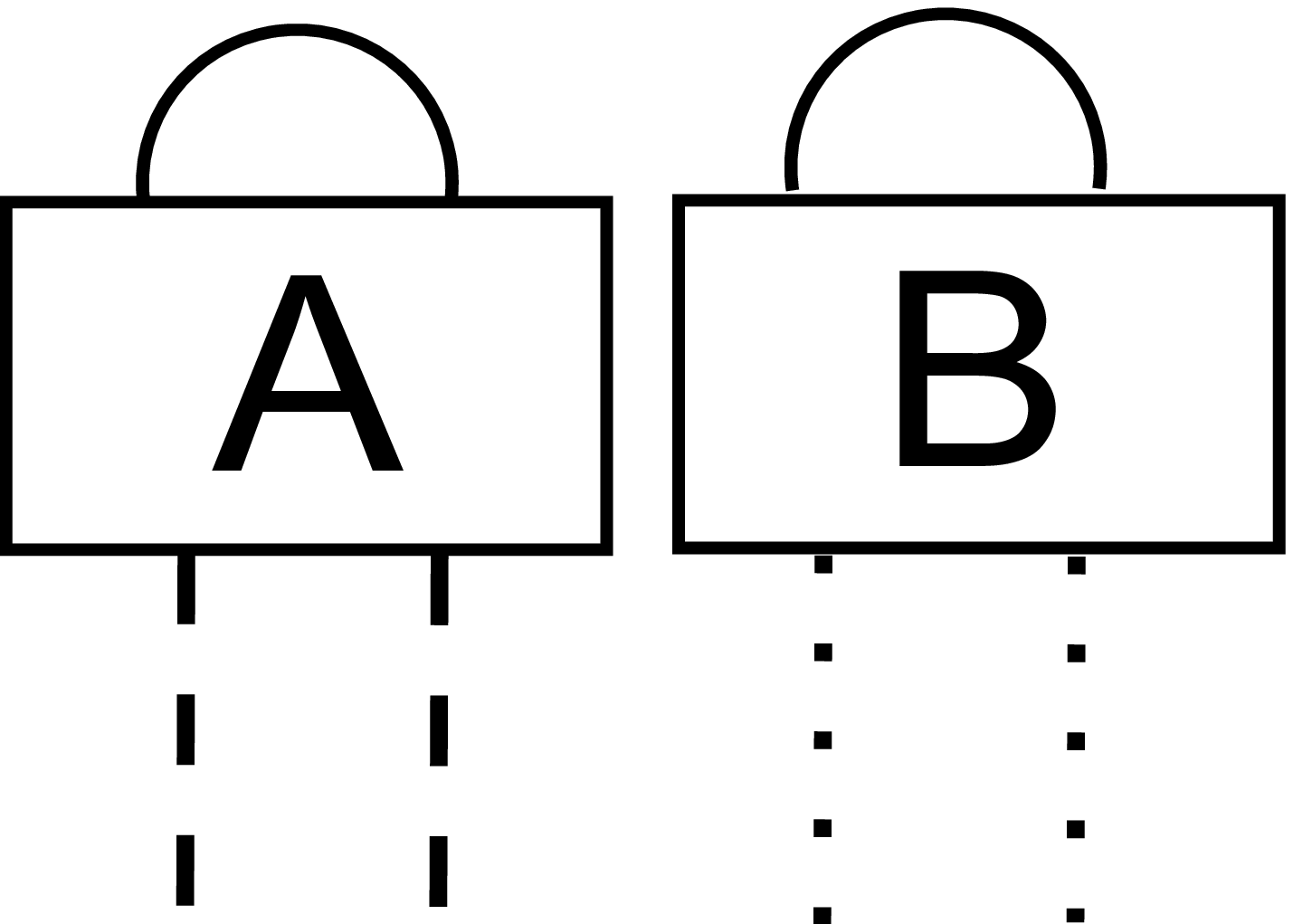}$, then the flatness of Temperley-Lieb elements.
The factor $\delta^{-1}=\delta_A^{-1}\delta_B^{-1}$ contributes $-1$ to the fudge factor of adding a cap on a shaded boundary.

(3) follows from $\graa{adjoint1}=\gra{adjoint3}$, then the flatness of Temperley-Lieb elements. The fudge factor is $0$ while adding a cap on a unshaded boundary.

(4) follows from
$\gra{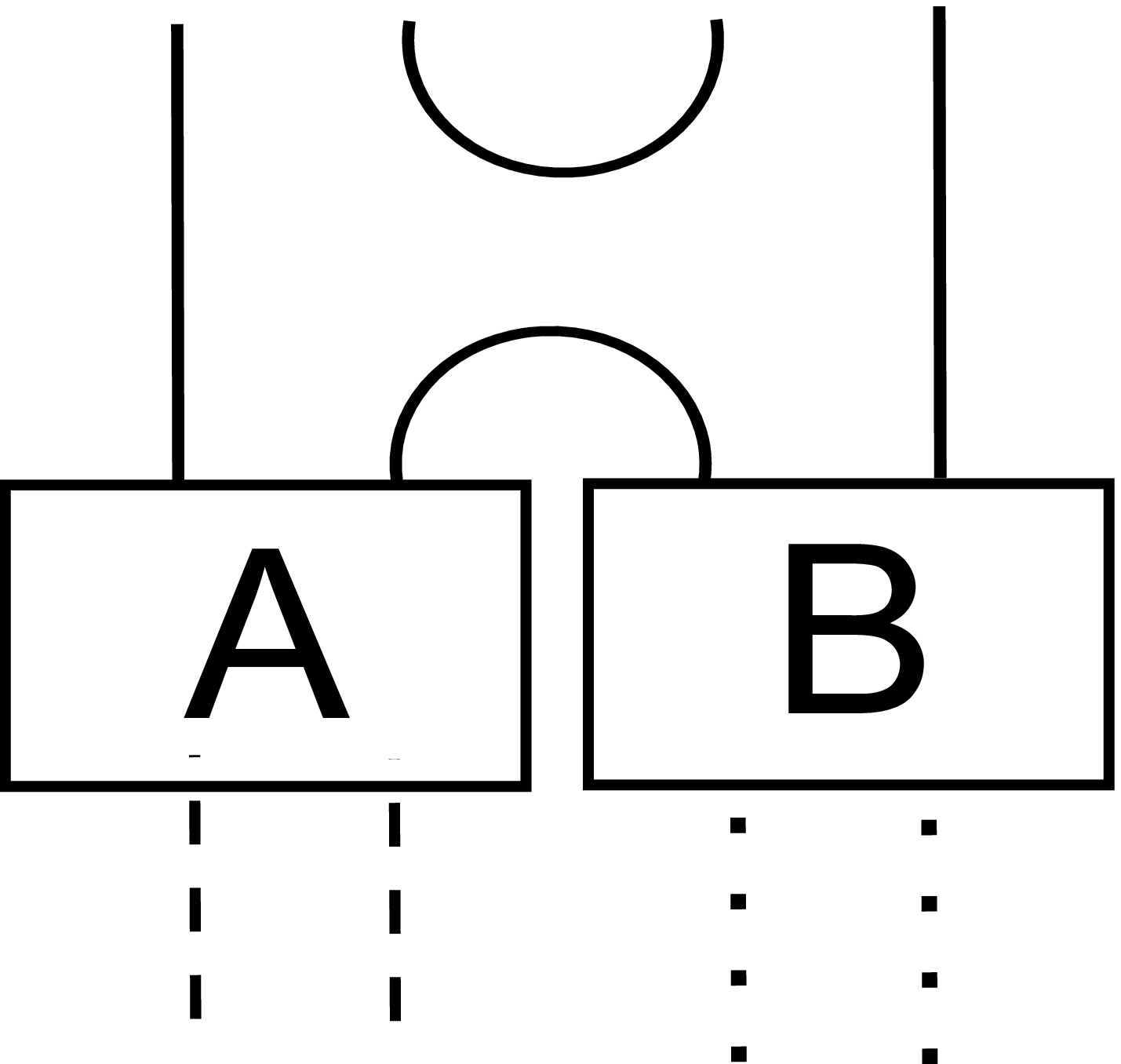}=\delta^2\grb{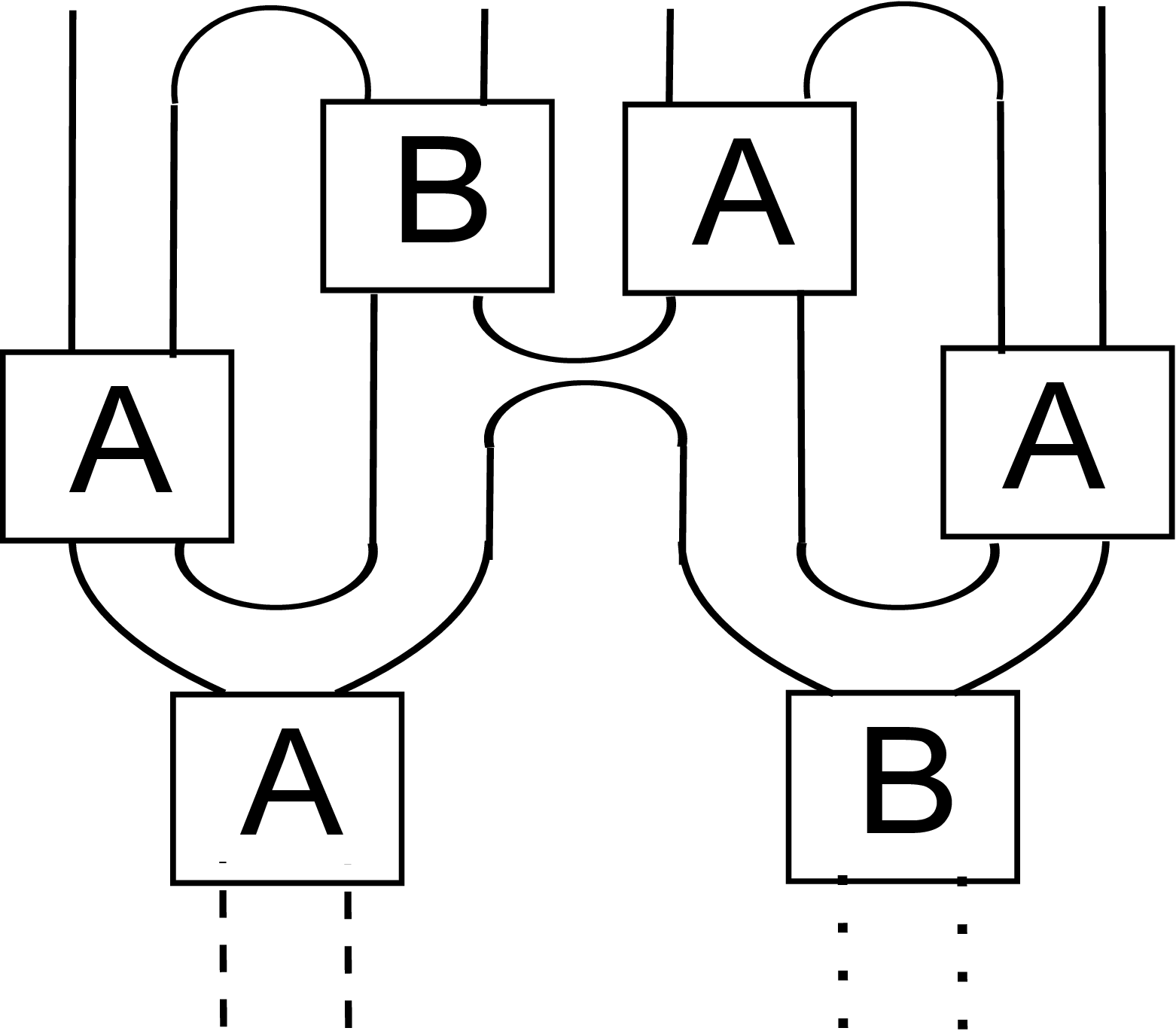}=\delta^2\grb{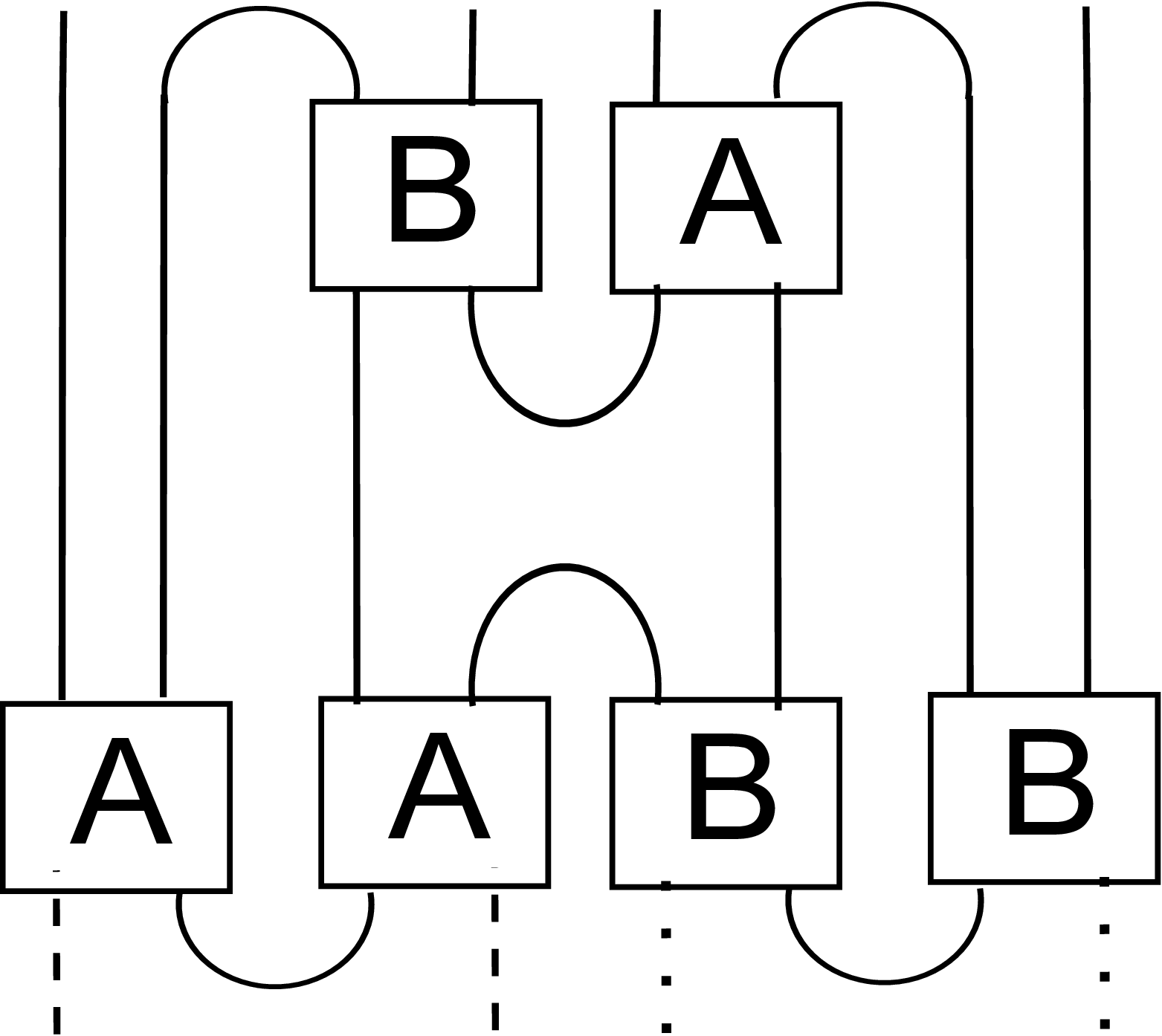}=\delta\grb{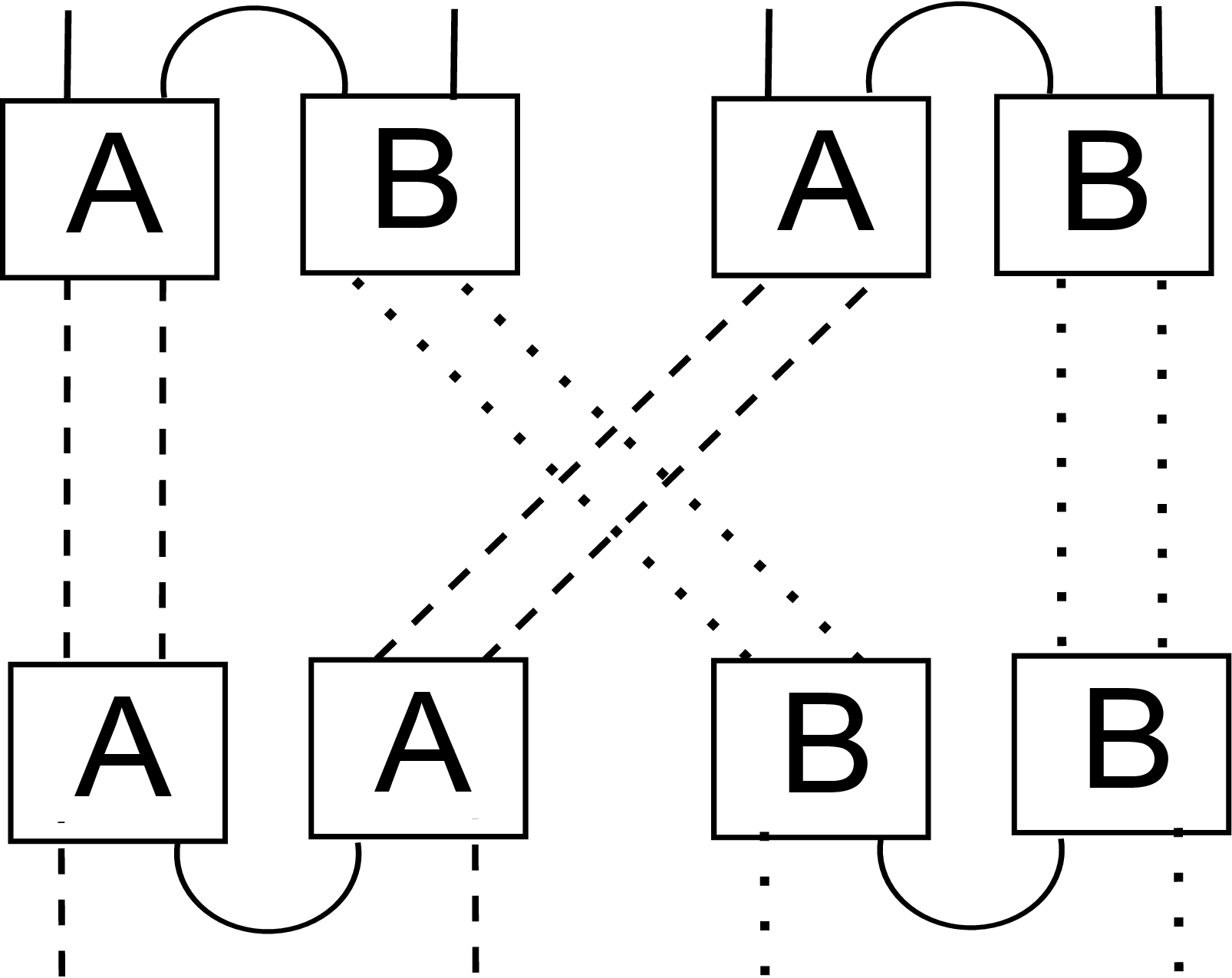}$, then the flatness of Temperley-Lieb elements.
The factor $\delta=\delta_A\delta_B$ contributes $+1$ to the fudge factor of adding a string in a shaded boundary.

(5) follows from $\grs{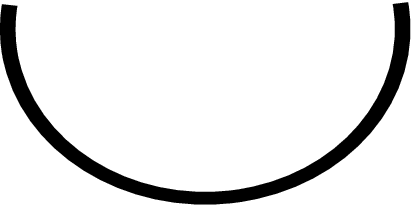}=\gra{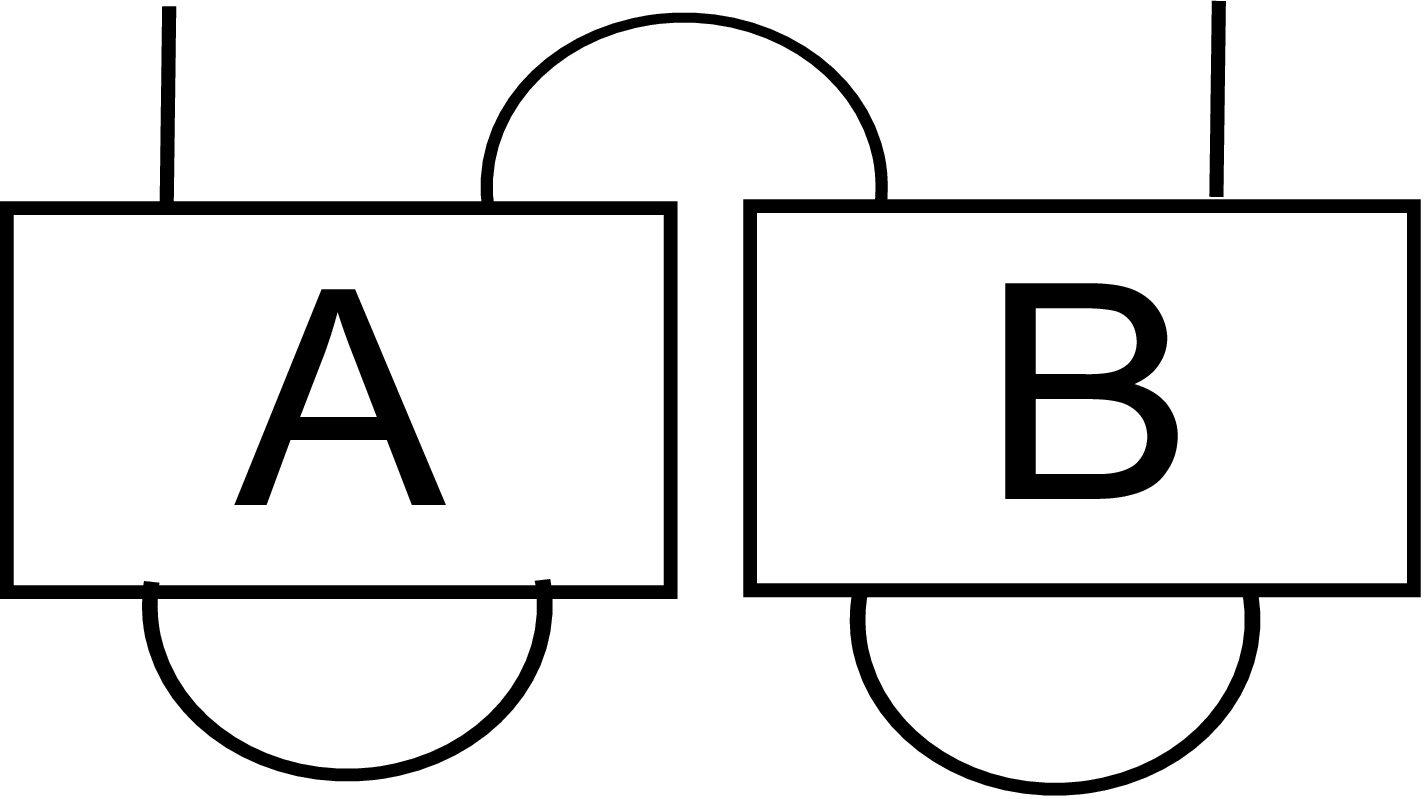}$, then the flatness of Temperley-Lieb elements. The fudge factor is $0$ while adding a string in a unshade boundary.\\
\end{proof}

\begin{theorem}\label{tensordecom}
Suppose $\mathscr{S}$ is a subfactor planar algebra with two biprojections $A,B$, such that $AB=e$; $A*B$ is a multiple of $id$; $a*B=B*a, ~\forall ~a\preceq A$; $A*b=b*A, ~\forall ~b\preceq B$;
and $\mathscr{S}$ is generated by $\{x\in \mathscr{S}_2| x\preceq A \quad or\quad x\preceq B\}$ as a planar algebra,
then the planar subalgebra $\mathscr{S}_A\vee\mathscr{S}_B$ of $\mathscr{S}$ generated by the two vector spaces $\mathscr{S}_A$ and $\mathscr{S}_B$ is naturally isomorphic to the tensor product $\mathscr{S}_A\otimes\mathscr{S}_B$, $\mathscr{S}_A$ and $\mathscr{S}_B$ are generated by 2-boxes, and $\mathscr{S}=\mathscr{S}_A\vee\mathscr{S}_B$.
In this case, $\mathscr{S}$ is said to be separated by the biprojections $A$ and $B$ as a tensor product.
\end{theorem}

\begin{proof}
For any $a\in(\mathscr{S}_A)_2$, $a>0$, we have $a\preceq A$, thus $a*B=B*a$ by assumption.
Then
$\Theta_B(a)=\delta^2A(B*(AaA)*B)A=\delta^2A(B*a*B)A=\delta^2A(B*B*a)A=\delta tr(B)A(B*a)A$. Using the exchange relation of $A$, we have
$A(B*a)A=A(B*(aA))A=A((BA)*a)A=A(e*a)A=\frac{1}{\delta}AaA=\frac{1}{\delta}a$. Thus $\Theta_B(a)=tr(B)a$, which implies $a\in Flat_B(\mathscr{S}_A)$ by Proposition \ref{annularflat}.
Any operator in $(\mathscr{S}_A)_2$ is a linear sum of four positive operators, so $(\mathscr{S}_A)_2\subset Flat_B(\mathscr{S}_A)$. Take $\mathscr{A}$ to be the planar subalgebra of $Flat_B(\mathscr{S}_A)$ generated by 2-boxes.
Symmetrically $(\mathscr{S}_B)_2\subset Flat_A(\mathscr{S}_B)$.
Take $\mathscr{B}$ to be the planar subalgebra of $Flat_A(\mathscr{S}_B)$ generated by 2-boxes.
Then $\mathscr{A}\vee\mathscr{B}\subset Flat_B(\mathscr{S}_A)\vee Flat_A(\mathscr{S}_B) \subset \mathscr{S}$.
On the other hand $x\preceq A$ implies $x\in (\mathscr{S}_A)_2$, and
$y\preceq B$ implies $y\in (\mathscr{S}_B)_2$.
By assumption, we have $\mathscr{S}\subset\mathscr{A}\vee\mathscr{B}$.
So  $\mathscr{A}\vee\mathscr{B}= Flat_B(\mathscr{S}_A)\vee Flat_A(\mathscr{S}_B) = \mathscr{S}$.
Then $Flat_B(\mathscr{S}_A)=\mathscr{A}$ and $Flat_A(\mathscr{S}_B)=\mathscr{B}$ by counting the dimension. So they generated by 2-boxes.
By Theorem \ref{tensorembed}, $\mathscr{S}=Flat_B(\mathscr{S}_A)\vee Flat_A(\mathscr{S}_B)$ is naturally isomorphic to $Flat_B(\mathscr{S}_A)\otimes Flat_A(\mathscr{S}_B)$.
\end{proof}

\begin{corollary}\label{tensordecom222}
Suppose $\mathscr{A}$, $\mathscr{B}$ are subfactor planar algebras. If $\mathscr{A}\otimes\mathscr{B}$ is generated by 2-boxes, then both $\mathscr{S}_A$ and $\mathscr{S}_B$ are generated by 2-boxes.
\end{corollary}

\begin{proof}
Considering $A=id\otimes e$, $B=e\otimes id$, and $a\otimes b=\delta^2(a\otimes e)*(e\otimes b)$, the statement follows from Theorem \ref{tensordecom}.
\end{proof}

\section{Biprojections}\label{biprojections}
In this section, we assume that $\mathscr{S}$ is a subfactor planar algebra.

For crossed product group subfactor planar algebra, all 2-boxes minimal projections are indexed by group elements, and their coproduct behaves like the group multiplication. Motivated by this fact, we will discuss the coproduct of 2-box positive operators of $\mathscr{S}$, specifically the lattice of the supports of those positive operators, based on the Schur Product theorem.
For a positive operator $x$, and $e\preceq x$, the support of $x^{*k}$ will be increasing. The limit is a projection $P$, such that $P*P\sim P$. We shall expect $P$ to be a biprojection, viewed as the biprojection generated by $x$.

If we assume $P$ is a projection, $e\preceq p$, $P=P'$ and $P*P\sim P$, then by GJS construction $\cite{GJS}$, we obtain a factor $\mathcal{N}=Gr_0$ consisting of the linear span of diagrams with all boundary points on the left, and a factor $\mathcal{M}=Gr_1$ consisting of the linear span of diagrams with all boundary points on the left except one and the top and one on the bottom. Let us construct $\mathcal{P}$ as a cut down of $\mathcal{M}$ by an action of the projection $P$ on the right. By our assumptions, we can show that $\mathcal{P}$ is an intermediate subfactor of $\mathcal{N}\subset\mathcal{M}$, and $P$ is a biprojection.

Furthermore we can drop the assumptions $e\preceq p$, $P=P'$, because the support of $P^{*k}$ will contain $e$ and $P'$ when $k$ large enough. This phenomenon is similar to the fact that any group element $g$ of a finite group generates identity and $g^{-1}$. We will not expect this phenomenon without the finite-index condition. A well known result of representation theory of finite groups is also reduce to this phenomenon and Stone-Weierstrass theorem. See Theorem \ref{P*P=P}, \ref{maxcoeff}, Proposition \ref{maxcoeffremark} and the following remark.

Let $Q$ be a projection. Take the maximal projection $P$ subject to the condition $P*Q\preceq Q$, then $P*P\sim P$. By our new result, $P$ is a biprojection. If $Q$ behaves like a normaliser under coproduct, then the planar algebra $\mathscr{S}$ is a free product, see Theorem \ref{virtual normalizer}. Moreover we can find out such a normaliser $Q$ by looking at the principal graph, see Lemma \ref{copro connect}. The combination of these two results will be the key break of our main classification results.

\begin{theorem}[Schur Product Theorem]\label{P*P>0}
If $A$, $B$ are positive operators in $\mathscr{S}_2$,
then $A*B$ is a positive operator.
\end{theorem}

\begin{proof}
Set $A=a^*a, B=b^*b$.
Then $A*B=\delta\Phi((a(1\boxdot b))^*(a(1\boxdot b)))$, where $\Phi$ is the conditional expectation from $\mathscr{S}_3$ to $\mathscr{S}_2$.
Since $(a(1\boxdot b))^*(a(1\boxdot b))\geq0$ in $\mathscr{S}_3$, we have $A*B\geq0$ in $\mathscr{S}_2$. While $tr(A*B)>0$, so $A*B>0$.
\end{proof}

\begin{remark}
If we consider the subfactor planar algebra as a planar subalgebra of it graph planar algebra \cite{Jon00,JonPen}, then this Schur Product Theorem reduces to the canonical one.
\end{remark}

\begin{definition}
Suppose $X\in\mathscr{S}_2$ is a positive operator, and $X=\sum_{i=1}^kC_iP_i$ for some mutually orthogonal minimal projections $P_i\in\mathscr{S}_2$ and $C_i>0$, for $1\leq i\leq k.$
Let us define the rank of $X$ to be $k$, denoted by $r(X)$.
\end{definition}
It is easy to see that $r(X)$, the rank of $X$ in $\mathscr{S}_2$, is independent of the decomposition. And $r(X)=1$ means $X$ is a multiple of a minimal projection.

\begin{lemma}\label{copro connect}
Suppose $P$ and $Q$ are projections in $\mathscr{S}_2$. Then
$$r(P'*Q)\leq \dim(Q\mathscr{S}_3 P),$$
where $Q\mathscr{S}_3 P=\{QxP| x\in\mathscr{S}_3\}$.
\end{lemma}

\begin{proof}
Suppose $P'*Q=\Sigma_{i=1}^kC_iR_i$, for some mutually orthogonal minimal projections $R_i, ~1\leq i\leq k,$ in $\mathscr{S}_2$, and $C_i>0, ~1\leq i\leq k$.
Let $v_i\in Q\mathscr{S}_3P$ be $Q(1\boxdot R_i')P$, for $1\leq i \leq k$.
It is easy to check that $tr_3(v_i^*v_i)=tr((P'*Q)R_i)>0$ and $tr_3((v_i^*v_j)=0$, for $1\leq i,j \leq k$, $i\neq j$.
So $r(P'*Q)=k\leq \dim(Q(1\boxdot R_i')P)$.
\end{proof}

If $P,Q$ are two minimal projections in $\mathscr{S}_2$, then $P,Q$ correspond to two points in the principal graph, and $\dim(Q\mathscr{S}_3P)$ is the number of length 2 paths between the two points. If $\dim(\mathscr{S}_3)$ is small, then for most depth 2 points in the principal, the number of length 2 paths between them is small. So the rank of their coproduct is small.

\begin{definition}
Suppose $\mathscr{S}$ is a subfactor planar algebra generated by 2-boxes. Let us define its rank to be the maximal number of length 2 paths between a pair depth 2 points in the principal graph, for all pairs of depth 2 points.
\end{definition}

\begin{proposition}
Suppose $\mathscr{S}$ is a subfactor planar algebra and its rank is $1$, then it is a group subfactor planar algebra $\mathscr{S}^G$, for some group $G$. Verse Visa.
\end{proposition}

The following three lemmas are basic facts about positive operators and planar algebras. They will be used several times without being mentioned.

\begin{lemma}
Suppose $A,B,C,D$ are positive operators in $\mathscr{S}_2$. If $A\preceq C, B\preceq D$. Then $A*B\preceq C*D$.
\end{lemma}
\begin{proof}
Because $\mathscr{S}_2$ is finite dimensional, the spectrum of a positive operator is finite.
If $A\preceq C, B\preceq D$, then $A<\lambda C$ and $B<\lambda D$, for some $\lambda$ large enough.
Thus $A*B<\lambda^2 C*D$, by Theorem \ref{P*P>0}. Then $A*B\preceq C*D$.
\end{proof}

\begin{lemma}\label{CDE0}
Suppose $C$  and $D$ are two positive operators. If $tr(CD)=0$, then $CD=0$. Furthermore if $E$ is a positive operator and $E\preceq C$, then $ED=0$
\end{lemma}

\begin{proof}
If $tr(CD)=0$, then $tr(C^{\frac{1}{2}}DC^{\frac{1}{2}})=0$. Thus $C^{\frac{1}{2}}D^{\frac{1}{2}}=0$. We have $CD=0$. Furthermore if $E\preceq C$, then $E<\lambda C$ for some $\lambda>0$. Thus
 $0\leq tr(ED)\leq \lambda tr(CD)=0$. Thus $ED=0$.
\end{proof}

\begin{lemma}\label{CDE=}
If $C,D,E\in\mathscr{S}_2$, then
$$tr((C*D)E')=tr((D*E)C')=tr((E*C)D')$$
$$=tr(E'(C*D))=tr(C'(D*E))=tr(D'(E*C))$$
$$=tr(C(E'*D'))=tr(D(C'*E'))=tr(E(D'*C'))$$
$$=tr((E'*D')C)=tr((C'*E')D)=tr((D'*C')E)¡£$$
\end{lemma}

Recall that $a'$ means the contragredient of $a$.

\begin{proof}
It follows directly from the isotopy and the spherical property of a planar algebra.
\end{proof}

Suppose $x$ is a positive operator in $\mathscr{S}_2$ and $e\preceq x$.
Then the support of $x^{*k}$ is an increasing sequence bounded by $Q$, when $k$ approaches to infinity. We will prove that their union is a biprojection.

Before that, let us prove a lemma. For convenience, we mark the boundary points of a $3$-box by $1,2,3,4,5,6$ from the dollar sign clockwise as
$\graa{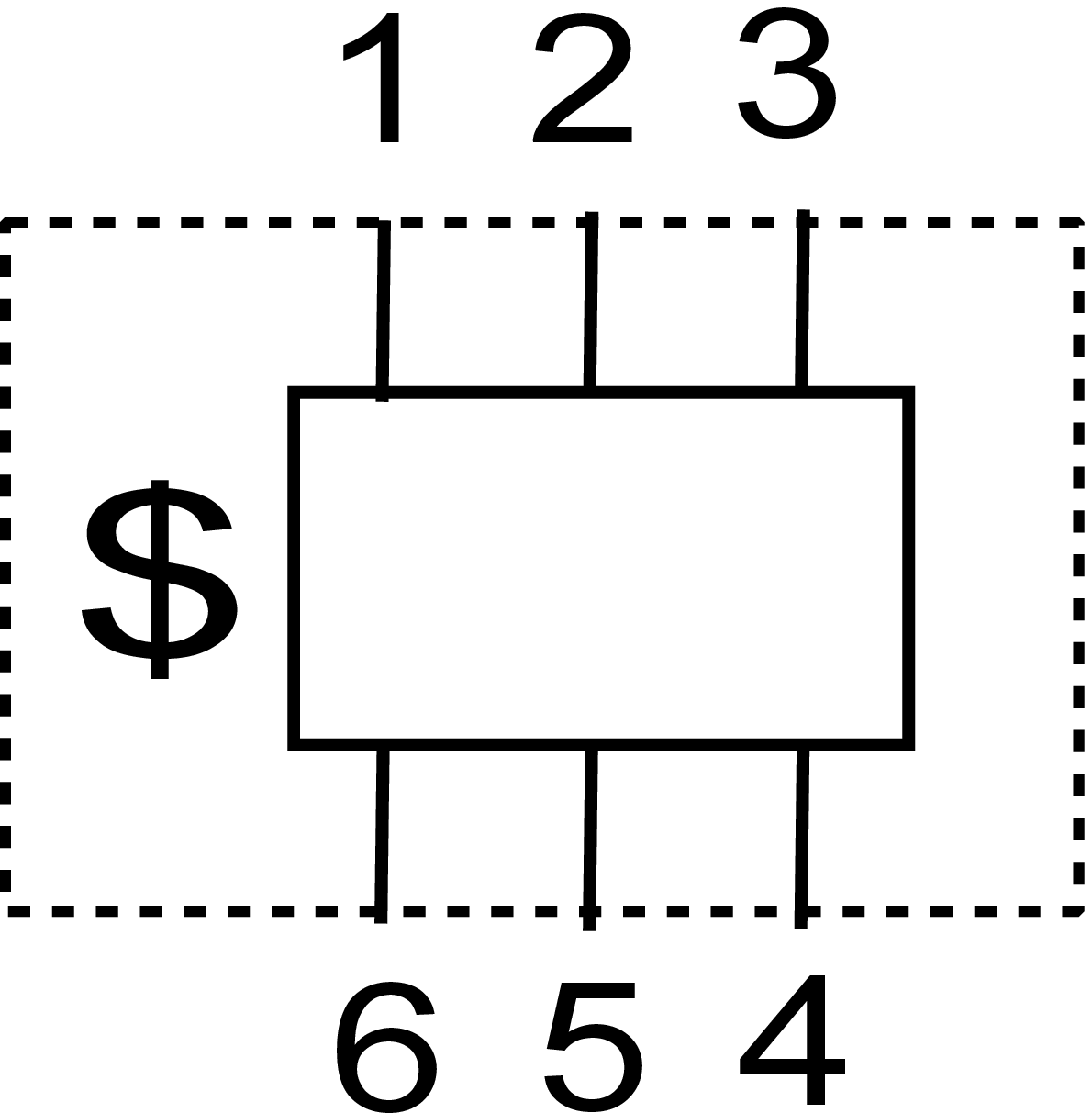}$.
\begin{lemma}\label{PQR=0}
Suppose $P,Q,R$ are projections in $\mathscr{S}_2$. Then the following are equivalent,

(1)$P(1\boxdot Q)R=0$;

(2)$(P*Q)R=0$;

(2')$P(R*Q')=0$;

(3)$tr((P*Q)R)=0$;
\end{lemma}
\begin{proof}
(1)$\Rightarrow$ (2)(resp. (2')) follows from adding a cap which connects boundary points $4,5$ (resp. $2,3$) of the 3-box $P(1\boxdot Q)R$.

(2),(2')$\Rightarrow$ (3) is obvious.

(3)$\Rightarrow$ (1) follows from $tr((P(1\boxdot Q)R)^*(P(1\boxdot Q)R))=tr(P(1\boxdot Q)R)=0$, and the positivity of the trace.
\end{proof}

\begin{theorem}\label{P*P=P}
Suppose $P$ is a projection in $\mathscr{S}_2$.
If $P*P\preceq P$, equivalently $(P*P)P=P*P$, then $P$ is a biprojection.
\end{theorem}

\begin{proof}
Set $Q=id-P$.
The condition $(P*P)P=P*P$ implies $(P*P)Q=0$. By Lemma($\ref{PQR=0}$), we have $P(1\boxdot P)Q=0$. Then $P(1\boxdot P)=P(1\boxdot P)P$.
Adding a cap which connects boundary points 2,3, we obtain $\frac{tr(P)}{\delta}P=P(P*P')$. Then $tr((Q*P)P)=tr(Q(P*P'))=tr(P*P')-tr(P(P*P'))=tr(P*P')-\frac{tr(P)}{\delta}tr(P)=0$. Using Lemma \ref{PQR=0} again, we have $(Q*P)P=0$. Thus $P*P=(P*P)P=((id-Q)*P)P=(id*P)P=\frac{tr(P)}{\delta}P$. Then $P$ is a biprojection.
\end{proof}

For an element $y\in\mathscr{S}_2$, take the positive operator $x=e+y^*y+yy^*$. Let $P$ be the union of the support of $x^{*k}$, for $k=1,2,\cdots$. Then $P*P\preceq P$. So $P$ is a biprojection. Moreover $PyP=y$.
If $Q$ is a biprojection, such that $QyQ=y$, then $QxQ=Q$. So $x^*k\preceq Q$.  Then $P\leq Q$. That means $P$ is the smallest biprojection ``containing" $y$.
\begin{definition}\label{biprogen}
Suppose $y$ is an element in $\mathscr{S}_2$ and $P$ is the smallest biprojection satisfying $PyP=y$. Then we call $P$ the biprojection generated by $y$.
\end{definition}

Let $\iota:\mathscr{S}_2\rightarrow \mathscr{S}_3$ be the inclusion by adding one string to the right, and $\mathscr{I}_3$ be the two sided ideal of $\mathscr{S}_3$ generated by $e_2$.
Then $\mathscr{S}_2 \cap \mathscr{I}_3=\{x\in\mathscr{S}_2|\iota(x)\in\mathscr{I}_3\}$ is a two sided ideal of $\mathscr{S}_2$.
Thus the support of $\mathscr{S}_2 \cap \mathscr{I}_3$ is a central projection.
A minimal projection of $\mathscr{S}_2$ belongs to $\mathscr{S}_2 \cap \mathscr{I}_3$ if and only if the corresponding point in the principal graph is not adjacent to a depth $3$ point.

\begin{theorem}\label{P bipro}
Let $P$ be the support of $\mathscr{S}_2 \cap \mathscr{I}_3$, then $P$ is a central biprojection, and $\mathscr{S}_P$ is depth 2.
\end{theorem}
\begin{proof}
By definition $\iota(P)\in \mathscr{S}_3$, we have $\iota(P)=\sum_i c_i(1\boxdot id) d_i$, for finitely many $c_i,d_i\in \mathscr{S}_2$. Then
$$\iota(P*P)=\sum_i\grb{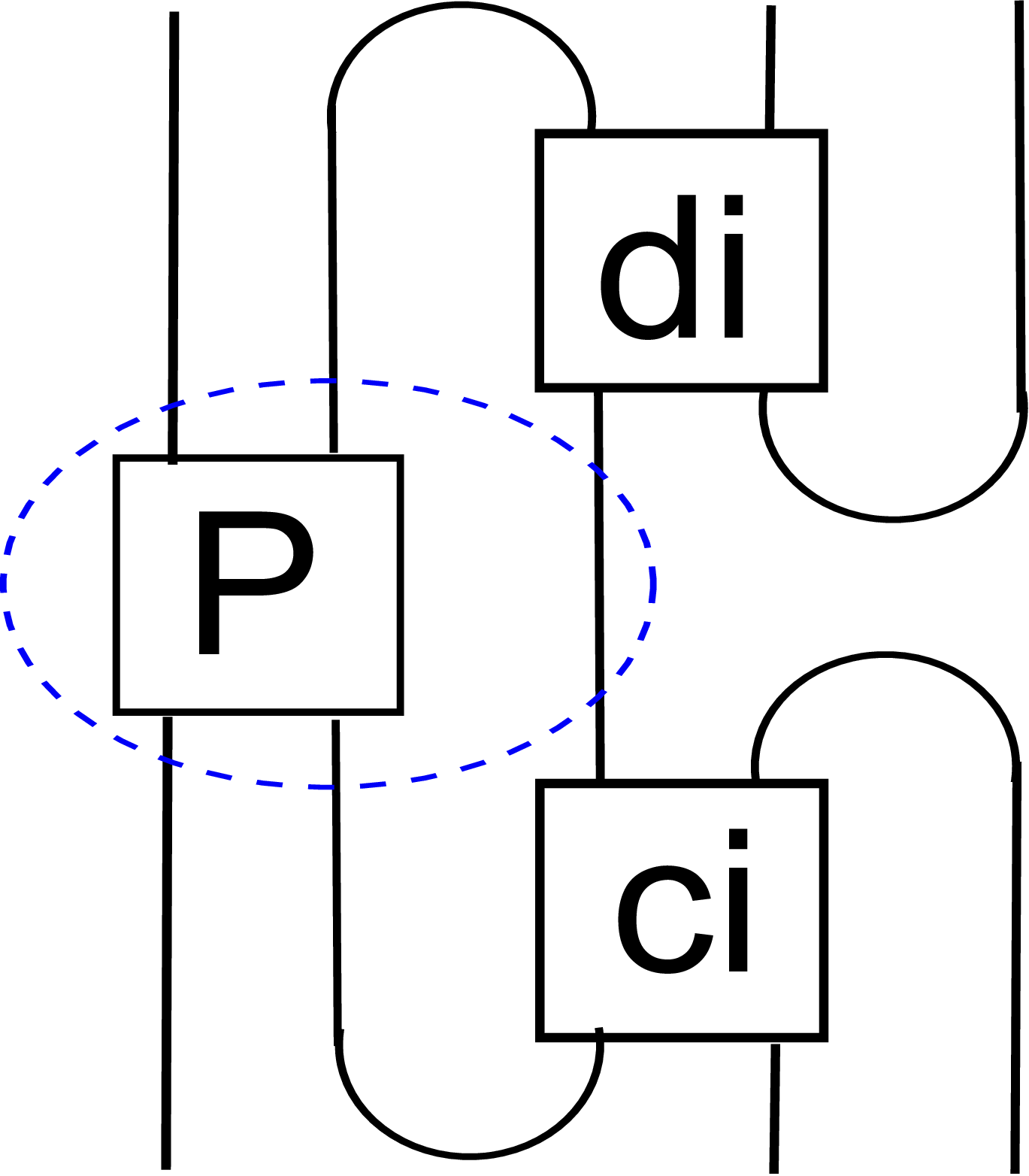}=\sum_{i,i'}\grb{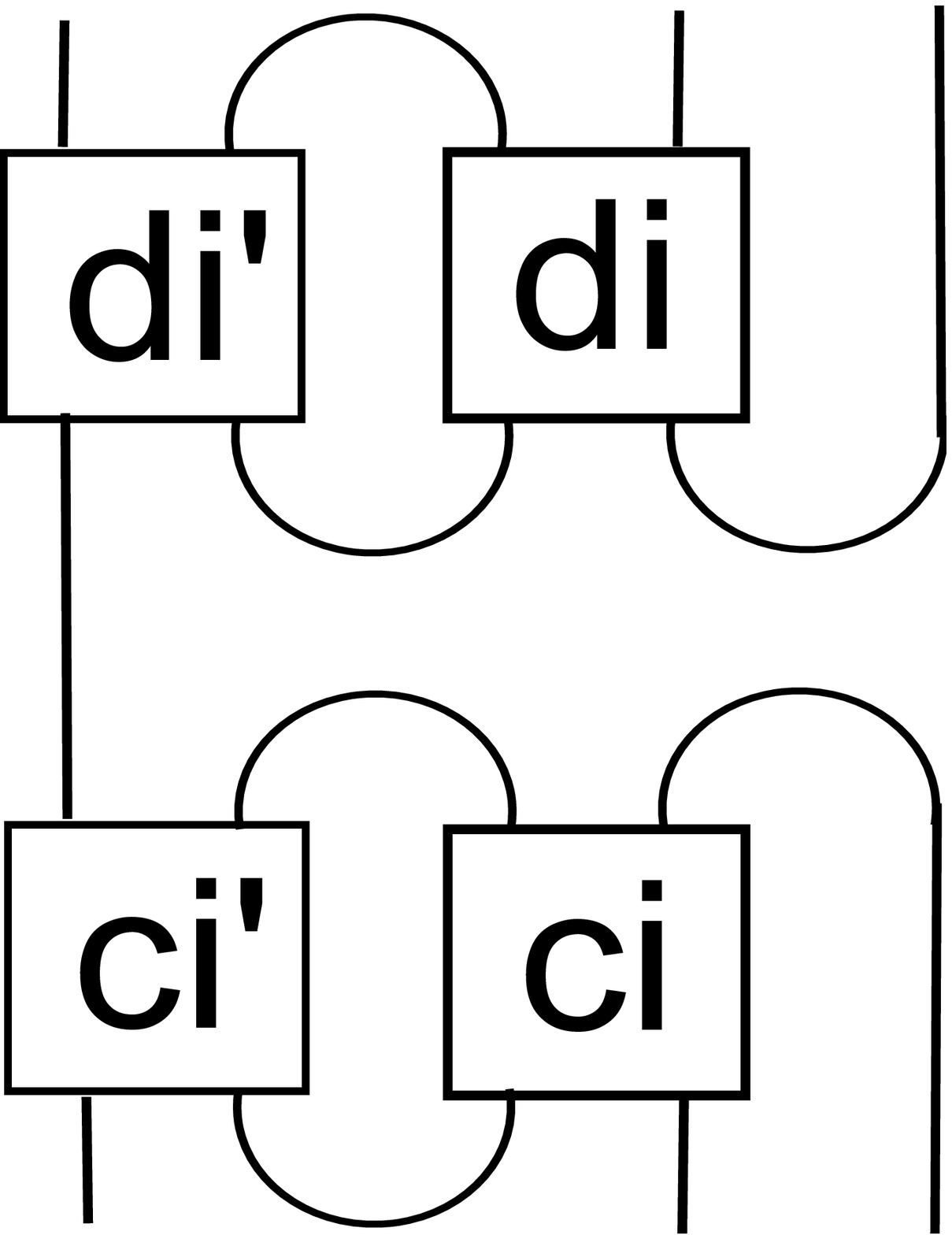}.$$
Thus $P*P\in \mathscr{S}_2 \cap \mathscr{I}_3.$ Then $(P*P)P=P*P$. By Theorem \ref{P*P=P}, we have $P$ is a biprojection.
So $P=P'$.

Let $\iota_P: (\mathscr{S}_P)_2\rightarrow (\mathscr{S}_P)_3$ be the inclusion by adding one string to the right.
Then for any $x\in (\mathscr{S}_P)_2$, we have $\iota_P(x)$ is a multiple of $P\iota(x)P$. While $\iota(x)\in \mathscr{I}_3$, so $\iota_P(x)=\sum_i Pf_i(1\boxdot id) g_iP$ for finitely many $f_i,g_i\in \mathscr{S}_2$. Note that $P$ is central, so $\iota_P(x)=\sum_i f_iP(1\boxdot id) Pg_i=\sum_i f_iP(1\boxdot P) Pg_i$. The last equation follows from the exchange relation of the biprojection $P$. Observe that $f_iP,Pg_i\in (\mathscr{S}_P)_2$, $P(1\boxdot P)P$ is a multiple of the Jones projection in $(\mathscr{S}_P)_3$. Thus $\iota_P(x)$ is in the two sided ideal of $(\mathscr{S}_P)_3$ generated the Jones projection.
Then $\mathscr{S}_P$ is depth 2.
\end{proof}

\begin{corollary}
If $A\in \mathscr{S}_2\cap \mathscr{I}_3$, then $A'\in \mathscr{S}_2\cap \mathscr{I}_3$.
\end{corollary}

\begin{proof}
If $A\in \mathscr{S}_2\cap \mathscr{I}_3$, then $A=PA$. So $A'=A'P$. Then  $A'\in \mathscr{S}_2\cap \mathscr{I}_3$.
\end{proof}

If $\mathscr{S}$ is the planar algebra of an irreducible subfactor $\mathcal{N}\subset \mathcal{M}$, then the support $P$ of $\mathscr{S}_2 \cap \mathscr{I}_3$ as a central biprojection corresponds to an intermediate subfactor $\mathcal{P}$ of $\mathcal{N}\subset \mathcal{M}$. Note that the planar algebra of $\mathcal{N}\subset\mathcal{P}$ is $\mathscr{S}_P$. So $\mathcal{P}=\mathcal{N}\rtimes K$, where $K$ is the dual of $(\mathscr{S}_P)_2$ as a Kac algebra.

For a positive operator $A$ in $\mathscr{S}_2$, let $P$ be the biprojection generated by $A$. We will show that $\frac{\delta}{tr(P)}1\boxdot P$ is the spectral projection $E_{\{\|1\boxdot (A+A')\|\}}(1\boxdot (A+A'))$,
where $\|\cdot\|$ is the operator norm.
To prove this, we need some lemmas.

\begin{lemma}\label{PAk}
Suppose $A\in\mathscr{S}_2$ is a positive operator, and $P$ is the biprojection generated by $A$, then
$P\sim \sum_{i=1}^k A^{*i}$ for $k$ large enough.
\end{lemma}
\begin{proof}
Let $X_k$ be the support of $\Sigma_{i=1}^k A^{*i}$, for $k=1,2,\cdots$. Note that $A^{*i}>0$, so $X_k\leq X_{k+1}$. While $\mathscr{S}_2$ is finite dimensional, there is an $m$, such that $X_m=X_{m+1}$. Then $X_m*A\preceq X_m$. Thus $X_m*X_m\preceq X_m$. By Theorem \ref{P*P=P}, $X_m$ is a biprojection. Moreover $A\preceq P$ implies $X_m\preceq P$. Thus $X_m=P$. Then $P\sim \Sigma_{i=1}^k A^{*i}$, whenever $k\geq m$.
\end{proof}

\begin{lemma}\label{norm1a}
Suppose $A$ is a positive operator in $\mathscr{S}_2$, then $\|1\boxdot A\|=\frac{tr(A)}{\delta}$.
\end{lemma}
\begin{proof}
Note that $(1\boxdot A) e_2=\frac{tr(A)}{\delta}e_2$. Thus $\|1\boxdot A\|\geq\frac{tr(A)}{\delta}$.
Recall that  $(1\boxdot A)^*=1\boxdot A'$. So $(1\boxdot A)(1\boxdot A')>0$.
Then $\|(1\boxdot A)(1\boxdot A')\|=\lim_{k\rightarrow \infty} tr(((1\boxdot A)(1\boxdot A'))^{k})^{\frac{1}{k}}$.
By an isotopy, we have
$$tr(((1\boxdot A)(1\boxdot A'))^{k+1})=\delta tr((A*A')^{*k}(A*A'))$$
$$\leq\delta \|A*A'\|tr((A*A')^{*k})=\delta^2\|A*A'\|(\frac{tr(A)}{\delta})^{2k}.$$
Thus
$\|(1\boxdot A)(1\boxdot A')\|\leq(\frac{tr(A)}{\delta})^2$. Then
$\|(1\boxdot A)\|\leq\frac{tr(A)}{\delta}$. Therefore $\|(1\boxdot A)\|=\frac{tr(A)}{\delta}$.
\end{proof}
\begin{lemma}\label{QAQ}
Suppose $A$ is a positive operator in $\mathscr{S}_2$ and $1\boxdot Q$ is a minimal projection in $\mathscr{S}_{1,3}$.
If $Q*A*Q=\frac{tr(A)}{\delta}Q$, then for any self-adjoint operator $B\in\mathscr{S}_2$, $B\preceq A$, we have $Q*B*Q=\frac{tr(B)}{\delta}Q$.
\end{lemma}

\begin{proof}
For a positive operator $C<A$, by Lemma \ref{norm1a}, we have $\|1\boxdot C\|=\frac{tr(C)}{\delta}$ and $\|1\boxdot (A-C)\|=\frac{tr(A-C)}{\delta}$.
By assumption $1\boxdot Q$ is minimal, thus $(1\boxdot Q)(1\boxdot C)(1\boxdot Q)=\lambda 1\boxdot Q$, for some $|\lambda|\leq\frac{tr(C)}{\delta}$,
and $(1\boxdot Q)(1\boxdot (A-C))(1\boxdot Q)=\mu 1\boxdot Q$, for some $|\mu|\leq\frac{tr(A-C)}{\delta}$.
If $Q*A*Q=\frac{tr(A)}{\delta}Q$, then $(1\boxdot Q)(1\boxdot A)(1\boxdot Q)=\frac{tr(A)}{\delta} 1\boxdot Q$.
Therefore $\frac{tr(A)}{\delta}=\lambda+\mu$. Then $\lambda=\frac{tr(C)}{\delta}$ and $\mu=\frac{tr(A-C)}{\delta}$.
Thus $Q*C*Q=\frac{tr(C)}{\delta}Q$.
In general, if $B$ is a self-adjoint operator and $B\preceq A$, then $B$ is a linear sum of $C_i$'s, such that $0<C_i<A$.
So $Q*B*Q=\frac{tr(B)}{\delta}Q$.
\end{proof}

\begin{theorem}\label{maxcoeff}
Suppose $A$ is a positive operator in $\mathscr{S}_2$, $P$ is the biprojection generated by $A$ and $1\boxdot Q$ is a minimal projection in $\mathscr{S}_{1,3}$. Then the following are equivalent,

(1)$Q*P*Q=\frac{tr(P)}{\delta}Q$;

(2)$Q*A*Q=\frac{tr(A)}{\delta}Q$;

(3)$Q*A=A*Q=\frac{tr(A)}{\delta}Q$;

(4)$Q*(A+A')*Q=\frac{tr(A+A')}{\delta}Q$.

Consequently $\frac{\delta}{tr(P)}1\boxdot P$ is the spectral projection $E_{\{\|1\boxdot (A+A')\|\}}(1\boxdot (A+A'))$.
\end{theorem}
\begin{proof}
(1)$\Rightarrow$ (4) $\Rightarrow$ (2) follows from Lemma \ref{QAQ} and the fact $A\preceq A+A'\preceq P$.

(2)$\Rightarrow (3)$: If $Q*A*Q=\frac{tr(A)}{\delta}Q$, then $(1\boxdot Q)*(1\boxdot A)*(1\boxdot Q)=\frac{tr(A)}{\delta}(1\boxdot Q)$. By Lemma \ref{norm1a}, we have $\|1\boxdot A\|=\frac{tr(A)}{\delta}$. Thus $(1\boxdot Q)*(1\boxdot A)=(1\boxdot A)*(1\boxdot Q)=\frac{tr(A)}{\delta}(1\boxdot Q)$. Then $Q*A=A*Q=\frac{tr(A)}{\delta}Q$.

(3)$\Rightarrow$ (1): By assumption $1\boxdot Q$ is a projection, thus $Q*Q=Q$. If $Q*A=A*Q=\frac{tr(A)}{\delta}Q$, then $Q*(\sum_{i=1}^kA^{*i})*Q=\delta^{-1}tr(\sum_{i=1}^kA^{*i})Q$. By Lemma \ref{PAk}, we have $P\preceq \sum_{i=1}^k A^{*i}$, for $k$ large enough. So $Q*P*Q=\frac{tr(P)}{\delta} Q$, by Lemma \ref{QAQ}.

Note that
 $1\boxdot Q$ is a subprojection of $\frac{\delta}{tr(P)}1\boxdot P \iff \frac{\delta}{tr(P)}(1\boxdot Q)(1\boxdot P))(1\boxdot Q)=1\boxdot Q \iff Q*P*Q=\frac{tr(P)}{\delta} Q$.
And $1\boxdot Q$ is a subprojection of $E_{\{\|1\boxdot (A+A')\|\}}(1\boxdot (A+A')) \iff (1\boxdot Q)(1\boxdot (A+A'))(1\boxdot Q)=\frac{tr(A+A')}{\delta} 1\boxdot Q \iff Q*(A+A')*Q=\frac{tr(A+A')}{\delta} Q$.
Thus (1) $\iff$ (4) implies $\frac{\delta}{tr(P)}1\boxdot P=E_{\{\|1\boxdot (A+A')\|\}}(1\boxdot (A+A'))$.
\end{proof}

\begin{proposition}\label{maxcoeffremark}
For a finite group $G$, take $\mathscr{S}=\mathscr{S}^G$ to be a group subfactor planar algebra, and $A$ to be the (minimal) central projection corresponding to an (irreducible) representation $V$ of $G$.
Then the following are equivalent

(1) the representation $V$ is faithful;

(2) $E_{\|1\boxdot A\|}(1\boxdot A)=e_2$;

(2') $E_{\{\|1\boxdot (A+A')\|\}}(1\boxdot (A+A'))=e_2$;

(3) the biprojection generated by $A$ is $id$;

(4) every irreducible representation of $G$ is contained in some tensor power of $V$.
\end{proposition}

\begin{proof}
(1)$\iff$ (2) and (3)$\iff$ (4) are analogy between subfactors and representation theory. (2)$\iff$ (2') follows from Lemma \ref{norm1a}. (2')$\iff$ (3) follows from Theorem \ref{maxcoeff}.
\end{proof}

\begin{remark}
Note that (1) $\iff$ (4) is a well known result in Representation theory.
If the condition of (4) is replaced by the tensor power of $V$ and its contragredient, then (1) $\iff$ (4) reduces to Stone-Weierstrass theorem.
\end{remark}

It is well known that a trace-$1$ minimal projection in $\mathscr{S}_2$ induces a normalizer.
Consequently we have the following proposition.
\begin{proposition}
Suppose $P, Q\in\mathscr{S}_2$ are minimal projections and $tr(P)=1$, then $\delta P*Q$ is a minimal projection.
\end{proposition}

\begin{definition}
Suppose $P$ is a central minimal projection in $\mathscr{S}_2$, such that $tr(P)>1$ and $r(P*Q)=1$ (resp. $r(Q*P)=1$), for any minimal projection $Q$ in $\mathscr{S}_2$, $Q\neq P'$. Then we call $P$ a left (resp. right) virtual normalizer. If $P$ is a left and right virtual normalizer, then we call it a virtual normalizer.
\end{definition}

Obviously $P$ is a left virtual normaliser if and only if $P'$ is a right virtual normaliser.

\begin{theorem}\label{virtual normalizer}
Suppose $\mathscr{S}$ is a subfactor planar algebra generated by $\mathscr{S}_2$.
If $\mathscr{S}_2$ contains a left (or right) virtual normalizer, then either $\mathscr{S}$ is Temperley-Lieb or $\mathscr{S}$ is separated by a non-trivial biprojection as a free product.
\end{theorem}

\begin{proof}
When $\dim(\mathscr{S}_2)=2$, we have $\mathscr{S}$ is Temperley-Lieb.

When $\dim(\mathscr{S}_2)\geq3$, we assume that $P$ is a virtual normalizser.

Case 1: If $P'*P\preceq e+P'$, then $P'*P\sim e+P'$, otherwise $P'*P\sim e$ implies $tr(P)=1$ which contradicts to the assumption $tr(P)>1$. Thus $P=P'$. Then $e+P$ is a biprojection by Theorem \ref{P*P=P}. So $P'*P=\frac{tr(P)}{\delta}e+\frac{tr(P)-1}{\delta}P'$.
For a minimal projection $Q$ orthogonal to $(e+P)$, we have $r(P*Q)=1$. Moreover $P*Q\neq P$, because $tr((P*Q)P)=tr(Q(P*P))=0$. Then $r(P*(P*Q))=1$. Note that $P*(P*Q)=(P*P)*Q=\frac{tr(P)}{\delta}e*Q+\frac{tr(P)-1}{\delta}P*Q$, thus $P*Q\sim Q$. Then $(e+P)*Q\sim Q$.
Therefore, by Theorem \ref{freedecom}, $\mathscr{S}$ is separated by $e+P$ as a free product.

Case 2: Otherwise $(id-e-P')(P'*P)>0$. Let $S$ be the support of $(id-P')(P'*P)$. Then $e<S<id$.
We will show that $S$ is a central biprojection which separates $B_2$ as a free product.

For a minimal projection $Q$, $Q\neq P'$, we have
$tr((P'*P)Q)>0 \iff tr((P*Q)P)>0 \iff
P*Q=\frac{tr(Q)}{\delta}P \iff
tr((P*Q)P)=\frac{tr(P)tr(Q)}{\delta} \iff tr((P'*P)Q)=\frac{tr(P)tr(Q)}{\delta}$.
If $Q_1,Q_2$ are two minimal projections such that $\|Q_1-Q_2\|<\frac{tr(P)tr(Q)}{\delta\|P'*P\|}$ and $Q_i\neq P'$, for $i=1,2$, then $tr((P'*P)Q_1)>0 \Rightarrow tr((P'*P)Q_1)=\frac{tr(P)tr(Q_1)}{\delta} \Rightarrow tr((P'*P)Q_2)>0$.
By Theorem \ref{P*P>0}, $tr((P'*P)Q_i)\geq0$, for $i=1,2$. Thus $tr((P'*P)Q_2)=0 \Rightarrow tr((P'*P)Q_1)=0$.
Combining with Lemma \ref{CDE0}, we have $SQ_2=0 \Rightarrow (S+P')Q_2=0 \Rightarrow (P'*P)Q_2=0 \Rightarrow tr((P'*P)Q_2)=0 \Rightarrow tr((P'*P)Q_1)=0 \Rightarrow SQ_1=0$.
Thus $S$ is central.

For a minimal projection $R$, such that $R\leq S$, we have $R\neq P'$ and $R\preceq P'*P$.
Then $tr((P*R)P)=tr(R(P'*P))>0$. Thus $P*R\sim P$. Then $P*S\sim P$.
Therefore $P*(S*S)=(P*S)*S\sim P$. Then for a minimal projection $U$, $U\preceq S*S$, we have $P*U\sim P$. Thus $U\neq P'$ and $tr(U(P'*P))=tr((P*U)P)>0$. Then $tr(US)>0$. We have proved that $S$ is central, so $U\leq S$. Then $S*S\preceq S$. By Theorem \ref{P*P=P}, $S$ is biprojection.

Recall that $P*S\sim P$, so $S*P'\sim P'$.
Suppose $R_1$ is a minimal projection orthogonal to $S+P'$.
Then $tr((S*R_1)(S+P'))=tr(R_1(S*(S+P'))=0$. Thus $S*R_1$ is orthogonal to $S+P'$.
Moreover $P*(S*R_1)=(P*S)*R_1\sim P*R_1$. We claim that $S*R_1\sim R_1$.
Then by Theorem \ref{freedecom}, we have $\mathscr{S}_2$ is separated by $S$ as a free product.

Now we prove that $S*R_1\sim R_1$. Suppose $R_2$ is a minimal projection, such that $R_2\preceq S*R_1$. Then
$R_2$ is orthogonal $S+P'$, and $P*R_2 \preceq P*(S*R_1)\sim P*R_1$.
So $P*R_1=P*R_2$. It is enough to show that $R_2=R_1$.

For $n=1,2,\cdots$, we have $P^{*n}*R_1=P^{*n}*R_2$.
If $P^{*n}*R_1\nsim P'$ and $r(P^{*n}*R_1)=1$, then $r(P^{*(n+1)}*R_1)=1$.

(1) If $P^{*n}*R_1\nsim P'$, $\forall ~n>0$, then $r(P^{*n}*R_1)=1$, $\forall ~n>0$. By Lemma \ref{PAk}, $e\preceq P^{*m}$ for some $m>0$, thus $R_1\sim P^{*m}*R_1=P^{*m}*R_2\sim R_2$. Then $R_1=R_2$.

(2) If $P^{*n}*R_1\sim P'$, for some $n>0$, assuming this $n$ is the minimal one, then $r(P^{*j}*R_1)=1$, $\forall ~1\leq j\leq n$.

(a) If $P' \preceq P^{*k}$, for some $1\leq k\leq n-1$, then $e\preceq P*P'\preceq P^{*(k+1)}$.
Thus $R_1\sim P^{*(k+1)}*R_1 \sim P^{*(k+1)}*R_2\sim R_2$. Then $R_1=R_2$.

(b) If $P'\preceq P^{*n}$, then $P'*R_1\preceq P^{*n}*R_1\sim P'$. So $tr(R_1(P'*P))=tr(R_1(P*P'))=tr((P'*R_1)P')>0$. On the other hand, we have $R_1(S+P')=0$ and $P*P'\preceq S+P'$, so $R_1(P'*P)=0$. It is a contradiction.

(c) Otherwise $P'\nsim P^{*j}$ and $r(P^{*(j+1)})=1$, $\forall ~1\leq j \leq n$. We will show that $(P')^{*l}$ is central by induction, for $1\leq l \leq n+1$.
The virtually normalizer $P$ is central, so $P'$ is central. For $1\leq l \leq n$, suppose $(P')^{*l}$ is central.
Take a minimal projection $V$, such that $tr(((P')^{*(l+1)})V)>0$. If $V=P'$, then $(P')^{*(l+1)}=P'$, and it is central.
If $V\neq P'$, then $r(P*V)=1$. Note that $tr((P')^{*l}(P*V))=tr(((P')^{*(l+1)})V)>0$, thus $P*V\sim (P')^{*l}$. Then $P*V$ only depends on $tr(V)$.
So $tr(((P')^{*(l+1)})V)$ only depends on $tr(V)$. When the minimal projection $V$ moves continuously in its central support, the assumption $tr(((P')^{*(l+1)})V)>0$ always holds. So $(P')^{*(l+1)}$ is central.
Then we have $(P')^{*(n+1)}$ is a central minimal projection.
Recall that $P^{*n}*R_1\sim P'$, thus $tr(R_1((P')^{(n+1)}))=tr((P^{*n}*R_1)P')>0$.
Then $R_1\sim(P')^{*(n+1)}$. Similarly $R_2\sim(P')^{*(n+1)}$. So $R_1\sim R_2$. Then $R_1=R_2$.

\end{proof}

\begin{remark}
In $\mathscr{S}^S$, either $tr(P)=1$, or $P$ is a virtual normalizer. In the latter case, $e+P$ is a biprojection and the planar algebra $\mathscr{S}^S$ is separated by $e+P$ a free product. Furthermore $(\mathscr{S}^S)_{e+P}$ is Temperley-Lieb, and $P$ is the second Jones-Wenzl projection.
\end{remark}

\section{Constructions and Decompositions}\label{code}
\subsection{Exchange Relation Planar Algebras}
In general, it is not easy to show a subfactor planar algebra is an exchange relation planar algebra. In this section we will give two general constructions of exchange relation planar algebras by the free product and the tensor product. Moreover we will show that the subgroup subfactor planar algebra $\mathscr{S}^{\mathbb{Z}_2\subset\mathbb{Z}_p\rtimes\mathbb{Z}_2}$ is an exchange relation planar algebra, for an odd prime number $p$.
For the classification, we will show how an exchange relation planar algebra decomposed as a free product or a tensor product.

\begin{proposition}
Suppose $\mathscr{S}$ is an exchange relation planar algebra, then its dual is an exchange relation planar algebra.
\end{proposition}
\begin{proof}
Recall that the dual of a subfactor planar algebra is given by switching its shading.
Thus the dual of an exchange relation planar algebra is still generated by 2-boxes,
and its exchange relation is given by the $180^o$ rotation of the adjoint of the original exchange relation.
\end{proof}

\begin{proposition}\label{exfreedecom}
Suppose $\mathscr{A}*\mathscr{B}$ is an exchange relation planar algebra, then both $\mathscr{A}$ and $\mathscr{B}$ are exchange relation planar algebras.
\end{proposition}

\begin{proof}
By Corollary \ref{freedecom}, both $\mathscr{A}$ and $\mathscr{B}$ are generated by 2-boxes.
Suppose $Q=id\otimes e$, the central biprojection separating $\mathscr{A}*\mathscr{B}$ as a free product.
Then $\mathscr{A}$ is isomorphic to $(\mathscr{A}*\mathscr{B})_Q$.
For any $x,y \preceq Q$, we have $(1\boxdot x)y=\sum_i c_i(1\boxdot d_i)+f_i(1\boxdot id)g_i$, for finitely many two boxes $c_i,d_i,f_i,g_i$.
Then $Q(1\boxdot Qx)yQ=\sum_i Qc_i(1\boxdot Qd_i)Q+Qf_i(1\boxdot Q)g_iQ$.
Note that $Qc_iQ, d_iQ, Qf_i, g_iQ \preceq Q$, so that is the exchange relation of $(\mathscr{A}*\mathscr{B})_Q$.
Thus $\mathscr{A}$ is an exchange relation planar algebra.
Considering the duality of exchange relation planar algebras, we have $\mathscr{B}$ is an exchange relation planar algebra.
\end{proof}

\begin{proposition}\label{exfreeprod}
Suppose $\mathscr{A}$ and $\mathscr{B}$ are exchange relation planar algebras, then $\mathscr{A}*\mathscr{B}$ is an exchange relation planar algebra.
\end{proposition}

\begin{proof}
By Corollary \ref{free222}, we have $\mathscr{A}*\mathscr{B}$ is generated by 2-boxes.
Suppose $Q=id\otimes e$ is the central biprojection which separates $\mathscr{A}*\mathscr{B}$ as a free product.
Then any 2-box in $\mathscr{A}*\mathscr{B}$ is of the form $x\otimes e+id\otimes y$, for some $x\in\mathscr{A}_2$, $y\in\mathscr{B}_2$.
We need to check the exchange relation for four cases.
For any $x_1,x_2\in\mathscr{A}_2$, $y_1,y_2\in\mathscr{B}_2$,

(1)the exchange relation of $(1\boxdot (x_1\otimes e)) (x_2\otimes e)$ follows from the exchange relation of $\mathscr{A}$;

(2)the exchange relation of $(1\boxdot (id\otimes y_1)) (id\otimes y_2)$ follows from the exchange relation of $\mathscr{B}$;

(3)$(1\boxdot (x_1\otimes e)) (id\otimes y_1)=(id\otimes y_1)(1\boxdot (x_1\otimes e))$;

(4)$(1\boxdot (id\otimes y_1)) (x_1\otimes e)=(id \otimes y_1') (1\boxdot id) (x_1\otimes e)$.

Therefore $\mathscr{A}*\mathscr{B}$ is an exchange relation planar algebra.
\end{proof}

\begin{proposition}\label{extensordecom}
Suppose $\mathscr{A}\otimes\mathscr{B}$ is an exchange relation planar algebra, then both $\mathscr{S}_A$ and $\mathscr{S}_B$ are exchange relation planar algebras.
\end{proposition}

\begin{proof}
By Corollary \ref{tensordecom222}, both $\mathscr{S}_A$ and $\mathscr{S}_B$ are generated by 2-boxes.
The rest is the same as the proof of Proposition \ref{exfreedecom}.
\end{proof}

Its converse is not true. The tensor product of two Temperley-Lieb subfactor planar algebras may not be an exchange relation planar algebra. It is a corollary of Theorem \ref{main1}.
But a weak version is true.

\begin{proposition}\label{extensorprod}
Suppose $\mathscr{A}$ is a Depth 2 subfactor planar algebra and $\mathscr{B}$ is an exchange relation planar algebra, then $\mathscr{S}_A\otimes\mathscr{S}_B$ is an exchange relation planar algebra.
\end{proposition}

\begin{proof}
By Corollary \ref{tensor222}, we have $\mathscr{A}\otimes\mathscr{B}$ is generated by 2-boxes.
Suppose $A=id\otimes e$ and $B=e\otimes id$, the central biprojections separating $\mathscr{S}_A\otimes\mathscr{S}_B$ as a tensor product.
Then any 2-boxes in $\mathscr{S}_A\otimes\mathscr{S}_B$ is a finite sum of $x*y$'s for which $x\preceq A,y\preceq B$.
So we only need to check the exchange relation for $(1\otimes(x_1*y_1))(x_2*y_2)$, for any $x_1,x_2\preceq A, y_1, y_2\preceq B$.
Since $\mathscr{B}$ is an exchange relation planar algebra, and $(\mathscr{A}\otimes\mathscr{B})_B$ is isomorphic to $\mathscr{B}$, we have
$B(1\boxdot By_1)y_2B=\sum_i Bc_i(1\boxdot Bd_i)B+Bf_i(1\boxdot B)g_iB$, for finitely many $c_i,d_i,f_i,g_i\preceq B$.
Then $(1\boxdot y_1)y_2=\sum_i c_i(1\boxdot d_i)+f_i(1\boxdot )g_i$, by the exchange relation of the biprojection $B$.
Similarly $x_1=\sum_j s_j(1\boxdot id)t_j$, $x_2=\sum_k u_k(1\otimes id)v_k$, for finitely many $s_j,t_j,u_k,v_k\preceq A$, because $\mathscr{A}$ is a Depth 2 subfactor planar algebra.
Then
$$\grb{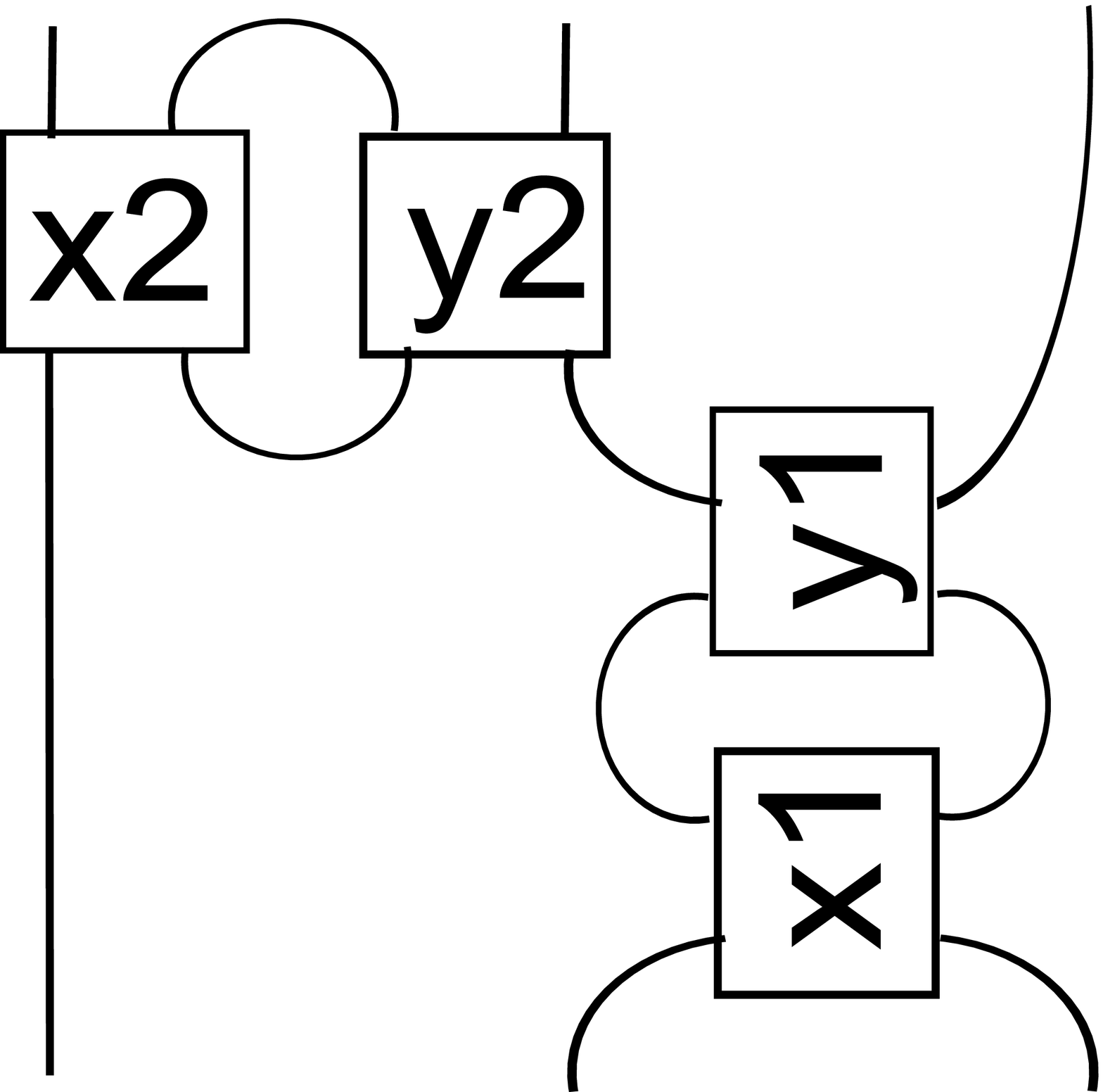}=\sum_i\grb{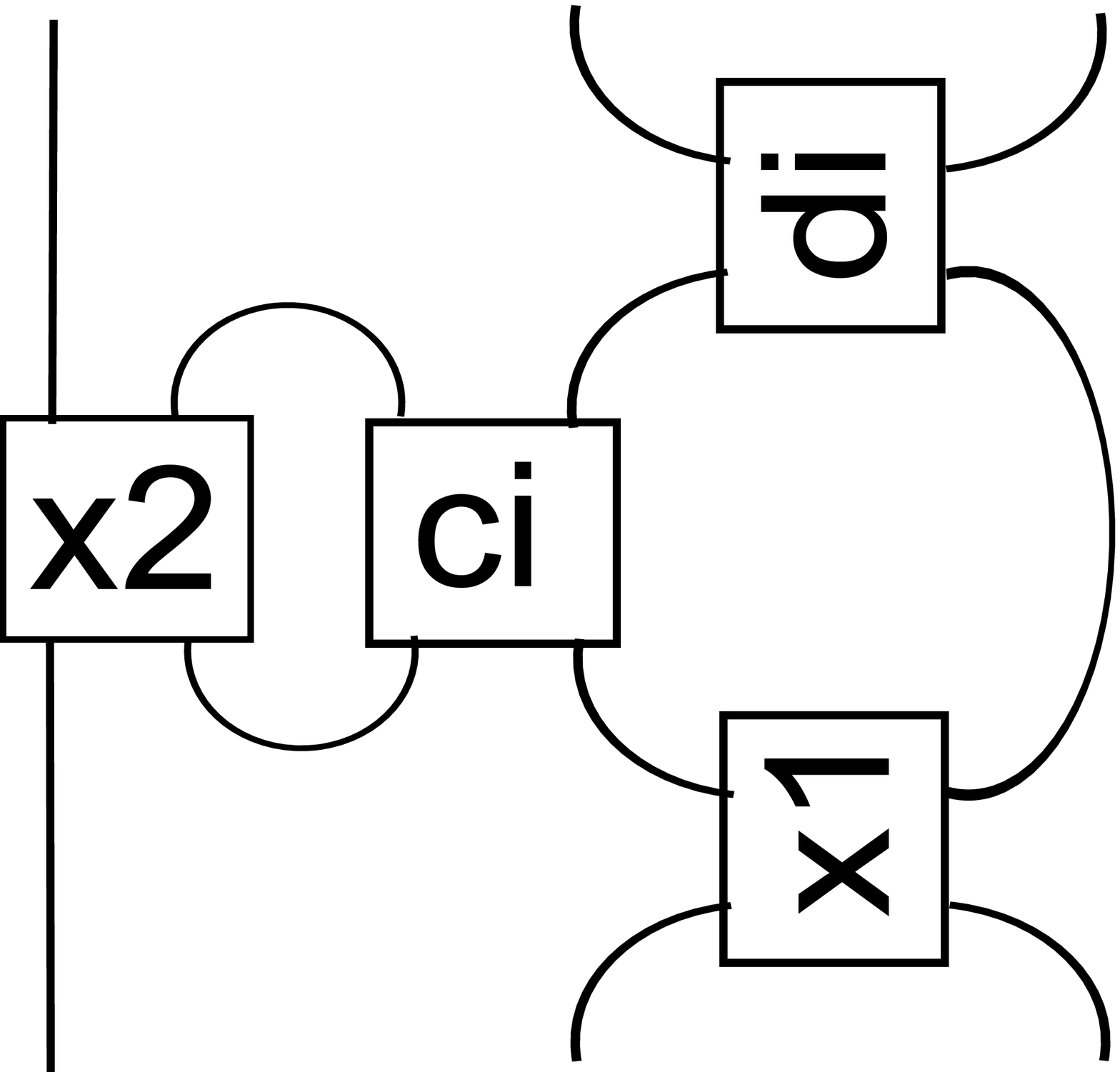}+\grb{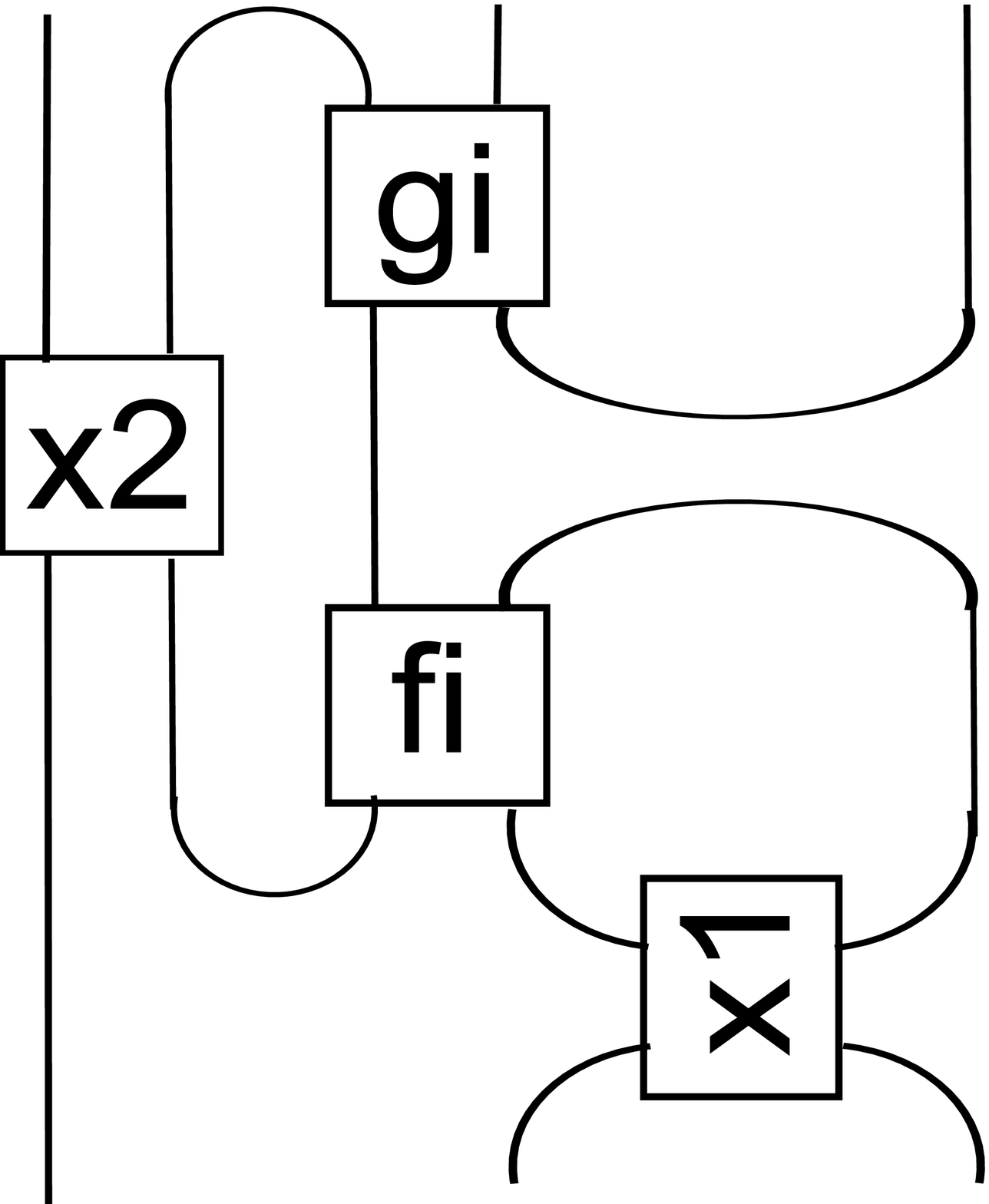}=\sum_{i,j}\grb{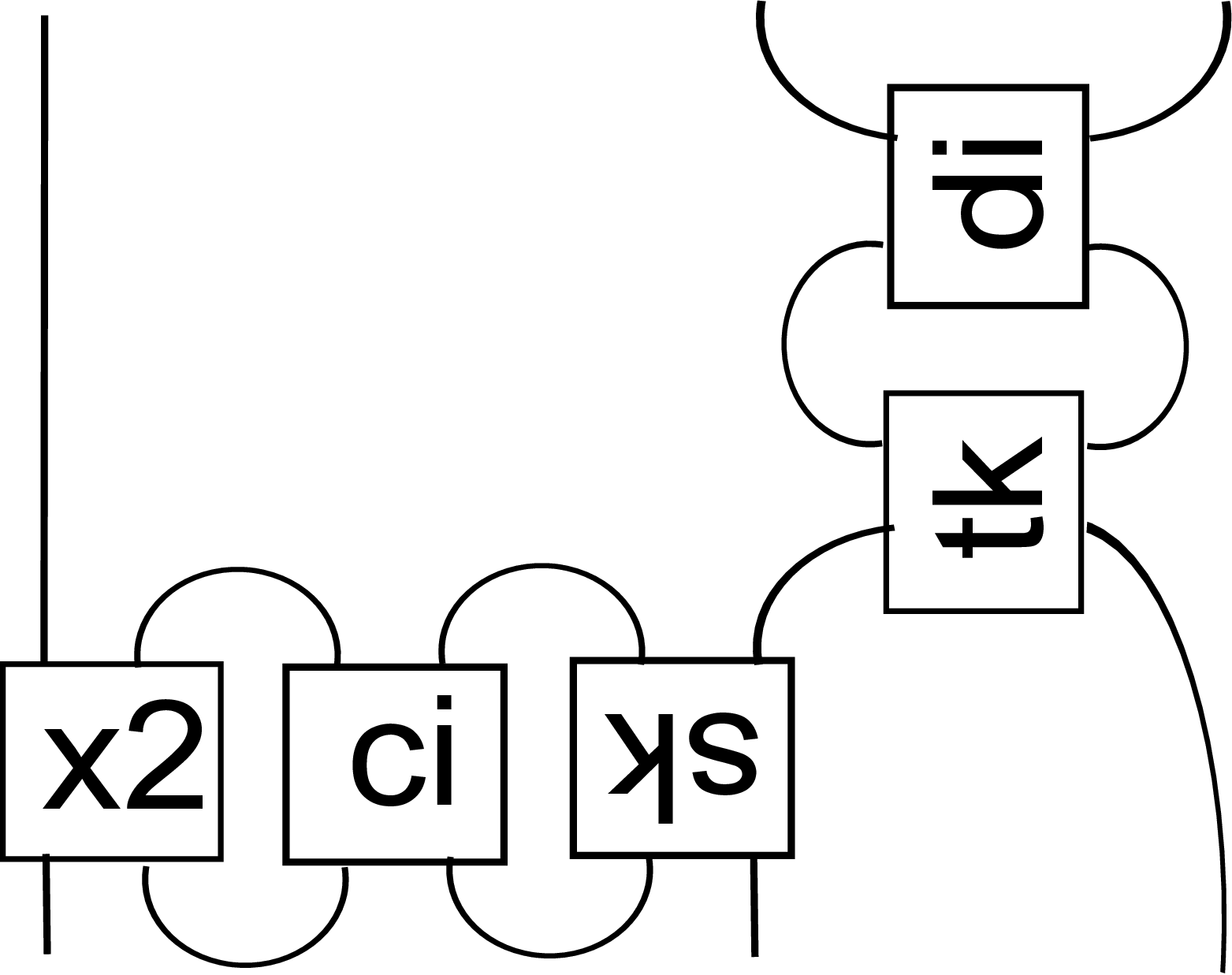}+\sum_{j,k}\grb{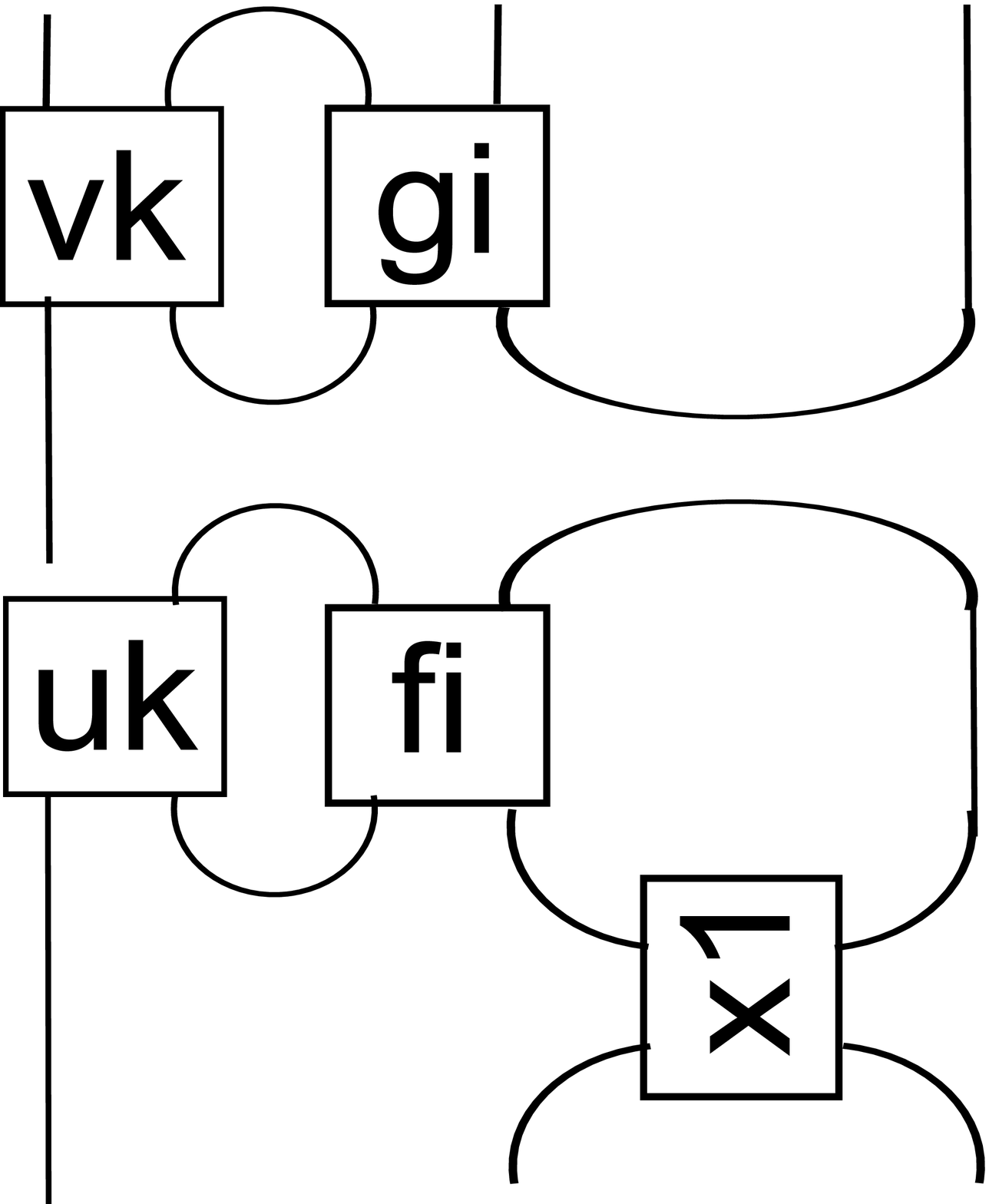}.$$
Thus $\mathscr{S}_A\otimes\mathscr{S}_B$ is an exchange relation planar algebra.
\end{proof}

\begin{theorem}\label{ex2p2}
For an odd prime number $p$, the subgroup subfactor planar algebra $\mathscr{S}^{\mathbb{Z}_2\subset\mathbb{Z}_p\rtimes\mathbb{Z}_2}$ is an exchange relation planar algebra, where $\mathbb{Z}_p\rtimes\mathbb{Z}_2=\{a,t| a^p=1, t^2=1, tat=a^{-1} \}$.
\end{theorem}

\begin{proof}
Note that the principal graph of $\mathscr{S}^{\mathbb{Z}_2\subset\mathbb{Z}_p\rtimes\mathbb{Z}_2}$ is $\gra{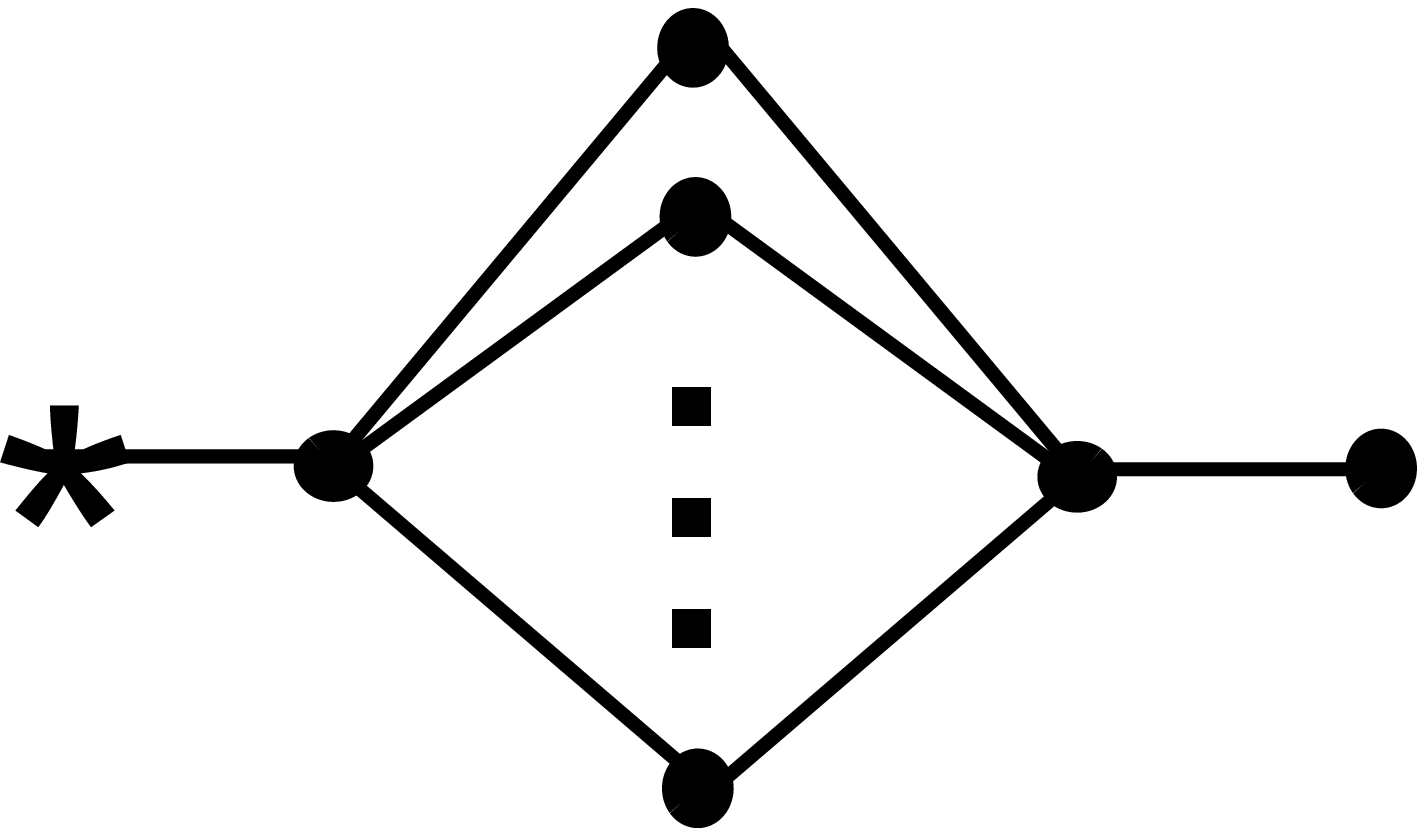}$, with $\frac{p-1}{2}$ depth 2 points.
So $\dim((\mathscr{S}^{\mathbb{Z}_2\subset\mathbb{Z}_p\rtimes\mathbb{Z}_2})_2)=\frac{p+1}{2}$ and $\dim((\mathscr{S}^{\mathbb{Z}_2\subset\mathbb{Z}_p\rtimes\mathbb{Z}_2})_3)=(\frac{p+1}{2})^2+(\frac{p-1}{2})^2$.
Considering $\mathscr{S}^{\mathbb{Z}_2\subset\mathbb{Z}_p\rtimes\mathbb{Z}_2}$ as a biprojection cutdown of $\mathscr{S}^{\mathbb{Z}_p\rtimes\mathbb{Z}_2}$, it is easy to check that the minimal projections $e,g_1,\cdots,g_{\frac{p-1}{2}}$ of $(\mathscr{S}^{\mathbb{Z}_2\subset\mathbb{Z}_p\rtimes\mathbb{Z}_2})_2$ satisfy the relation
$$\delta g_m*g_n=g_{m+n}+g_{m-n}, ~\forall ~0\leq m,n \leq \frac{p-1}{2},$$
where $g_0=2e$; $g_{m+n}=g_{p-m-n}$, when $m+n>\frac{p-1}{2}$; and $g_{m-n}=g_{n-m}$, when $m-n<0$. Take
$$\chi_k=\frac{\delta}{p}(1\boxdot 2e+ \sum_{m=1}^{\frac{p-1}{2}}(\omega^{mk}+\omega^{-mk})(1\boxdot g_m)),$$
 where $\omega=e^{\frac{2\pi i}{p}}$, for $k=0,1,\cdots,\frac{p-1}{2}$. Then $\chi_0=2e_2$ and $\{e_2\}\cup \{\chi_k\}_{k=1}^{\frac{p-1}{2}}$ is the set of minimal projections of $(\mathscr{S}^{\mathbb{Z}_2\subset\mathbb{Z}_p\rtimes\mathbb{Z}_2})_{1,3}$.

Suppose $\mathscr{S}$ is the planar subalgebra of $\mathscr{S}^{\mathbb{Z}_2\subset\mathbb{Z}_p\rtimes\mathbb{Z}_2}$ generated by 2-boxes, and $\mathscr{I}_3$ is the two sided ideal of $\mathscr{S}_3$ generated by the Jones Projection. Then $\mathscr{S}_3=\mathscr{I}_3\oplus \mathscr{S}_3/\mathscr{I}_3$. Let us define $s_3$ to be the support of $\mathscr{S}_3/\mathscr{I}_3$, then $s_3P_m=P_m-\frac{\delta}{tr(P_m)}P_m(1\boxdot id)P_m$, by Wenzl's formula. By a direct computation, we have $tr(s_3P_m\chi_k)\neq 0$, for $1\leq m,k\leq \frac{p-1}{2}$, and $\{s_3P_m\chi_k\}_{1\leq m,k\leq \frac{p-1}{2}}$ is a set of pairwise orthogonal vectors. By Proposition \ref{dimpn}, we have $\dim(\mathscr{S}_3/\mathscr{I}_3)\leq (\dim((\mathscr{S}_2)-1)^2=(\frac{p-1}{2})^2$. Thus $\{s_3P_m\chi_k\}_{1\leq m,k\leq \frac{p-1}{2}}$ forms a basis of $\mathscr{S}_3/\mathscr{I}_3$, and $\dim(\mathscr{S}_3/\mathscr{I}_3)=(\frac{p-1}{2})^2$. That means $\mathscr{S}$ is an exchange relation planar algebra. Moreover $\mathscr{S}_3=\mathscr{S}^{\mathbb{Z}_2\subset\mathbb{Z}_p\rtimes\mathbb{Z}_2}_3$. So they have the same principal graph up to depth 3. Then their principal graphs have to be the same by the restriction of the index. So $\mathscr{S}=\mathscr{S}^{\mathbb{Z}_2\subset\mathbb{Z}_p\rtimes\mathbb{Z}_2}$.
\end{proof}

\subsection{$\mathscr{S}_3/\mathscr{I}_3$ is abelian}
To classify subfactor planar algebras generated by 2-boxes subject to the condition that the quotient of 3-boxes by the basic construction ideal is abelian, let us prove two lemmas for the construction and the decomposition of such planar algebras via the free product. For convenience we use following notations.
\begin{notation}\label{comtype}
A subfactor planar algebra $\mathscr{S}$ is called a commute relation planar algebra, if it is generated 2-boxes and $\mathscr{S}_3/\mathscr{I}_3$ is abelian.
Moreover it is of type AN, if $\mathscr{S}_2$ is abelian;
of type NA, if $\mathscr{S}_{1,3}$ is abelian;
of type AA, if both $\mathscr{S}_2$ and $\mathscr{S}_{1,3}$ are abelian.
\end{notation}

A Temperley-Lieb subfactor planar algebra is a commute relation planar algebra of type AA.
A depth 2 subfactor planar algebra is a commute relation planar algebra.
Furthermore it is of type AN if and only if it is $\mathscr{S}^G$, for a group $G$;
It is of type NA if and only if it is the dual of $\mathscr{S}^{G}$, for a group $G$;
It is of type AA if and only if it is $\mathscr{S}^G$, for an abelian group $G$.

\begin{proposition}\label{comex}
Suppose $\mathscr{S}$ is a commute relation planar algebra, then it is an exchange relation planar algebra.
Consequently $\mathscr{S}_3$ is generated by $\mathscr{S}_2$ and $\mathscr{S}_{1,3}$  as an algebra.
\end{proposition}

\begin{proof}
If $\mathscr{S}_3/\mathscr{I}_3$ is abelian,
then for any $a,b\in \mathscr{S}_2$, we have
$(1\boxdot a)b-b(1\boxdot a)\in \mathscr{I}_3$. Thus
$(1\boxdot a)b=b(1\boxdot a)+f_i(1\boxdot id)g_i$, for finitely many $f_i,g_i\in \mathscr{S}_2$.
\end{proof}


\begin{lemma}\label{comfreepro}
Suppose $\mathscr{A},\mathscr{B}$ are commute relation planar algebras of type NA and AN respectively,
then $\mathscr{S}=\mathscr{A}*\mathscr{B}$ is a commute relation planar algebra.
Furthermore if $\mathscr{A}$ is of type AA, then $\mathscr{S}$ is of type AN;
If $\mathscr{B}$ is of type AA, then $\mathscr{S}$ is of type NA.
\end{lemma}

\begin{proof}
Suppose $\mathscr{A},\mathscr{B}$ are commute relation planar algebras of type NA and AN respectively,
then by Proposition(\ref{exfreeprod}),(\ref{comex}), we have  $\mathscr{S}=\mathscr{A}*\mathscr{B}$ is an exchange relation planar algebra.
Therefore $\mathscr{S}$ is generated by $\mathscr{S}_2$ and $\mathscr{S}_{1,3}$ as an algebra.
To show $\mathscr{S}_3/\mathscr{I}_3$ is abelian, it is enough to show that for any $x,y\in \mathscr{S}_2$, we have

(1) $(1\boxdot x)y-y(1\boxdot x)\in \mathscr{I}_3$;

(2) $xy-yx\in \mathscr{I}_3$;

(3) $(1\boxdot x)(1\boxdot y)-(1\boxdot y)(1\boxdot x)\in \mathscr{I}_3$;

Note that $x=a\otimes e + id\otimes b$, $y=c\otimes e +id \otimes b$, for some $a,c\in \mathscr{A}_2$ and $b,d\in\mathscr{B}_2$. Each case splits into four subcases. Now we use the A,B-colour diagrams to express elements in the free product. We omit the labels of the boundary of a diagram, which should be ordered as $ABBA~ ABBA\cdots ABBA$.

(1.1) By assumption $\mathscr{A}$ is a commute relation planar algebra, so $(1\boxdot a)c-c(1\boxdot a)=\sum_i f_i(1\boxdot id)g_i$, for finitely many $f_i,g_i\in \mathscr{A}_2$.
Then $\grb{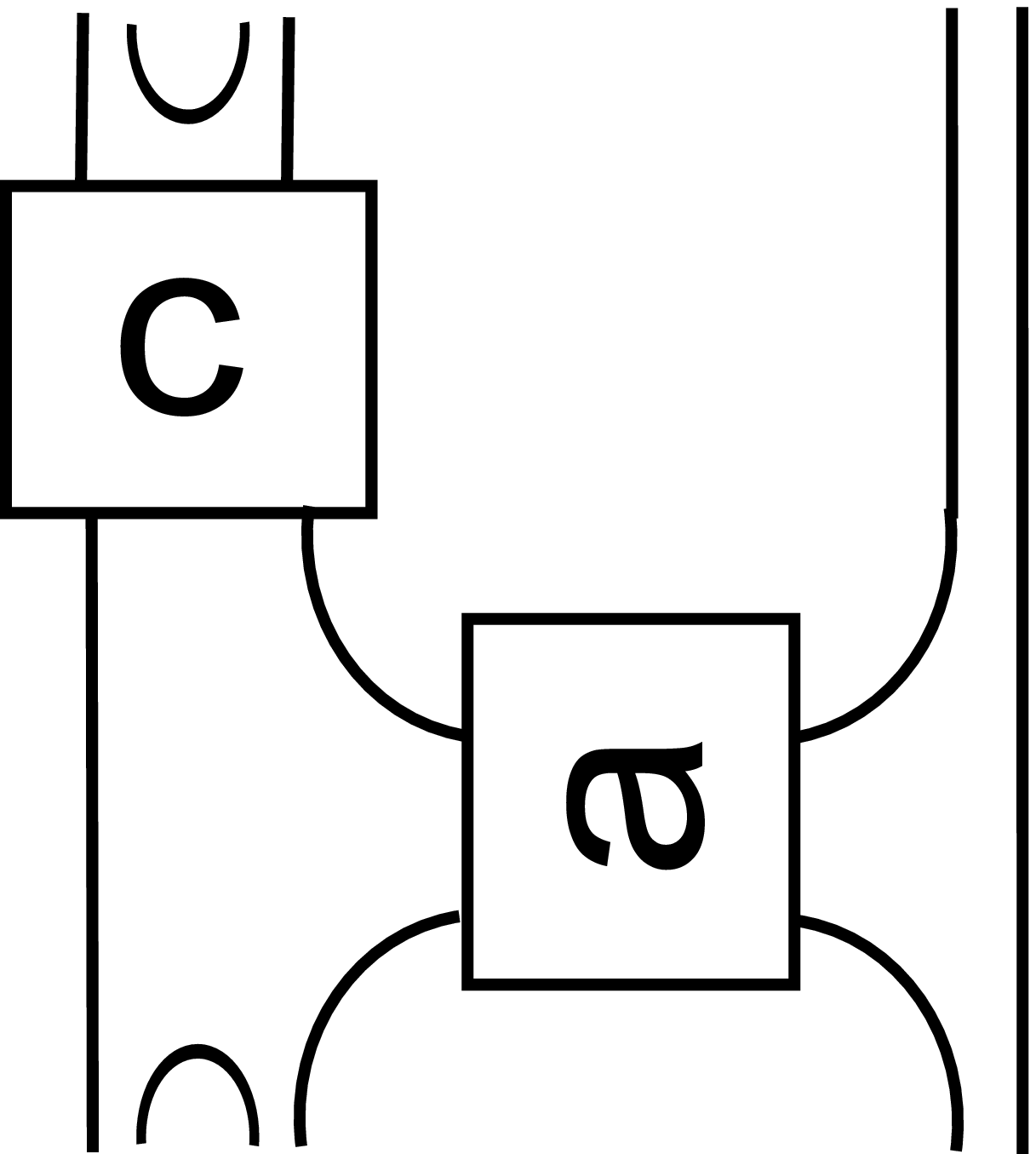}-\grb{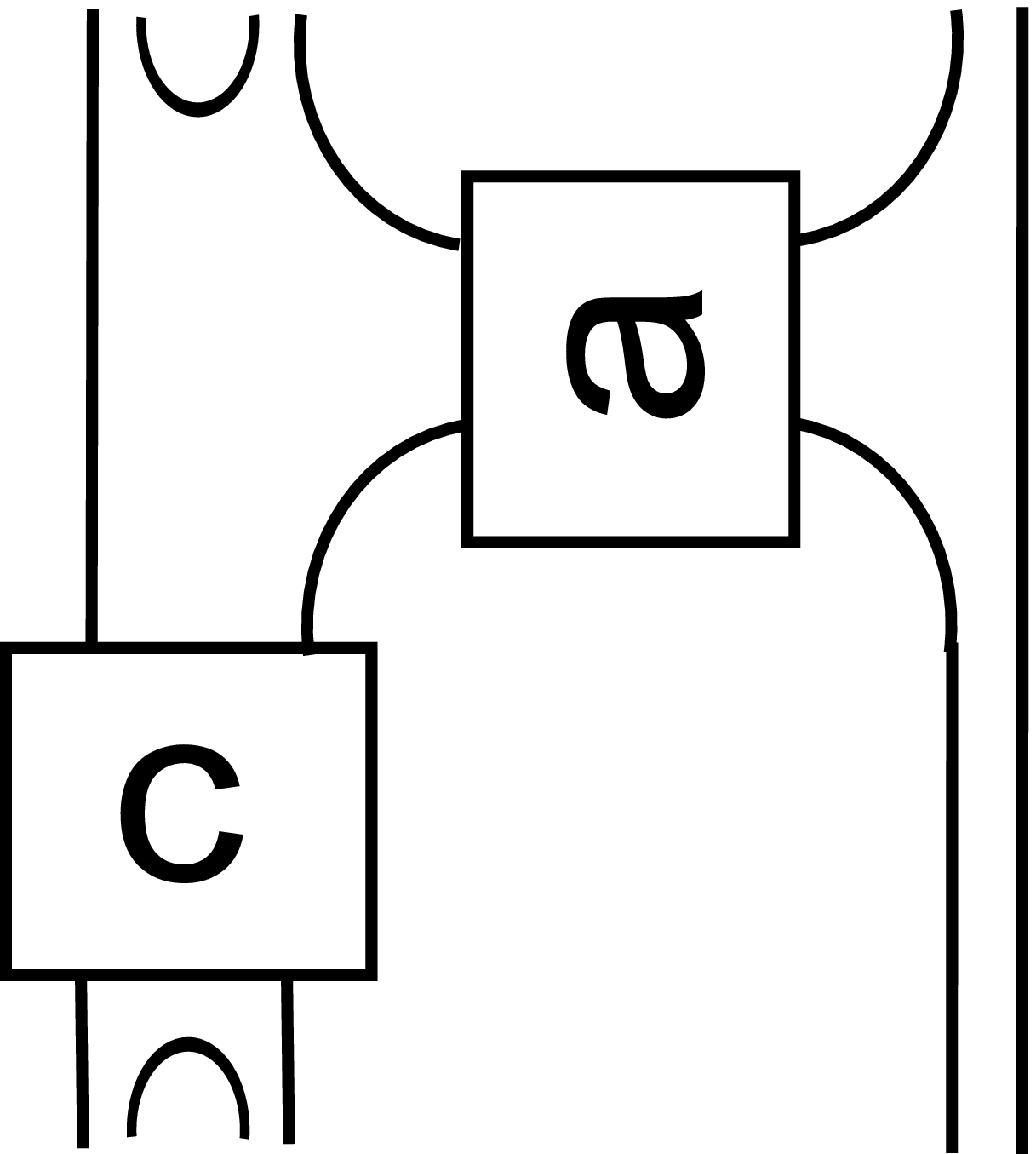}=\sum_i\grb{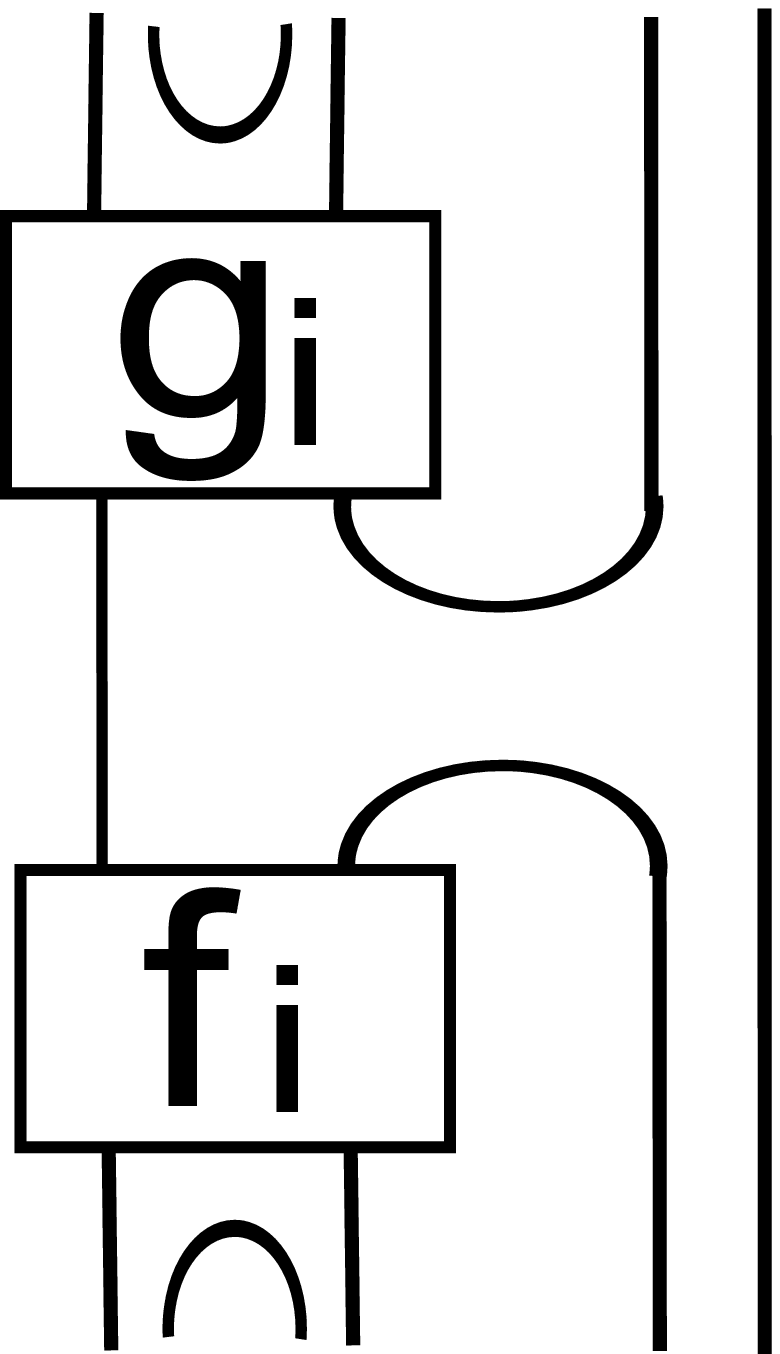}$. Thus $(1\boxdot (a\otimes e))(c\otimes e)-(c\otimes e)(1\boxdot (a\otimes e))\in \mathscr{I}_3$;

(1.2) By assumption $\mathscr{B}$ is a commute relation planar algebra. Similarly we have $(1\boxdot (id\otimes b))(id\otimes d)-(id\otimes d)(1\boxdot (id\otimes b))\in \mathscr{I}_3$;

(1.3) $(1\boxdot (a\otimes e))(id\otimes d)-(id\otimes d)(1\boxdot (a\otimes e))=0$;

(1.4) Note that $\grb{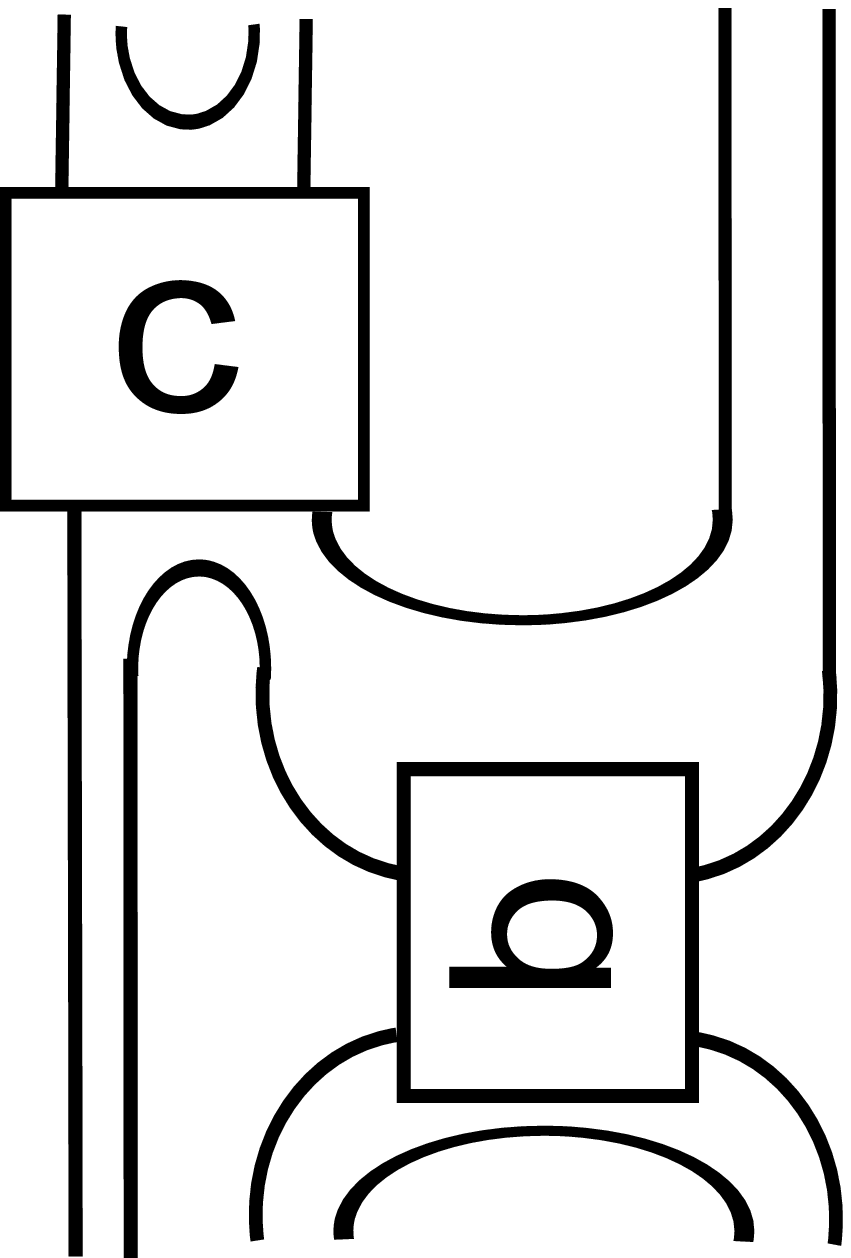},\grb{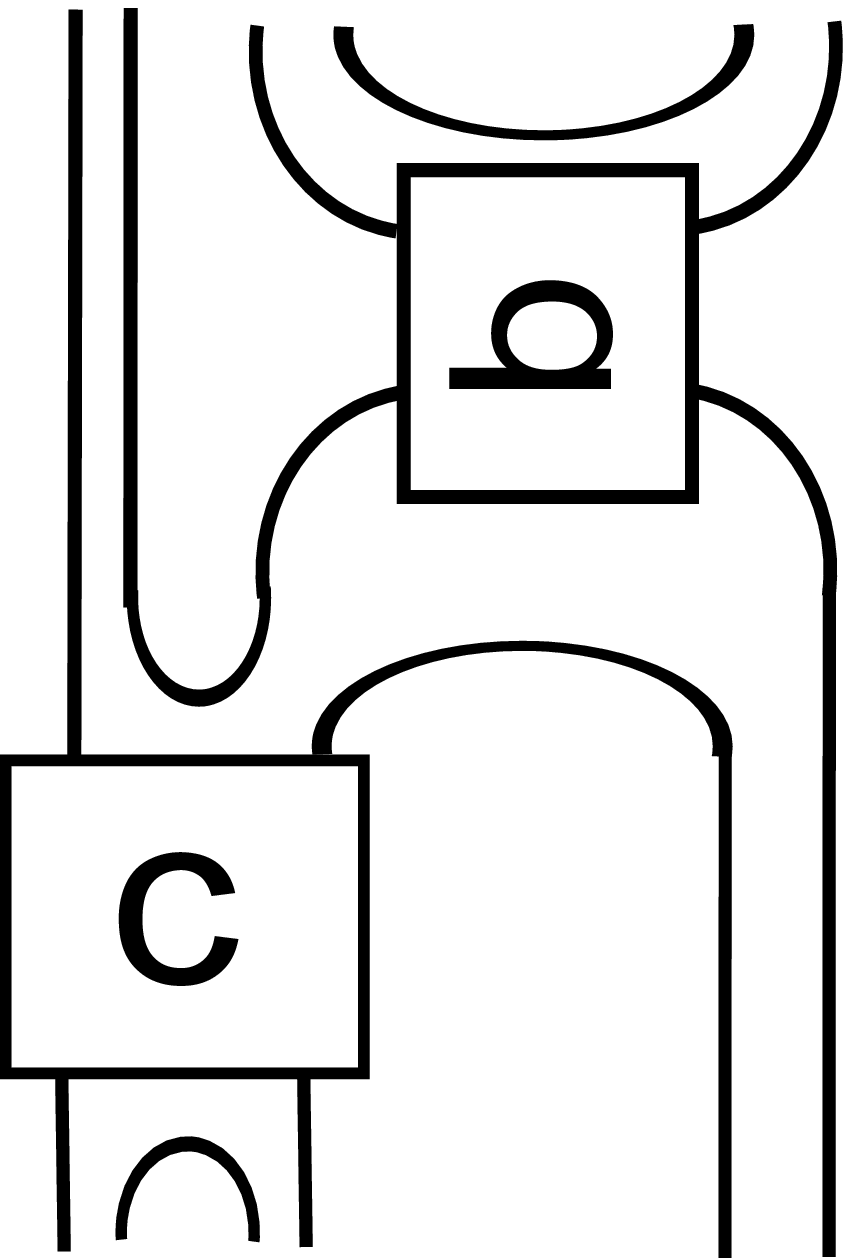} \in \mathscr{I}_3$, thus
$(1\boxdot (id\otimes b))(c\otimes e)-(c\otimes e)(1\boxdot (id\otimes b))\in \mathscr{I}_3$;

(2.1) By assumption $\mathscr{A}$ is a commute relation planar algebra, so $ac-ca=\sum_j f_j(1\boxdot id)g_j$, for finitely many $f_j,g_j\in \mathscr{A}_2$.
Then $\grb{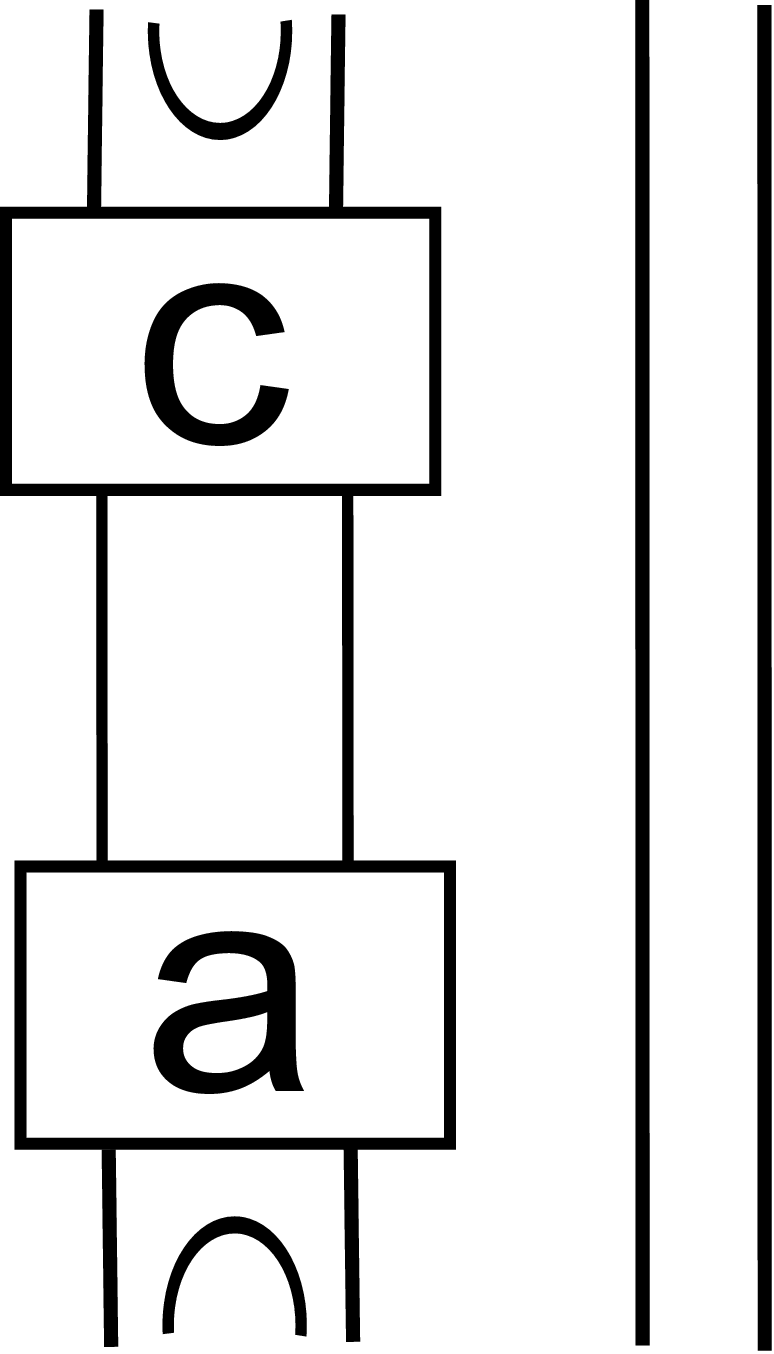}-\grb{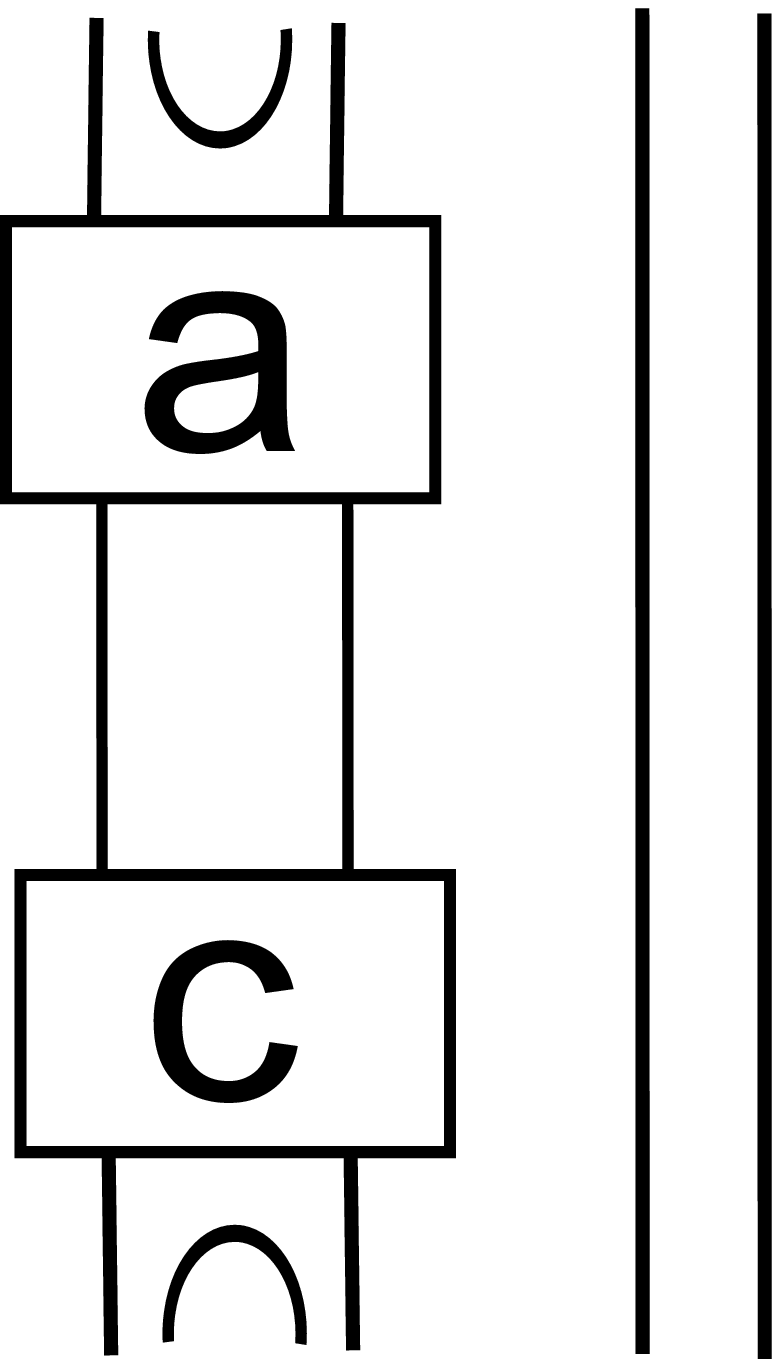}=\sum_j\grb{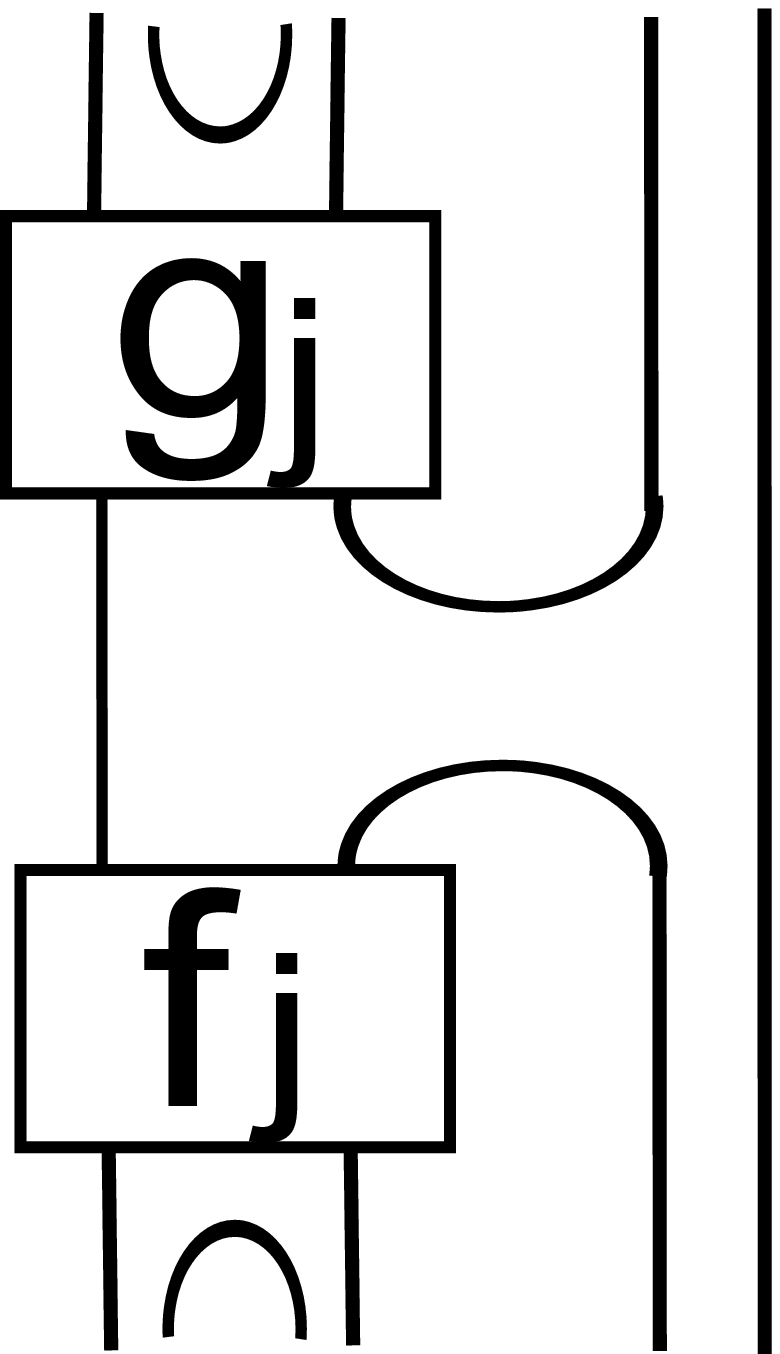}$. Thus $(a\otimes e)(c\otimes e)-(c\otimes e)(a\otimes e) \in \mathscr{I}_3$.

(2.2) By assumption $\mathscr{B}$ is of type AN, so $bd-db=0$. Thus $(id\otimes b)(id\otimes d)-(id\otimes b)(id\otimes d)=0$;

(2.3) $(a\otimes e)(id\otimes d)-(id\otimes d)(a\otimes e)=a\otimes ed- a\otimes de=0$;

(2.4) $(id\otimes b)(c\otimes e)-(c\otimes e)(id\otimes b)=c\otimes be- c\otimes eb=0$;

Considering the $180^\circ$ rotation, the proof of (3) is the same as that of (2).

Therefore $\mathscr{S}$ is a commute relation planar algebra.
\end{proof}


\begin{lemma}\label{comfreedecom}
Suppose $\mathscr{A}, \mathscr{B}$ are subfactor planar algebras with circle parameters greater than 1.
If $\mathscr{S}=\mathscr{A}*\mathscr{B}$ is a commute relation planar algebra,
then $\mathscr{A},\mathscr{B}$ are commute relation planar algebras of type NA and AN respectively.
Furthermore if $\mathscr{S}$ is of type AN, then $\mathscr{A}$ is of type AA;
If $\mathscr{S}$ is of type NA, then $\mathscr{B}$ is of type AA.
\end{lemma}

\begin{proof}
Because $\mathscr{S}_2$ is a commute relation planar algebra,
for any $a,c$ in $\mathscr{A}_2$, we have
$$\graa{com1}-\graa{com2}=\sum_i \graa{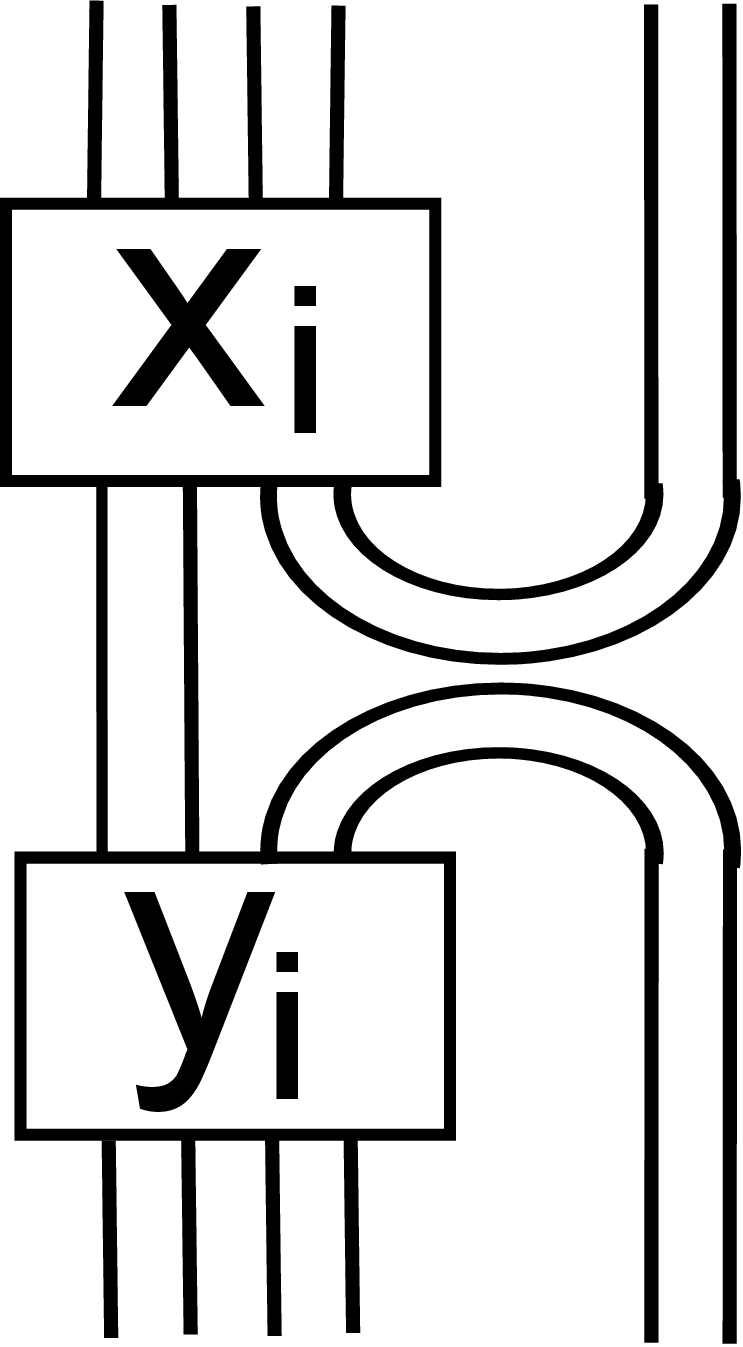},$$
for finitely many $x_i, y_i \in \mathscr{S}_2$. Moreover
$$(id \otimes e)(\graa{com1}-\graa{com2})(id\otimes e)=\graa{com1}-\graa{com2},$$
so we may assume that $x_i, y_i \preceq id\otimes e$. Then
$$\graa{com1}-\graa{com2}=\sum_i\graa{com3},$$
for finitely many $f_i,g_i\in \mathscr{A}_2$. Therefore
$(1\boxdot a)c-c(1\boxdot a)=\sum_i f_i(1\boxdot id)g_i.$

Similarly
$$\graa{com6}-\graa{com7}=\sum_j \graa{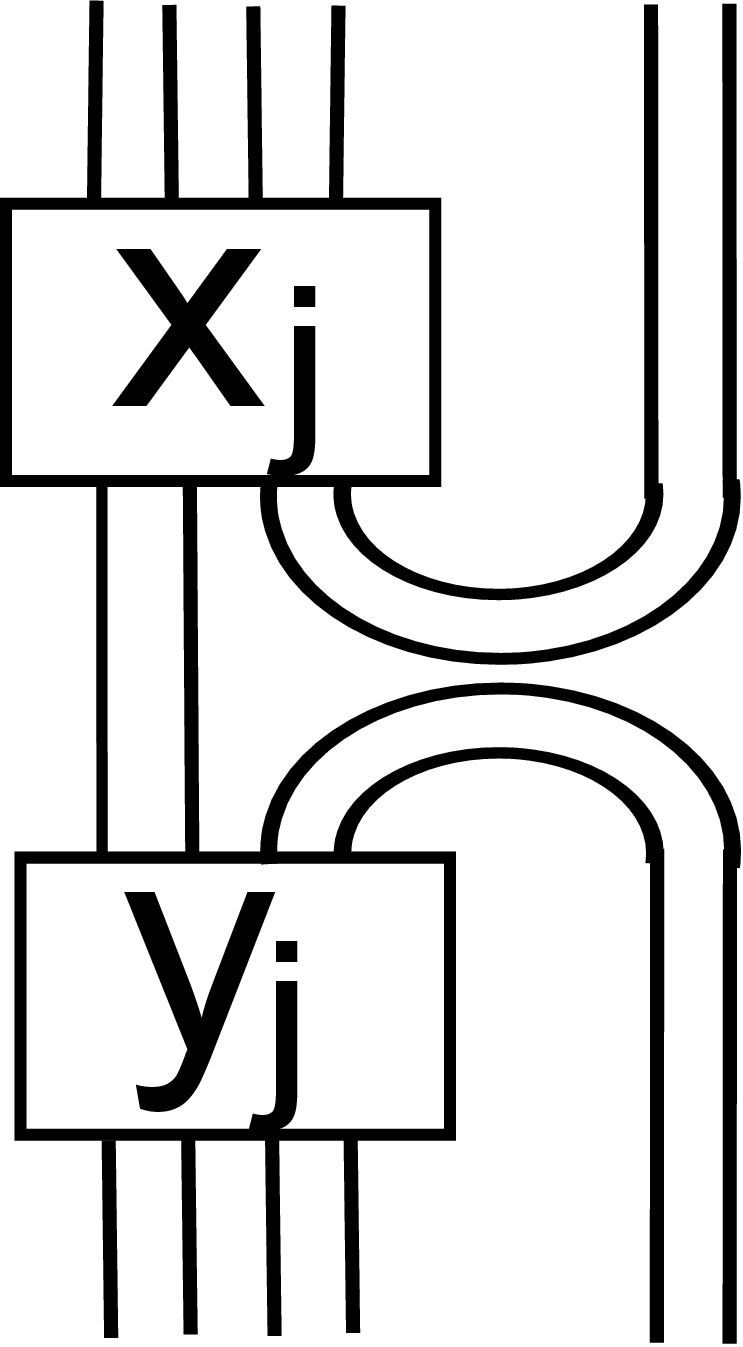},$$
for finitely many $x_j, y_j \in \mathscr{S}_2$.
Moreover
$$(id \otimes e)(\graa{com6}-\graa{com7})(id\otimes e)=\graa{com6}-\graa{com7},$$
so we may assume that $x_j, y_j \preceq id\otimes e$. Then
$$\graa{com6}-\graa{com7}=\sum_j\graa{com8},$$
for finitely many $f_j,g_j\in \mathscr{A}_2$.
Therefore $ac-ca=\sum_j f_j(1\boxdot id)g_j$.

If $\mathscr{A}_{1,3}$ is not abelian, then there is a matrix system $\{u_{11},u_{12},u_{21},u_{22}\}$ in $\mathscr{A}_{1,3}$. By assumption the index of $\mathscr{B}$ is greater than 1, so there is a projection $p\in\mathscr{B}_2$ orthogonal to the Jones Projection. Then $\{\gra{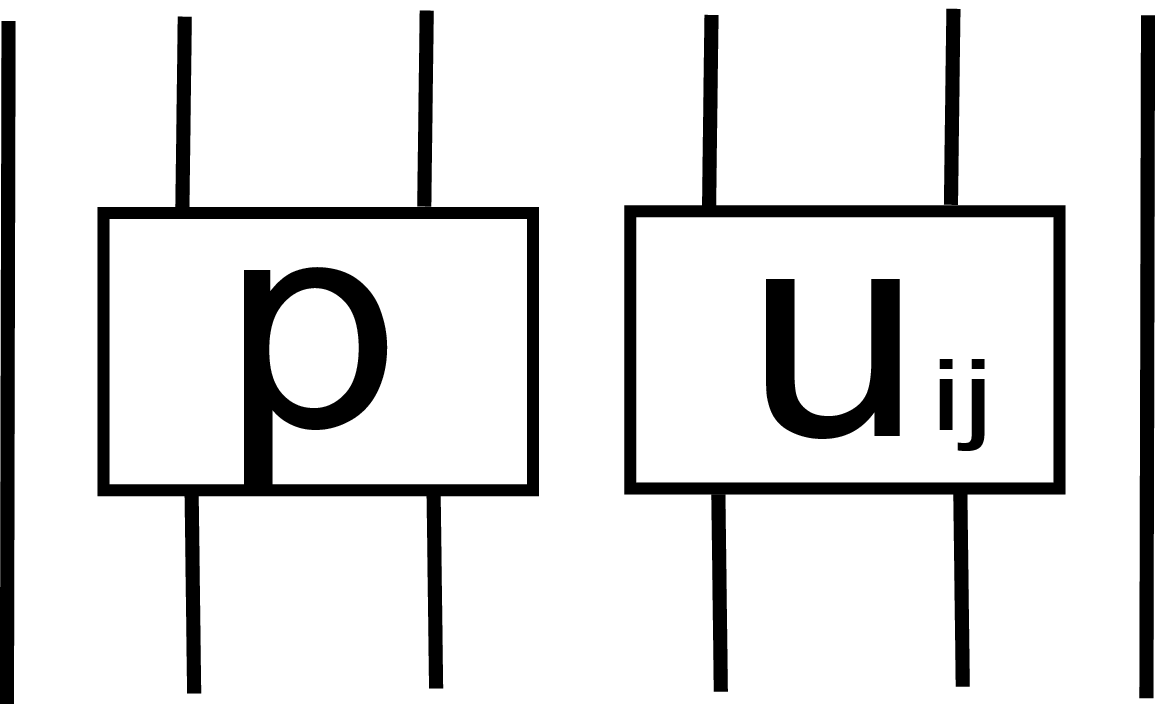}\}_{1\leq i,j\leq 2}$ forms a matrix system in $\mathscr{S}_3$, and they are in the orthogonal complement of $\mathscr{I}_3$. So $\mathscr{S}_3/\mathscr{I}_3$ is not abelian. It is a contradiction.

Therefore $\mathscr{A}$ is a commute relation planar algebra of type NA.
Furthermore if $\mathscr{S}$ is of type AN, then $\mathscr{S}_2$ is abelian. Note that $\mathscr{A}$ is isomorphic to $\mathscr{S}_{id\otimes e}$, so $\mathscr{A}_2$ is abelian. Then $\mathscr{A}$ is of type AA.

Considering the duality, we have  $\mathscr{B}$ is a commute relation planar algebra of type AN. Furthermore if $\mathscr{S}$ is of type NA, then $\mathscr{B}$ is of type AA.
\end{proof}

\section{Classifications}\label{cla expa}
Recall that the classification of subfactor planar algebra generated by a non-trivial 2-box are given by $\mathscr{S}^{\mathbb{Z}_3}$, $TL*TL$, for at most 12 dimensional 3-boxes \cite{BisJon97}; $\mathscr{S}^{\mathbb{Z}_2\subset \mathbb{Z}_5\rtimes\mathbb{Z}_2}$, for 13 dimensional 3-boxes \cite{BisJon02}; BMW's, precisely one family from quantum $Sp(4,\mathbb{R})$ and one from quantum $O(3,\mathbb{R})$, for 14 dimensional 3-boxes \cite{BJL}.
Now let us prove the main classification results.

\begin{theorem}\label{main1}
Suppose $\mathscr{S}$ is an exchange relation planar algebra with $\dim(\mathscr{S}_2)=4$, then $\mathscr{S}$ is one of the follows,

(1)$\mathscr{S}^{\mathbb{Z}_4}$, or $\mathscr{S}^{\mathbb{Z}_2\oplus \mathbb{Z}_2}$;

(2)$\mathscr{A}*TL$ or $TL*\mathscr{A}$, for an exchange relation planar algebra $\mathscr{A}$ with $\dim(\mathscr{A}_2)=3$;

(3)$\mathscr{S}^{\mathbb{Z}_2}\otimes TL$;

(4)$\mathscr{S}^{\mathbb{Z}_2\subset \mathbb{Z}_7\rtimes\mathbb{Z}_2}$.
\end{theorem}

Recall that an exchange relation planar algebra $\mathscr{A}$ with $\dim(\mathscr{A}_2)=3$ is equivalent to say a subfactor planar algebra $\mathscr{A}$ generated by a nontrivial 2-box with $\dim(\mathscr{A}_3)\leq 13$, see the arguments at the end of section \ref{sub ex}.

\begin{proof}
Suppose $\mathscr{S}$ is an exchange relation planar algebra with $\dim(\mathscr{S}_2)=4$. Then $\mathscr{S}_2$ and $\mathscr{S}_{1,3}$ are abelian.
Suppose $e,P_1,P_2,P_3$ are mutually orthogonal minimal projections of $\mathscr{S}_2$, and $e_2, 1\boxdot Q_1, 1\boxdot Q_2, 1\boxdot Q_3$ are mutually orthogonal minimal projections of $\mathscr{S}_{1,3}$.
Then $P_i*Q_j=\lambda_{i,j} Q_j$, for some $\lambda_{i,j}\in\mathbb{C}$.
Let $\mathscr{I}_3$ be the two sided ideal of $\mathscr{S}_3$ generated by the Jones Projection, then $\mathscr{S}_3=\mathscr{I}_3\oplus \mathscr{S}_3/\mathscr{I}_3$. Let us define $s_3$ to be the support of $\mathscr{S}_3/\mathscr{I}_3$.
Then $\{s_3(1\boxdot Q_j)P_i\}_{1\leq i,j\leq 3}$ are pairwise orthogonal. While $\mathscr{S}$ is an exchange relation planar algebra, so $\mathscr{S}_3/\mathscr{I}_3$ is generated by $\{s_3P_i(1\boxdot Q_j)\}_{1\leq i,j\leq 3}$ as a linear space. Then $\dim(s_3(1\boxdot Q_j)\mathscr{S}_3P_i)\leq 1$, and $\dim(s_3(1\boxdot Q_j)\mathscr{S}_3P_i))= 1 \iff s_3P_i(1\boxdot Q_j)\neq 0$.
It is easy to check that $s_3P_i(1\boxdot Q_j)=0 \iff tr(s_3P_i(1\boxdot Q_j))=0 \iff tr((P_i-\frac{\delta}{tr(P_i)}P_i(1\boxdot id)P_i)(1\boxdot Q_j))=0 \iff |\lambda_{i,j}|=\frac{tr(P_i)}{\delta}$.
Note that $P_i$ corresponds to a depth 2 point in the principal graph, denoted by $a_i$; and $1\boxdot Q_j$ corresponds to a depth 2 point in the dual principal graph denoted by $b_j$, and $\dim(\hom(s_31\boxdot Q_j,P_i))$ is the number of length 2 paths from $a_i$ to $b_j$ passing through a depth 3 point of the principal graph in the 4-partite principal graph. Thus $\dim(\hom(s_31\boxdot Q_j,P_i))\leq 1$ implies the number of edges connecting a depth 3 point of the principle graph with $a_i$ (or $b_j$) is at most $1$.

(1) If $\mathscr{S}$ is depth 2, then it is $\mathscr{S}^{\mathbb{Z}_4}$, or $\mathscr{S}^{\mathbb{Z}_2\oplus \mathbb{Z}_2}$; Conversely a depth 2 subfactor planar algebra is an exchange relation planar algebra.
Therefore we obtain class (1), $\mathscr{S}^{\mathbb{Z}_4}$, and $\mathscr{S}^{\mathbb{Z}_2\oplus \mathbb{Z}_2}$.

(2) If $\mathscr{S}$ is separated by a non-trivial biprojection $Q$ as a free product, then by Proposition \ref{exfreedecom}, both $\mathscr{S}^Q$ and $\mathscr{S}_Q$ are exchange relation planar algebras. Note that $\dim{(\mathscr{S}^Q)_2}+\dim{(\mathscr{S}_Q)_2}=\dim(\mathscr{S}_2)+1=5$. Thus one of them is Temperley-Lieb, and the other is an exchange relation planar algebra $\mathscr{A}$ with $\dim(\mathscr{A}_2)=3$. Conversely by Proposition \ref{exfreeprod}, a free product of them is an exchange relation planar algebra. Thus we obtain class (2), the free product of an exchange relation planar algebra $\mathscr{A}$ with $\dim(\mathscr{A}_2)=3$ and $TL$.

(2') If $id-P_i$ is a biprojection for some $1\leq i\leq 3$, then $tr(((id-P_i)*P_i)(id-P_i))=tr(P_i((id-P_i)(id-P_i)))=0$. Thus $(id-P_i)*P_i\sim P_i$. By Theorem \ref{freedecom}, we have $\mathscr{S}$ is separated by $id-P_i$ as a free product.

(3) By Proposition \ref{extensorprod}, the tensor product of $\mathscr{S}^{\mathbb{Z}_2}$ and $TL$ is an exchange relation planar algebra $\mathscr{S}$ with $\dim(\mathscr{S}_2)=2\times 2=4$.

(4) By Theorem \ref{ex2p2}, The subgroup subfactor planar algebra $\mathscr{S}^{\mathbb{Z}_2\subset \mathbb{Z}_7\rtimes\mathbb{Z}_2}$ is an exchange relation planar algebra $\mathscr{S}$ with $\dim(\mathscr{S}_2)=\frac{7+1}{2}=4$.

By Proposition \ref{dimpn}, we have $\dim(\mathscr{S}_3/\mathscr{I}_3)\leq 9$.
We need to consider the following four cases:

(a) $\mathscr{S}_3/\mathscr{I}_3$ is abelian;

(b) $\mathscr{S}_3/\mathscr{I}_3$ is a direct sum of $M_{2\times2}$ and $\mathbb{C}$'s.

(c) $\mathscr{S}_3/\mathscr{I}_3$ contains $M_{2\times2}\otimes M_{2\times2}$;

(d) $\mathscr{S}_3/\mathscr{I}_3=M_{3\times3}$;

Case (a): If $\mathscr{S}_3/\mathscr{I}_3$ is abelian, then in the principal graph, each depth 3 point only connects with one depth 2 point. Thus $\dim(P_i\mathscr{S}_3P_j))=1$, whenever $1\leq i,j \leq 3$ and $i\neq j$. By Lemma \ref{copro connect}, we have $r(P_j'*P_i)=1$. Thus either $tr(P_i)=1$, or $P_i'$ is a virtual normalizer. If $tr(P_i)=1$, for $i=1,2,3$, then $\mathscr{S}_2$ is depth 2 which is in class (1). Otherwise by Theorem \ref{virtual normalizer}, $\mathscr{S}$ is separated by a non-trivial projection as a free product. Then it is in class (2).

Case (b): If $\mathscr{S}_3/\mathscr{I}_3$ is a direct sum of $M_{2\times2}$ and $\mathbb{C}$'s, then there is one depth 3 point $v$ with multiplicity 2 in the principal graph. Without loss of generality, we assume $v$ is connected with $a_2$, $a_3$. Then in the principal graph, there is only one length 2 path between $a_1$ and $a_l$, for $l=2,3$. By Lemma \ref{copro connect}, either $P_1'$ is a left virtual normalizer or $tr(P_1)=1$.
If $P_1'$ is a left virtual normalizer, then by Theorem \ref{virtual normalizer}, we have $\mathscr{S}$ in class (2). Otherwise $tr(P_1)=1$. Note that $tr(P_2)>1$ and $tr(P_3)>1$. Thus $P_1=P_1'$ and $e+P_1$ is a biprojection.
Note that the dual of $\mathscr{S}$ satisfies the same assumption, so it contains either a left virtual normalizer or a trace-2 biprojection. In the former case, by Theorem \ref{virtual normalizer}, we have $\mathscr{S}$ is in class (2). In the latter case, there is a biprojection in $\mathscr{S}$ with trace $\frac{\delta^2}{2}$. If its rank is $3$, then $\mathscr{S}$ is in class (2), by the discussion in (2'). If its rank is 2, then it is $e+P_m$, for some $1\leq m\leq 3$. If $m=1$, then $\delta^2=4$. Thus $\mathscr{S}$ is in class (1). Otherwise, without loss of generality, we assume that $m=2$. That means $e+P_2$ is a biprojection, and $tr(e+P_2)=\frac{\delta^2}{2}$. Then $P_2=P_2'$, and $tr((P_1*P_2)P_2)=tr(P_1(P_2*P_2))=0$. So $P_1*P_2=\frac{1}{\delta}P_3$. Then $(e+P_1)*(e+P_2)=\frac{1}{\delta}id$. By Theorem \ref{tensordecom222}, we have $\mathscr{S}$ is separated by $e+P_1$ and $e+P_2$ as a tensor product. Moreover $\mathscr{S}_{e+P_1}$ is $\mathscr{S}^{\mathbb{Z}_2}$, because $tr(e+P_1)=2$, and $\mathscr{S}_{e+P_2}$ is Temperley-Lieb. So $\mathscr{S}$ is in class (3).

Case (c): If $\mathscr{S}_3/\mathscr{I}_3$ contains $M_{2\times2}\otimes M_{2\times2}$, then there are two depth 3 points with multiplicity 2 in the principal graph. Thus there is a point $a_i$ connects with both of them. Moreover there is a point $b_j$ connects with both of them in the 4-partite principal graph. Then there are two length 2 paths from $a_i$ to $b_j$ passing through a depth 3 point of the principal graph. It is a contradiction.

Case (d): If $\mathscr{S}_3/\mathscr{I}_3=M_{3\times3}$, then $\dim(\mathscr{S}_3/\mathscr{I}_3)=9$. Thus $s_3P_i(1\boxdot Q_j)\neq0$, for any $1\leq i,j \leq 3$. Then $|\lambda_{i,j}|\neq \frac{tr(P_i)}{\delta}$.
By Theorem \ref{maxcoeff}, the biprojection generated by $P_i$ is $id$. In the principal graph, there is only one depth 3 point, and it connects with each depth 2 point with one edge. Thus $tr(P_1)=tr(P_2)=tr(P_3)>1$, and $\dim(P_i\mathscr{S}_3 P_j)\leq 2$, $\forall ~1\leq i,j\leq 3$. Take $c=tr(P_1)$. By Lemma \ref{copro connect}, we have $r(P_i*P_j)\leq 2$, $\forall ~1\leq i,j\leq 3$. Note that among the three projections, at least one is self contragredient, we assume that $P_1=P_1'$. Then $P_1*P_1=\frac{c}{\delta}e+\frac{c-1}{\delta}P_k$ for some $1\leq k \leq 3$. By Theorem \ref{P*P=P}, we have $k\neq 1$, otherwise the biprojection generated by $P_1$ is $e+P_1$. Without loss of generality, we assume that $k=2$. Then $P_2=P_2'$, and $tr((P_1*P_2)P_1)=tr(P_2(P_1*P_1))=\frac{c(c-1)}{\delta}$. Thus $P_1*P_2=\frac{c-1}{\delta}P_1+\frac{1}{\delta}P_l$, for some $2\leq l\leq 3$. By Theorem \ref{P*P=P}, we have $l\neq 2$, otherwise the biprojection generated by $P_1$ is $e+P_1+P_2$. So $P_1*P_2=\frac{c-1}{\delta}P_1+\frac{1}{\delta}P_3$. Then $P_1*P_3=\frac{1}{\delta}P_2+\frac{c-1}{\delta}P_3$. Moreover $P_2*P_2=\frac{c}{\delta}e+\frac{c-1}{\delta}P_3$, otherwise $e+P_1+P_2$ is a biprojection. Then $P_2*P_3=\frac{1}{\delta}P_1+\frac{c-1}{\delta}P_3$ and $P_3*P_3=\frac{c}{\delta}e+\frac{c-1}{\delta}P_1$.
Therefore $P_1*(P_1*P_2)=\frac{c(c-1)}{\delta^2}e+\frac{(c-1)^2}{\delta^2}P_2+\frac{1}{\delta^2}P_2+\frac{c-1}{\delta^2}P_3$, and
$(P_1*P_1)*P_2=\frac{c}{\delta^2}P_2+\frac{c(c-1)}{\delta^2}e+\frac{(c-1)^2}{\delta^2}P_3$. Comparing the coefficient of $P_3$, we obtain $c=2$.
Then by Theorem \ref{p2ex}, we have $\mathscr{S}$ is isomorphic to $\mathscr{S}^{\mathbb{Z}_2\subset \mathbb{Z}_7\rtimes\mathbb{Z}_2}$.
\end{proof}

\begin{theorem}
Suppose $\mathscr{S}$ is a subfactor planar algebra generated by 2-boxes with $\dim(\mathscr{S}_{2,\pm})=4$ and $\dim(\mathscr{S}_{3,\pm})\leq 23$, then
then $\mathscr{S}$ is one of the follows

(1) $\mathscr{S}^{\mathbb{Z}_4}$ or $\mathscr{S}^{\mathbb{Z}_2\oplus \mathbb{Z}_2}$;

(2a) $\mathscr{A}*TL$ or $TL*\mathscr{A}$, where $\mathscr{A}$ is generated by a non-trivial 2-box with $\dim(\mathscr{A}_{3,\pm})\leq 13$;

(2b) $\mathscr{B}*\mathscr{S}^{\mathbb{Z}_2}$ or $\mathscr{S}^{\mathbb{Z}_2}*\mathscr{B}$, where $\mathscr{B}$ is generated by a non-trivial 2-box with $\dim(\mathscr{A}_{3,\pm})\leq 14$;

(3) $\mathscr{S}^{\mathbb{Z}_2}\otimes TL$.
\end{theorem}

\begin{proof}
Suppose $\mathscr{S}$ is a subfactor planar algebra generated by 2-boxes with $\dim(\mathscr{S}_{2,\pm})=4$ and $\dim(\mathscr{S}_{3,\pm})\leq 23$, and $\mathscr{I}_3$ is the two sided ideal of $\mathscr{S}_3$ generated by the Jones Projection, then $\dim(\mathscr{S}_3/\mathscr{I}_3)\leq 23-4^2=7$. Thus we only need to consider case (a) and (b) in the proof of Theorem \ref{main1}. Note that we never assumed that $\mathscr{S}$ is an exchange relation planar algebras in the arguments of these two cases, so either $\mathscr{S}$ is one of $\mathscr{S}^{\mathbb{Z}_4}$, $\mathscr{S}^{\mathbb{Z}_2\oplus \mathbb{Z}_2}$ and $\mathscr{S}^{\mathbb{Z}_2}\otimes TL$, corresponding to class (1) and (3) in the statement; or $\mathscr{S}$ is separated by a non-trivial biprojection $Q$ as a free product. In the latter case, both $\mathscr{S}^Q$ and $\mathscr{S}_Q$ are generated by 2-boxes.  Counting the dimensions of 2-boxes, we have $\dim{(\mathscr{S}^Q)_2}-1+\dim((\mathscr{S}_Q)_2)-1=\dim(\mathscr{S}_2)-1=3$. Thus one of them is $TL$, and the other is a subfactor planar algebra generated by 2-boxes with 3 dimensional 2-boxes. Furthermore counting the dimensions of 3-boxes, we have $\dim((\mathscr{S}^Q)_3)+\dim((\mathscr{S}_Q)_3)=3^2+2^2-(3-1)(2-1)+\dim(\mathscr{S}_2)-4^2\leq 18$. Thus either the Templey-Lieb one is $\mathscr{S}^{\mathbb{Z}_2}$ and the other has at most 14 dimensional 3-boxes, or the other one has at most 13 dimensional 3-boxes. They correspond to class (2a) and (2b) in the statement.
\end{proof}

\begin{theorem}\label{clas comAA}
Suppose $\mathscr{S}$ is a subfactor planar algebra generated by 2-boxes, such that $\mathscr{S}_3/\mathscr{I}_3$, $\mathscr{S}_2$ and $\mathscr{S}_{1,3}$ are abelian, then it is a free product of Temperley-Lieb subfactor planar algebras and group subfactor planar algebras for abelian groups. Verse Visa.
\end{theorem}

\begin{proof}
By assumption $\mathscr{S}_3/\mathscr{I}_3$ is abelian, so the multiplicity of a depth 3 point in the principal graph is $1$.
Furthermore $\mathscr{S}_2$ is abelian, so the multiplicity of a depth 2 point in the principal graph is $1$.
Then for two distinct depth 2 points, there is only one length 2 path between them.
Thus $\dim(P_i\mathscr{S}_3P_j)=1$ in $\mathscr{S}_3$, for any two distinct minimal projections $P_i,P_j$ of $\mathscr{P_2}$. By Lemma \ref{copro connect}, we have $r(P_j'*P_i)=1$. Thus either $tr(P_i)=1$, or $P_i'$ is a virtual normalizer.
If $tr(P_i)=1$, for any minimal projection $P_i$ in $\mathscr{S}_2$, then $\mathscr{S}$ is depth 2. By assumption it is of type AA, so it is a group subfactor planar algebra for some abelian group.
Otherwise $\mathscr{S}$ contains a virtual normalizer. By Theorem \ref{virtual normalizer}, either $\mathscr{S}$ is Temperley-Lieb or $\mathscr{S}$ is separated by a non-trivial biprojection as a free product $\mathscr{A}*\mathscr{B}$. In the latter case, both $\mathscr{A}$ and $\mathscr{B}$ have smaller index. By Lemma \ref{comfreedecom}, both $\mathscr{A}$ and $\mathscr{B}$ are commute relation planar algebras of type AA. Then we may decompose them again until they are either depth 2 or Temperley-Lieb. Therefore $\mathscr{S}$ is a free product of Temperley-Lieb subfactor planar algebras and depth 2 subfactor planar algebras, and each of them is of type AA.

Conversely both Temperley-Lieb subfactor planar algebras and group subfactor planar algebras for abelian groups are commute relation planar algebras of type AA. By Lemma \ref{comfreepro}, their free product is a commute relation planar algebra of type AA.
\end{proof}

\begin{theorem}\label{clas com}
Suppose $\mathscr{S}$ is a subfactor planar algebra generated by 2-boxes, and $\mathscr{S}_3/\mathscr{I}_3$ is abelian,
then $\mathscr{S}$ is either depth 2 or the free product
$\mathscr{A}_1*\mathscr{A}_2*\cdots *\mathscr{A}_n$, such that
$\mathscr{A}_1$ is Temperley-Lieb or the dual of $\mathscr{S}^{G_1}$, for a group $G_1$;
$\mathscr{A}_n$ is Temperley-Lieb or $\mathscr{S}^{G_n}$, for a group $G_n$;
$\mathscr{A}_m$, for $1<m<n$, is Temperley-Lieb or $\mathscr{S}^{G_m}$, for an abelian group $G_m$.
Verse Visa.
\end{theorem}

In this general case, we still want to show the existence of a virtual normalizer in a commute relation planar algebra, whenever it is not depth 2. Then we may decompose a commute relation planar algebra as a free product of commute relation planar algebras, until they are either Temperley-Lieb or depth 2.

\begin{notation}
Recall that $\mathscr{S}_2\cap\mathscr{I}_3$ is a two sided ideal of $\mathscr{S}_2$. Let us define $\mathscr{S}_2/\mathscr{I}_3$ to be the orthogonal complement of $\mathscr{S}_2\cap\mathscr{I}_3$ in $\mathscr{S}_2$.
Then $\mathscr{S}_2=(\mathscr{S}_2\cap\mathscr{I}_3) \oplus \mathscr{S}_2/\mathscr{I}_3$.
\end{notation}

\begin{lemma}\label{113}
Suppose $\mathscr{S}$ is a commute relation planar algebra. If $\mathscr{S}$ is not depth 2,
then each minimal projection $P_i$ in $\mathscr{S}_2/\mathscr{I}_3$ is a virtual normalizer.
\end{lemma}

Based on Lemma \ref{113}, the proof of Theorem \ref{clas com} is similar to the proof of Theorem \ref{clas comAA}.

\begin{proof}[Proof of Theorem \ref{clas com}]
Suppose $\mathscr{S}$ is a commute relation planar algebra. If it is not depth 2, then by Lemma \ref{113}, it contains a virtual normalizer. If $\mathscr{S}$ is not Temperley-Lieb, then by Lemma \ref{virtual normalizer} and Proposition \ref{comfreedecom}, we have $\mathscr{S}$ is a free product of two commute relation planar algebras with smaller index. Repeating this process, we have
$\mathscr{S}=\mathscr{A}_1*\mathscr{A}_2*\cdots *\mathscr{A}_n$, such that each $\mathscr{A}_i$ is either Temperley-Lieb or depth 2. By Lemma \ref{comfreedecom}, we have $\mathscr{A}_1$ is of type NA; $\mathscr{A}_n$ is of type AN; the others are of type AA.

Conversely by Lemma \ref{comfreepro}, their free product is a commute relation planar algebra.
\end{proof}

To prove Lemma \ref{113}, let us prove some basic results first.

\begin{lemma}\label{111}
Suppose $\mathscr{S}$ is a commute relations planar algebra. If $P_i,P_j$ are distinct minimal projections in $\mathscr{S}_2/\mathscr{I}_3$. Then $r(P_i'*P_j)=1$.
\end{lemma}

\begin{proof}
Suppose $P_i,P_j$ are distinct minimal projections in $\mathscr{S}_2/\mathscr{I}_3$. And $v_i, v_j$ are the corresponding depth 2 points in the principal graph.
By assumption $\mathscr{S}_3/\mathscr{I}_3$ is abelian, so $P_i,P_j$ are central in $\mathscr{S}_2$. Then the multiplicity of $v_i,v_j$ are $1$. Note that the multiplicity of a depth 3 point is $1$, so there is only one length 2 path between $v_i$ and $v_j$. By Lemma \ref{copro connect}, we have $r(P_i'*P_j)=1$.
\end{proof}

We want to show that $r(P_i*P_j)=1$, whenever $P_i$ is a minimal projection in $\mathscr{S}_2/\mathscr{I}_3$ and $P_j$ is a minimal projection in $\mathscr{S}_2\cap\mathscr{I}_3$. If $tr(P_j)=1$, then $r(P_i*P_j)=1$ is a minimal. If $tr(P_j)>1$, we will see $P_i*P_j\sim P_i$.

\begin{lemma}\label{112}
Suppose $P_i,P_j,P_k$ are minimal projections of $\mathscr{S}_2$, $P_i\in \mathscr{S}_2/\mathscr{I}_3$, and $P_j\in \mathscr{S}_2\cap\mathscr{I}_3$.
If $P_k\preceq P_i*P_j$, then $P_k\in \mathscr{S}_2/\mathscr{I}_3$.
\end{lemma}

\begin{proof}
If $P_k\in \mathscr{S}_2\cap\mathscr{I}_3$, and $P_k\preceq P_i*P_j$, then $tr(P_i(P_k*P_j'))=tr((P_i*P_j)P_k)>0$. But $P_j,P_k\in \mathscr{S}_2\cap\mathscr{I}_3$, by Theorem \ref{P bipro}, we have $P_k*P_j' \in \mathscr{S}_2\cap \mathscr{I}_3$. So $P_i(P_k*P_j')=0$. It is a contradiction.
\end{proof}

\begin{lemma}\label{Tijk}
Suppose $P_i,P_j,P_k$ are minimal projections of $\mathscr{S}_2$, $P_i,P_k\in \mathscr{S}_2/\mathscr{I}_3$, and $P_i\neq P_k$.
If $tr(P_j(P_i'*P_k))>0$, then $tr(P_j(P_i'*P_k))=\frac{tr(P_i)tr(P_k)}{\delta}$,
which implies $P_j\in \mathscr{S}_2/\mathscr{I}_3$ or $tr(P_j)=1$.
\end{lemma}

\begin{proof}
If $tr(P_j(P_i'*P_k))>0$, then $tr((P_i*P_j)P_k)>0$.
Thus $(P_i*P_j)P_k=\lambda P_k$, for some $\lambda>0$.
Note that $\mathscr{I}_3$ is the two sided ideal of $\mathscr{S}_3$, take $s_3$ to be the support of $\mathscr{S}_3$.
Then $s_3P_i$ is in the center of $\mathscr{S}_3$, since $\mathscr{S}_3/\mathscr{I}_3$ is abelian.
So $(s_3P_i)(1\boxdot P_j)P_{k}
=(1\boxdot P_j)(s_3P_i)P_{k}=0$.
By Wenzl's formula, we have $s_3P_i=P_i-\frac{\delta}{tr(P_i)}P_i(1\boxdot id)P_i.$
So $P_i(1\boxdot P_j)P_{k}=\frac{\delta}{tr(P_i)}P_i(1\boxdot id)P_i(1\boxdot P_j)P_{k}
=\frac{\delta}{tr(P_i)}P_i(1\boxdot id)(P_i*P_j)P_{k}=\frac{\delta}{tr(P_i)}\lambda P_ie_2(1\boxdot id)P_{k}$, i.e.,
$$\graa{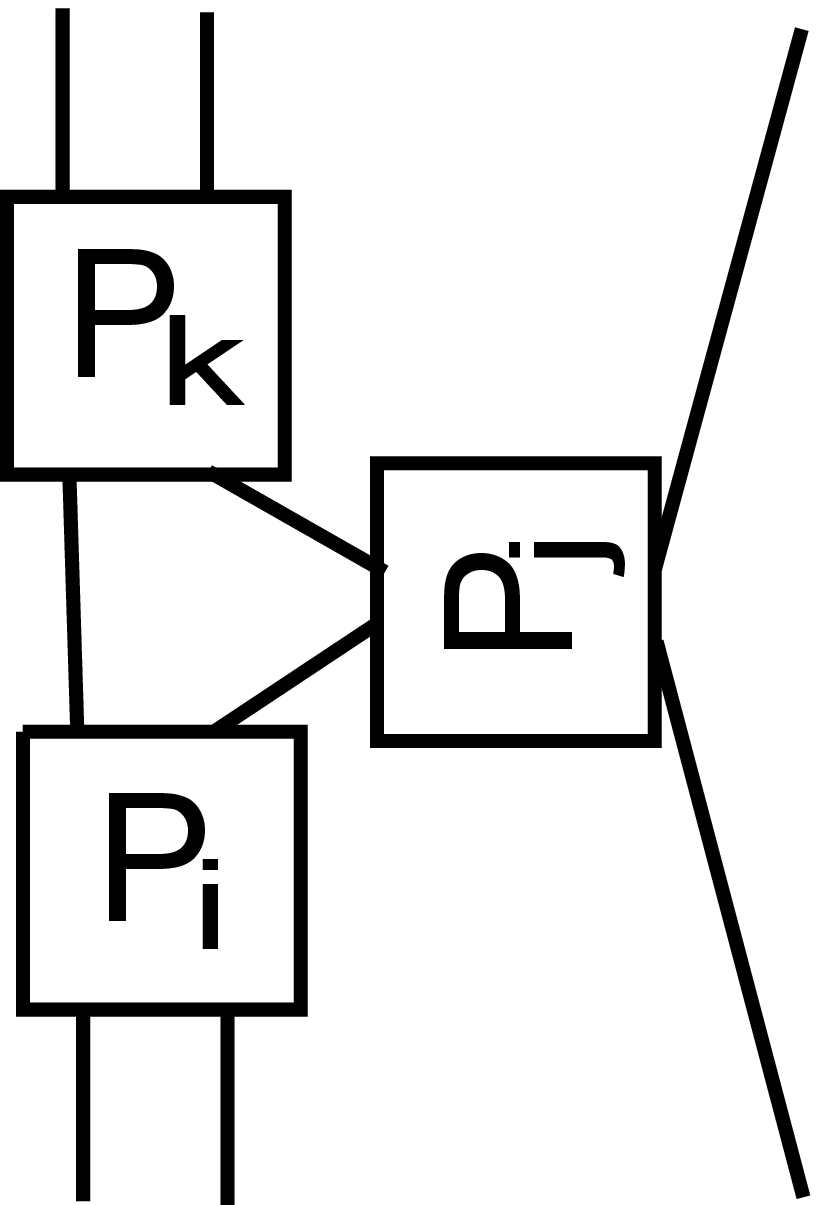}=\frac{\delta}{tr(P_i)}\lambda \graa{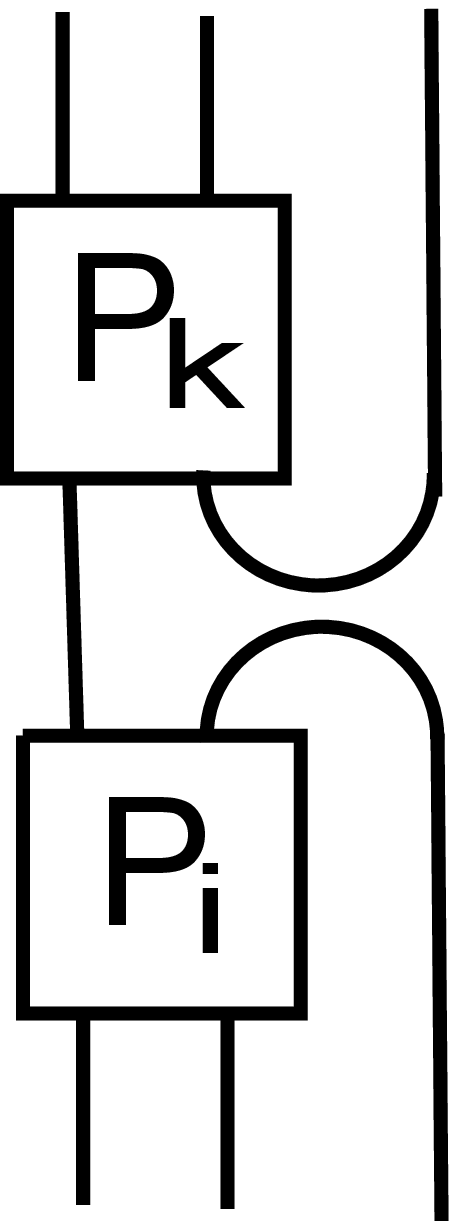}.$$
Then $P_i(1\boxdot P_j)P_{k}\neq0$.
Multiplying $P_j$ on the right side, (the 3,4 position,) we have
$P_i(1\boxdot P_j)P_{k}=\frac{\delta}{tr(P_i)}\lambda P_i(1\boxdot P_j)P_{k}$.
Thus $\lambda=\frac{tr(P_i)}{\delta}$, and $tr(P_j(P_i'*P_k))=tr((P_i*P_j)P_k)=tr(\lambda P_k)=\frac{tr(P_i)tr(P_k)}{\delta}$.
If $P_l$ is a minimal projection in $\mathscr{S}_2$, such that $\|P_l-P_j\|<1$, then $tr(P_l(P_i'*P_k))>0$. Thus $tr(P_l(P_i'*P_k))=\frac{tr(P_i)tr(P_k)}{\delta}$.
So $P_i'*P_k$ is central.
By assumption $P_i,P_k\in \mathscr{S}_2/\mathscr{I}_3$ and $P_i\neq P_k$, so $r(P_i'*P_k)=1$, by Lemma \ref{111}. Thus $P_j\sim P_i'*P_k$, and $P_j$ is central. Then $P_j\in \mathscr{S}_2/\mathscr{I}_3$, or $tr(P_j)=1$.
\end{proof}

\begin{proof}[Proof of Lemma \ref{113}]
Suppose $\mathscr{S}$ is not depth 2, $P_i$ is a minimal projection in $\mathscr{S}_2/\mathscr{I}_3$, $P_j$ is a minimal projection in $\mathscr{S}_2$, and $P_j\neq P_i'$.

If $P_j\in \mathscr{S}_2/\mathscr{I}_3$, then $r(P_i*P_j)=1$, by Lemma \ref{111}.

If $tr(P_j)=1$, then $r(P_i*P_j)=1$.

Otherwise $P_j\in \mathscr{S}_2\cap \mathscr{I}_3$ and $tr(P_j)>1$.
By Lemma \ref{112}, if $P_k\preceq P_i*P_j$, then $P_k\in\mathscr{S}_2/\mathscr{I}_3$, and $tr(P_j(P_i'*P_k))=tr((P_i*P_j)P_k)>0$.
Furthermore by Lemma \ref{Tijk}, if $P_k\neq P_i$, then $P_j\in \mathscr{S}_2/\mathscr{I}_3$ or $tr(P_j)=1$. It is a contradiction.
So $P_k=P_i$. Then $P_i*P_j\sim P_i$.

Therefore $P_i$ is a left virtual normalizer.
By Theorem \ref{P bipro}, $P_i$ is a minimal projection in $\mathscr{S}_2/\mathscr{I}_3$. So it is a left virtual normalizer.
So $P_i$ is a virtual normalizer.
\end{proof}

\begin{definition}
A subfactor (or a subfactor planar algebra) is said to be k-supertransitive, if its principal graph is the Dynkin diagram $A_{k+1}$ up to depth $k$.
\end{definition}

From a subfactor perspective, we have the following weak version of Theorem \ref{clas com}.

\begin{theorem}
Suppose $\mathcal{N}\subset \mathcal{M}$ is a finite index irreducible subfactor,
such that the quotient of $\mathcal{N}'\cap \mathcal{M}_2$ by the basic construction ideal $(\mathcal{N}'\cap \mathcal{M}_2) e_2 (\mathcal{N}'\cap \mathcal{M}_2)$ is abelian, where $\mathcal{N}\subset \mathcal{M}\subset \mathcal{M}_1\subset \mathcal{M}_2\subset\cdots$ is the Jones tower and $e_2$ is the Jones Projection onto $L^2(\mathcal{M})$. Then either $\mathcal{N}\subset \mathcal{M}$ is depth 2 or
there exists a sequence of intermediate subfactors $\mathcal{N}\subset \mathcal{R}_1\subset \cdots\subset \mathcal{R}_n\subset \mathcal{M}$,
such that

(1) either $\mathcal{N}\subset \mathcal{R}_1$ is 2-supertransitive or $\mathcal{R}_1=\mathcal{N}\rtimes G$ for an outer action of a group $G$;

(2) either $\mathcal{R}_n\subset \mathcal{M}$ is 2-supertransitive or $\mathcal{R}_n=\mathcal{M}^H$ for an outer action of a group $H$, where $\mathcal{M}^H$ is the fix point algebra under the action of $H$;

(3) either $\mathscr{R}_i\subset \mathscr{R}_{i+1}$ is 2-supertransitive or $\mathcal{R}_{i+1}=\mathcal{R}_i\rtimes A$ for an outer action of an abelian group $A$, for $1\leq i\leq n-1$.

Furthermore any intermediate subfactor of $\mathcal{N}\subset \mathcal{M}$ is either one of the sequence or an intermediate subfactor of some adjacent pair of the sequence.
\end{theorem}

\begin{proof}
Suppose $\mathscr{F}$ is the planar algebra of $\mathcal{N}\subset \mathcal{M}$, and $\mathscr{S}$ is the planar subalgebra of $\mathscr{F}_2$ generated by 2-boxes.
Since $\mathscr{S}_2=\mathscr{F}_2$, the two sided ideal $\mathscr{I}_3$ generated by the Jones projection of $\mathscr{F}_3$ is also the two sided ideal of $\mathscr{S}_3$.
Note that $(\mathcal{N}'\cap \mathcal{M}_2)/((\mathcal{N}'\cap \mathcal{M}_2) e_2 (\mathcal{N}'\cap \mathcal{M}_2))$ is abelian means $\mathscr{F}_3/\mathscr{I}_3$ is abelian. Then $\mathscr{S}_3/\mathscr{I}_3$ is abelian. So $\mathscr{S}$ is a commute relation planar algebra.
Recall that intermediate subfactors correspond to biprojections in $\mathscr{F}_2=\mathscr{S}_2$. By Theorem \ref{clas com}, we have $\mathscr{S}$ is either depth 2 or a free product of Temperley-Lieb subfactor planar algebras and depth 2 subfactor planar algebras. Thus either $\mathcal{N}\subset \mathcal{M}$ is depth 2 or there exists a sequence of intermediate subfactors $\mathcal{N}\subset \mathcal{R}_1\subset \cdots\subset \mathcal{R}_n\subset \mathcal{M}$ correspond to the sequence of biprojections $P_1,P_2,\cdots, P_n$ which separate $\mathscr{S}$ as a free product, such that $\mathscr{S}_{P_{i+1}}^{P_i}$ is either Temperley-Lieb or depth 2, for $1\leq i \leq n-1$.
Note that $(\mathscr{S}_{P_{i+1}}^{P_i})_2=(\mathscr{F}_{P_{i+1}}^{P_i})_2$, so $(\mathscr{F}_{P_{i+1}}^{P_i})$ is either 2-supertransitive or depth 2. Moreover if $\mathcal{N}\subset \mathcal{R}_1$ is depth 2, then its planar algebra is of type NA, thus $\mathcal{R}_1=\mathcal{N}\rtimes G$, for an outer action of a group $G$; If $\mathcal{R}_n\subset \mathcal{M}$ is depth 2, then its planar algebra is of type AN, thus $\mathcal{R}_n=\mathcal{M}^H$, for an outer action of a finite group $H$; If $\mathscr{R}_i\subset \mathscr{R}_{i+1}$ is depth 2, for some $1\leq i\leq n-1$, then its planar algebra is of type AA, thus $\mathcal{R}_{i+1}=\mathcal{R}_i\rtimes A$, for an outer action of an abelian group $A$.

Furthermore by Theorem \ref{inter free}, any intermediate subfactor of $\mathcal{N}\subset \mathcal{M}$ is either one of the sequence or an intermediate subfactor of some adjacent pair of the sequence.
\end{proof}

\bibliography{bibliography}

\providecommand{\bysame}{\leavevmode\hbox to3em{\hrulefill}\thinspace}
\providecommand{\MR}{\relax\ifhmode\unskip\space\fi MR }
\providecommand{\MRhref}[2]{%
  \href{http://www.ams.org/mathscinet-getitem?mr=#1}{#2}
}
\providecommand{\href}[2]{#2}
\begin{thebibliography}{GdlHJ89}

\bibitem[AH99]{AsaHaa}
M.~Asaeda and U.~Haagerup, \emph{Exotic subfactors of finite depth with jones
  indices $(5+\sqrt{13})/2$ and $(5+\sqrt{17})/2$}, Commun. Math. Phys.
  \textbf{202} (1999), 1--63.

\bibitem[Bis94]{Bis94}
D.~Bisch, \emph{A note on intermediate subfactors}, Pacific J. Math.
  \textbf{163} (1994), 201--216.

\bibitem[Bis97]{Bis97}
\bysame, \emph{Bimodules, higher relative commutants and the fusion algebra
  associated to a subfactor}, Operator algebras and their applications
  (Waterloo, ON, 1994/1995), 13-63, Fields Inst. Commun., 13, Amer. Math. Soc.,
  Providence, RI, 1997.

\bibitem[Bis98]{Bis98}
\bysame, \emph{Principal graphs of subfactors with small {J}ones index}, Math.
  Ann. \textbf{311} (1998), no.~2, 223--231.

\bibitem[BJ]{BisJonfree}
D.~Bisch and V.~F.~R. Jones, \emph{The free product of planar algebras, and
  subfactors}, unpublished.

\bibitem[BJ97a]{BisJonFC}
\bysame, \emph{Algebras associated to intermediate subfactors}, Invent. Math.
  \textbf{128} (1997), 89--157.

\bibitem[BJ97b]{BisJon97}
\bysame, \emph{Singly generated planar algebras of small dimension}, Duke Math.
  J. \textbf{128} (1997), 89--157.

\bibitem[BJ03]{BisJon02}
\bysame, \emph{Singly generated planar algebras of small dimension, part {II}},
  Advances in Mathematics \textbf{175} (2003), 297--318.

\bibitem[BJL]{BJL}
D.~Bisch, V.~Jones, and Z.~Liu, \emph{Singly generated planar algebras of small
  dimension, part {III}}.

\bibitem[BL]{BhaLan}
B.~Bhattacharyya and Z.~Landau, \emph{Intermediate standard invariants and
  intermediate planar algebras}, Submitted to Journal of Functional Analysis.

\bibitem[BMPS12]{BMPS}
S.~Bigelow, S.~Morrison, E.~Peters, and N.~Snyder, \emph{Constructing the
  extended {H}aagerup planar algebra}, Acta Math. (2012), 29--82.

\bibitem[BW89]{BirWen}
J.~Birman and H.~Wenzl, \emph{Braids, link polynomials and a new algebra},
  Trans. AMS \textbf{313(1)} (1989), 249--273.

\bibitem[GdlHJ89]{GHJ}
F.~Goodman, P.~de~la Harpe, and V.F.R. Jones, \emph{Coxeter graphs and towers
  of algebras}, vol.~14, Springer-Verlag, MSRI publications, 1989.

\bibitem[GJS10]{GJS}
A.~Guionnet, V.~F.~R. Jones, and D.~Shlyakhtenko, \emph{Random matrices, free
  probability, planar algebras and subfactors}, A Quanta of Maths:
  Non-commutative Geometry Conference in Honor of Alain Connes, Clay Math.
  Proc. \textbf{11} (2010), 201--240.

\bibitem[Haa94]{Haa94}
U.~Haagerup, \emph{Principal graphs of subfactors in the index range $4 <
  [{M}\colon {N}] < 3+\sqrt{2}$}, Subfactors (Kyuzeso,1993), World Sci. Publ.,
  River Edge, NJ, 1994, pp.~1--38.

\bibitem[IJMS12]{ind53}
M.~Izumi, V.~F.~R. Jones, S.~Morrison, and N.~Snyder, \emph{Subfactors of index
  less than 5, part 3: Quadruple points}, Comm. Math. Phys. \textbf{316(2)}
  (2012), 531--554.

\bibitem[Izu91]{Izu91}
M.~Izumi, \emph{Applications of fusion rules to classification of subfactors},
  Publ. RIMS, Kyoto Univ. \textbf{27} (1991), 953--994.

\bibitem[JMSa]{JMN}
V.~F.~R. Jones, S.~Morrison, and N.~Snyder, \emph{The classification of
  subfactors of index at most 5}, arXiv:1304.6141v2.

\bibitem[JMSb]{ind50}
\bysame, \emph{The classification of subfactors of index at most 5.},
  arXiv:1304.6141v2.

\bibitem[Jon]{JonPA}
V.~F.~R. Jones, \emph{Planar algebras, {I}}, arXiv:math.QA/9909027.

\bibitem[Jon83]{Jon83}
\bysame, \emph{Index for subfactors}, Invent. Math. \textbf{72} (1983), 1--25.

\bibitem[Jon00]{Jon00}
\bysame, \emph{The planar algebra of a bipartite graph}, Knots in {H}ellas '98
  ({D}elphi), 94-117, Ser. Knots Everything, 24, World Sci. Publ., River Edge,
  NJ, 2000.

\bibitem[Jon12]{Jon12}
\bysame, \emph{Quadratic tangles in planar algebras}, Duke Math. J.
  \textbf{161} (2012), 2257--2295.

\bibitem[JP11]{JonPen}
V.~F.~R. Jones and D.~Penneys, \emph{The embedding theorem for finite depth
  subfactor planar algebras}, Quantum Topol. \textbf{2} (2011), 301--337.

\bibitem[JS97]{JS}
V.~F.~R. Jones and V.S. Sunder, \emph{Introduction to subfactors}, vol. 234,
  Cambridge University Press, 1997.

\bibitem[KLS03]{KLS03}
V.~Kodiyalam, Z.~Landau, and V.S. Sunder, \emph{The planar algebra associated
  to a {K}ac algebra}, Proc. Indian Acad. Sci. Math. Sci. \textbf{113} (2003),
  15--51.

\bibitem[Lan02]{Lan02}
Z.~Landau, \emph{Exchange relation planar algebras}, Geometriae Dedicata
  \textbf{95} (2002), 183--214.

\bibitem[MPPS12]{ind52}
S.~Morrison, D.~Penneys, E.~Peters, and N.~Snyder, \emph{Subfactors of index
  less than 5, part 2: triple points}, International Journal of Mathematics
  \textbf{23(3)} (2012).

\bibitem[MPS10]{MPSD2n}
S.~Morrison, E.~Peters, and N.~Snyder, \emph{Skein theory for the ${D}_{2n}$
  planar algebras}, Journal of Pure and Applied Algebra \textbf{214} (2010),
  117--139.

\bibitem[MS12]{ind51}
S.~Morrison and N.~Snyder, \emph{Subfactors of index less than 5, part 1: the
  principal graph odometer}, Comm. Math. Phys. \textbf{117} (2012), 1--35.

\bibitem[Mur87]{Mur87}
J.~Murakami, \emph{The kauffman polynomial of links and representation theory},
  Osaka J. Math. \textbf{24(4)} (1987), 745--758.

\bibitem[Ocn88]{Ocn88}
A.~Ocneanu, \emph{Quantized groups, string algebras and {G}alois theory for
  algebras}, Operator algebras and applications, Vol.\ 2, London Math. Soc.
  Lecture Note Ser., vol. 136, Cambridge Univ. Press, Cambridge, 1988,
  pp.~119--172.

\bibitem[Pet10]{Pet10}
E.~Peters, \emph{A planar algebra construction of the {H}aagerup subfactor},
  International Journal of Mathematics \textbf{21} (2010), 987--1045.

\bibitem[Pop90]{Pop90}
S.~Popa, \emph{Classification of subfactors: reduction to commuting squares},
  Invent. Math. \textbf{101} (1990), 19--43.

\bibitem[Pop94]{Pop94}
\bysame, \emph{Classification of amenable subfactors of type {II}}, Acta Math.
  \textbf{172} (1994), 352--445.

\bibitem[Pop95]{Pop95}
\bysame, \emph{An axiomatization of the lattice of higher relative commutants},
  Invent. Math. \textbf{120} (1995), 237--252.

\bibitem[PT12]{ind54}
D.~Penneys and J.~Tener, \emph{Subfactors of index less than 5, part 4:
  Quadruple points}, International Journal of Mathematics \textbf{23(3)}
  (2012), 18 pages.

\bibitem[Sat97]{Sat97}
N.~Sato, \emph{Fourier transform for paragroups and its application to the
  depth two case}, Publ. Res. Inst. Math. Sci. \textbf{33} (1997), 189--222.

\bibitem[SV93]{SunVij}
V.S. Sunder and A.K. Vijayarajan, \emph{On the nonoccurrence of the coxeter
  graphs $\beta_{2n+1}$, ${D}_{2n+1}$ and ${E}_7$ as the principal graph of an
  inclusion of {II}$_1$ factors}, Pacific J. Math. \textbf{161} (1993),
  185--200.

\bibitem[Wen87]{Wen87}
H.~Wenzl, \emph{On sequences of projections}, C. R. Math. Rep. Acad. Sci.
  Canada \textbf{9(1)} (1987), 5--9.

\bibitem[Wen98]{Wen}
\bysame, \emph{Quantum groups and subfactors of type {B}, {C} and {D}}, J.
  Amer. Math. Soc. \textbf{11} (1998), 455--487.

\end{thebibliography}
\bibliographystyle{amsalpha}

\end{document}